\documentclass[reqno,english]{amsart}
\usepackage{amsfonts,amsmath,latexsym,verbatim,amscd,mathrsfs,color,array}
\usepackage{amssymb,amsthm,graphicx }
\usepackage{amsfonts}
\usepackage[colorlinks=true]{hyperref}
\usepackage[normalem]{ulem}

\usepackage[hmargin=2.3cm, vmargin=2.3cm]{geometry}
\usepackage{bbm}
\usepackage{pdfsync}
\usepackage{epstopdf}
\usepackage{cite}
\usepackage{graphicx}
\usepackage{multirow}
\usepackage[font=bf,aboveskip=15pt]{caption}
\usepackage[toc,page]{appendix}
\usepackage[colorlinks=true]{hyperref}
\usepackage{bookmark}
\DeclareFontShape{OT1}{cmr}{bx}{sc}{<-> cmbcsc10}{}
\usepackage[square,numbers]{natbib}

\usepackage{float}
\usepackage{ulem}
\usepackage{syntonly}
\usepackage{mathtools}
\usepackage{bm}
\usepackage{amsfonts,amsmath,latexsym,verbatim,amscd,mathrsfs,color,array}
\usepackage[colorlinks=true]{hyperref}

\usepackage{amsmath,amssymb,amsthm,amsfonts,graphicx,color}
\usepackage{amssymb}
\usepackage{epstopdf}

\allowdisplaybreaks

\newcommand{\1}{\mathbf{1}}

\newcommand{\R}{ {\mathbb{R}}}

\newcommand{\pp}{ {\partial} }

\newcommand{\RR}{{{\mathbb R}}}

\newcommand{\cc}{ {\mathbf c}}

\newcommand{\cuad}{{\sqcap\kern-.68em\sqcup}}

\newcommand{\be}{\begin{equation}}
	\newcommand{\ee}{\end{equation}}

\newcommand{\la}{\lambda}

\newcommand{\dx}{\; dx}

\newcommand{\inn}{{\quad\mbox{in } }}

\renewcommand{\div}{\mathop{\rm div}}
\newcommand{\curl}{\mathop{\rm curl}}

\theoremstyle{plain}
\newtheorem{theorem}{Theorem}[section]
\newtheorem{lemma}[theorem]{Lemma}
\newtheorem{prop}[theorem]{Proposition}

\newtheorem{remark}{Remark}[theorem]
\newcommand{\bremark}{\begin{remark} \em}
	\newcommand{\eremark}{\end{remark} }

\numberwithin{equation}{section}

\title[$H$-system and its heat flow]{Finite-time singularity formation for the heat flow of the $H$-system}
\author[Y. Sire]{Yannick Sire}
\address{\noindent Department of Mathematics, Johns Hopkins University, 3400 N. Charles Street, Baltimore, MD 21218, USA}
\email{ysire1@jhu.edu}
\author[J. Wei]{Juncheng Wei}
\address{\noindent Department of Mathematics, University of British Columbia, Vancouver, B.C., V6T 1Z2, Canada}
\email{jcwei@math.ubc.ca}
\author[Y. Zheng]{Youquan Zheng}
\address{\noindent School of Mathematics, Tianjin University, Tianjin 300072, P. R. China}
\email{zhengyq@tju.edu.cn}
\author[Y. Zhou]{Yifu Zhou}
\address{\noindent
School of Mathematics and Statistics, Wuhan University, Wuhan 430072, China}
\email{yifuzhou@whu.edu.cn}
\begin{document}
\begin{abstract}
We construct the first example of finite time blow-up solutions for the heat flow of the $H$-system, describing the evolution of surfaces with constant mean curvature
\begin{equation*}
\begin{cases}
u_t = \Delta u - 2u_{x_1}\wedge u_{x_2}~&\text{ in }~\R^2\times\R_+,\\
 u(\cdot, 0) = u_0~&\text{ in }~\R^2,
  \end{cases}
\end{equation*}
where $u$: $\R^2\times\R_+\to \R^3$.
The singularity at finite time forms as a scaled least energy $H$-bubble, denoted as $W$, exhibiting type II blow-up speed. One key observation is that the linearized operators around $W$ projected onto $W^\perp$ and in the $W$-direction are in fact  decoupled. On $W^\perp$, the linearization is the linearized harmonic map heat flow, while in the $W$-direction, it is the linearized Liouville-type flow. Based on this, we also prove the non-degeneracy of the $H$-bubbles with any degree.
\end{abstract}
\maketitle

{
  \hypersetup{linkcolor=black}
  \tableofcontents
}

\section{Introduction and main results}

\medskip

A classical problem in geometric analysis is the following {\sl Plateau} problem: for a given curve $\Gamma$ , find a surface with boundary $\Gamma$, with mean curvature $H(x)$ for a point $x$ on the surface, where $H$ is some (smooth) function.  In the case of the ambient space $\R^3$,  a  parametric surface with prescribed mean curvature, satisfies the following equation, also known as the $H$-surface system
\begin{equation}\label{eqn-Hsys}
\Delta u-2H(u) u_{x_1}\wedge u_{x_2}=0 ~\mbox{ on } D,
\end{equation}
where $u$: $D\subset \R^2\to\R^3$ with $D$ being the unit disk, $H$ is a given scalar, ``$\wedge$'' denotes the wedge product, and, for instance, $u_{x_1} = \frac{\partial }{\partial x_1}u$. The geometric significance of system \eqref{eqn-Hsys} is that conformal solutions $u$, i.e. solutions satisfying additionally,
$$
|u_{x_1}|^2-|u_{x_2}|^2=u_{x_1}\cdot u_{x_2}=0 ~\mbox{ on } D,
$$
 parameterize immersed 2D disk-type surfaces of prescribed mean curvature $H$.  Solutions of \eqref{eqn-Hsys} may arise as ``soap bubbles'', that is, surfaces of least area enclosing a given volume. Concerning the existence and optimal estimates, the Dirichlet problem of the $H$-system was studied intensively in many seminal works, such as Heinz \cite{Heinz1954}, Hildebrandt \cite{Hildebrandt1969,Hildebrandt1970}, Gulliver-Spruck \cite{Spruck1971,Spruck1972}, Steffen \cite{Steffen1976-1,Steffen1976-2} and Wente \cite{Wente1969}. Struwe \cite{StruweActa} considered the Plateau problem of the $H$-system and proved its existence; see also Duzaar-Steffen \cite{DuzaarSteffen}. For more geometric motivations and backgrounds, we refer to the comprehensive monographs of Struwe \cite{Struwe1988,Struwebook}, Duzaar-Steffen \cite{DuzaarSteffen1}, Steffen \cite{Steffen1}, Bethuel-Caldiroli-Guida \cite{Bethuel1} and their references.

\medskip

For the Dirichlet problem in a smooth bounded domain $\Omega\subset \R^2$
\begin{equation}\label{stationary-H-system}
\begin{cases}
\Delta u =  H(x, u, \nabla u)u_{x_1}\wedge u_{x_2}~&\text{ in }~\Omega,\\
 u(\cdot) = \tilde g~&\text{ on }~\partial\Omega
  \end{cases}
\end{equation}
with the scalar $H$ being smooth, it was proved by Wente \cite{Wente1975} and Chanillo-Malchiodi \cite{chanillomalchiodi2005cagasymptotic} that there is no nonzero solution in simply connected domain when $\tilde g = 0$. For $H(x, u, \nabla u) = 2$ and $\|\tilde g\|_{L^\infty} < 1$, a solution with minimal energy was constructed by Hildebrandt in \cite{Hildebrandt1970},  while in Brezis-Coron \cite{BrezisCoroncpam}, Steffen \cite{SteffenARMA}, Struwe \cite{Struwe1988}, the authors considered large energy solutions, and it was proved by Heinz \cite{HeinzARMA} that the condition $\|\tilde g\|_{L^\infty} < 1$ is sharp. For general $H(x, u, \nabla u)$, the existence of solutions was proved in a series of works by Caldiroli-Musina \cite{CaldiroliMusinaCCM,Caldiroli2004,CaldiroliMusinaRMI,MusinaJAM2004} via the variational perturbative method introduced by Ambrosetti and Badiale in \cite{AmbrosettiBadiale}. Furthermore, bubbling and multi-bubble solutions have been constructed in Caldiroli-Musina
\cite{caldirolimusinaduke2004} and Chanillo-Malchiodi \cite{chanillomalchiodi2005cagasymptotic}. Regularity of weak solutions was studied by Musina \cite{MusinaAALLMA}. The asymptotic behavior of the solutions for (\ref{stationary-H-system}) was studied, for instance, in Sasahara
\cite{Yasuhiro1995}, Isobe \cite{Takeshi2000,Takeshi2001a,Takeshi2001b}, Caldiroli-Musina \cite{CaldiroliMusina2006arma,CaldiroliMusinaJFA2007} and the references therein.

\medskip

Rivi\`ere proved in the important work \cite{Riviere-2008} that  two-dimensional conformally invariant nonlinear elliptic PDEs, including the prescribed mean curvature equation and harmonic map equation, can be written in terms of suitable conservation laws. Based on this special compensated-compactness structure,  he proved that the solutions of the prescribed bounded mean curvature equation and the harmonic map equation in any manifolds are continuous, and that critical points of two-dimensional continuously differentiable conformally invariant elliptic Lagrangians of the form
 \begin{equation}\label{H-2-form}
E[u] = \frac{1}{2}\int_{\mathbb{R}^2}|\nabla u|^2dx_1\wedge dx_2 + \int_{\mathbb{R}^2}\omega(u)(u_{x_1}, u_{x_2})dx_1\wedge\dx_2
\end{equation}
  are continuous. Here $\omega$ is a $C^1$ differential two-form on a $C^2$-submanifolds $N^k$ of $\mathbb{R}^m$, and $k$ and $m$ are arbitrary integers satisfying $1\leq k \leq m$.

\medskip

For the $H$-system in $\R^2$ with constant mean curvature $H=1$, finite-energy entire solutions were classified in the classical paper \cite{BrezisCoron} by Brezis-Coron as
\begin{equation*}
u(z) = \Pi\left(\frac{P(z)}{Q(z)}\right)+C, \quad z = (x, y) = x+iy
\end{equation*}
where $\Pi:\mathbb{C}\to \mathbb{S}^2$ is the inverse stereographic projection defined by $$\Pi(z) =\frac{1}{1+|z|^2}\begin{bmatrix}
2z\\
|z|^2-1
\end{bmatrix},$$ $P$, $Q$ are polynomials and $C$ is a constant vector in $\mathbb{R}^3$.
From this result we know that a typical class of solutions to
\begin{equation}\label{eqn-H=1}
\Delta u-2 u_{x_1}\wedge u_{x_2}=0 ~\mbox{ in }~\R^2
\end{equation}
 are $W^{(m)}(z) = \Pi\left(z^m\right)$ for $m\in \mathbb{Z}^+$, where $m$ is the degree of the map. Here we consider the so called non-degenerate property of {\it the degree $m$ bubble} $W^{(m)}$, which concerns the bounded kernel functions of the linearized equation around $W^{(m)}$ as follows
\begin{equation}\label{e:linearized}
\begin{aligned}
\Delta\phi = 2\left(W^{(m)}_{x_1}\wedge \phi_{x_2} + \phi_{x_1}\wedge W^{(m)}_{x_2}\right).
\end{aligned}
\end{equation}
The non-degeneracy plays an important role in the construction of $H$-bubbles; see Caldiroli-Musina \cite{caldirolimusinaduke2004} for instance. In the polar coordinates, $W^{(m)}$ can be written as
\begin{equation*}
W^{(m)}(x) = W^{(m)}(r, \theta) =
\begin{bmatrix}
e^{im\theta}\sin(w_m)\\
\cos(w_m)\\
\end{bmatrix}
=
\begin{bmatrix}
\frac{2 r^m \cos (m\theta)}{r^{2 m}+1}\\
\frac{2 r^m \sin (m\theta)}{r^{2 m}+1}\\
\frac{r^{2 m}-1}{r^{2 m}+1}
\end{bmatrix}
,\quad x = (r\cos\theta, r\sin\theta)\in \mathbb{R}^2,\quad  m\in \mathbb{Z}_+,
\end{equation*}
and
$$
w_m=\pi-2\arctan(r^m).
$$
The linearized equation (\ref{e:linearized}) then becomes
\begin{equation}\label{e:linearized1}
\begin{aligned}
\Delta\phi = \frac2r\left(W^{(m)}_r\wedge \phi_\theta + \phi_r\wedge W^{(m)}_\theta\right)
\end{aligned}
\end{equation}
for $\phi = \phi(r, \theta)$. For $m=\pm 1$, Chanillo and Malchiodi  \cite{chanillomalchiodi2005cagasymptotic} proved the non-degeneracy result; see also   \cite{Yasuhiro1995,Takeshi2001a,Takeshi2001b,caldirolimusinaduke2004,MusinaJAM2004} and the references therein. Chanillo and Malchiodi further conjectured such non-degeneracy holds true for {\it the degree $m$ bubble} ($|m|\geq 2$). Our first result confirms this:
\begin{theorem}[Non-degeneracy of $H$-bubbles]\label{thm}
The solution to \eqref{eqn-H=1}
\begin{equation*}
W^{(m)}(x) = W^{(m)}(r, \theta) =
\begin{bmatrix}
\frac{2 r^m \cos (m\theta)}{r^{2 m}+1}\\
\frac{2 r^m \sin (m\theta)}{r^{2 m}+1}\\
\frac{r^{2 m}-1}{r^{2 m}+1}
\end{bmatrix}
,\quad x = (r\cos\theta, ~r\sin\theta)\in \mathbb{R}^2,\quad  m\in \mathbb{Z}^+
\end{equation*}
is non-degenerate in the sense that all bounded solutions of the linearized equation (\ref{e:linearized1}) are linear combinations of $4m+5$ functions defined as follows,
\begin{equation}\label{thm-kernels}
\begin{aligned}
&\frac{r^{2m} - 1}{r^{2m} + 1}W^{(m)}, \quad \frac{r^m}{1 + r^{2m}}\cos (m\theta)W^{(m)},\quad \frac{r^m}{(1 +r^{2m})}\sin (m\theta)W^{(m)},\\
&\frac{r^{m-k}}{1+r^{2m}}\left(\cos (k\theta)E_1^{(m)} + \sin (k\theta)E_2^{(m)}\right), \quad \frac{r^{m-k}}{1+r^{2m}}\left(\sin (k\theta)E_1^{(m)} - \cos (k\theta)E_2^{(m)}\right),\\
&\frac{r^{m+l}}{1+r^{2m}}\left(\cos (l\theta)E_1^{(m)}-\sin (l\theta)E_2^{(m)}\right), \quad \frac{r^{m+l}}{1+r^{2m}}\left(\sin (l\theta)E_1^{(m)} + \cos (l\theta)E_2^{(m)}\right)
\end{aligned}
\end{equation}
for $k=0, 1, \cdots, m$ and $l= 1, \cdots, m$. Here \begin{equation*}
E_1^{(m)} = E_1^{(m)}(r, \theta) =
\begin{bmatrix}
\frac{r^{2m}-1}{r^{2m}+1}\cos(m\theta)\\\frac{r^{2m}-1}{r^{2m}+1}\sin(m\theta)\\ \frac{-2r^m}{r^{2m}+1}
\end{bmatrix}
,\quad
E_2^{(m)} = E_2^{(m)}(r, \theta) =
\begin{bmatrix}
-\sin(m\theta)\\ \cos(m\theta)\\ 0
\end{bmatrix}.
\end{equation*}
\end{theorem}
\begin{remark}
For $m\leq -1$, $W^{(m)}(z) = \Pi'\left({\bar z}^{|m|}\right)$ with
$$\Pi'(z) =\frac{1}{1+|z|^2}\begin{bmatrix}
2z\\
1-|z|^2
\end{bmatrix},
$$
and the same non-degeneracy result holds with $m$ replaced by $|m|$.
\end{remark}
Note that the method in \cite{chanillomalchiodi2005cagasymptotic} for the case $m=\pm 1$ depends on the spectrum of $\Delta_{\mathbb{S}^2}$ and the spherical decomposition on $L^2(\mathbb{S}^2)$. In \cite{MusinaJAM2004}, Musina further studied the role of the spectrum of the Laplace operator on $\mathbb{S}^2$ in the $H$-bubble problem. Here our proof of Theorem \ref{thm} is based on a decoupling property of the linearized operator and an ODE argument.

\medskip

A key observation we make is that the linearized equations on $[W^{(m)}]^\perp$ and in the $W^{(m)}$-direction are decoupled, see the decompositions in Section \ref{appendixB} and Lemma \ref{linearization-lemma-0} (for $m=1$). On $[W^{(m)}]^\perp$, the linearized problem can be solved in a similar manner as the one for harmonic maps, see \cite[Corollary 1]{chenliuwei}. Indeed, the kernel functions \eqref{thm-kernels} in Theorem \ref{thm} are the same as those for the linearized harmonic map equation except for the first three functions parallel to $W^{(m)}$ at each point. Due to the critical growth in the nonlinearity, the $H$-system \eqref{eqn-H=1} shares some similarities with the harmonic map equation with target manifold being $\mathbb S^2$
$$
\Delta u+|\nabla u|^2u=0 ~\mbox{ in }~\R^2.
$$
In fact, \eqref{eqn-H=1} is the Euler-Lagrange equation of the energy functional
\begin{equation}\label{Henergy}
E_H[u]=\frac12\int_{\R^2 } |\nabla u|^2 +\frac23 \int_{\R^2 } u\cdot(u_{x_1}\wedge u_{x_2}):=E_D[u]+E_V[u],
\end{equation}
that is scaling invariant. Here, $E_D[u]$ is the usual Dirichlet energy that appears in the context of harmonic maps, and $E_V[u]$ may be referred to as the $H$-volume functional of the surface parametrized by $u$.

\medskip

Non-degeneracy of harmonic maps proved in \cite{chenliuwei} is a rather important property in the study of singularity formation for the harmonic map heat flow and the quantitative stability of harmonic maps, see \cite{17HMF,DengSunWeiharmonicmaps} and the references therein. In a related context, we refer the interested readers to \cite{LenzmanSchikorra,SireWeiZheng,DengSunWei} for the non-degeneracy of half harmonic maps. Inspired by the works \cite{chenliuwei,DengSunWeiharmonicmaps} on harmonic maps, it might be natural to expect the non-degeneracy of general bubble $u(z) = \Pi\left(\frac{P(z)}{Q(z)}\right)+C$ and the locally quantitative stability; see also \cite{rupflin}.

\medskip

The previous discussion, though important for our purposes, concerns the stationnary problem. In this paper, we focus on the geometric flow that describes the evolution of parametric surfaces with constant mean curvature. It is the associated heat flow of the $H$-system
\begin{equation}\label{e:main}
\begin{cases}
u_t = \Delta u - 2u_{x_1}\wedge u_{x_2}~&\text{ in }~\R^2\times\R_+,\\
 u(\cdot, 0) = u_0~&\text{ in }~\R^2,
  \end{cases}
\end{equation}
where $u(x,t)=u(x_1,x_2,t):\R^2\times \R_+\to \R^3$, and $u_0 :\R^2\to \R^3$ is a given smooth map. The system (\ref{e:main}) is the negative $L^2$-gradient flow of the energy \eqref{Henergy}, and its stationary equation, namely $H$-system, is the equation satisfied by surfaces of mean curvature $H=1$ in conformal representation.
Global existence and regularity for weak solutions of the initial-boundary value problem to the heat flow of the $H$-system were established by Rey in \cite{Rey1991heatflow}. Partial regularity of weak solution was studied in the works of Wang \cite{Wangchangyou1997,Wangchangyou1999}. Existence and uniqueness of short time regular solution were proved by Chen-Levine \cite{ChenLevine}, where they also analyzed the bubbling phenomenon at the first singular time. In \cite{Duzaar2013TAMS,Duzaar2015IMRN}, B\"ogelein-Duzaar-Scheven showed short-time regularity, and the global existence of weak solution which is regular except from finitely many singular times; see Duzaar-Scheven \cite{Duzaar2015AIHP} for the global existence for a Plateau problem. 

\medskip

We are interested in the formation of singularity at finite time for the heat flow \eqref{e:main}.  Even though it is known that the heat flow for $H$-system shares many similarities with harmonic map flows (and there are now plenty of examples of finite time singularities for harmonic map flows, e.g. \cite{17HMF} and references therein), there is no example  of solutions to \eqref{e:main} which exhibit singularity formation at finite time, even in the equivariant case. The aim of this paper is to investigate the blow-up mechanism of the system \eqref{e:main} and its precise asymptotics, and we will prove the first instance of finite time blowup for \eqref{e:main}. Though the elliptic theory for such systems share many similarities with the one for harmonic maps, the geometric flow exhibits behaviour which are surprising and make the analysis substantially harder. As alluded before the splitting of the linearization between directions behaving like the linearization of the harmonic map heat flow around a bubble and the one sharing strong similarities with a Liouville flow (i.e. an exponential nonlinearity) introduces major adjustments to the general strategy due to interaction issues. This behaviour at the linear level, though critical, is nevertheless decoupled which allows to close the argument.

\medskip

The building block of our construction is the following  least energy entire, nontrivial $H$-bubble ($m=1$) given by
\begin{equation}\label{e:H-bubble}
W(x) = \frac{1}{1+|x|^2}\left[
\begin{matrix}
2x\\
|x|^2-1
\end{matrix}
\right], \quad x\in \mathbb{R}^2,
\end{equation}
where we write $W=W^{(1)}$, $E_1=E_1^{(1)}$ and $E_2=E_2^{(1)}$ for simplicity. $W$ satisfies
\begin{equation*}
\int_{\mathbb{R}^2}|\nabla W|^2 = 8\pi, \quad W(\infty) =(0,0,1)^{\mathsf T}.
\end{equation*}
Note that $W(x)$ can be expressed as the $1$-equivariant form
\begin{equation}\label{corotationalform}
W(x) = \left[
\begin{matrix}
e^{i\theta}\sin w(r)\\
\cos w(r)
\end{matrix}
\right], \,\,w(r) = \pi - 2\arctan(r),\,\, x = re^{i\theta}.
\end{equation}
Define $\gamma $-rotation matrix around $z$-axis as
\begin{equation}\label{Matrix}
	Q_{\gamma }:=
	\begin{bmatrix}
		\cos\gamma  & -\sin\gamma  & 0 \\[0.3em]
		\sin \gamma  & \cos\gamma   & 0 \\[0.3em]
		0 & 0 & 1
	\end{bmatrix}.
\end{equation}

\medskip

Our main result is stated as follows.
\begin{theorem}[Type II finite-time blow-up]\label{t:main}
For any point $q\in\R^2$ and sufficiently small $T > 0$, there exists $u_0$ such that $\nabla_x u(x, t)$ with $u(x,t)$ solving (\ref{e:main}) blows up at $q$ as $t\to T$. More precisely, there exist $\kappa \in\R_+$, $\gamma _*\in\R$, and a map $u_*\in H^1_{{\rm loc}}(\R^2;\R^3)\cap L^{\infty}(\R^2;\R^3)$ such that
\begin{equation*}
u(x, t) - u_*(x) - Q_{\gamma _*}\left[W\left(\frac{x-\xi(t)}{\lambda(t)}\right)-W(\infty)\right]\to 0~\text{ as }~t\to T
\end{equation*}
in $H^1_{{\rm loc}}(\R^2;\R^3)\cap L^{\infty}(\R^2;\R^3)$ with
\begin{equation*}
\lambda(t) = \kappa \frac{T-t}{|\log(T-t)|^2}(1 + o(1)), \quad \xi(t)=q+o(1),
\end{equation*}
where $o(1)\to 0$ as $t\to T$. In particular, we have
\begin{equation*}
|\nabla u(\cdot, t)|^2dx \overset{\ast}{\rightharpoonup} |\nabla u_*|^2dx+ 8\pi\delta_q~\text{ as }~t\to T
\end{equation*}
in the sense of Radon measures.
\end{theorem}

\begin{remark}
\noindent
\begin{itemize}
\item The fine asymptotics of $u$ is obtained in the construction:
$$
u(x,t)=Q_{\gamma (t)}W\left(\frac{x-\xi(t)}{\lambda(t)}\right)+\Phi_{{\rm per}},
$$
where the perturbation $\Phi_{{\rm per}}$ is small in the sense that
$|\Phi_{{\rm per}}|\lesssim T^{\epsilon_1}, ~ |\nabla_x \Phi_{{\rm per}}|\lesssim \la^{\epsilon_2}(t)$
for some $\epsilon_1>0$ and $-1<\epsilon_2<0$. More precise asymptotic behavior can be found in Section \ref{gluingsystem}.

\medskip

\item The same construction works for the Cauchy-Dirichlet problem in a smooth bounded domain $\Omega \subset \R^2$:
\begin{equation*}
\begin{cases}
u_t = \Delta u - 2u_{x_1}\wedge u_{x_2}~&\rm{ in }~\Omega\times\R_+,\\
u=(0,0,1)^{\mathsf T}~&\rm{ on }~\pp\Omega\times\R_+,\\
 u(\cdot, 0) = u_0~&\rm{ in }~\Omega
  \end{cases}
\end{equation*}
for given smooth map $u_0$ with $u_0|_{\pp\Omega}=(0,0,1)^{\mathsf T}$. See \cite{WangHS11} for related existence results as well as blow-up and regularity criteria for \eqref{e:main} with Dirichlet boundary. It is worth mentioning that with such Dirichlet boundary, the solutions to \eqref{e:main} are in fact different from those discussed in \cite{WangHS11}, where zero boundary condition was imposed.

\medskip

\item Motivated by \cite{17HMF} concerning the blow-up for the harmonic map heat flow, it is reasonable to expect that the finite time singularities might also happen in the gradient flow for the general energy functional \eqref{H-2-form}, and different blow-up mechanism may arise depending on the 2-form $\omega$ presumably.

\medskip

\item Bubbling solutions taking higher degree profile might be possible by virtue of the non-degeneracy in Theorem \ref{thm}. However, the construction is more involved due to the presence of more bounded kernel functions for the linearized operator.
\end{itemize}
\end{remark}

The general framework of the construction is based on the inner-outer gluing method, first developed by Cort\'azar-del Pino-Musso  \cite{ContazarDelPinoMusso} and D\'avila-del Pino-Wei \cite{17HMF}, and it is a versatile approach designed for the singularity formation for evolution PDEs.

\medskip

The heat flow of the $H$-system is intimately related to the harmonic map heat flow
$$
u_t=\Delta u+|\nabla u|^2u ~\mbox{ in }~\R^2\times\R_+
$$
due to their similar quadratic nonlinearities, and our construction is primarily inspired by D\'avila-del Pino-Wei \cite{17HMF} concerning the blow-up for the latter geometric flow. Despite similar criticality and structure shared with harmonic map heat flow, the blow-up mechanism of \eqref{e:main} turns out to be more delicate due to the system in the $W$-direction. In fact, as a consequence of Lemma \ref{linearization-lemma-0}, the elliptic linearization decouples as the linearized harmonic map equation on $W^\perp$ and the linearized Liouville equation in the $W$-direction. Compared to the case of the harmonic map heat flow, the treatment required in the $W$-direction for \eqref{e:main} results in several technical novelties and a somewhat more unstable blow-up beyond the equivariant symmetry class. We shall describe these in what follows.

\medskip

We now describe informally our construction. Recall we are interested in constructing solutions without symmetry assumptions. The major difficulties in the construction arise from the non-radial symmetry and the non-local/global effects, triggered by the slow spatial decay, in {\it both} $W$-direction and on $W^\perp$. Indeed, the ansatz for the desired blow-up solution with equivariant-type symmetry \eqref{ansatz-symm} turns out to be relatively simpler without the introduction of extra modulation parameters $c_1$ and $c_2$ corresponding to the parabolic linearized Liouville equation in the $W$-direction. Due to the non-radial symmetry, the ansatz has to be modified to further improve the slowly decaying errors in non-radial modes, yielding an extra non-local system governing the modes $\pm1$ parameters $c_1$ and $c_2$ in the $W$-direction (cf. the $c_1$-$c_2$ system in Section \ref{sec-lt-c1c2}). Moreover, the $c_1$-$c_2$ system is in fact coupled with another $\la$-$\gamma$ system, governing the scaling parameter $\la$ and the rotational parameter $\gamma$ on $W^\perp$ in a non-local way, and this is a consequence of slowly decaying error in mode $0$ on $W^\perp$. The $\la$-$\gamma$ system was first observed and  derived in the context of the harmonic map heat flow \cite{17HMF}. For \eqref{e:main},  the non-locality of the $c_1$-$c_2$ system and its coupling with the $\la$-$\gamma$ system seem to be a new feature beyond the symmetry class.

\medskip

Another aspect is the use of the distorted Fourier transform for the spectral analysis of mode $0$ in the $W$-direction (i.e., the linearized Liouville equation, cf. Appendix \ref{sec-DFT}). The corresponding kernel is of order $1$ at space infinity which makes the analysis more subtle, and the techniques for linear theories in all the other modes seem not to be sufficient to ensure a solution with sufficient decay. Instead, we use the techniques of the distorted Fourier transform for this specific mode and carry out a spectral analysis for the associated half-line Schr\"odinger operator, see Appendix \ref{sec-DFT}. The motivation is from a series important works of Krieger-Miao-Schlag-Tataru \cite{Krieger08,KST09Duke,krieger2020stability,KMS20WM} in the hyperbolic settings and from a recent one \cite{LLG} in a dissipative-dispersive setting. These techniques might be of use in the study of semilinear elliptic and parabolic equations with exponential-type nonlinearity.

\medskip

\noindent{\bf A roadmap to the construction.}
We elaborate in a bit more detailed fashion below. The first step of our proof is to find a good approximate solution. Let us begin with a simple motivational ansatz to solution $u$ of \eqref{e:main}, taking the co-rotational form similar to \eqref{corotationalform}:
\begin{equation}\label{ansatz-symm}
u(x, t) = \psi(r, t)\left[
\begin{matrix}
e^{i\theta}\sin \varphi(r, t)\\
\cos\varphi(r, t)
\end{matrix}
\right].
\end{equation}
Then system (\ref{e:main}) becomes the following system of 1D evolution equations
\begin{equation}\label{sys-cor}
\left\{
\begin{aligned}
&\varphi_t = \varphi_{rr}+\frac{\varphi_r}{r}+\frac{(-\sin\varphi+2r\varphi_r)\psi_r}{r\psi}-\frac{\sin(2\varphi)}{2r^2},\\
&\psi_t = \psi_{rr} + \frac{\psi_r}{r} -\frac{2\psi^2\varphi_r\sin\varphi}{r}+\frac{-1+\cos(2\varphi)-2r^2\varphi_r^2}{2r^2}\psi.
\end{aligned}
\right.
\end{equation}
One can check directly that $(\pi-2\arctan(r), 1)$ is a stationary solution. Then a natural approximation solution is
\begin{equation*}
\varphi_0(r, t) = \pi-2\arctan\left(\frac{r}{\lambda(t)}\right) \text{ and }\psi_0(r, t) = 1.
\end{equation*}
The linearized system of \eqref{sys-cor} around $\left(\varphi_0,~\psi_0\right)$ is given by
\begin{equation*}
\left\{
\begin{aligned}
&(\phi_1)_t = (\phi_1)_{rr}+\frac{(\phi_1)_r}{r} -\frac{r^4-6r^2\lambda^2+\lambda^4}{r^2(r^2+\lambda^2)^2}\phi_1-\frac{6\lambda}{r^2+\lambda^2}(\phi_2)_r,\\
&(\phi_2)_t = (\phi_2)_{rr} + \frac{(\phi_2)_r}{r} + \frac{8\lambda^2}{(r^2+\lambda^2)^2}\phi_2
\end{aligned}
\right.
\end{equation*}
for perturbations $\phi_1$ and $\phi_2$ of $\varphi_0$ and $\psi_0$, respectively. Note the second equation is the linearization for the 2D Liouville equation
$$\Delta u+e^u=0$$
around the  bubble $\log\frac{8\lambda^2}{(r^2+\lambda^2)^2}$.

\medskip

Let us emphasize that the ansatz in the motivational example above enjoys the equivariant symmetry, but beyond such class, the ansatz has to be modified significantly due to the non-radial modes.  Inspired by the simple ansatz \eqref{ansatz-symm} and the non-degeneracy in Theorem \ref{thm}, we will find a suitable profile that approximates well the real solution to \eqref{e:main} and then further improve it by adding several corrections. First we define the error of $u$ as
\begin{equation*}\label{e:error}
S[u]:= -u_t + \Delta u - 2u_{x_1}\wedge u_{x_2}.
\end{equation*}
We take the approximate solution as
\begin{equation*}\label{def-approx}
\begin{aligned}
U_{\la,\gamma,\xi,c_1,c_2}
&:=~Q_{\gamma(t)} \left[W\left(\frac{x-\xi(t)}{\lambda(t)}\right)+c_1(t)\frac{2\rho}{\rho^2+1}\cos\theta  W\left(\frac{x-\xi(t)}{\lambda(t)}\right) + c_2(t) \frac{2\rho}{\rho^2+1}\sin\theta  W\left(\frac{x-\xi(t)}{\lambda(t)}\right)\right]\\
&:=~Q_{\gamma}(W+c_1\mathcal Z_{1,1}+c_2\mathcal Z_{1,2}),
\end{aligned}
\end{equation*}
where $Q_\gamma$ is given in \eqref{Matrix}, and the modulation parameters $\la(t),~\xi(t),~\gamma(t),~c_1(t),~c_2(t)$ are to be adjusted. The most notable difference in the ansatz compared to the radial case is the introduction of the new parameters $c_1$, $c_2$, which in fact correspond to the modes $\pm1$ in the linearized Liouville problem in the $Q_\gamma W$-direction. Indeed, $\mathcal Z_{1,1}$ and $\mathcal Z_{1,2}$ are the corresponding kernels given in \eqref{kernels}.

\medskip

As we will see in Section \ref{approximateandimprovement}, the non-radial ansatz $U_{\la,\gamma,\xi,c_1,c_2}$ generates slowly decaying errors in both $W$-direction and on $W^\perp$, and thus needs to be improved by adding several non-local corrections. We choose the corrected approximate solution as
\begin{equation*}\label{def-U_*}
\begin{aligned}
U_*:=&~U_{\la,\gamma,\xi,c_1,c_2}+\eta_1\Big(\Phi^{(0)}+\Phi^{(1)}+\Phi^{(2)}_{U^\perp}+\Phi^{(-2)}_{U^\perp}+\Phi^{(2)}_U\Big),
\end{aligned}
\end{equation*}
where the purpose of the cut-off function
$
\eta_1:=\eta(x-\xi(t))
$
is to avoid potential slow spatial decay in the remote region, and the five corrections all have their own role. Here the corrections $\Phi^{(0)}$ and $\Phi^{(1)}$, in non-local form, are to improve slow spatial decay in mode $0$ on $[Q_\gamma W]^\perp$ and in modes $\pm1$ in the $Q_\gamma W$-direction. The other three corrections $\Phi^{(\pm 2)}_{U^\perp}$, $\Phi^{(2)}_U$ are to improve slow time decay in modes $\pm2$ on $[Q_\gamma W]^\perp$ and in mode $2$ in the $Q_\gamma W$-direction, respectively. Such bulky ansatz is necessary when designing weighted spaces for the gluing process, especially in the choice of constants measuring these spaces that the desired solutions reside in (see \eqref{finalchoice} in Section \ref{gluingsystem}). In other words, without these corrections, weighted spaces for the perturbation cannot be found in the current set-up. This is one of the most important parts in the construction, and we do not know if there are simpler ansatzes for the approximate solution.

\medskip

We next look for a solution with the following form
$$
u=U_*+\Phi,
$$
so $S[u]=0$ yields
$$
\pp_t \Phi=\Delta \Phi-2\pp_{x_1} U_*\wedge \pp_{x_2}\Phi-2\pp_{x_1}\Phi\wedge\pp_{x_2}U_*-2\pp_{x_1}\Phi\wedge \pp_{x_2}\Phi+S[U_*].
$$
To implement the inner-outer gluing procedure, we decompose $\Phi$ into
$$
\Phi(x,t)= \eta_R Q_\gamma  \Big[\Phi_W\left(y, t\right) +\Phi_{W^\perp}\left(y, t\right)\Big] + \Psi(x, t), \quad y=\frac{x-\xi(t)}{\lambda(t)},
$$
where $\Phi_W$ is in the $W$-direction, $\Phi_{W^\perp}$ is on $W^\perp$,
$
\eta_R := \eta\left(\frac{x-\xi(t)}{R(t)\lambda(t)}\right),
$
and $R(t)=\la^{-\beta}(t)$ with $\beta>0$ to be chosen later.
Then it is sufficient to find a desired solution $u$ to \eqref{e:main} if the triple $(\Phi_W,~\Phi_{W^\perp},~\Psi)$ satisfies the coupled gluing system:
\begin{align} \label{gluing-sysa}
\la^2\pp_t \Phi_W=&~\Delta_y\Phi_{W } - 2\pp_{y_1}W\wedge \pp_{y_2}\Phi_{W } - 2\pp_{y_1} \Phi_{W }\wedge \pp_{y_2}W+{\rm RHS}_{{\rm in},W} \quad\mbox{  in  }~~ \mathcal{D}_{2R},\\
\label{gluing-sysb}
\la^2\pp_t \Phi_{W^\perp}=&~\Delta_y\Phi_{W^\perp} - 2\pp_{y_1}W\wedge \pp_{y_2}\Phi_{W^\perp} - 2\pp_{y_1} \Phi_{W^\perp}\wedge \pp_{y_2}W+{\rm RHS}_{{\rm in},W^\perp}\quad\mbox{  in  }~~ \mathcal{D}_{2R},\\
\label{gluing-sysc}
\pp_t \Psi=&~\Delta_x \Psi + {\rm RHS}_{{\rm out}}  \mbox{  in  }~~ \R^2 \times(0,T),
\end{align}
i.e., $\Phi_W$ and $\Phi_{W^\perp}$ satisfy inner problem in the $W$-direction and on $W^\perp$, respectively, and $\Psi$ solves the outer problem. The detailed form of the right hand sides ${\rm RHS}_{{\rm in},W}$, ${\rm RHS}_{{\rm in},W^\perp}$ and ${\rm RHS}_{{\rm out}}$ with couplings will be given in Section \ref{gluingsystem-innerouter}.

\medskip

To attack the inner-outer gluing system \eqref{gluing-sysa}-\eqref{gluing-sysc}, a key property derived in Lemma \ref{linearization-lemma-0} is that the equation (\ref{gluing-sysb}) for $\Phi_{W^\perp}$ is essentially a perturbation of the linearized harmonic map heat flow around the bubble $W$, while the equation (\ref{gluing-sysa}) for $\Phi_W$ is a perturbation of the linearized Liouville-type flow, i.e., linearization of
$$
u_t=\Delta u+e^u ~\mbox{ in }~\R^2\times \R_+
$$
around the canonical bubble $\log\frac{8\lambda^{-2}}{(1+|y|^2)^2},$
and these two linearized operators are in fact decoupled. This remarkable structure allows us to develop the linear theories separately in the $W$-direction and on $W^\perp$, and the full nonlinear systems for $\Phi_{W^\perp}$ and $\Phi_W$ are weakly coupled provided their weighted spaces are properly designed. We will give the full linear theory in Section \ref{lineartheories}, and the gluing system will be solved in Section \ref{gluingsystem}.

\medskip

As in the ansatz $U_{\la,\gamma,\xi,c_1,c_2}$, we introduce two modulation parameters $c_1$ and $c_2$ that correspond to the slowly decaying kernels $\mathcal Z_{1,1}$, $\mathcal Z_{1,2}$ (Fourier modes $\pm1$) for the linearized Liouville operator. The non-local corrections $\Phi^{(0)}$ and $\Phi^{(1)}$, taking care of the slowly decaying errors respectively on $[Q_\gamma W]^\perp$ and in the $Q_\gamma W$-direction, in turn  yield two non-local systems that govern the blow-up dynamics of $\la$-$\gamma$ and $c_1$-$c_2$ through the orthogonality conditions at corresponding modes. For the $\la$-$\gamma$ system, the influence of $c_1$, $c_2$ turns out to be a perturbation. However, in the $c_1$-$c_2$ system, the coupling from $\la$ and $\gamma$ in fact serves as one of the leading parts, and these parameters obey
\begin{equation}\label{integral-differentialequation}
\begin{aligned}
&\int_{-T}^{t-\la^2(t)}\frac{p_1(s)}{t-s}ds +\frac23\left(\int_{-T}^{t-\la^2(t)}\frac{{\rm Re}[p_0(s) e^{-i\gamma(t)}]}{t-s}ds\right)  \cc=f(t),\\
&\quad p_1(t) = -2(\lambda \cc)', \quad  \cc(t) = c_1(t) + i c_2(t), \quad p_0(t) = -2(\dot\lambda + i\lambda\dot\gamma )e^{i\gamma }
\end{aligned}
\end{equation}
for some $f(t)\to 0$ as $t\to T$. Our strategy is to approximate \eqref{integral-differentialequation} by the local dynamics \eqref{eqn-532532532}, where the second term involving $\mathbf Z$ can be regarded as the main contribution from the $\la$-$\gamma$ system, and the role of the initial data can be seen; see Section \ref{leading-dynamics}. We also observe from here that, compared to the one for the heat flow of harmonic map, the blow-up for \eqref{e:main} seems to be somewhat more unstable beyond equivariant symmetry class due to the restrictive assumptions on the initial data. The solvability of the full problem (\ref{integral-differentialequation}) is more subtle due to the double integro-differential operators, and its resolution will be done by linearization in Section \ref{sec-lt-c1c2}.

\medskip

The rest of this paper is organized as follows. In Section \ref{Sec-notations}, notations and necessary formulas for the linearized operators will be given. In Section \ref{appendixB}, we prove Theorem   \ref{thm}. This part is of independent interest. In Section \ref{approximateandimprovement}, we discuss the approximate solution for the desired blow-up and its improvement by introducing a couple of corrections, where some technical analysis is postponed to Appendix \ref{appendixA}. In Section \ref{gluingsystem-innerouter}, we make the final ansatz and formulate the gluing system. Asymptotics of the modulation parameters will be derived in Section \ref{leading-dynamics}. In Section \ref{lineartheories}, linear theories for the outer problem and the inner problems in the $W$-direction and on $W^\perp$ will be developed, where the spectral analysis and the pointwise control for the mode $0$ in the $W$-direction via distorted Fourier transform will be carried out in Appendix \ref{sec-DFT}. The full gluing system will be solved in Section \ref{gluingsystem}. In Section \ref{sec-lt-c1c2}, the linear theory for the $c_1$-$c_2$ system will be established.

\bigskip

\section{Notations and preliminaries}\label{Sec-notations}

\medskip

We first list in this section notations, useful properties and formulas for the linearized operator of the $H$-system around the $H$-bubble. Our building block is
$W\left(\frac{x-\xi}{\la}\right)$ given in \eqref{e:H-bubble} with $y=\frac{x-\xi}{\la},$
 and
\begin{equation*}
E_1(y) = \left[
\begin{matrix}
e^{i\theta}\cos w(\rho)\\
-\sin w(\rho)
\end{matrix}
\right],\quad E_2(y) =
\left[
\begin{matrix}
ie^{i\theta}\\
0
\end{matrix}
\right]
\end{equation*}
form a Frenet basis associated to $W$. Here
$$w(\rho) = \pi - 2\arctan(\rho), \quad y = \rho e^{i\theta},$$
so
\begin{equation*}
w_\rho = -\frac{2}{\rho^2+1},\quad \sin w =-\rho w_\rho = \frac{2\rho}{\rho^2+1}, \quad \cos w = \frac{\rho^2-1}{\rho^2+1}.
\end{equation*}

\noindent $\bullet$ {\bf Notations.}

\begin{itemize}
\item A map $f$ is said to be in the $W$-direction if there exists a scalar $c$ such that $f=cW$. Similarly, $f$ is said to be on $W^\perp$ (or $f\in W^\perp$) if $f\cdot W=0$.
\item The complex form of a map $f=f_1 E_1+f_2 E_2$ on $W^\perp$ is defined by
$$
(f)_{\mathbb C}=f_1 +if_2.
$$
\item We write
$$
f:=\Pi_W[f]+\Pi_{W^\perp}[f]~\mbox{ for }~\Pi_{W^\perp}[f]:=f-(f\cdot W)W.
$$
\item
A map $f$ is said to be in mode $k$ on $W^\perp$ if $\Pi_{W^\perp}[f]$ can be written as
$$
\Pi_{W^\perp}[f]={\rm Re}(f_k(r)e^{ik\theta})E_1+{\rm Im}(f_k(r)e^{ik\theta})E_2,\quad x=r e^{i\theta}.
$$
\item
The notation $(\Pi_{W^\perp}[f])_{\mathbb C,j}$ denotes the projection of the complex form of $\Pi_{W^\perp}[f]$ onto mode $j$.

\item For $x\in \mathbb{R}^2$, $t\leq T$  and admissiable functions $g(x), h(x,t)$, denote
	\begin{equation*}
		\begin{aligned}
			&
			\left(\Gamma_{\R^2}\circ g \right)(x,t) := (4\pi t)^{-1}\int_{\R^2} e^{-\frac{|x-y|^2}{4t}} g(y) dy,\\
			&
			\left( \Gamma_{\R^2} \bullet h \right) (x,t) :=
			\int_{0}^t \int_{\R^2}
			\left[ 4\pi(t-s)\right]^{-1}
			e^{ -\frac{|x-y|^2}{4(t-s)}  } h(y,s) dy ds .
		\end{aligned}
	\end{equation*}
	
\item Denote $\langle y \rangle = \sqrt{1+|y|^2}$ for any $y\in \RR^2$.
	
\item The symbol $``\,\lesssim\,"$ means $``\,\leq\, C\,"$ for a positive constant $C$ independent of $t$ and $T$. Here $C$ might be different from line to line.
\end{itemize}

\medskip

\noindent $\bullet$ {\bf Linearized operator.}

\medskip

The linearization around $W$ reads
\begin{equation}\label{def-linearizedoperator}
L_W[\varphi] := \Delta_y\varphi - 2W_{y_1}\wedge \varphi_{y_2} - 2\varphi_{y_1}\wedge W_{y_2},
\end{equation}
and as a consequence of the non-degeneracy of degree 1 map $W$ proved by Chanillo-Malchiodi \cite{chanillomalchiodi2005cagasymptotic}, all bounded kernel functions of the operator $L_W$ must be linear combinations of
\begin{equation}\label{kernels}
\left\{
\begin{aligned}
&Z_{0,1}(y)=\rho w_{\rho}  E_1(y),\\
&Z_{0,2}(y)=\rho w_{\rho}  E_2(y),\\
&Z_{1,1}(y)=w_{\rho} [\cos \theta E_1(y)+\sin\theta E_2(y)],\\
&Z_{1,2}(y)=w_{\rho} [\sin \theta E_1(y)-\cos\theta E_2(y)],\\
&Z_{-1,1}(y)=\rho^2 w_{\rho} [\cos \theta E_1(y)-\sin \theta E_2(y)],\\
&Z_{-1,2}(y)=\rho^2 w_{\rho} [\sin \theta E_1(y)+\cos \theta E_2(y)],\\
&\mathcal Z_0=\cos w W,\\
&\mathcal Z_{1,1}=\cos \theta \sin w W,\\
&\mathcal Z_{1,2}=\sin \theta \sin w W.\\
\end{aligned}
\right.
\end{equation}
For
\begin{equation*}
U:=Q_{\gamma } W,
\end{equation*}
we also define
\begin{equation*}
L_U[\phi]:=\Delta_x \phi - 2U_{x_1}\wedge \phi_{x_2} - 2\phi_{x_1}\wedge U_{x_2}.
\end{equation*}
Here, $Q_\gamma$ is defined in \eqref{Matrix}. Clearly, one has
$$L_U[Q_{\gamma }\varphi]=\la^{-2}Q_{\gamma }L_W[\varphi].
$$
We now give several useful formulations of $L_W$  acting on $\varphi$ in different forms.

\medskip

\begin{lemma}\label{linearization-lemma-0}
If we set
$$
\Phi(y) =   \phi_1(\rho, \theta)E_1 + \phi_2(\rho, \theta)E_2+\phi_3(\rho, \theta)W
$$
and suppose that $L_W[\Phi] = 0$, then the scalars $\phi_1$, $\phi_2$ and $\phi_3$ should satisfy the following equations
\begin{align}\label{eqn-E_1}
&~\partial_{\rho\rho}\phi_1+\frac1\rho \partial_{\rho}\phi_1+\frac{1}{\rho^2}\partial_{\theta\theta}\phi_1-\frac1{\rho^2}\phi_1+\frac8{(1+\rho^2)^2}\phi_1-\frac{2(\rho^2-1)}{\rho^2(\rho^2+1)}\pp_{\theta}\phi_2=0,\\
\label{eqn-E_2}
&~\partial_{\rho\rho}\phi_2+\frac1\rho \partial_{\rho}\phi_2+\frac{1}{\rho^2}\partial_{\theta\theta}\phi_2-\frac{1}{\rho^2}\phi_2+\frac8{(1+\rho^2)^2}\phi_2+\frac{2(\rho^2-1)}{\rho^2(\rho^2+1)}\pp_{\theta}\phi_1=0,\\
\label{eqn-U}
&~\partial_{\rho\rho}\phi_3+\frac1\rho \partial_{\rho}\phi_3+\frac{1}{\rho^2}\partial_{\theta\theta}\phi_3+\frac{8}{(1+\rho^2)^2}\phi_3=0.
\end{align}
\end{lemma}

\begin{remark}\label{rmk-211211}
\noindent
\begin{itemize}
\item We observe from above Lemma that the the components in $W$-direction and on $W^\perp$ are in fact decoupled under linearization.
\item Notice that the scalar equation \eqref{eqn-U} can be regarded as the linearization of the Liouville equation
$$
\Delta u+e^u=0 ~\mbox{ in }~\R^2
$$
around $\log\frac{8}{(1+\rho^2)^2},$
which is non-degenerate in the sense that the linearization only has following bounded kernels
$$
\frac{\rho^2-1}{\rho^2+1},\quad \frac{2\rho e^{i\theta}}{\rho^2+1}.
$$
See also \eqref{kernels}.
\end{itemize}
\end{remark}

The next three lemmas concern the expansion of
\begin{equation*}
\tilde{L}_U[\Phi]:= -2U_{x_1}\wedge \Phi_{x_2} - 2\Phi_{x_1}\wedge U_{x_2}
\end{equation*}
in the linearization $L_U[\Phi]$, and these will be useful in analyzing the couplings in the gluing system.

\begin{lemma}\label{linearization-lemma-1}
In the polar coordinate system
$$
\Phi(x) = \Phi(r, \theta),\quad x = re^{i\theta},\quad \rho = \frac{r}{\lambda},
$$
the term
$
\tilde{L}_U[\Phi]$
can be expressed as
\begin{equation*}
\begin{aligned}
\tilde{L}_U[\Phi]
 =&~- \frac{2}{\lambda}w_{\rho}\left(\Phi_r\cdot(Q_\gamma   W )\right)Q_\gamma   E_1 +\frac{2}{\lambda r}w_{\rho}(\Phi_\theta\cdot (Q_\gamma  W ))Q_\gamma   E_2\\
 &~ +\frac{2}{\lambda r}w_{\rho}\left[ r(\Phi_r\cdot(Q_\gamma   E_1))-(\Phi_\theta\cdot (Q_\gamma  E_2))  \right] Q_\gamma  W .
\end{aligned}
\end{equation*}
\end{lemma}

\medskip

We consider a $C^1$ function $\Phi:\R^2 \to \mathbb{C}\times \mathbb{R}$, that we express in the form
\begin{equation*}
\Phi(x) = \left[
\begin{matrix}
\varphi_1(x) + i\varphi_2(x)\\
\varphi_3(x)
\end{matrix}
\right].
\end{equation*}
We also denote
\begin{equation*}
\varphi = \varphi_1 + i\varphi_2,\quad \bar{\varphi} = \varphi_1 - i\varphi_2
\end{equation*}
and define the operators
\begin{equation*}
\div \varphi = \partial_{x_1}\varphi_1 + \partial_{x_2}\varphi_2,\quad \text{curl } \varphi = \partial_{x_1}\varphi_2 - \partial_{x_2}\varphi_1.
\end{equation*}

\medskip

\begin{lemma}\label{linearization-lemma-2}
In the polar coordinate system
$$
\Phi(x) = \Phi(r, \theta)=(\varphi_1,~\varphi_2,~\varphi_3)^T,\quad x = re^{i\theta},\quad \rho = \frac{r}{\lambda},
$$
the term $\tilde{L}_U[\Phi]$ can be decomposed as follows
\begin{equation*}
\begin{aligned}
\tilde{L}_U[\Phi] &= \tilde{L}_U[\Phi]_0 + \tilde{L}_U[\Phi]_1 + \tilde{L}_U[\Phi]_2
\end{aligned}
\end{equation*}
with
\begin{equation*}
\begin{aligned}
\tilde{L}_U[\Phi]_0 = \frac{1}{\lambda}\rho w_\rho^2[\div(e^{-i\gamma }\varphi)Q_\gamma  E_1 + {\rm{curl }}(e^{-i\gamma }\varphi)Q_\gamma  E_2]+\frac{1}{\lambda}w_{\rho}^2\div(e^{-i\gamma }\varphi))Q_\gamma  W,
\end{aligned}
\end{equation*}
\begin{equation*}
\begin{aligned}
\tilde{L}_U[\Phi]_1 & = -\frac{2}{\lambda} w_\rho\cos w\Big[[\partial_{x_1}\varphi_3\cos\theta + \partial_{x_2}\varphi_3\sin\theta] Q_\gamma  E_1 + [\partial_{x_1}\varphi_3\sin\theta - \partial_{x_2}\varphi_3\cos\theta]Q_\gamma  E_2\Big] \\ & \quad -\frac{2}{\lambda}w_\rho\sin w[\partial_{x_1}\varphi_3\cos\theta +\partial_{x_2}\varphi_3\sin\theta] Q_\gamma  W,
\end{aligned}
\end{equation*}
\begin{equation*}
\begin{aligned}
\tilde{L}_U[\Phi]_2
&= \frac{1}{\lambda} \rho w_\rho^2\Big[[\div (e^{i\gamma }\bar\varphi)\cos(2\theta) - {\rm{curl }}(e^{i\gamma }\bar\varphi)\sin(2\theta)]Q_\gamma  E_1 + [\div (e^{i\gamma }\bar\varphi)\sin(2\theta) + {\rm{curl }}(e^{i\gamma }\bar\varphi)\cos(2\theta)] Q_\gamma  E_2\Big] \\ & \quad +\frac{2}{\lambda} w_\rho\left[-\frac{1}{2}\rho^2w_\rho\div(e^{i\gamma }\bar\varphi) \cos(2\theta) +\frac{1}{2}\rho^2w_\rho{\rm{curl }}(e^{i\gamma }\bar\varphi)\sin(2\theta)\right] Q_\gamma  W .
\end{aligned}
\end{equation*}
\end{lemma}

\medskip

\begin{lemma}\label{linearization-lemma-3}
For
\begin{equation*}
\Phi(x) = \left[
\begin{matrix}
\phi(r)e^{i\theta}\\
0
\end{matrix}
\right]
\end{equation*}
with $\phi$ complex-valued,
the term
$
\tilde{L}_U[\Phi]$
can be expressed as
\begin{equation*}
\begin{aligned}
\tilde{L}_U [\Phi]& = \frac{2}{\lambda}\rho w_{\rho}^2\left[\emph{Re}(e^{-i\gamma }\phi_r(r))Q_\gamma  E_1 + \frac{1}{r}\emph{Im}(e^{-i\gamma }\phi(r))Q_\gamma  E_2\right] \\
&\quad + \frac{2}{\lambda}w_{\rho}\left(\emph{Re}(\phi_re^{-i\gamma })\cos w - \frac{1}{r}\emph{Re}(\phi e^{-i\gamma }) \right)Q_\gamma  W .
\end{aligned}
\end{equation*}
Assume $\psi$ is real-valued and
\begin{equation*}
\Psi(x) = \left[
\begin{matrix}
0\\
0\\
\psi(r)\\
\end{matrix}
\right].
\end{equation*}
Then
\begin{equation*}
\begin{aligned}
\tilde{L}_U [\Psi] =&~ -\frac2{\la} w_\rho\cos w \left(\psi_r Q_\gamma E_1-\frac{1}r \psi_\theta Q_\gamma E_2\right)+\frac2{\la}\rho w_\rho^2\psi_r  Q_\gamma W.
\end{aligned}
\end{equation*}
\end{lemma}

\medskip

The proof of all the lemmas in this section will be postponed to Appendix \ref{appendix-linearization}.

\bigskip

\section{Proof of Theorem \ref{thm}: Non-degeneracy of the degree $m$ bubble}\label{appendixB}

\medskip

For notational simplicity, we write throughout this section
$$
W^{(m)}=W,\quad E_1^{(m)}=E_1,\quad E_2^{(m)}=E_2.
$$

\begin{proof}[Proof of Theorem \ref{thm}]

We divide the proof into the following steps.

\medskip

\noindent {\bf Step 1}. In the polar coordinate system, we write the solution $\phi = \phi(r, \theta)$ of (\ref{e:linearized1}) as
\begin{equation*}
\begin{aligned}
&\phi(r, \theta) = \xi(r, \theta)E_1 + \eta(r, \theta)E_2 + \zeta(r, \theta)W.
\end{aligned}
\end{equation*}
Then by direct computation, we know the linearized operator
$$L[\phi] = \Delta \phi - \frac{2}{r}\phi_r\wedge W_\theta - \frac{2}{r}W_r\wedge \phi_\theta:=(L^{(1)}[\phi], L^{(2)}[\phi], L^{(3)}[\phi])$$
 can be expressed as follows,
\begin{align*}
&L^{(1)}[\phi]\\
&= \frac{1}{r^2 \left(r^{2 m}+1\right)^3}\cos (m\theta)\left[-m^2 (-1+7 r^{2 m}-7 r^{4 m}+r^{6 m})\xi+16 m^2 r^{3 m}\zeta\right.\\
&\quad +(-2m+2m r^{2m}+2m r^{4m}-2m r^{6m})\eta_{\theta}\\
&\quad +(-1-r^{2m}+r^{4m}+r^{6m})\xi_{\theta\theta}+(-r-r^{1+2m}+r^{1+4m}+r^{1+6m})\xi_{r}\\
&\quad +(-r^2-r^{2+2m}+r^{2+4m}+r^{2+6m})\xi_{rr}+(2r^{2+m}+4r^{2+3m}+2r^{2+5m})\zeta_{rr}\\
&\quad \left. +(2r^{1+m}+4r^{1+3m}+2r^{1+5m})\zeta_{r}+(2r^m+4r^{3m}+2r^{5m})\zeta_{\theta\theta}\right]\\
&\quad +\frac{(r^{2 m}+1)}{r^2 \left(r^{2 m}+1\right)^3}\sin (m\theta)\left[(2m-2m r^{4m})\xi_{\theta} + m^2 (1-6 r^{2 m}+r^{4 m}) \eta\right.\\
&\quad\quad +(-1-2r^{2m}-r^{4m})\eta_{\theta\theta}+(-r-2r^{1+2m}-r^{1+4m})\eta_{r}\\
&\quad\quad\left.+(-r-2r^{2+2m}-r^{2+4m})\eta_{rr})\right],
\end{align*}
\begin{align*}
&L^{(2)}[\phi]=\frac{-1-r^{2m}}{r^2\left(r^{2 m}+1\right)^3}\cos (m\theta)\left[(2m-2m r^{4m})\xi_{\theta}+m^2(1-6r^{2m}+r^{4m})\eta\right.\\
&\quad +(-1-2r^{2m}-r^{4m})\eta_{\theta\theta}+(-r-2r^{1+2m}-r^{1+4m})\eta_{r}\\
&\quad \left.+(-r^2-2r^{2+2m}-r^{2+4m})\eta_{rr}\right]\\
&\quad +\frac{1}{r^2 \left(r^{2 m}+1\right)^3}\sin (m\theta)\left[-m^2(-1+7r^{2m}-7r^{4m}+r^{6m})\xi+16m^2r^{3m}\zeta\right.\\
&\quad +(-2m+2mr^{2m}+2mr^{4m}-2mr^{6m})\eta_{\theta}+(-1-r^{2m}+r^{4m}+r^{6m})\xi_{\theta\theta}\\
&\quad +(-r-r^{1+2m}+r^{1+4m}+r^{1+6m})\xi_{r}+(-r^2-r^{2+2m}+r^{2+4m}+r^{2+6m})\xi_{rr}\\
&\quad +(2r^m+4r^{3m}+2r^{5m})\zeta_{\theta\theta}+(2r^{1+m}+4r^{1+3m}+2r^{1+5m})\zeta_{r}\\
&\quad \left. +(2r^{2+m}+4r^{2+3m}+2r^{2+5m})\zeta_{rr}\right]
\end{align*}
and
\begin{align*}
&L^{(3)}[\phi]\\
&= \frac{1}{r^2\left(r^{2 m}+1\right)^3}\left[2m^2(r^m-6r^{3m}+r^{5m})\xi+8m^2r^{2m}\zeta+(-4mr^m+4mr^{5m})\eta_{\theta}\right.\\
&\quad +(-2r^m-4r^{3m}-2r^{5m})\xi_{\theta\theta} + (-1-r^{2m}+r^{4m}+r^{6m})\zeta_{\theta\theta}\\
&\quad +(-2r^{1+m}-4r^{1+3m}-2r^{1+5m})\xi_{r} +(-r-r^{1+2m}+r^{1+4m}+r^{1+6m})\zeta_{r}\\
&\quad \left.+(-2r^{2+m}-4r^{2+3m}-2r^{2+5m})\xi_{rr}+(-r^2-r^{2+2m}+r^{2+4m}+r^{2+6m})\xi_{rr}\right].
\end{align*}

\medskip

\noindent {\bf Step 2}. Using Fourier expansion, we set $\xi(r, \theta) = a(r) \cos (k\theta)+b(r) \sin (k\theta)$, $\eta(r, \theta) = c(r) \cos (k\theta)+d(r) \sin (k\theta)$ and $\zeta(r, \theta) = e(r)\cos(k\theta)+f(r)\sin(k\theta)$, $k$ is an integer, and we write
$$L_\perp(\phi):=(L^{(1)}_{\perp}[\phi], L^{(1)}_{\perp}[\phi], L^{(1)}_{\perp}[\phi]) : = (L^{(1)}[\phi], L^{(2)}[\phi], L^{(3)}[\phi]) - ((L^{(1)}[\phi], L^{(2)}[\phi], L^{(3)}[\phi])\cdot W)W.$$
Then we have
\begin{equation*}
L^{(1)}_{\perp}[\phi]=\frac{1}{r^2 \left(r^{2 m}+1\right)^3}((1+r^{2m})\sin(m\theta)A_1+(-1+r^{2m})\cos(m\theta)A_2)
\end{equation*}
with
\begin{align*}
A_1 &= \cos(k\theta)(-2km(r^{4m}-1)b(r)+(k^2(1+r^{2m})^2+m^2(1-6r^{2m}+r^{4m}))c(r))\\
&\quad -\cos(k\theta)(r(1+r^{2m})^2(c'(r)+rc''(r)))\\
&\quad +\sin(k\theta)(2km(r^{4m}-1)a(r)+(k^2(1+r^{2m})^2+m^2(1-6r^{2m}+r^{4m}))d(r))\\
&\quad -\sin(k\theta)(r(1+r^{2m})^2(d'(r)+rd''(r))),
\end{align*}

\begin{align*}
A_2 &= \cos(k\theta)((-k^2(1+r^{2m})^2-m^2(1-6r^{2m}+r^{4m}))a(r)-2km(r^{4m}-1)d(r))\\
&\quad +\cos(k\theta)(r(1+r^{2m})^2(a'(r)+r a''(r))\\
&\quad +\sin(k\theta)(-k^2(1+r^{2m})^2-m^2(1-6r^{2m}+r^{4m}))b(r))\\
&\quad +\sin(k\theta)(2km(r^{2m}-1)(r^{2m}+1)c(r)+r(1+r^{2m})^2(b'(r)+r b''(r)),
\end{align*}

\begin{equation*}
L^{(2)}_{\perp}[\phi]=\frac{1}{r^2 \left(r^{2 m}+1\right)^3}((1+r^{2m})\cos(m\theta)B_1+(-1+r^{2m})\sin(m\theta)B_2)
\end{equation*}
with
\begin{align*}
B_1 &= \cos(k\theta)(2km(r^{4m}-1)b(r)-(k^2(1+r^{2m})^2+m^2(1-6r^{2m}+r^{4m}))c(r))\\
&\quad +\cos(k\theta)(r(1+r^{2m})^2(c'(r)+rc''(r)))\\
&\quad -\sin(k\theta)(2km(r^{4m}-1)a(r)+(k^2(1+r^{2m})^2+m^2(1-6r^{2m}+r^{4m}))d(r))\\
&\quad +\sin(k\theta)(r(1+r^{2m})^2(d'(r)+rd''(r))),
\end{align*}

\begin{align*}
B_2 &= \cos(k\theta)(-(k^2(1+r^{2m})^2+m^2(1-6r^{2m}+r^{4m}))a(r)-2km(r^{4m}-1)d(r)\\
&\quad +r(1+r^{2m})^2(a'(r)+r a''(r))\\
&\quad +\sin(k\theta)(-(k^2(1+r^{2m})^2+m^2(1-6r^{2m}+r^{4m}))b(r)+2km(r^{4m}-1)c(r)\\
&\quad +r(1+r^{2 m})^2(b'(r)+r b''(r)),
\end{align*}

\begin{equation*}
L^{(3)}_{\perp}[\phi]=\frac{2r^{m-2}}{(1+r^{2m})^3}\cos(k\theta)C_1+\frac{2r^{m-2}}{(1+r^{2m})^3}\sin(k\theta)C_2,
\end{equation*}
with
\begin{equation*}
\begin{aligned}
C_1 &= (k^2(1+r^{2m})^2+m^2(1-6r^{2m}+r^{4m}))a(r)+(1+r^{2m})2km(r^{2m}-1)d(r)\\
&\quad -r(1+r^{2m})(a'(r)+ra''(r)),
\end{aligned}
\end{equation*}

\begin{equation*}
\begin{aligned}
C_2 &= (k^2(1+r^{2m})^2+m^2(1-6r^{2m}+r^{4m}))b(r)-(1+r^{2m})2km(r^{2m}-1)c(r)\\
&\quad +r(1+r^{2m})(b'(r)+rb''(r)).
\end{aligned}
\end{equation*}

In the direction of $W$, we have
\begin{equation*}
(L^{(1)}[\phi], L^{(2)}[\phi], L^{(3)}[\phi])\cdot W =\frac{1}{r^2(1+r^{2m})^2}\cos(k\theta)D_1+\frac{2r^{m-2}}{(1+r^{2m})^3}\sin(k\theta)D_2
\end{equation*}
with
\begin{equation*}
D_1 = -(-8m^2r^{2m}+k^2(1+r^{2m})^2)e(r)+r(1+r^{2m})^2(e'(r)+re''(r)),
\end{equation*}
\begin{equation*}
D_2 = -(-8m^2r^{2m}+k^2(1+r^{2m})^2)f(r)+r(1+r^{2m})^2(f'(r)+rf''(r)).
\end{equation*}

We also write $L_\perp(\phi)$ as
\begin{equation*}
\begin{aligned}
L_\perp(\phi) & = (L_\perp(\phi)\cdot E_1) E_1 + (L_\perp(\phi)\cdot E_2)E_2
\end{aligned}
\end{equation*}
with
\begin{equation*}
\begin{aligned}
&L_\perp(\phi)\cdot E_1\\
& = ((-k^2(r^{2m}+1)^2-m^2(1-6r^{2m}+r^{4m}))a(r)-2km(r^{4m}-1)d(r))\cos(k\theta)\\
&\quad +r(1+r^{2m})^2(a'(r)+ra''(r))\cos(k\theta)\\
&\quad + ((-k^2(r^{2m}+1)^2-m^2(1-6r^{2m}+r^{4m}))b(r)+2km(r^{4m}-1)c(r))\sin(k\theta)\\
&\quad + r(1+r^{2m})^2(b'(r)+rb''(r))\sin(k\theta)
\end{aligned}
\end{equation*}
and
\begin{equation*}
\begin{aligned}
&L_\perp(\phi)\cdot E_2\\ &= ((-k^2(r^{2m}+1)^2-m^2(1-6r^{2m}+r^{4m}))c(r)+2km(r^{4m}-1)b(r))\cos(k\theta)\\
&\quad +r(1+r^{2m})^2(c'(r)+rc''(r))\cos(k\theta)\\
&\quad + ((-k^2(r^{2m}+1)^2-m^2(1-6r^{2m}+r^{4m}))d(r)-2km(r^{4m}-1)a(r))\sin(k\theta)\\
&\quad +r(1+r^{2m})^2(d'(r)+rd''(r))\sin(k\theta).
\end{aligned}
\end{equation*}

Therefore, to solve the linearized equation (\ref{e:linearized1}), we need to solve the following system of ODEs,
\begin{equation}\label{e:ode1}
\left\{
\begin{aligned}
a''(r)+\frac{1}{r}a'(r)-\frac{k^2+\frac{m^2(1-6r^{2m}+r^{4m})}{(r^{2m}+1)^2}}{r^2}a(r)-\frac{2km(r^{2m}-1)}{r^2(1+r^{2m})}d(r) = 0,\\
d''(r)+\frac{1}{r}d'(r)-\frac{k^2+\frac{m^2(1-6r^{2m}+r^{4m})}{(r^{2m}+1)^2}}{r^2}d(r)-\frac{2km(r^{2m}-1)}{r^2(1+r^{2m})}a(r) = 0,
\end{aligned}
\right.
\end{equation}
\begin{equation}\label{e:ode2}
\left\{
\begin{aligned}
b''(r)+\frac{1}{r}b'(r)-\frac{k^2+\frac{m^2(1-6r^{2m}+r^{4m})}{(r^{2m}+1)^2}}{r^2}b(r)+\frac{2km(r^{2m}-1)}{r^2(1+r^{2m})}c(r) = 0,\\
c''(r)+\frac{1}{r}c'(r)-\frac{k^2+\frac{m^2(1-6r^{2m}+r^{4m})}{(r^{2m}+1)^2}}{r^2}c(r)+\frac{2km(r^{2m}-1)}{r^2(1+r^{2m})}b(r) = 0,
\end{aligned}
\right.
\end{equation}
and
\begin{equation}\label{e:ode3}
\left\{
\begin{aligned}
e''(r)+\frac{1}{r}e'(r)+\frac{8m^2r^{2m}-k^2(1+r^{2m})^2}{r^2(r^{2m}+1)^2}e(r) = 0,\\
f''(r)+\frac{1}{r}f'(r)+\frac{8m^2r^{2m}-k^2(1+r^{2m})^2}{r^2(r^{2m}+1)^2}f(r) = 0.
\end{aligned}
\right.
\end{equation}

\medskip

\noindent {\bf Step 3}. We solve the system (\ref{e:ode1})-(\ref{e:ode3}) in the following four cases.

{\bf Case 1}: If $k = 0$, then the solutions of (\ref{e:ode1})-(\ref{e:ode3}) are
\begin{equation*}
a(r) = \frac{(c_1+c_3)}{2}\frac{r^m}{1+r^{2m}} + \frac{(c_2+c_4)}{2}\frac{2r^m\log(r^{2m})+r^{3m}-r^{-m}}{1+r^{2m}},
\end{equation*}
\begin{equation*}
d(r) = \frac{(c_1-c_3)}{2}\frac{r^m}{1+r^{2m}} + \frac{(c_2-c_4)}{2}\frac{2r^m\log(r^{2m})+r^{3m}-r^{-m}}{1+r^{2m}},
\end{equation*}
\begin{equation*}
b(r) = \frac{(c_5+c_7)}{2}\frac{r^m}{1+r^{2m}} + \frac{(c_6+c_8)}{2}\frac{2r^m\log(r^{2m})+r^{3m}-r^{-m}}{1+r^{2m}},
\end{equation*}
\begin{equation*}
c(r) = \frac{(c_5-c_7)}{2}\frac{r^m}{1+r^{2m}} + \frac{(c_6-c_8)}{2}\frac{2r^m\log(r^{2m})+r^{3m}-r^{-m}}{1+r^{2m}},
\end{equation*}
\begin{equation*}
e(r) = \frac{c_9(r^{2m} - 1)}{1 + r^{2m}}+ \frac{c_{10}r^m((r^m - r^{-m})\log(r^{2m}) - 4r^{-m})}{1 + r^{2m}},
\end{equation*}
\begin{equation*}
f(r) = \frac{c_{11}(r^{2m} - 1)}{1 + r^{2m}} + \frac{c_{12}r^m((r^m - r^{-m})\log(r^{2m}) - 4r^{-m})}{1 + r^{2m}}.
\end{equation*}
Therefore the bounded kernel functions take the form
\begin{equation*}
\begin{aligned}
&\phi(r, \theta) = \frac{(c_1+c_3)}{2}\frac{r^m}{1+r^{2m}}\cos (k\theta)E_1 + \frac{(c_5-c_7)}{2}\frac{r^m}{1+r^{2m}}\cos (k\theta)E_2\\
&\quad\quad\quad\quad + \frac{c_9(r^{2m} - 1)}{1 + r^{2m}}\cos (k\theta)W.
\end{aligned}
\end{equation*}

{\bf Case 2}: If $k = 1$, then the solutions of (\ref{e:ode1})-(\ref{e:ode3}) are
\begin{equation*}
\begin{aligned}
a(r)& = \frac{c_{13}r^{m-1}}{2(1 + r^{2m})} + \frac{c_{14}r((-1-m)r^{-m}+2m^2r^m+mr^{3m}-2r^m-r^{3m})}{2(1 + r^{2m})}\\
&\quad + \frac{c_{15}r^{m+1}}{2(1 + r^{2m})} + \frac{c_{16}r^{-1}((-1+m)r^{-m}+2m^2r^m-mr^{3m}-2r^m-r^{3m})}{2(1 + r^{2m})},
\end{aligned}
\end{equation*}
\begin{equation*}
\begin{aligned}
d(r)& = \frac{c_{13}r^{m-1}}{2(1 + r^{2m})} + \frac{c_{14}r((-1-m)r^{-m}+2m^2r^m+mr^{3m}-2r^m-r^{3m})}{2(1 + r^{2m})}\\
&\quad - \frac{c_{15}r^{m+1}}{2(1 + r^{2m})} - \frac{c_{16}r^{-1}((-1+m)r^{-m}+2m^2r^m-mr^{3m}-2r^m-r^{3m})}{2(1 + r^{2m})},
\end{aligned}
\end{equation*}
\begin{equation*}
\begin{aligned}
b(r) &= \frac{c_{17}r^{m-1}}{2(1 + r^{2m})} + \frac{c_{18}r((-1-m)r^{-m}+2m^2r^m+mr^{3m}-2r^m-r^{3m})}{2(1 + r^{2m})}\\
&\quad + \frac{c_{19}r^{m+1}}{2(1 + r^{2m})} + \frac{c_{20}r^{-1}((-1+m)r^{-m}+2m^2r^m-mr^{3m}-2r^m-r^{3m})}{2(1 + r^{2m})},
\end{aligned}
\end{equation*}
\begin{equation*}
\begin{aligned}
c(r) &= -\frac{c_{17}r^{m-1}}{2(1 + r^{2m})} - \frac{c_{18}r((-1-m)r^{-m}+2m^2r^m+mr^{3m}-2r^m-r^{3m})}{2(1 + r^{2m})}\\
&\quad + \frac{c_{19}r^{m+1}}{2(1 + r^{2m})} + \frac{c_{20}r^{-1}((-1+m)r^{-m}+2m^2r^m-mr^{3m}-2r^m-r^{3m})}{2(1 + r^{2m})},
\end{aligned}
\end{equation*}
\begin{equation*}
\begin{aligned}
e(r)& = \frac{c_{21}}{1 + r^{2m}}((1 - m)r^{-m} + (1 + m)r^m)r^{m-1}\\
&\quad + \frac{c_{22}}{1 + r^{2m}}((1 + m)r^{-m} + (1 - m)r^m)r^{m+1},
\end{aligned}
\end{equation*}
\begin{equation*}
\begin{aligned}
f(r)& = \frac{c_{23}}{1 + r^{2m}}((1 - m)r^{-m} + (1 + m)r^m)r^{m-1}\\
&\quad + \frac{c_{24}}{1 + r^{2m}}((1 + m)r^{-m} + (1 - m)r^m)r^{m+1}.
\end{aligned}
\end{equation*}

If $m=1$, the bounded kernel functions take the form
\begin{equation*}
\begin{aligned}
&\phi(r, \theta)\\
&\quad = \left(\left(\frac{c_{13}}{2(1 + r^{2})}+ \frac{c_{15}r^{2}}{2(1 + r^{2})}\right)\cos (\theta)+\left(\frac{c_{17}}{2(1 + r^{2})}+ \frac{c_{19}r^{2}}{2(1 + r^{2})}\right)\sin (\theta)\right)E_1\\
&\quad\quad + \left(\left(-\frac{c_{17}}{2(1 + r^{2})}+ \frac{c_{19}r^{2}}{2(1 + r^{2})}\right)\cos (\theta)+\left(\frac{c_{13}}{2(1 + r^{2})}-\frac{c_{15}r^{2}}{2(1 + r^{2})}\right)\sin (\theta)\right)E_2\\
&\quad\quad + \left(\frac{(c_{21}+c_{22})2r}{1 + r^{2}}\cos (\theta)+\frac{(c_{23}+c_{24})2r}{1 + r^{2}}\sin (\theta)\right)W.
\end{aligned}
\end{equation*}

If $m>1$, the bounded kernel functions take the form
\begin{equation*}
\begin{aligned}
&\phi(r, \theta)=\\ &\left(\left(\frac{c_{13}r^{m-1}}{2(1 + r^{2m})}+ \frac{c_{15}r^{m+1}}{2(1 + r^{2m})}\right)\cos (\theta)+\left(\frac{c_{17}r^{m-1}}{2(1 + r^{2m})}+ \frac{c_{19}r^{m+1}}{2(1 + r^{2m})}\right)\sin (\theta)\right)E_1\\
& + \left(\left(-\frac{c_{17}r^{m-1}}{2(1 + r^{2m})}+ \frac{c_{19}r^{m+1}}{2(1 + r^{2m})}\right)\cos (\theta)+\left(\frac{c_{13}r^{m-1}}{2(1 + r^{2m})}-\frac{c_{15}r^{m+1}}{2(1 + r^{2m})}\right)\sin (\theta)\right)E_2.
\end{aligned}
\end{equation*}

{\bf Case 3}: If $k = m$, then the solutions of (\ref{e:ode1})-(\ref{e:ode3}) are
\begin{equation*}
\begin{aligned}
a(r)& = \frac{c_{25}}{2(1 + r^{2m})} + \frac{c_{26}r^m(r^{3m} + 2r^{-m}\log(r^{2m}) + 4r^m)}{2(1 + r^{2m})}\\
&\quad + \frac{c_{27}r^{2m}}{2(1 + r^{2m})} + \frac{c_{28}r^{-m}(-2r^{3m}\log(r^{2m}) + 4r^m + r^{-m})}{2(1 + r^{2m})},
\end{aligned}
\end{equation*}
\begin{equation*}
\begin{aligned}
d(r)& = \frac{c_{25}}{2(1 + r^{2m})} + \frac{c_{26}r^m(r^{3m} + 2r^{-m}\log(r^{2m}) + 4r^m)}{2(1 + r^{2m})}\\
&\quad - \frac{c_{27}r^{2m}}{2(1 + r^{2m})} - \frac{c_{28}r^{-m}(-2r^{3m}\log(r^{2m}) + 4r^m + r^{-m})}{2(1 + r^{2m})},
\end{aligned}
\end{equation*}
\begin{equation*}
\begin{aligned}
b(r) &= \frac{c_{29}r^{2m}}{2(1 + r^{2m})} + \frac{c_{30}r^{-m}(-2r^{3m}\log(r^{2m}) + 4r^m + r^{-m})}{2(1 +r^{2m})}\\
&\quad + \frac{c_{31}}{2(1 + r^{2m})} + \frac{c_{32}r^m(r^{3m} + 2r^{-m}\log(r^{2m}) + 4r^m)}{2(1 + r^{2m})},
\end{aligned}
\end{equation*}
\begin{equation*}
\begin{aligned}
c(r) &= \frac{c_{29}r^{2m}}{2(1 + r^{2m})} + \frac{c_{30}r^{-m}(-2r^{3m}\log(r^{2m}) + 4r^m + r^{-m})}{2(1 + r^{2m})}\\
&\quad - \frac{c_{31}}{2(1 + r^{2m})} - \frac{c_{32}r^m(r^{3m} + 2r^{-m}\log(r^{2m}) + 4r^m)}{2(1 + r^{2m})},
\end{aligned}
\end{equation*}
\begin{equation*}
e(r) = \frac{c_{33}r^m}{(1 + r^{2m})} + \frac{c_{34}(2r^m\log(r^{2m}) + r^{3m} - r^{-m})}{1 + r^{2m}},
\end{equation*}
\begin{equation*}
f(r) = \frac{c_{35}r^m}{(1 +r^{2m})} + \frac{c_{36}(2r^m\log(r^{2m}) + r^{3m} - r^{-m})}{1 + r^{2m}}.
\end{equation*}
Therefore the bounded kernel functions take the form
\begin{equation*}
\begin{aligned}
&\phi(r, \theta)\\ &= \left(\left(\frac{c_{25}}{2(1 + r^{2m})}+\frac{c_{27}r^{2m}}{2(1 + r^{2m})}\right)\cos (k\theta)+\left(\frac{c_{29}r^{2m}}{2(1 + r^{2m})}+ \frac{c_{31}}{2(1 + r^{2m})}\right)\sin (k\theta)\right)E_1\\
&\quad + \left(\left(\frac{c_{29}r^{2m}}{2(1 + r^{2m})}- \frac{c_{31}}{2(1 + r^{2m})}\right)\cos (k\theta)+\left(\frac{c_{25}}{2(1 + r^{2m})}- \frac{c_{27}r^{2m}}{2(1 + r^{2m})}\right)\sin (k\theta)\right)E_2\\
&\quad + (\frac{c_{33}r^m}{(1 + r^{2m})}\cos (k\theta)+\frac{c_{34}r^m}{(1 +r^{2m})}\sin (k\theta))W.
\end{aligned}
\end{equation*}

{\bf Case 4}: The solutions of (\ref{e:ode1})-(\ref{e:ode3}) in the case $k\neq 0$, $k\neq 1$ and $k\neq m$ are
\begin{equation*}
\begin{aligned}
a(r) &= c_{37}\frac{r^{m-k}}{2(1+r^{2m})}+c_{38}\frac{r^k(\frac{k(k+m)}{2}r^{-m}+(k-m)(kr^{3m}/2+(k+m)r^m))}{2(1+r^{2m})}\\
&\quad + c_{39}\frac{r^{m+k}}{2(1+r^{2m})}+c_{40}\frac{r^{-k}(\frac{k(k-m)}{2}r^{-m}+(k+m)(kr^{3m}/2+(k-m)r^m))}{2(1+r^{2m})},
\end{aligned}
\end{equation*}
\begin{equation*}
\begin{aligned}
d(r) &= c_{37}\frac{r^{m-k}}{2(1+r^{2m})}+c_{38}\frac{r^k(\frac{k(k+m)}{2}r^{-m}+(k-m)(kr^{3m}/2+(k+m)r^m))}{2(1+r^{2m})}\\
&\quad - c_{39}\frac{r^{m+k}}{2(1+r^{2m})}-c_{40}\frac{r^{-k}(\frac{k(k-m)}{2}r^{-m}+(k+m)(kr^{3m}/2+(k-m)r^m))}{2(1+r^{2m})},
\end{aligned}
\end{equation*}
\begin{equation*}
\begin{aligned}
b(r) &= \frac{r^{-m-k}}{1+r^{2m}}\left(c_{41}(1+(k+m)r^{2m}(\frac{2}{k}+\frac{r^{2m}}{k-m}))+c_{42}r^{2(k+m)}\right.\\
&\quad\quad\quad\quad\quad\left.+c_{43}\frac{r^{2k}(k(k+m)+2(k^2-m^2)r^{2m}+k(k-m)r^{4m})}{k(k+m)}+c_{44}r^{2m}\right),
\end{aligned}
\end{equation*}
\begin{equation*}
\begin{aligned}
c(r) &= \frac{r^{-m-k}}{1+r^{2m}}\left(c_{41}(1+(k+m)r^{2m}(\frac{2}{k}+\frac{r^{2m}}{k-m}))+c_{42}r^{2(k+m)}\right.\\
&\quad\quad\quad\quad\quad\left.-c_{43}\frac{r^{2k}(k(k+m)+2(k^2-m^2)r^{2m}+k(k-m)r^{4m})}{k(k+m)}-c_{44}r^{2m}\right),
\end{aligned}
\end{equation*}
\begin{equation*}
\begin{aligned}
e(r)& = \frac{c_{45}}{1 + r^{2m}}((k - m)r^{-m} + (k + m)r^m)r^{m-k}\\
&\quad + \frac{c_{46}}{1 + r^{2m}}((k + m)r^{-m} + (k - m)r^m)r^{m+k},
\end{aligned}
\end{equation*}
\begin{equation*}
\begin{aligned}
f(r)& = \frac{c_{47}}{1 + r^{2m}}((k - m)r^{-m} + (k + m)r^m)r^{m-k}\\
&\quad + \frac{c_{48}}{1 + r^{2m}}((k + m)r^{-m} + (k - m)r^m)r^{m+k}.
\end{aligned}
\end{equation*}
If $k > m$, there are no bounded kernel functions. If $k =2, \cdots, m-1$, the bounded kernel functions take the form
\begin{equation*}
\begin{aligned}
\phi(r, \theta)&=\left(c_{37}\frac{r^{m-k}}{2(1+r^{2m})}+ c_{39}\frac{r^{m+k}}{2(1+r^{2m})}\right)\cos (k\theta)E_1\\
&\quad +\frac{r^{-m-k}}{1+r^{2m}}\left(c_{42}r^{2(k+m)}+c_{44}r^{2m}\right)\sin (k\theta)E_1\\
&\quad + \frac{r^{-m-k}}{1+r^{2m}}\left(c_{42}r^{2(k+m)}- c_{44}r^{2m}\right)\cos (k\theta)E_2\\
&\quad +\left(c_{37}\frac{r^{m-k}}{2(1+r^{2m})}- c_{39}\frac{r^{m+k}}{2(1+r^{2m})}\right)\sin (k\theta)E_2.
\end{aligned}
\end{equation*}
From the above computations, we conclude the validity of Theorem \ref{thm}.

\end{proof}

\bigskip

\section{Approximation and improvement}\label{approximateandimprovement}

\medskip

In this section, we find a suitable profile that approximates well the real solution to \eqref{e:main} and then further improve it by adding several corrections. We first define the error of $u$ as
\begin{equation}\label{e:error}
S[u]:= -u_t + \Delta u - 2u_{x_1}\wedge u_{x_2}.
\end{equation}
Recall \eqref{kernels}. We take the approximate solution as
\begin{equation}\label{def-approx}
\begin{aligned}
&~U_{\la,\gamma,\xi,c_1,c_2}\\
:=&~Q_{\gamma(t)} \left[W\left(\frac{x-\xi(t)}{\lambda(t)}\right)+c_1(t)\frac{2\rho}{\rho^2+1}\cos\theta  W\left(\frac{x-\xi(t)}{\lambda(t)}\right) + c_2(t) \frac{2\rho}{\rho^2+1}\sin\theta  W\left(\frac{x-\xi(t)}{\lambda(t)}\right)\right]\\
=&~Q_{\gamma}(W+c_1\mathcal Z_{1,1}+c_2\mathcal Z_{1,2}),
\end{aligned}
\end{equation}
where $Q_\gamma$ in \eqref{Matrix} can be written as
$$
Q_{\gamma}=e^{\gamma J_z},\quad J_z=\begin{bmatrix}
0&-1&0\\
1&0&0\\
0&0&0\\
\end{bmatrix},
$$
and the modulation parameters $\la(t),~\xi(t),~\gamma(t),~c_1(t),~c_2(t)$ are to be determined.

\begin{remark}
By the non-degeneracy of degree $1$ $H$-bubble, there are nine kernels for the associated linearized operator. In the above ansatz, we only introduce six modulation parameters. The other three correspond to the infinitesimal generators of rigid motions: rotations around $x$- and $y$-axes, and another translation for the map in $\R^3$, and the kernels are all of order one at space infinity. For technical reasons, we take advantage of regularity properties to control these modes in the absence of modulation.
\end{remark}

We have
\begin{equation}\label{eqn-333333}
\begin{aligned}
&\pp_\rho  \mathcal  Z_{1,1}=\cos\theta w_{\rho}(\cos w W+\sin w E_1),\\
&\pp_\theta  \mathcal  Z_{1,1}=-\sin w  \sin\theta W +\sin^2 w \cos \theta E_2,\\
&\pp_\rho  \mathcal  Z_{1,2}=\sin\theta w_{\rho}(\cos w W+\sin w E_1),\\
&\pp_\theta  \mathcal  Z_{1,2}=\sin w  \cos\theta W +\sin^2 w \sin \theta E_2.\\
\end{aligned}
\end{equation}
Then the error of the approximation \eqref{def-approx} is given by
\begin{align*}
&~S[U_{\la,\gamma,\xi,c_1,c_2}]\\
=&~-\pp_t (Q_{\gamma} W)-\sum_{j=1}^2\left[\dot c_j Q_{\gamma} \mathcal Z_{1,j}+c_j\pp_t(Q_{\gamma} \mathcal Z_{1,j})\right]\\
&~-2c_1c_2Q_{\gamma} \pp_{x_1}  \mathcal  Z_{1,1}\wedge \pp_{x_2}  \mathcal Z_{1,2}-2c_1c_2Q_{\gamma} \pp_{x_1} \mathcal  Z_{1,2}\wedge \pp_{x_2} \mathcal  Z_{1,1}\\
=&~-\pp_t (Q_{\gamma} W)-\sum_{j=1}^2\dot c_j Q_{\gamma} \mathcal Z_{1,j}-\sum_{j=1}^2 c_j\dot\gamma J_z e^{\gamma J_z} \mathcal Z_{1,j}-\sum_{j=1}^2 c_j Q_{\gamma}(\pp_\rho \mathcal Z_{1,j} \rho_t+\pp_\theta \mathcal Z_{1,j} \theta_t)\\
&~-\frac{2\la^{-2} c_1c_2}{\rho} Q_{\gamma}\left(\pp_\rho \mathcal Z_{1,1}\wedge \pp_{\theta} \mathcal Z_{1,2}+\pp_\rho \mathcal Z_{1,2}\wedge \pp_{\theta} \mathcal Z_{1,1}\right)\\
=&~\la^{-1}\dot{\lambda}\rho w_\rho Q_{\gamma }E_1+\dot\gamma \rho w_{\rho}Q_{\gamma }E_2\\
&~+\la^{-1}\dot\xi_{1}w_{\rho}Q_{\gamma }\left[\cos\theta E_1+\sin\theta E_2\right]+ \la^{-1}\dot\xi_{2} w_{\rho}Q_{\gamma }\left[\sin\theta E_1-\cos\theta E_2\right]\\
&~-\frac{2\rho}{1+\rho^2}(\dot c_1\cos\theta +\dot c_2\sin\theta) Q_\gamma W-\frac{2\dot\gamma \rho \sin w}{1+\rho^2}\left(c_1 \cos\theta+c_2\sin\theta \right) Q_\gamma E_2\\
&~-\bigg[\la^{-1}\rho^{-1}\sin w(\dot\xi_1\sin\theta-\dot\xi_2\cos\theta)(c_2\cos\theta-c_1\sin\theta)\\
&~\quad-\la^{-1}w_\rho \cos w(\dot\xi_1\cos\theta+\dot\xi_2 \sin\theta+ \dot\la \rho)(c_1\cos\theta+c_2\sin\theta)\bigg] Q_{\gamma} W\\
&~+\bigg[\la^{-1}w_\rho\sin w(\dot\xi_1\cos\theta+\dot\xi_2 \sin\theta+ \dot\la \rho)(c_1\cos\theta+c_2\sin\theta)\bigg]Q_{\gamma} E_1\\
&~-\bigg[\la^{-1}\rho^{-1}\sin^2 w(\dot\xi_1\sin\theta-\dot\xi_2 \cos\theta)(c_1\cos\theta+c_2\sin\theta)\bigg]Q_{\gamma} E_2\\
&~-\frac{2\la^{-2} c_1c_2}{\rho} Q_{\gamma}\left(-\sin2\theta w_\rho \sin^2 w\cos w E_1-\cos 2\theta w_\rho \sin^2 w E_2+\sin2\theta w_\rho \sin^3 w W\right),
\end{align*}
where we have used \eqref{eqn-333333}. We then arrange terms and write
\begin{equation}\label{def-error1}
\begin{aligned}
&~S[U_{\la,\gamma,\xi,c_1,c_2}]\\
=&~\underbrace{\la^{-1}\dot{\lambda}\rho w_\rho Q_{\gamma }E_1+\dot\gamma \rho w_{\rho}Q_{\gamma }E_2}_{:=\mathcal{E}_{U^\perp}^{(0)}}~+~\mathcal R_{U^\perp}\\
&~+\underbrace{(-2)\la^{-2}c_1c_2 w_\rho^2 \sin w(\sin2\theta \cos w Q_\gamma E_1+\cos2\theta Q_\gamma E_2)}_{:=\mathcal E_{U^\perp}^{(\pm 2)}}\\
&~+\underbrace{\frac{\dot\xi_{1}}{\la}w_{\rho}Q_{\gamma }\left[\cos\theta E_1+\sin\theta E_2\right]+ \frac{\dot\xi_{2}}{\la} w_{\rho}Q_{\gamma }\left[\sin\theta E_1-\cos\theta E_2\right]}_{:=\mathcal{E}_{U^\perp}^{(1)} }\\
&~+\underbrace{\rho w_\rho\left[\cos\theta (\dot c_1+\la^{-1}\dot\la c_1) +\sin\theta (\dot c_2+\la^{-1}\dot\la c_2)\right] Q_\gamma W}_{:=\mathcal E_{U}^{(\pm1)}}\\
&~+\underbrace{2\la^{-2}c_1c_2\sin 2\theta \rho w_\rho^2 \sin^2 w Q_\gamma W}_{:=\mathcal E_{U}^{(\pm2)}}~+~\mathcal R_{U},
\end{aligned}
\end{equation}
where
\begin{equation*}
\begin{aligned}
\mathcal R_{U^\perp}:=&~\dot\gamma \rho w_\rho \sin w\left(c_1 \cos\theta+c_2\sin\theta \right) Q_\gamma E_2\\
&~+
\bigg[\la^{-1}w_\rho\sin w(\dot\xi_1\cos\theta+\dot\xi_2 \sin\theta+ \dot\la \rho)(c_1\cos\theta+c_2\sin\theta)\bigg]Q_{\gamma} E_1\\
&~-\bigg[\la^{-1}\rho^{-1}\sin^2 w(\dot\xi_1\sin\theta-\dot\xi_2 \cos\theta)(c_1\cos\theta+c_2\sin\theta)\bigg]Q_{\gamma} E_2\\
=&~\bigg[\la^{-1}w_\rho\sin w(\dot\xi_1\cos\theta+\dot\xi_2 \sin\theta+ \dot\la \rho)(c_1\cos\theta+c_2\sin\theta)\bigg]Q_{\gamma} E_1\\
&~+\bigg[\Big(\dot\gamma \rho w_\rho \sin w-\la^{-1}\rho^{-1}\sin^2 w(\dot\xi_1\sin\theta-\dot\xi_2 \cos\theta)\Big)(c_1\cos\theta+c_2\sin\theta)\bigg]Q_{\gamma} E_2,
\end{aligned}
\end{equation*}
\begin{equation*}
\begin{aligned}
\mathcal R_{U}:=&~-\bigg[\la^{-1}\rho^{-1}\sin w(\dot\xi_1\sin\theta-\dot\xi_2\cos\theta)(c_2\cos\theta-c_1\sin\theta)\\
&~\quad-\la^{-1}w_\rho [\cos w(\dot\xi_1\cos\theta+\dot\xi_2 \sin\theta)+\dot\la \rho w_\rho](c_1\cos\theta+c_2\sin\theta)\bigg] Q_{\gamma} W.
\end{aligned}
\end{equation*}
We observe that, compared to $\mathcal R_{U^\perp}$ and $\mathcal R_{U}$, the error terms $\mathcal E_{U^\perp}^{(0)}$ and $\mathcal E_{U}^{(\pm 1)}$ decay slower in space, and $\mathcal E_{U^\perp}^{(\pm 2)}$ and $\mathcal E_{U}^{(\pm 2)}$ decay slower in time (in view of the blow-up dynamics that we search for). So we add several global corrections to improve the errors $\mathcal E_{U^\perp}^{(0)}$, $\mathcal E_{U}^{(\pm 1)}$, $\mathcal E_{U^\perp}^{(\pm 2)}$, and $\mathcal E_{U}^{(\pm 2)}$.

To deal with $\mathcal E_{U^\perp}^{(0)}$ and $\mathcal E_{U}^{(\pm 1)}$, we add two global/non-local corrections to improve the spatial decay at leading order. We aim to find $\Phi^{(0)}$ and $\Phi^{(1)}$ solving approximately
\begin{align*}
&~\pp_t \Phi^{(0)}\approx\Delta \Phi^{(0)} +\mathcal E_{U^\perp}^{(0)},\\
&~\pp_t \Phi^{(1)}\approx\Delta \Phi^{(1)} +\mathcal E_{U}^{(\pm 1)}.
\end{align*}
Here $\Phi^{(0)}=\Phi^{(0)}[\lambda, \gamma , \xi, c_1,c_2] $ and $\Phi^{(1)}=\Phi^{(1)}[\lambda, \gamma , \xi, c_1,c_2] $ are non-local corrections, depending on the choice of parameters, to be specified below.
To handle $\mathcal E_{U^\perp}^{(\pm 2)}$ and $\mathcal E_{U}^{(\pm 2)}$, we solve the following two linearized problems:
\begin{align*}
&~\pp_t\Phi^{(j)}_{U^\perp}=L_U[\Phi^{(j)}_{U^\perp}]+\mathcal E_{U^\perp}^{(\pm 2)},\quad j=\pm2,\\
&~\pp_t \Phi^{(2)}_U=L_U[\Phi^{(2)}_U]+\mathcal E_{U}^{(\pm 2)},
\end{align*}
where $\Phi^{(j)}_{U^\perp}=\Phi^{(j)}_{U^\perp}(\rho,t)$ is on $U^\perp$, and $\Phi^{(2)}_U=\Phi^{(2)}_U(\rho,t)$ is in the $U$-direction.

\medskip

\noindent $\bullet$ {\bf Approximate form of $\Phi^{(0)}$.}

\medskip

We notice that
\begin{align*}
\mathcal E_{U^\perp}^{(0)}
=&~\la^{-1}\dot{\lambda}\rho w_\rho Q_{\gamma }E_1+\dot\gamma \rho w_{\rho}Q_{\gamma }E_2\\
\approx&~ -\frac{2r(\dot\la+i\la\dot\gamma)}{\la^2+r^2} \begin{bmatrix}
e^{i(\theta+\gamma)}\\
0\\
\end{bmatrix}~\mbox{ for }~r=|x-\xi|\gg \la.
\end{align*}
So we assume that
\begin{equation*}
\Phi^{(0)}(x, t) = \left[\begin{matrix}
\varphi^0(x, t)\\
0
\end{matrix}\right]
\end{equation*}
and $\varphi^0$ is an approximation of
\begin{equation*}
\varphi^0_t = \Delta \varphi^0 - \frac{2r}{r^2+\lambda^2}(\dot\lambda + i\lambda\dot\gamma )e^{i(\theta+\gamma )} .
\end{equation*}
Let
$$\varphi^0(x, t) = re^{i\theta}\psi^0(z(r), t),\quad z(r) = \sqrt{r^2+\lambda^2(t)},$$
for $\psi^0(z, t)$ satisfying the equation
\begin{equation*}
\psi^0_t = \psi^0_{zz} + \frac{3\psi^0_z}{z} + \frac{p_0(t)}{z^2},
\end{equation*}
where we define
\begin{equation}\label{def-p_0}
p_0(t) := -2(\dot\lambda + i\lambda\dot\gamma )e^{i\gamma }.
\end{equation}
From the Duhamel's principle, we know that
\begin{equation}\label{eqn-psi^0}
\psi^0(z, t) = \int_{-T}^tp_0(s)k(z, t-s)ds, \quad k(z, t) = \frac{1-e^{-\frac{z^2}{4t}}}{z^2}
\end{equation}
is a weak solution. Notice $k_t=k_{zz}+\frac3zk_z$. Then we have
\begin{equation*}
-\partial_t\Phi^{(0)} + \Delta_x \Phi^{(0)} = \tilde{\mathcal R}_0^0 + \tilde{\mathcal R}_1^0:=\tilde{\mathcal R}^0, \quad \tilde{\mathcal R}_0^0 = \left[\begin{matrix}
\mathcal R_0^0\\
0
\end{matrix}\right], \quad \tilde{\mathcal R}_1^0 = \left[\begin{matrix}
\mathcal R_1^0\\
0
\end{matrix}\right]
\end{equation*}
where
\begin{equation*}
\mathcal R_0^0 = -r e^{i\theta}\frac{p_0(t)}{z^2} + re^{i\theta}\frac{\lambda^2}{z^4}\int_{-T}^tp_0(s)(zk_z-z^2k_{zz})(z(r), t-s)ds
\end{equation*}
and
\begin{equation*}
\mathcal R_1^0 = e^{i\theta}{\rm Re}[(\dot\xi e^{-i\theta})]\int_{-T}^tp_0(s)k(z(r), t-s)ds - \frac{r}{z^2}e^{i\theta}\left(\lambda\dot\lambda-{\rm Re}(re^{-i\theta}\dot\xi(t))\right)\int_{-T}^tp_0(s)zk_z(z(r), t-s)ds.
\end{equation*}

\medskip

\noindent $\bullet$ {\bf Approximate form of $\Phi^{(1)}$.}

\medskip

Since
\begin{align*}
\mathcal E_{U}^{(\pm 1)}=&~\rho w_\rho\left[\cos\theta (\dot c_1+\la^{-1}\dot\la c_1) +\sin\theta (\dot c_2+\la^{-1}\dot\la c_2)\right] Q_\gamma W\\
\approx &~-\frac{2r}{r^2+\la^2}\begin{bmatrix}
0\\
0\\
\cos\theta (\la c_1)' +\sin\theta (\la c_2)'
\end{bmatrix} ~\mbox{ for }~r\gg \la,
\end{align*}
we assume that
\begin{equation*}
\Phi^{(1)}(x, t) = \left[\begin{matrix}
0\\
0\\
{\rm Re}(\varphi^1(x, t))
\end{matrix}\right]
\end{equation*}
and $\varphi^1$ solves approximately
\begin{equation*}
\varphi^1_t = \Delta \varphi^1- \frac{2r}{r^2+\lambda^2} e^{-i\theta}(\la \mathbf c)',
\end{equation*}
where we use the notation
\begin{equation}\label{def-cc}
\cc(t) := c_1(t)+i c_2(t).
\end{equation}
We look for
$$\varphi^1(x, t) = re^{-i\theta}\psi^1(z(r), t),\quad z(r) = \sqrt{r^2+\lambda^2},$$ where $\psi^1(z, t)$ satisfies the equation
\begin{equation*}
\psi^1_t = \psi^1_{zz} + \frac{3\psi^1_z}{z} + \frac{p_1(t)}{z^2},
\end{equation*}
and here we define
\begin{equation}\label{def-p_1}
\quad p_1(t) := -2(\la \cc)'.
\end{equation}
Similar to \eqref{eqn-psi^0}, we see that
\begin{equation}\label{eqn-psi^1}
\psi^1(z, t) = \int_{-T}^tp_1(s)k(z, t-s)ds.
\end{equation}
So we add a correction
\begin{equation*}
\Phi^{(1)}(x, t) = \left[\begin{matrix}
0\\
0\\
{\rm Re}(\varphi^1(r, t))
\end{matrix}\right]
\end{equation*}
with
$$
\varphi^1(r, t) = re^{-i\theta}\int_{-T}^tp_1(s)k(z(r), t-s)ds.
$$
We next compute the new error produced by $\Phi^{(1)}$:
\begin{equation*}
-\partial_t\Phi^{(1)} + \Delta_x \Phi^{(1)} ={\rm Re}(\tilde{\mathcal R}_0^1 + \tilde{\mathcal R}_1^1):=\tilde{\mathcal R}^1, \quad \tilde{\mathcal R}_0^1 = \left[\begin{matrix}
0\\
0\\
\mathcal R_0^1
\end{matrix}\right], \quad \tilde{\mathcal R}_1^1 = \left[\begin{matrix}
0\\
0\\
\mathcal R_1^1
\end{matrix}\right]
\end{equation*}
where
\begin{equation*}
\mathcal R_0^1 = -re^{-i\theta}\frac{p_1(t)}{z^2} + re^{-i\theta}\frac{\lambda^2}{z^4}\int_{-T}^tp_1(s)(zk_z-z^2k_{zz})(z(r), t-s)ds
\end{equation*}
and
\begin{equation*}
\mathcal R_1^1 = e^{-i\theta}{\rm Re}[(\dot\xi e^{-i\theta})]\int_{-T}^tp_1(s)k(z(r), t-s)ds - \frac{re^{-i\theta}}{z^2}\left(\lambda\dot\lambda-{{\rm Re}}(re^{-i\theta}\dot\xi(t))\right)\int_{-T}^tp_1(s)zk_z(z(r), t-s)ds.
\end{equation*}

\medskip

\noindent $\bullet$ {\bf Equations for $\Phi^{(\pm2)}_{U^\perp}$.}

\medskip

To handle $\mathcal E_{U^\perp}^{(\pm 2)}$  in \eqref{def-error1}, we first write it in complex form:
\begin{align*}
\left( Q_{-\gamma} \mathcal E_{U^\perp}^{(\pm 2)}\right)_{\mathbb C}=&~-2\la^{-2} c_1c_2  w^2_\rho \sin w(\sin2\theta \cos w+i\cos2\theta)\\
=&~-i\la^{-2} c_1c_2  w^2_\rho \sin w\left[e^{2i\theta} (1-\cos w)+
e^{-2i\theta}(1+\cos w)\right]\\
:=&~\left(Q_{-\gamma}\mathcal E_{U^\perp}^{( 2)}\right)_{\mathbb C}+\left(Q_{-\gamma}\mathcal E_{U^\perp}^{(- 2)}\right)_{\mathbb C}.
\end{align*}
We try to add two corrections, expressed in $(\rho,t)$ coordinates
\begin{equation}\label{eqn-correctpm2'}
\begin{aligned}
&\Phi^{(-2)}_{U^\perp}(\rho,t)=\phi^{(-2)}_1 Q_\gamma E_1 + \phi^{(-2)}_2 Q_\gamma E_2, \quad (\Phi_{U^\perp}^{(-2)})_{\mathbb C}=\phi^{(-2)}_1+i\phi^{(-2)}_2:=\varphi_{-2}(\rho,t)e^{-2i\theta}\\
&\Phi_{U^\perp}^{(2)}(\rho,t)=\phi^{(2)}_1 Q_\gamma E_1 + \phi^{(2)}_2 Q_\gamma  E_2, \quad (\Phi_{U^\perp}^{(2)})_{\mathbb C}=\phi^{(2)}_1+i\phi^{(2)}_2:=\varphi_{2}(\rho,t)e^{2i\theta},\\
\end{aligned}
\end{equation}
where the complex-valued $\varphi_{-2}$ and $\varphi_2$ solve
\begin{equation}\label{eqn-correctpm2}
\begin{aligned}
\la^2\pp_t \varphi_{j}=\pp_{\rho\rho} \varphi_{j}+\frac1{\rho}\pp_\rho \varphi_{j}+\left[\frac8{(1+\rho^2)^2}-\frac{1+j^2}{\rho^2}-\frac{2j(\rho^2-1)}{\rho^2(\rho^2+1)}\right]\varphi_{j} +\la^2 e^{-ji\theta}\left(Q_{-\gamma}\mathcal E_{U^\perp}^{(j)}\right)_{\mathbb C}, \quad j=\pm 2.
\end{aligned}
\end{equation}
In other words, we solve the   linearized problem
\begin{align*}
&~\pp_t\Phi^{(j)}_{U^\perp}=L_U[\Phi^{(j)}_{U^\perp}]+\mathcal E_{U^\perp}^{(\pm 2)},\quad j=\pm2,
\end{align*}
where $\Phi^{(j)}_{U^\perp}=\Phi^{(j)}_{U^\perp}(\rho,t)$ is on $U^\perp$. A solution to  \eqref{eqn-correctpm2} with zero initial data will be ensured by a linear theory developed later in Section \ref{sec-linearinner}.

\medskip

\noindent $\bullet$ {\bf Equation for $\Phi^{(2)}_U$.}

\medskip

To deal with $\mathcal E_{U}^{(\pm2)}$ given in \eqref{def-error1}, we need a correction $\Phi^{(2)}_U(\rho,t)$ in the form
\begin{equation}\label{eqn-psi_2}
\Phi^{(2)}_U(\rho,t)=\sin2\theta \psi_2(\rho,t)  Q_\gamma W
\end{equation}
with $\psi_2$ solving
\begin{equation}\label{eqn-correctL2}
\la^2 \pp_t \psi_2 = \pp_{\rho\rho} \psi_2 +\frac1\rho \pp_\rho \psi_2-\frac4{\rho^2}\psi_2+\frac8{(1+\rho^2)^2}\psi_2 + 2c_1c_2 \rho w_\rho^2 \sin^2 w
\end{equation}
since it is exactly mode $\pm 2$ of linearization in the $U$-direction (cf. \eqref{eqn-U}). Equation \eqref{eqn-correctL2} with zero Cauchy data will be solved by the linear theory in Section \ref{sec-linearinner}.

\bigskip

We now compute the new error $S[U_*]$ of the corrected approximation
\begin{equation}\label{def-U_*}
\begin{aligned}
U_*:=&~U_{\la,\gamma,\xi,c_1,c_2}+\eta_1\Big(\Phi^{(0)}+\Phi^{(1)}+\Phi^{(2)}_{U^\perp}+\Phi^{(-2)}_{U^\perp}+\Phi^{(2)}_U\Big)\\
:=&~Q_{\gamma}(W+c_1\mathcal Z_{1,1}+c_2\mathcal Z_{1,2})+\eta_1 \Phi_*,
\end{aligned}
\end{equation}
where the purpose of the cut-off function
$$
\eta_1:=\eta(x-\xi(t))
$$
is to avoid potential slow spatial decay in the remote region. Here $\eta(s)$ is a smooth cut-off function  with $\eta(s)=1$ for $s<1$ and $\eta(s) = 0$ for $ s > 2$.
We first analyze the leading terms
$$
\mathcal R_{U^\perp}+\mathcal R_{U}+\mathcal E_{U^\perp}^{(1)}+\eta_1 (\mathcal E_{U^\perp}^{(0)}+\tilde{\mathcal R}^0)+\eta_1 (\mathcal E_{U}^{(\pm1)}+\tilde{\mathcal R}^1)+\eta_1 \tilde L_U[\Phi^{(0)}+\Phi^{(1)}]
$$
in the corrected error $S[U_*]$. By definition, we have
\begin{equation}\label{eqn-321}
\begin{aligned}
&~\mathcal R_{U^\perp}+\mathcal R_{U}+\mathcal E_{U^\perp}^{(1)}\\
=&~\bigg[\la^{-1}w_\rho\sin w(\dot\xi_1\cos\theta+\dot\xi_2 \sin\theta+ \dot\la \rho)(c_1\cos\theta+c_2\sin\theta)\bigg]Q_{\gamma} E_1\\
&~+\bigg[\Big(\dot\gamma \rho w_\rho \sin w-\la^{-1}\rho^{-1}\sin^2 w(\dot\xi_1\sin\theta-\dot\xi_2 \cos\theta)\Big)(c_1\cos\theta+c_2\sin\theta)\bigg]Q_{\gamma} E_2\\
&~-\bigg[\la^{-1}\rho^{-1}\sin w(\dot\xi_1\sin\theta-\dot\xi_2\cos\theta)(c_2\cos\theta-c_1\sin\theta)\\
&~\quad-\la^{-1}w_\rho [\cos w(\dot\xi_1\cos\theta+\dot\xi_2 \sin\theta)+\dot\la \rho w_\rho](c_1\cos\theta+c_2\sin\theta)\bigg] Q_{\gamma} W\\
&~+\frac{\dot\xi_{1}}{\la}w_{\rho}Q_{\gamma }\left[\cos\theta E_1+\sin\theta E_2\right]+ \frac{\dot\xi_{2}}{\la} w_{\rho}Q_{\gamma }\left[\sin\theta E_1-\cos\theta E_2\right],
\end{aligned}
\end{equation}
whose complex form on $W^\perp$ is
\begin{equation}\label{eqn-322}
\begin{aligned}
&~\Big(\mathcal R_{U^\perp}+\mathcal R_{U}+\mathcal E_{U^\perp}^{(1)}\Big)_{\mathbb C}\\
=&~\rho w_\rho \sin w (\la^{-1}\dot\la+i\dot\gamma)(c_1\cos\theta+c_2\sin\theta)\\
&~+\la^{-1}w_\rho \sin w (\dot\xi_1 -i\dot\xi_2)e^{i\theta} (c_1\cos\theta+c_2\sin\theta)\\
&~+\la^{-1}w_\rho (\dot\xi_1-i\dot\xi_2) e^{i\theta}.
\end{aligned}
\end{equation}
Next, by using
\begin{equation}\label{eqn-325325325}
\begin{aligned}
\left[
\begin{matrix}
f(\rho)e^{i\theta}\\
g(\rho, \theta)\\
\end{matrix}\right]
 =&~ \Big[\cos w {\rm Re}(f e^{-i\gamma })-g \sin w\Big]Q_\gamma  E_1 + {\rm Im}(f e^{-i\gamma })Q_\gamma  E_2 \\
 &~+ \Big[\sin w {\rm Re}(f e^{-i\gamma })+ g \cos w\Big]Q_\gamma  W,
\end{aligned}
\end{equation}
one has
\begin{equation}\label{eqn-323323323323}
\begin{aligned}
&~\mathcal E_{U^\perp}^{(0)}+\tilde{\mathcal R}^0\\
=&~\rho w_\rho \begin{bmatrix}
(\la^{-1}\dot \la +i\dot\gamma)e^{i(\theta+\gamma)}\\
0\\
\end{bmatrix}
+
\begin{bmatrix}
\mathcal R_0^0\\
0\\
\end{bmatrix}
+
\begin{bmatrix}
\mathcal R_1^0\\
0\\
\end{bmatrix}
+
\rho w_\rho \la^{-1}\dot\la
\begin{bmatrix}
 e^{i(\theta+\gamma)}w_\rho\\
-\sin w\\
\end{bmatrix}\\
=&~
\begin{bmatrix}
(r\frac{\lambda^2}{z^4}\int_{-T}^tp_0(s)(zk_z-z^2k_{zz})(z(r), t-s)ds)e^{i\theta}
\\
0\\
\end{bmatrix}
+
\begin{bmatrix}
(\la^{-1}\dot\la e^{i\gamma}\rho w_\rho^2)e^{i\theta}\\
\la^{-1}\dot\la\rho^2 w_\rho^2\\
\end{bmatrix}
+
\begin{bmatrix}
\mathcal R_1^0\\
0\\
\end{bmatrix}\\
=&~\Bigg[\cos w {\rm Re}[(f_{U^\perp}^{(0)}+f_{U^\perp}^{(1)})e^{-i\gamma }]-\la^{-1}\dot\la\rho^2 w_\rho^2\sin w\Bigg]Q_\gamma  E_1 + {\rm Im}[(f_{U^\perp}^{(0)}+f_{U^\perp}^{(1)})e^{-i\gamma }]Q_\gamma  E_2\\
&~ + \Big[\sin w {\rm Re}[(f_{U^\perp}^{(0)}+f_{U^\perp}^{(1)})e^{-i\gamma }]+ \la^{-1}\dot\la\rho^2 w_\rho^2\cos w\Big]Q_\gamma  W,
\end{aligned}
\end{equation}
where
\begin{equation}\label{eqn-324324324324}
\begin{aligned}
f_{U^\perp}^{(0)}:=&~r\frac{\lambda^2}{z^4}\int_{-T}^tp_0(s)(zk_z-z^2k_{zz})(z , t-s)ds+\la^{-1}\dot\la e^{i\gamma}\rho w_\rho^2-\frac{r}{z^2}\la\dot\la \int_{-T}^t p_0(s)z k_z(z,t-s)ds,\\
f_{U^\perp}^{(1)}:=&~{\rm Re}[(\dot\xi e^{-i\theta})]\int_{-T}^tp_0(s)\Big(k(z , t-s)+\frac{r^2}{z}k_z(z,t-s)\Big)ds.\\
\end{aligned}
\end{equation}

Similarly,
\begin{equation*}
\begin{aligned}
&~\mathcal E_{U}^{(\pm1)}+\tilde{\mathcal R}^1\\
=&~\rho w_\rho\left[\cos\theta (\dot c_1+\la^{-1}\dot\la c_1) +\sin\theta (\dot c_2+\la^{-1}\dot\la c_2)\right]
\begin{bmatrix}
e^{i(\theta+\gamma)}\sin w\\
w_\rho\\
\end{bmatrix}\\
&~+\begin{bmatrix}
0\\
0\\
{\rm Re}\left[re^{-i\theta}\frac{\lambda^2}{z^4}\int_{-T}^tp_1(s)(zk_z-z^2k_{zz})(z(r), t-s)ds\right]\\
\end{bmatrix}
+
\begin{bmatrix}
0\\
0\\
{\rm Re}[\mathcal R_1^1]\\
\end{bmatrix},
\end{aligned}
\end{equation*}
and from \eqref{eqn-325325325}, it follows that
\begin{align} \notag
\Big(\mathcal E_{U}^{(\pm1)}+\tilde{\mathcal R}^1\Big)\cdot Q_\gamma W=&~-\sin w  \left(\rho^2 w^2_\rho\left[\cos\theta (\dot c_1+\la^{-1}\dot\la c_1) +\sin\theta (\dot c_2+\la^{-1}\dot\la c_2)\right] \right)\\ \notag
&~+ \left(\rho w^2_\rho\left[\cos\theta (\dot c_1+\la^{-1}\dot\la c_1) +\sin\theta (\dot c_2+\la^{-1}\dot\la c_2)\right] \right)\cos w\\ \notag
&~+{\rm Re}\left[re^{-i\theta}\frac{\lambda^2}{z^4}\int_{-T}^tp_1(s)(zk_z-z^2k_{zz})(z(r), t-s)ds+\mathcal R_1^1\right]\cos w\\ \notag
=&~-\rho w^2_\rho\left[\cos\theta (\dot c_1+\la^{-1}\dot\la c_1) +\sin\theta (\dot c_2+\la^{-1}\dot\la c_2)\right]  \\ \notag
&~+{\rm Re}\left[re^{-i\theta}\frac{\lambda^2}{z^4}\int_{-T}^tp_1(s)(zk_z-z^2k_{zz})(z(r), t-s)ds+\mathcal R_1^1\right]\cos w,\\ \label{eqn-327}
\Big(\mathcal E_{U}^{(\pm1)}+\tilde{\mathcal R}^1\Big)\cdot Q_\gamma E_1=&~-\cos w  \left(\rho^2 w^2_\rho\left[\cos\theta (\dot c_1+\la^{-1}\dot\la c_1) +\sin\theta (\dot c_2+\la^{-1}\dot\la c_2)\right] \right)\\ \notag
&~-\left(\rho w^2_\rho\left[\cos\theta (\dot c_1+\la^{-1}\dot\la c_1) +\sin\theta (\dot c_2+\la^{-1}\dot\la c_2)\right] \right)\sin w\\ \notag
&~-{\rm Re}\left[re^{-i\theta}\frac{\lambda^2}{z^4}\int_{-T}^tp_1(s)(zk_z-z^2k_{zz})(z(r), t-s)ds+\mathcal R_1^1\right]\sin w \\ \notag
=&~-\rho^2 w^2_\rho\left[\cos\theta (\dot c_1+\la^{-1}\dot\la c_1) +\sin\theta (\dot c_2+\la^{-1}\dot\la c_2)\right] \\ \notag
&~-{\rm Re}\left[re^{-i\theta}\frac{\lambda^2}{z^4}\int_{-T}^tp_1(s)(zk_z-z^2k_{zz})(z(r), t-s)ds+\mathcal R_1^1\right]\sin w,\\ \notag
\Big(\mathcal E_{U}^{(\pm1)}+\tilde{\mathcal R}^1\Big)\cdot Q_\gamma E_2=&~0.
\end{align}

Also, from Lemma \ref{linearization-lemma-3}, we have
\begin{align} \notag
&~\tilde L_U[\Phi^{(0)}+\Phi^{(1)}]\\ \notag
=&~\frac{2}{\lambda}\rho w_{\rho}^2\left[{\rm Re}(e^{-i\gamma }(\psi^0+\frac{r^2}{z}\pp_z \psi^0))Q_\gamma  E_1 + \frac{1}{r}{\rm Im}(e^{-i\gamma }r\psi^0)Q_\gamma  E_2\right] \\ \notag
&~ + \frac{2}{\lambda}w_{\rho}\left({\rm Re}(e^{-i\gamma }(\psi^0+\frac{r^2}{z}\pp_z \psi^0))\cos w - \frac{1}{r}{\rm Re}(e^{-i\gamma }r\psi^0) \right)Q_\gamma  W\\ \notag
&~-\frac2{\la} w_\rho\cos w \left({\rm Re}[e^{-i\theta}(\psi^1+\frac{r^2}{z}\psi_z^1)] Q_\gamma E_1-\frac{1}r {\rm Im}[e^{-i\theta}r\psi^1] Q_\gamma E_2\right)\\  \label{eqn-327327327}
&~+\frac2{\la}\rho w_\rho^2{\rm Re}[e^{-i\theta}(\psi^1+\frac{r^2}{z}\psi_z^1)]  Q_\gamma W\\ \notag
=&~\frac{2}{\lambda}w_{\rho}\left({\rm Re}(e^{-i\gamma }(\psi^0+\frac{r^2}{z}\pp_z \psi^0))\cos w - \frac{1}{r}{\rm Re}(e^{-i\gamma }r\psi^0)+\rho w_\rho{\rm Re}[e^{-i\theta}(\psi^1+\frac{r^2}{z}\psi_z^1)] \right)Q_\gamma  W\\ \notag
&~+\frac{2}{\lambda}w_{\rho}\left(\rho w_{\rho}{\rm Re}(e^{-i\gamma }(\psi^0+\frac{r^2}{z}\pp_z \psi^0))-\cos w{\rm Re}[e^{-i\theta}(\psi^1+\frac{r^2}{z}\psi_z^1)] \right)Q_\gamma  E_1\\ \notag
&~+\frac{2}{\lambda r}  w_{\rho}\left(\rho w_\rho{\rm Im}(e^{-i\gamma }r\psi^0)+\cos w {\rm Im}[e^{-i\theta}r\psi^1]\right)Q_\gamma  E_2.
\end{align}

\bigskip

We denote the remaining terms in the error $S[U_*]$ by
\begin{equation}\label{def-R_*}
\mathcal R_* :=S[U_*]-\Big[\mathcal R_{U^\perp}+\mathcal R_{U}+\mathcal E_{U^\perp}^{(1)}+\eta_1 (\mathcal E_{U^\perp}^{(0)}+\tilde{\mathcal R}^0)+\eta_1 (\mathcal E_{U}^{(\pm1)}+  \tilde{\mathcal R}^1)+\eta_1 \tilde L_U[\Phi^{(0)}+\Phi^{(1)}]\Big]-E_{\eta_1},
\end{equation}
where $E_{\eta_1}$ is defined in \eqref{def-E_eta1},
and we claim:
\begin{lemma}\label{lemma-newerror}
The remainder $\mathcal R_*$ in the corrected error $S[U_*]$ projected in each direction $Q_\gamma W$, $Q_\gamma E_1$ and $Q_\gamma E_2$ is given by
\begin{align*}
&~\mathcal R_*\cdot Q_\gamma W\\
=&-2\eta_1 \la^{-1}w_\rho(\dot\xi_1\cos\theta+\dot\xi_2 \sin\theta+\dot\la \rho)(\phi_1^{(2)}+\phi_1^{(-2)})\\
&~+\eta_1 \sin w[\dot\gamma+\la^{-1}\rho^{-1}(\dot\xi_1 \sin\theta-\dot\xi_2 \cos\theta)] (\phi_2^{(2)}+\phi_2^{(-2)})\\
&~+\eta_1 \la^{-1}(\dot\xi_1\cos\theta+\dot\xi_2 \sin\theta+\dot\la \rho)\sin2\theta \pp_\rho \psi_2 \\
&~-2\eta_1 \la^{-1}\rho^{-1}(\dot\xi_1 \sin\theta-\dot\xi_2 \cos\theta)\cos2\theta \psi_2 \\
&~-\frac{2\eta_1^2}r \Bigg[\left(\cos w {\rm Re}\left[e^{-i\gamma }(\psi^0+\frac{r^2}{z}\pp_z \psi^0)\right]-{\rm Re}[e^{-i\theta}(\psi^1+\frac{r^2}{z}\pp_z\psi^1)]\sin w\right)\\
&~\qquad+\la^{-1}\left(\sin2\theta w_\rho \psi_2+\pp_\rho\phi_1^{(2)}+\pp_\rho\phi_1^{(-2)}\right)\Bigg]\\
&~\quad\times \left({\rm Im}\left[e^{-i\gamma }ir\psi^0\right]+(\pp_\theta \phi_2^{(2)}+\pp_\theta \phi_2^{(-2)})+\cos w (\phi_1^{(2)}+\phi_1^{(-2)})+\sin2\theta \sin w\psi_2\right) \\
&~+\frac{2\eta_1^2}r \left({\rm Im}\left[e^{-i\gamma }(\psi^0+\frac{r^2}{z}\pp_z \psi^0)\right]+\la^{-1}(\pp_\rho\phi_2^{(2)}+\pp_\rho\phi_2^{(-2)})\right)\\
&~\quad \times \left(\cos w {\rm Re}\left[e^{-i\gamma }ir\psi^0\right]-{\rm Im}(r e^{-i\theta}\psi^1)\sin w +(\pp_\theta \phi_1^{(2)}+\pp_\theta \phi_1^{(-2)})-\cos w (\phi_2^{(2)}+\phi_2^{(-2)})\right) \\
&~-\frac{2\eta_1 }r \left({\rm Im}\left[e^{-i\gamma }ir\psi^0\right]+(\pp_\theta \phi_2^{(2)}+\pp_\theta \phi_2^{(-2)})+\cos w (\phi_1^{(2)}+\phi_1^{(-2)})+\sin2\theta \sin w\psi_2\right)\\
&~\quad \times\la^{-1}w_{\rho}\sin w(c_1 \cos\theta + c_2 \sin\theta) \\
&~-\frac{2\eta_1 }r \Bigg[\left(\cos w {\rm Re}\left[e^{-i\gamma }(\psi^0+\frac{r^2}{z}\pp_z \psi^0)\right]-{\rm Re}[e^{-i\theta}(\psi^1+\frac{r^2}{z}\pp_z\psi^1)]\sin w\right)\\
&~\qquad+\la^{-1}\left(\sin2\theta w_\rho \psi_2+\pp_\rho\phi_1^{(2)}+\pp_\rho\phi_1^{(-2)}\right)\Bigg] \times \sin^2 w(c_1\cos\theta+ c_2\sin \theta) ,\\
\end{align*}
\begin{align*}
&~\mathcal R_*\cdot Q_\gamma E_1\\
=&~\eta_1 \cos w[\dot\gamma+\la^{-1}\rho^{-1}(\dot\xi_1 \sin\theta-\dot\xi_2 \cos\theta)](\phi_2^{(2)}+\phi_2^{(-2)})\\
&~+2\eta_1 \la^{-1}\rho^{-1}(\dot\xi_1 \sin\theta-\dot\xi_2 \cos\theta)[\phi_2^{(2)}-\phi_2^{(-2)}]\\
&~+\eta_1 \la^{-1}(\dot\xi_1\cos\theta+\dot\xi_2 \sin\theta+\dot\la \rho) (\pp_\rho \phi_1^{(2)}+\pp_\rho \phi_1^{(-2)})\\
&~+\eta_1 \la^{-1}w_\rho(\dot\xi_1\cos\theta+\dot\xi_2 \sin\theta+\dot\la \rho) \sin2\theta \psi_2 \\
&-\frac{2\eta_1^2 }r \left({\rm Im}\left[e^{-i\gamma }(\psi^0+\frac{r^2}{z}\pp_z \psi^0)\right]+\la^{-1}(\pp_\rho\phi_2^{(2)}+\pp_\rho\phi_2^{(-2)})\right)\\
&~\quad \times \left(\sin w {\rm Re}\left[e^{-i\gamma }ir\psi^0\right]+{\rm Im}(r e^{-i\theta}\psi^1)\cos w-\sin w(\phi_2^{(2)}+\phi_2^{(-2)})+2\psi_2\cos2\theta\right) \\
&~+\frac{2\eta_1^2 }r \left(\sin w {\rm Re}\left[e^{-i\gamma }(\psi^0+\frac{r^2}{z}\pp_z \psi^0)\right]+{\rm Re}[e^{-i\theta}(\psi^1+\frac{r^2}{z}\pp_z\psi^1)]\cos w\right)\\
&~\quad\times \left({\rm Im}\left[e^{-i\gamma }ir\psi^0\right]+(\pp_\theta \phi_2^{(2)}+\pp_\theta \phi_2^{(-2)})+\cos w (\phi_1^{(2)}+\phi_1^{(-2)})+\sin2\theta \sin w\psi_2\right) \\
&~+\frac{2\eta_1 }r \left({\rm Im}\left[e^{-i\gamma }ir\psi^0\right]+(\pp_\theta \phi_2^{(2)}+\pp_\theta \phi_2^{(-2)})+\cos w (\phi_1^{(2)}+\phi_1^{(-2)})+\sin2\theta \sin w\psi_2\right)\\
&~\quad\times \la^{-1}w_{\rho} \cos w (c_1 \cos\theta + c_2 \sin\theta)  \\
&~+\frac{2\eta_1 }r \Bigg[\left(\sin w {\rm Re}\left[e^{-i\gamma }(\psi^0+\frac{r^2}{z}\pp_z \psi^0)\right]+{\rm Re}[e^{-i\theta}(\psi^1+\frac{r^2}{z}\pp_z\psi^1)]\cos w\right)\\
&~\qquad+\la^{-1}\left[\sin2\theta\pp_\rho \psi_2-w_\rho(\phi_1^{(2)}+\phi_1^{(-2)})\right]\Bigg]\times\sin^2 w(c_1\cos\theta+ c_2\sin \theta) \\
&~-\frac{2\eta_1 }r \left({\rm Im}\left[e^{-i\gamma }(\psi^0+\frac{r^2}{z}\pp_z \psi^0)\right]+\la^{-1}(\pp_\rho\phi_2^{(2)}+\pp_\rho\phi_2^{(-2)})\right)\times \sin w (c_2 \cos\theta-c_1\sin\theta),  \\
\end{align*}
\begin{align*}
&~\mathcal R_*\cdot Q_\gamma E_2\\
=&~-\eta_1 \cos w[\dot\gamma+\la^{-1}\rho^{-1}(\dot\xi_1 \sin\theta-\dot\xi_2 \cos\theta)](\phi_1^{(2)}+\phi_1^{(-2)})  \\
&~+2\eta_1 \la^{-1}\rho^{-1}(\dot\xi_1 \sin\theta-\dot\xi_2 \cos\theta)[\phi_1^{(-2)}-\phi_1^{(2)}] \\
&~+\eta_1 \la^{-1}(\dot\xi_1\cos\theta+\dot\xi_2 \sin\theta+\dot\la \rho) (\pp_\rho \phi_2^{(2)}+\pp_\rho \phi_2^{(-2)}) \\
&~-\eta_1 \sin w[\dot\gamma+\la^{-1}\rho^{-1}(\dot\xi_1 \sin\theta-\dot\xi_2 \cos\theta)] \sin2\theta \psi_2  \\
&~-\frac{2\eta_1^2 }r \Bigg[\left(\sin w {\rm Re}\left[e^{-i\gamma }(\psi^0+\frac{r^2}{z}\pp_z \psi^0)\right]+{\rm Re}[e^{-i\theta}(\psi^1+\frac{r^2}{z}\pp_z\psi^1)]\cos w\right)\\
&~\qquad+\la^{-1}\left[\sin2\theta\pp_\rho \psi_2-w_\rho(\phi_1^{(2)}+\phi_1^{(-2)})\right]\Bigg]\\
&~\quad \times \left(\cos w {\rm Re}\left[e^{-i\gamma }ir\psi^0\right]-{\rm Im}(r e^{-i\theta}\psi^1)\sin w +(\pp_\theta \phi_1^{(2)}+\pp_\theta \phi_1^{(-2)})-\cos w (\phi_2^{(2)}+\phi_2^{(-2)})\right) \\
&~+\frac{2\eta_1^2 }r \Bigg[\left(\cos w {\rm Re}\left[e^{-i\gamma }(\psi^0+\frac{r^2}{z}\pp_z \psi^0)\right]-{\rm Re}[e^{-i\theta}(\psi^1+\frac{r^2}{z}\pp_z\psi^1)]\sin w\right)\\
&~\qquad+\la^{-1}\left(\sin2\theta w_\rho \psi_2+\pp_\rho\phi_1^{(2)}+\pp_\rho\phi_1^{(-2)}\right)\Bigg]\\
&~\quad \times \left(\sin w {\rm Re}\left[e^{-i\gamma }ir\psi^0\right]+{\rm Im}(r e^{-i\theta}\psi^1)\cos w-\sin w(\phi_2^{(2)}+\phi_2^{(-2)})+2\psi_2\cos2\theta\right)\\
&~-\frac{2\eta_1 }r \left(\cos w {\rm Re}\left[e^{-i\gamma }ir\psi^0\right]-{\rm Im}(r e^{-i\theta}\psi^1)\sin w +(\pp_\theta \phi_1^{(2)}+\pp_\theta \phi_1^{(-2)})-\cos w (\phi_2^{(2)}+\phi_2^{(-2)})\right)\\
&~\quad\times \la^{-1}w_{\rho} \cos w (c_1 \cos\theta + c_2 \sin\theta)  \\
&~+\frac{2\eta_1 }r \left(\sin w {\rm Re}\left[e^{-i\gamma }ir\psi^0\right]+{\rm Im}(r e^{-i\theta}\psi^1)\cos w-\sin w(\phi_2^{(2)}+\phi_2^{(-2)})+2\psi_2\cos2\theta\right)\\
&~\quad \times \la^{-1}w_{\rho}\sin w(c_1 \cos\theta + c_2 \sin\theta)  \\
&~+\frac{2\eta_1 }r \Bigg[\left(\cos w {\rm Re}\left[e^{-i\gamma }(\psi^0+\frac{r^2}{z}\pp_z \psi^0)\right]-{\rm Re}[e^{-i\theta}(\psi^1+\frac{r^2}{z}\pp_z\psi^1)]\sin w\right)\\
&~\qquad+\la^{-1}\left(\sin2\theta w_\rho \psi_2+\pp_\rho\phi_1^{(2)}+\pp_\rho\phi_1^{(-2)}\right)\Bigg]\times \sin w (c_2 \cos\theta-c_1\sin\theta) .
\end{align*}
\end{lemma}

The derivation of above Lemma is rather lengthy, and we postpone it in Appendix \ref{app-remainder}.

\medskip

\bigskip

\section{Formulating the gluing system}\label{gluingsystem-innerouter}

\medskip

We aim to find a real solution
$$
u=U_*+\Phi,
$$
so $S[u]=0$ yields
$$
\pp_t \Phi=\Delta \Phi-2\pp_{x_1} U_*\wedge \pp_{x_2}\Phi-2\pp_{x_1}\Phi\wedge\pp_{x_2}U_*-2\pp_{x_1}\Phi\wedge \pp_{x_2}\Phi+S[U_*].
$$
We decompose $\Phi$ into
$$
\Phi(x,t)= \eta_R Q_\gamma  \Big[\Phi_W\left(y, t\right) +\Phi_{W^\perp}\left(y, t\right)\Big] + \Psi(x, t), \quad y=\frac{x-\xi(t)}{\lambda(t)},
$$
where $\Phi_W$ is in the $W$-direction, $\Phi_{W^\perp}$ is on $W^\perp$,
\begin{equation*}
\eta_R := \eta\left(\frac{x-\xi(t)}{R(t)\lambda(t)}\right),
\end{equation*}
and $R(t)=\la^{-\beta}(t)$ with $\beta>0$ to be chosen later.

Then it is sufficient to find a desired solution $u$ to \eqref{e:main} if the triple $(\Phi_W,~\Phi_{W^\perp},~\Psi)$ satisfies the gluing system:
\begin{equation}\label{gluing-sys}
\begin{aligned}
\la^2\pp_t \Phi_W=&~\Delta_y\Phi_{W } - 2\pp_{y_1}W\wedge \pp_{y_2}\Phi_{W } - 2\pp_{y_1} \Phi_{W }\wedge \pp_{y_2}W\\
&~+\la^2\Pi_W\Big[Q_{-\gamma}\tilde L_U[\Psi]\Big]+\la^2 \Pi_W\Big[Q_{-\gamma}(S[U_*])\Big]+\la^2\mathcal H_{{\rm in}}^{W} \quad\mbox{  in  }~~ \mathcal{D}_{2R},\\
\la^2\pp_t \Phi_{W^\perp}=&~\Delta_y\Phi_{W^\perp} - 2\pp_{y_1}W\wedge \pp_{y_2}\Phi_{W^\perp} - 2\pp_{y_1} \Phi_{W^\perp}\wedge \pp_{y_2}W\\
&~+\la^2\Pi_{W^\perp}\Big[Q_{-\gamma}\tilde L_U[\Psi]\Big]+\la^2 \Pi_{W^\perp}\Big[Q_{-\gamma} (S[U_*])\Big]+\la^2\mathcal H_{{\rm in}}^{W^\perp}\quad\mbox{  in  }~~ \mathcal{D}_{2R},\\
\pp_t \Psi=&~\Delta_x \Psi +(1-\eta_R) \tilde L_U[\Psi]+(1-\eta_R)  S[U_*] \\
&~ + \Big[Q_\gamma(\Phi_W +\Phi_{W^\perp})\Delta_x\eta_R  + 2\nabla_x\eta_R\cdot\nabla_x(Q_\gamma\Phi_W +Q_\gamma\Phi_{W^\perp})-Q_\gamma(\Phi_W +\Phi_{W^\perp})\partial_t\eta_R\Big]\\
&~-2(1-\eta_R)\pp_{x_1}(U_*-U)\wedge \pp_{x_2}\Big(\eta_R Q_\gamma(\Phi_W+\Phi_{W^\perp})+\Psi\Big)\\
&~-2(1-\eta_R)\pp_{x_1}\Big(\eta_R Q_\gamma(\Phi_W+\Phi_{W^\perp})+\Psi\Big)\wedge \pp_{x_2}(U_*-U)\\
&~+(1-\eta_R)\Big[-\eta_R\dot\gamma J_z Q_\gamma(\Phi_W +\Phi_{W^\perp}) +\eta_R (\la^{-1}\dot\xi+\la^{-1}\dot\la y)\cdot \nabla_y(Q_\gamma\Phi_W +Q_\gamma\Phi_{W^\perp})\Big]\\
&~-2(1-\eta_R)\pp_{x_1}\Big(\eta_R Q_\gamma(\Phi_W+\Phi_{W^\perp})+\Psi\Big)\wedge \pp_{x_2}\Big(\eta_R Q_\gamma(\Phi_W+\Phi_{W^\perp})+\Psi\Big)  \mbox{  in  }~~ \R^2 \times(0,T),
\end{aligned}
\end{equation}
i.e., $\Phi_W$ and $\Phi_{W^\perp}$ satisfy inner problem in the $W$-direction and on $W^\perp$, respectively, and $\Psi$ solves the outer problem. Here
$$
\mathcal{D}_{2R}:=\Big\{(y,t)\in\R^2\times \R_+: ~|y|\leq 2R(t),~t\in(0,T)\Big\},
$$
$\widetilde{\mathcal R}_*$, $E_{\eta_1}$ are defined in \eqref{def-tildeR_*}, \eqref{def-E_eta1}, and
\begin{equation}\label{def-HWHWperp}
\begin{aligned}
\mathcal H_{{\rm in}}^{W}=&~\Pi_W\Bigg\{Q_{-\gamma}\bigg[-2\pp_{x_1}(U_*-U)\wedge \pp_{x_2}\Big(\eta_R Q_\gamma(\Phi_W+\Phi_{W^\perp})+\Psi\Big)\\
&~\qquad\qquad\qquad-2\pp_{x_1}\Big(\eta_R Q_\gamma(\Phi_W+\Phi_{W^\perp})+\Psi\Big)\wedge \pp_{x_2}(U_*-U)\\
&~\qquad\qquad\qquad -\eta_R\dot\gamma J_z Q_\gamma(\Phi_W +\Phi_{W^\perp}) +\eta_R (\la^{-1}\dot\xi+\la^{-1}\dot\la y)\cdot \nabla_y(Q_\gamma\Phi_W +Q_\gamma\Phi_{W^\perp})\\
&~\qquad\qquad\qquad-2\pp_{x_1}\Big(\eta_R Q_\gamma(\Phi_W+\Phi_{W^\perp})+\Psi\Big)\wedge \pp_{x_2}\Big(\eta_R Q_\gamma(\Phi_W+\Phi_{W^\perp})+\Psi\Big)\bigg]\Bigg\},\\
\mathcal H_{{\rm in}}^{W^\perp}=&~\Pi_{W^\perp}\Bigg\{Q_{-\gamma}\bigg[-2\pp_{x_1}(U_*-U)\wedge \pp_{x_2}\Big(\eta_R Q_\gamma(\Phi_W+\Phi_{W^\perp})+\Psi\Big)\\
&~\qquad\qquad\qquad-2\pp_{x_1}\Big(\eta_R Q_\gamma(\Phi_W+\Phi_{W^\perp})+\Psi\Big)\wedge \pp_{x_2}(U_*-U)\\
&~\qquad\qquad\qquad -\eta_R\dot\gamma J_z Q_\gamma(\Phi_W +\Phi_{W^\perp}) +\eta_R (\la^{-1}\dot\xi+\la^{-1}\dot\la y)\cdot \nabla_y(Q_\gamma\Phi_W +Q_\gamma\Phi_{W^\perp})\\
&~\qquad\qquad\qquad-2\pp_{x_1}\Big(\eta_R Q_\gamma(\Phi_W+\Phi_{W^\perp})+\Psi\Big)\wedge \pp_{x_2}\Big(\eta_R Q_\gamma(\Phi_W+\Phi_{W^\perp})+\Psi\Big)\bigg]\Bigg\}.
\end{aligned}
\end{equation}
We will solve equations \eqref{gluing-sys}$_1$ and \eqref{gluing-sys}$_2$ with zero initial data and \eqref{gluing-sys}$_3$ with non-trivial initial data
$$
\Psi(x,0)=Z^*=(z_1^*,z_2^*,z_3^*),\quad x\in\R^2.
$$
Problems \eqref{gluing-sys}$_1$ and \eqref{gluing-sys}$_2$ are the inner problem which captures the blow-up property, while \eqref{gluing-sys}$_3$ describes the main character with singularities removed. We will solve these problems by developing corresponding linear theories and the fixed point argument in suitable weighted spaces. We notice that the inner problem \eqref{gluing-sys}$_2$ is essentially a perturbation of the linearized harmonic map heat flow around the bubble $W$, and the inner problem \eqref{gluing-sys}$_1$ is a perturbation of the parabolic linearized Liouville equation (cf. Lemma \ref{linearization-lemma-0} and Remark \ref{rmk-211211}).

\medskip

\bigskip

\section{Leading dynamics for the parameters}\label{leading-dynamics}

\medskip

In this section, we capture the leading dynamics of the parameters $\la(t)$, $\xi(t)$, $\gamma(t)$, $c_1(t)$, $c_2(t)$, and these are in fact determined by orthogonality conditions in corresponding modes that are needed to ensure the existence of desired solutions with fast space-time decay (in the linear theories in Section \ref{sec-linearinner}). We will choose $\la R\ll 1$, so $\eta_1(x)=\eta_1^2(x)=1$ and $E_{\eta_1}=0$ (cf. \eqref{def-E_eta1}) for $|x-\xi(t)|\leq 2\la R$. Our aim is to adjust the parameters such that the orthogonalities hold
\begin{equation}\label{orthogo}
\begin{aligned}
&~\int_{B_{2R}} \left( \Pi_{W^\perp}\Big[Q_{-\gamma}\tilde L_U[\Psi]\Big]+  \Pi_{W^\perp}\Big[Q_{-\gamma} (S[U_*])\Big]+\mathcal H_{{\rm in}}^{W^\perp}\right) \cdot Z_{i,j}(y) dy=0,\quad i=0,1,~j=1,2,\\
&~\int_{B_{2R}} \left( \Pi_{W }\Big[Q_{-\gamma}\tilde L_U[\Psi]\Big]+  \Pi_{W }\Big[Q_{-\gamma} (S[U_*])\Big]+\mathcal H_{{\rm in}}^{W }\right) \cdot \mathcal Z_{1,j}(y) dy=0,\quad j=1,2\\
\end{aligned}
\end{equation}
for all $t\in(0,T)$, where the kernels $Z_{i,j}$ and $\mathcal Z_{1,j}$ are defined in \eqref{kernels}. The reason for not requiring orthogonalities with the other three kernel functions $Z_{-1,1}$, $Z_{-1,2}$ and $\mathcal Z_0$ is to avoid further complications due to the introduction of new modulation parameters, and we shall use linear theories  without orthogonality conditions in Section \ref{sec-linearinner} together with regularity estimates to control these modes.

\medskip

The goal of this section is to derive the leading dynamics governing these parameters by approximating \eqref{orthogo}. To do this, we decompose the remainder $\mathcal R_*$ in $S[U_*]$ as
\begin{equation}\label{def-tildeR_*}
\mathcal R_*:=\widetilde{\mathcal R}_*+\mathcal R_{*,1},
\end{equation}
where $\mathcal R_{*,1}$ is defined as
\begin{equation}\label{def-R_*1}
\begin{aligned}
&~\mathcal R_{*,1}\\
:=&~- \frac{2\eta_1 }r \sin w(c_1\cos\theta+ c_2\sin \theta)\Bigg(\la^{-1} w_\rho {\rm Im}\left[e^{-i\gamma }ir\psi^0\right] +\sin w\cos w {\rm Re}\left[e^{-i\gamma }(\psi^0+\frac{r^2}{z}\pp_z \psi^0)\right]  \Bigg)  Q_\gamma W \\
&~+ \frac{2\eta_1 }r (c_1\cos\theta+ c_2\sin \theta)\Bigg(\la^{-1} w_\rho \cos w {\rm Im}\left[e^{-i\gamma }ir\psi^0\right] +\sin^3 w  {\rm Re}\left[e^{-i\gamma }(\psi^0+\frac{r^2}{z}\pp_z \psi^0)\right]  \Bigg)  Q_\gamma E_1\\
&~-\frac{2\eta_1 }r \sin w (c_2 \cos\theta-c_1\sin\theta) {\rm Im}\left[e^{-i\gamma }(\psi^0+\frac{r^2}{z}\pp_z \psi^0)\right]Q_\gamma E_1\\
&~-\frac{2\eta_1 }r \la^{-1} w_\rho \cos 2w (c_1\cos\theta+c_2 \sin\theta)  {\rm Re}\left[e^{-i\gamma }ir\psi^0\right] Q_\gamma E_2\\
&~+\frac{2\eta_1 }r \sin w\cos w (c_2\cos\theta-c_1\sin\theta) {\rm Re}\left[e^{-i\gamma }(\psi^0+\frac{r^2}{z}\pp_z \psi^0)\right] Q_\gamma E_2.
\end{aligned}
\end{equation}
So the error reads
\begin{equation}\label{error-decomp}
\begin{aligned}
S[U_*]=&~\Big[\mathcal R_{U^\perp}+\mathcal R_{U}+\mathcal E_{U^\perp}^{(1)}+\eta_1 (\mathcal E_{U^\perp}^{(0)}+\tilde{\mathcal R}^0)+\eta_1 (\mathcal E_{U}^{(\pm1)}+  \tilde{\mathcal R}^1)+\eta_1 \tilde L_U[\Phi^{(0)}+\Phi^{(1)}]\Big]+\widetilde{\mathcal R}_*+\mathcal R_{*,1}+E_{\eta_1}.
\end{aligned}
\end{equation}
In what follows, $S[U_*]$ in the orthogonality conditions will be approximated by its main terms, and the analysis of the remainders will be done in Section \ref{solve-ortho}.

\medskip

\subsection{$\la$-$\gamma$ system and translation parameter $\xi$}

We start from the RHS on $W^\perp$, and this turns out to yield the dynamics for $\la$ and $\gamma$ at mode $0$ and for $\xi:=(\xi_1,\ \xi_2)$ at mode $1$.

\medskip

\noindent $\bullet$ {\bf Mode $0$ on $W^\perp$: $\la$-$\gamma$ system.}

\medskip

The orthogonality condition \eqref{orthogo}$_1$ with $i=0$ and $j=1,2$ in the complex form implies
\begin{equation}\label{orthogo-Wperp0}
\int_0^{2\pi}\int_0^{2R} \left(\Pi_{W^\perp}\Big[Q_{-\gamma}\tilde L_U[\Psi]\Big]+  \Pi_{W^\perp}\Big[Q_{-\gamma} (S[U_*])\Big]+\mathcal H_{{\rm in}}^{W^\perp}\right)_{\mathbb C} \rho^2 w_\rho   d\rho d\theta=0,
\end{equation}
which is a complex system for $\la$ and $\gamma$. We single out the main terms at mode $0$:
\begin{equation*}
\left(\Pi_{W^\perp}\Big[Q_{-\gamma}\tilde L_U[Z_0^*(q)]\Big]+  \Pi_{W^\perp}\Big[Q_{-\gamma} \Big((\mathcal E_{U^\perp}^{(0)}+\tilde{\mathcal R}^0)+\tilde L_U[\Phi^{(0)}]\Big)\Big]\right)_{\mathbb C,0}.\\
\end{equation*}
We will deal with the remainder terms, turn out to be faster vanishing in time, when we solve the full problem via fixed point argument. From Lemma \ref{linearization-lemma-2}, one has
\begin{equation*}
\begin{aligned}
&~\left(\Pi_{W^\perp}\Big[Q_{-\gamma}\tilde L_U[Z_0^*(q)]\Big]\right)_{\mathbb C,0}\\
=&~\la^{-1}\rho w_\rho^2[\div(e^{-i\gamma }Z^*(q)) +i {\rm{curl }}(e^{-i\gamma }Z^*(q))].
\end{aligned}
\end{equation*}
Also, from \eqref{eqn-323323323323}, \eqref{eqn-324324324324},   \eqref{eqn-327327327}, and \eqref{eqn-psi^0}, we obtain
\begin{align*}
&~\left(\Pi_{W^\perp}\Big[Q_{-\gamma} \Big((\mathcal E_{U^\perp}^{(0)}+\tilde{\mathcal R}^0)+\tilde L_U[\Phi^{(0)}]\Big)\Big]\right)_{\mathbb C,0}\\
=&~\Big(\cos w {\rm Re}[(f_{U^\perp}^{(0)})e^{-i\gamma }]-\la^{-1}\dot\la\rho^2 w_\rho^2\sin w \Big)+i {\rm Im}[(f_{U^\perp}^{(0)})e^{-i\gamma }]\\
&~+2\la^{-1}\rho w^2_{\rho}{\rm Re}\Big[e^{-i\gamma }(\psi^0+\frac{r^2}{z}\pp_z \psi^0)\Big]+i2\la^{-1}  \rho w^2_\rho{\rm Im}(e^{-i\gamma }\psi^0)\\
=&~\cos w {\rm Re}\Bigg[\frac{\la^{-1}\rho}{(1+\rho^2)^2}\int_{-T}^t p_0(s) e^{-i\gamma(t)}\big(zk_z(z,t-s)-z^2k_{zz}(z,t-s)\big)ds\Bigg]\\
&~-\cos w {\rm Re}\Bigg[\frac{\dot\la \rho}{1+\rho^2}\int_{-T}^t p_0(s) e^{-i\gamma(t)} zk_z(z,t-s) ds\Bigg]-\la^{-1}\dot\la \rho w_\rho^2\\
&~+2\la^{-1}\rho w_\rho^2 {\rm Re}\Bigg[\int_{-T}^t p_0(s) e^{-i\gamma(t)}\bigg(k(z,t-s)+\frac{r^2}{z}k_{z}(z,t-s)\bigg)ds \Bigg]\\
&~+i{\rm Im}\Bigg[\frac{\la^{-1}\rho}{(1+\rho^2)^2}\int_{-T}^t p_0(s) e^{-i\gamma(t)}\big(zk_z(z,t-s)-z^2k_{zz}(z,t-s)\big)ds\Bigg]\\
&~-i{\rm Im}\Bigg[\frac{\dot\la \rho}{1+\rho^2}\int_{-T}^t p_0(s) e^{-i\gamma(t)} zk_z(z,t-s) ds\Bigg]\\
&~+i2\la^{-1}  \rho w^2_\rho {\rm Im}\Big[\int_{-T}^t p_0(s) e^{-i\gamma(t)}k(z,t-s)ds \Big]\\
=&~-\cos w {\rm Re}\Bigg[\la^{-1}\rho w_\rho^2 \int_{-T}^t \frac{p_0(s)e^{-i\gamma(t)}}{t-s}\zeta^2 K_{\zeta\zeta} ds\Bigg]\\
&~+\cos w {\rm Re} \Bigg[\dot\la \rho w_\rho  \int_{-T}^t \frac{p_0(s)e^{-i\gamma(t)}}{t-s}\zeta  K_{\zeta } ds\Bigg]-\la^{-1}\dot\la \rho w_\rho^2\\
&~+2\la^{-1}\rho w_\rho^2 {\rm Re}\Bigg[ \int_{-T}^t \frac{p_0(s)e^{-i\gamma(t)}}{t-s}\Big(K+\frac{2\rho^2}{1+\rho^2}\zeta K_\zeta\Big) ds\Bigg]\\
&~+i {\rm Im}\Bigg[\la^{-1}\rho w_\rho^2 \int_{-T}^t \frac{p_0(s)e^{-i\gamma(t)}}{t-s}\zeta^2 K_{\zeta\zeta} ds\Bigg]\\
&~+i {\rm Im} \Bigg[\dot\la \rho w_\rho  \int_{-T}^t \frac{p_0(s)e^{-i\gamma(t)}}{t-s}\zeta  K_{\zeta } ds\Bigg] +i2\la^{-1}  \rho w^2_\rho {\rm Im} \Bigg[ \int_{-T}^t \frac{p_0(s)e^{-i\gamma(t)}}{t-s}K ds\Bigg],
\end{align*}
where
\begin{equation}\label{def-Kzeta}
k(z, t) = \frac{1-e^{-\frac{z^2}{4t}}}{z^2}, \quad \zeta=\frac{z^2}{t-s}=\frac{\la^2(1+\rho^2)}{t-s}, \quad K(\zeta) = \frac{1-e^{-\frac{\zeta}{4}}}{\zeta}.
\end{equation}
Therefore, \eqref{orthogo-Wperp0} at leading order gives the following system in complex form:
\begin{equation}\label{eqn-575757}
\begin{aligned}
&~\int_{-T}^t \frac{{{\rm Re}}(p_0(s)e^{-i\gamma (t)})}{t-s}\Gamma_1\left(\frac{\lambda(t)^2}{t-s}\right)ds + 2\dot\lambda(t)+i\int_{-T}^t \frac{{{\rm Im}}(p_0(s)e^{-i\gamma (t)})}{t-s}\Gamma_2\left(\frac{\lambda(t)^2}{t-s}\right)ds\\
=&~2\Big[\div(e^{-i\gamma(t)}Z^*(q)) +i {\rm{curl }}(e^{-i\gamma(t)}Z^*(q))\Big],
\end{aligned}
\end{equation}
where
\begin{equation*}
\begin{aligned}
\Gamma_1(\tau) = &~ \int_{0}^\infty \Bigg(\rho^3w_\rho^3\left[\Big(2K(\zeta) + 4\zeta K_\zeta(\zeta) \frac{\rho^2}{1+\rho^2} \Big)- \cos w \zeta^2 K_{\zeta\zeta}(\zeta)\right]+\cos w \la\dot\la \rho^3 w^2_\rho \zeta K_\zeta\Bigg)_{\zeta =\tau(1+\rho^2)}d\rho,\\
\Gamma_2(\tau) = &~ \int_{0}^\infty \Big(\rho^3w_\rho^3\left[2K(\zeta) + \zeta^2 K_{\zeta\zeta}(\zeta)\right]+\la\dot\la \rho^3 w^2_\rho \zeta K_\zeta\Big)_{\zeta =\tau(1+\rho^2)}d\rho,
\end{aligned}
\end{equation*}
and we have used $\int_0^{+\infty} \rho^3 w_\rho^3 d\rho =-2.$ Clearly,
$$
\Gamma_1(\tau)=\left\{
\begin{aligned}
&-1+O(\tau), ~&~\tau\leq 1,\\
&O\left(\frac1{\tau}\right), ~&~\tau> 1,\\
\end{aligned} \quad
\right.
\quad
\Gamma_2(\tau)=\left\{
\begin{aligned}
&-1+O(\tau), ~&~\tau\leq 1,\\
&O\left(\frac1{\tau}\right), ~&~\tau> 1,\\
\end{aligned} \quad
\right.
$$
so it follows from \eqref{eqn-575757} that
\begin{equation}\label{eqn-585858}
\begin{aligned}
\int_{-T}^{t-\la^2} \frac{ p_0(s)  }{t-s} ds =-2\Big[\div Z^*(q) +i {\rm{curl }} Z^*(q)\Big]+O(p_0(t))+o(1),
\end{aligned}
\end{equation}
as $t\to T$.  Recall $p_0(t) := -2(\dot\lambda + i\lambda\dot\gamma )e^{i\gamma }$. We then proceed as \cite[Section 5, p. 372]{17HMF} to obtain
\begin{equation*}
\begin{aligned}
\la(t)\sim \frac{T-t}{|\log(T-t)|^2},\quad \gamma(t)= \gamma_*= \arctan\frac{\curl Z^*(q)}{\div Z^*(q)}
\end{aligned}
\end{equation*}
as $t\to T$ under the assumption that $\div Z^*(q) < 0$. So the leading part of $\la(t)$ is given by
\begin{equation*}
\la_*(t)=\frac{(T-t)|\log T|}{|\log(T-t)|^2},
\end{equation*}
where the $|\log T|$ is a normalization factor. As we will see later, the dynamics of $\la$-$\gamma$ system and $c_1$-$c_2$ system are in fact coupled, and this results in an extra restriction:
\begin{equation}\label{mathbfZ}
\cos\gamma_* \div Z^*(q) + \sin\gamma_* \curl Z^*(q)\leq 0,
\end{equation}
 roughly speaking, yielding more instability.

\begin{remark}
Both restrictions $\div Z^*(q) < 0$ and \eqref{mathbfZ} can be achieved at the same time by choosing initial data $Z^*$ such that
$$
\div  Z^*(q)<0,\quad \left|\frac{\curl Z^*(q)}{\div Z^*(q)}\right|\ll 1
$$
since once $Z^*$ is fixed the rotation angle $\gamma_*$ is determined automatically by
$$
\tan \gamma_* = \frac{\curl Z^*(q)}{\div Z^*(q)}.
$$
In such case, $\gamma_*\ll 1$ and thus one has \eqref{mathbfZ}.
\end{remark}

\bigskip

\noindent $\bullet$ {\bf Mode $1$ on $W^\perp$: $\xi$.}

\medskip

Similarly, the orthogonality condition \eqref{orthogo}$_1$ with $i=1$ and $j=1,2$ in the complex form gives
\begin{equation}\label{orthogo-Wperp1}
\int_0^{2\pi}\int_0^{2R} \left(\Pi_{W^\perp}\Big[Q_{-\gamma}\tilde L_U[\Psi]\Big]+  \Pi_{W^\perp}\Big[Q_{-\gamma} (S[U_*])\Big]+\mathcal H_{{\rm in}}^{W^\perp}\right)_{\mathbb C}   w_\rho e^{-i\theta} \rho d\rho d\theta=0.
\end{equation}
From \eqref{eqn-321}, \eqref{eqn-322}, \eqref{eqn-323323323323}, \eqref{eqn-324324324324}, \eqref{eqn-327},  \eqref{eqn-327327327}, \eqref{def-R_*} and \eqref{def-tildeR_*}, the main terms that have contribution in $B_{2R}$ in above integration are given by
\begin{align*}
&~\left(\Pi_{W^\perp}\Big[Q_{-\gamma}\tilde L_U[\Psi]\Big]+  \Pi_{W^\perp}\Big[Q_{-\gamma} (S[U_*]-\widetilde{\mathcal R}_*)\Big]\right)_{\mathbb C,1}\\
= &~\left(\Pi_{W^\perp}\Big[Q_{-\gamma}\tilde L_U[Z_0^*(q)]\Big]+  \Pi_{W^\perp}\Big[Q_{-\gamma} (S[U_*]-\widetilde{\mathcal R}_*)\Big]\right)_{\mathbb C,1}+\left(\Pi_{W^\perp}\Big[Q_{-\gamma}\tilde L_U[\Psi-Z_0^*(q)]\Big]\right)_{\mathbb C,1}\\
=&~-2\la^{-1}w_\rho \cos w e^{i\theta}\Big(\pp_{x_1} z^*_3(q)-i\pp_{x_2}z^*_3(q)\Big)+\left(\Pi_{W^\perp}\Big[Q_{-\gamma}\tilde L_U[\Psi-Z_0^*(q)]\Big]\right)_{\mathbb C,1}\\
&~+\la^{-1}w_\rho (\dot\xi_1-i\dot\xi_2) e^{i\theta}+
\rho w_\rho \sin w (\la^{-1}\dot\la+i\dot\gamma)\frac{e^{i\theta}}{2}(c_1-ic_2)\\
&~+
 \frac12 e^{i(\theta-\gamma)}(\dot\xi_1-i\dot\xi_2)\int_{-T}^tp_0(s)\Big(k+\frac{r^2}{z}k_z\Big)(z , t-s)ds
 \\
 &~+\frac{w_\rho}{4} e^{i\theta}(\dot\xi_1-i\dot\xi_2) \Bigg[\int_{-T}^t \Big(p_0(s)e^{-i\gamma(t)}+\bar{p}_0(s)e^{i\gamma(t)}\Big)\Big(k+\frac{r^2}{z}k_z\Big)(z,t-s)ds\Bigg]\\
 &~+\frac{\la^{-1}}{4} \rho^2 w_\rho^2 e^{i\theta} \bar{p}_1(t)+\frac{\la^{-1}}{8} \rho^2 w_\rho^3 e^{i\theta}\int_{-T}^t \bar{p}_1(s)(zk_z-z^2 k_{zz})(z,t-s)ds\\
 &~+\frac14 \rho^2 w_\rho^2 e^{i\theta} \dot\la \int_{-T}^t \bar{p}_1(s)zk_z(z,t-s)ds-2\la^{-1}w_\rho \cos w e^{i\theta} \int_{-T}^t \bar{p}_1(s)k(z,t-s)ds\\
 &~-\la^{-1}w_{\rho} \cos w \frac{r^2}{z} e^{i\theta} \int_{-T}^t \bar{p}_1(s)k_z(z,t-s)ds\\
&~+  \frac{2 }r (c_1\cos\theta+ c_2\sin \theta)\Bigg(\la^{-1} w_\rho \cos w {\rm Im}\left[e^{-i\gamma }ir\psi^0\right] +\sin^3 w  {\rm Re}\left[e^{-i\gamma }(\psi^0+\frac{r^2}{z}\pp_z \psi^0)\right]  \Bigg)   \\
&~-\frac{2  }r \sin w (c_2 \cos\theta-c_1\sin\theta) {\rm Im}\left[e^{-i\gamma }(\psi^0+\frac{r^2}{z}\pp_z \psi^0)\right] \\
&~-i\frac{2  }r \la^{-1} w_\rho \cos 2w (c_1\cos\theta+c_2 \sin\theta)  {\rm Re}\left[e^{-i\gamma }ir\psi^0\right]  \\
&~+i\frac{2  }r \sin w\cos w (c_2\cos\theta-c_1\sin\theta) {\rm Re}\left[e^{-i\gamma }(\psi^0+\frac{r^2}{z}\pp_z \psi^0)\right] .
\end{align*}
Then the dynamics given by \eqref{orthogo-Wperp1} is approximately the following ODE:
\begin{equation}\label{eqn-515515515}
\begin{aligned}
\dot\xi_1-i\dot\xi_2 \sim &~CR^{-2} \Big(\pp_{x_1} z^*_3(q)-i\pp_{x_2}z^*_3(q)\Big)+O(|\cc|)\\
&~-\frac{\la}{4\pi}\int_0^{2\pi}\int_0^{2R}   \left(\Pi_{W^\perp}\Big[Q_{-\gamma}\tilde L_U[\Psi-Z_0^*(q)]\Big]+\Pi_{W^\perp}\Big[Q_{-\gamma} \widetilde{\mathcal R}_*\Big]+\mathcal H_{{\rm in}}^{W^\perp}\right)_{\mathbb C,1} w_\rho e^{-i\theta} \rho d\rho d\theta\\
\end{aligned}
\end{equation}
for some constant $C$, where we have used the fact that $\int_0^{+\infty} \rho w_\rho^2 \cos w d\rho=0$.  For later purpose (when dealing with $c_1$-$c_2$ system), we will choose initial data such that
$$\pp_{x_1} z^*_3(q)=\pp_{x_2}z^*_3(q)=0,$$
and
$$
\la\Big|Q_{-\gamma}\tilde L_U[\Psi-Z_0^*(q)]\Big| \lesssim \la^{\Theta},\quad |\cc|\lesssim \la^{\Theta}(t)
$$
for some $0<\Theta<1$. For the gluing procedure to work, we will eventually choose $\Theta$ to be slightly bigger than $1/3$ (cf. the final choice \eqref{finalchoice}). Because of this, also by estimates \eqref{est-710710710} and \eqref{est-732732732} that will be carried out later, one has
$$
\la\left|\Pi_{W^\perp}\Big[Q_{-\gamma} \widetilde{\mathcal R}_*\Big]+\mathcal H_{{\rm in}}^{W^\perp}\right|\lesssim \la^{3\Theta-1}.
$$
Therefore, there exists a solution
$$
\xi(t)=q+o((T-t)^{3\Theta})
$$
for above ODE, where we recall  $\la(t)\sim \frac{T-t}{|\log(T-t)|^2}$.

\medskip

\subsection{$c_1$-$c_2$ system}

\medskip

Finally, we derive the asymptotic behavior of the translation parameters in the $W$-direction. From the linear theory of the inner problem in the $W$-direction, we need:
\begin{equation}\label{orthogo-W1}
\begin{aligned}
&~\int_0^{2\pi}\int_0^{2R} \left[ Q_{-\gamma}\tilde L_U[\Psi] \cdot W+   \Big(Q_{-\gamma}(S[U_*])+\mathcal H_{{\rm in}}^{W}\Big) \cdot W\right] \sin w \cos\theta \rho d\rho d\theta=0,\\
&~\int_0^{2\pi}\int_0^{2R} \left[ Q_{-\gamma}\tilde L_U[\Psi] \cdot W+   \Big(Q_{-\gamma}(S[U_*]) +\mathcal H_{{\rm in}}^{W}\Big) \cdot W\right] \sin w \sin\theta \rho d\rho d\theta=0.
\end{aligned}
\end{equation}
By \eqref{eqn-321}, \eqref{eqn-323323323323}, \eqref{eqn-327},   \eqref{eqn-327327327}, \eqref{def-R_*} and \eqref{def-tildeR_*}, we have
\begin{align*}
&~Q_{-\gamma}(S[U_*]-\widetilde{\mathcal R}_*)\cdot W\\
=&~Q_{-\gamma}\Big(\mathcal R_{U^\perp}+\mathcal R_{U}+\mathcal E_{U^\perp}^{(1)}+(\mathcal E_{U^\perp}^{(0)}+\tilde{\mathcal R}^0)+(\mathcal E_{U}^{(\pm1)}+\tilde{\mathcal R}^1)+\tilde L_U[\Phi^{(0)}+\Phi^{(1)}]\Big)\cdot W\\
&~+Q_{-\gamma}(\mathcal R_*-\widetilde{\mathcal R}_*)\cdot W\\
=&~-\la^{-1}\rho^{-1}\sin w(\dot\xi_1\sin\theta-\dot\xi_2\cos\theta)(c_2\cos\theta-c_1\sin\theta)\\
&~+\la^{-1}w_\rho [\cos w(\dot\xi_1\cos\theta+\dot\xi_2 \sin\theta)+\dot\la \rho w_\rho](c_1\cos\theta+c_2\sin\theta) \\
&~+\sin w {\rm Re}[(f_{U^\perp}^{(0)}+f_{U^\perp}^{(1)})e^{-i\gamma }]+ \la^{-1}\dot\la\rho^2 w_\rho^2\cos w\\
&~-\rho w^2_\rho\left[\cos\theta (\dot c_1+\la^{-1}\dot\la c_1) +\sin\theta (\dot c_2+\la^{-1}\dot\la c_2)\right]  \\
&~+{\rm Re}\left[re^{-i\theta}\frac{\lambda^2}{z^4}\int_{-T}^tp_1(s)(zk_z-z^2k_{zz})(z(r), t-s)ds+\mathcal R_1^1\right]\cos w\\
&~+\frac{2}{\lambda}w_{\rho}\left({\rm Re}(e^{-i\gamma }(\psi^0+\frac{r^2}{z}\pp_z \psi^0))\cos w - \frac{1}{r}{\rm Re}(e^{-i\gamma }r\psi^0)+\rho w_\rho{\rm Re}[e^{-i\theta}(\psi^1+\frac{r^2}{z}\psi_z^1)] \right),\\
&~-\frac2r \sin w(c_1\cos\theta+ c_2\sin \theta)\Bigg(\la^{-1} w_\rho {\rm Im}\left[e^{-i\gamma }ir\psi^0\right] +\sin w\cos w {\rm Re}\left[e^{-i\gamma }(\psi^0+\frac{r^2}{z}\pp_z \psi^0)\right]  \Bigg),
\end{align*}
and by Lemma \ref{linearization-lemma-2},
\begin{equation*}
\begin{aligned}
&~Q_{-\gamma}\tilde L_U[\Psi] \cdot W\approx Q_{-\gamma}\tilde L_U[Z^*(q)] \cdot W\\
=&~\frac{1}{\lambda}w_{\rho}^2\div(e^{-i\gamma }Z^*(q))) -\frac{2}{\lambda}w_\rho\sin w[\partial_{x_1}z^*_3(q)\cos\theta +\partial_{x_2}z^*_3(q)\sin\theta]  \\
&~+\frac{2}{\lambda} w_\rho\left[-\frac{1}{2}\rho^2w_\rho\div(e^{i\gamma }\bar{Z}^*(q)) \cos(2\theta) +\frac{1}{2}\rho^2w_\rho{\rm{curl }}(e^{i\gamma }\bar{Z}^*(q))\sin(2\theta)\right] .
\end{aligned}
\end{equation*}
So the terms in $\left( Q_{-\gamma}\tilde L_U[\Psi] \cdot W+   Q_{-\gamma}(S[U_*]-\widetilde{\mathcal R}_*) \cdot W\right)$ that might contribute in \eqref{orthogo-W1} are given by
\begin{align*}
&~-2\la^{-1}w_\rho\sin w[\partial_{x_1}z^*_3(q)\cos\theta +\partial_{x_2}z^*_3(q)\sin\theta]+Q_{-\gamma}\tilde L_U[\Psi-Z^*(q)] \cdot W\\
&~-\la^{-1}\rho^{-1}\sin w(\dot\xi_1\sin\theta-\dot\xi_2\cos\theta)(c_2\cos\theta-c_1\sin\theta)\\
&~+\la^{-1}w_\rho [\cos w(\dot\xi_1\cos\theta+\dot\xi_2 \sin\theta)+\dot\la \rho w_\rho](c_1\cos\theta+c_2\sin\theta) \\
&~+\sin w {\rm Re}[( f_{U^\perp}^{(1)})e^{-i\gamma }] -\rho w^2_\rho\left[\cos\theta (\dot c_1+\la^{-1}\dot\la c_1) +\sin\theta (\dot c_2+\la^{-1}\dot\la c_2)\right]  \\
&~+{\rm Re}\left[re^{-i\theta}\frac{\lambda^2}{z^4}\int_{-T}^tp_1(s)(zk_z-z^2k_{zz})(z(r), t-s)ds+\mathcal R_1^1\right]\cos w\\
&~+\frac{2}{\lambda}w_{\rho}\left(\rho w_\rho{\rm Re}[e^{-i\theta}(\psi^1+\frac{r^2}{z}\psi_z^1)] \right)\\
&~-\frac2r \sin w(c_1\cos\theta+ c_2\sin \theta) \la^{-1} w_\rho {\rm Im}\left[e^{-i\gamma }ir\psi^0\right] \\
&~-\frac2r \sin w(c_1\cos\theta+ c_2\sin \theta)\sin w\cos w {\rm Re}\left[e^{-i\gamma }(\psi^0+\frac{r^2}{z}\pp_z \psi^0)\right]   \\
=&~-2\la^{-1}w_\rho\sin w[\partial_{x_1}z^*_3(q)\cos\theta +\partial_{x_2}z^*_3(q)\sin\theta]+Q_{-\gamma}\tilde L_U[\Psi-Z^*(q)] \cdot W\\
&~+\la^{-1}w_\rho^2 \Big[\dot\xi_1 c_1\cos^2\theta+\dot\xi_2 c_2 \sin^2 \theta+\dot\xi_1 c_2 \sin\theta\cos \theta+\dot\xi_2 c_1 \sin\theta\cos\theta\Big]\\
&~+\la^{-1}w_\rho\Big[\sin2\theta(\dot\xi_1 c_2+\dot\xi_2 c_1)+\cos2\theta (\dot\xi_1 c_1 -\dot\xi_2 c_2)\Big]-\rho w^2_\rho\left[\cos\theta  \dot c_1 +\sin\theta  \dot c_2 \right]\\
&~+\sin w {\rm Re}[(\dot\xi e^{-i\theta})] {\rm Re}\Bigg[  \int_{-T}^tp_0(s) e^{-i\gamma(t)}\Big(k+\frac{r^2}{z}k_z \Big)(z , t-s)ds\Bigg]\\
&~+\cos w{\rm Re}\left[re^{-i\theta}\frac{\lambda^2}{z^4}\int_{-T}^tp_1(s)(zk_z-z^2k_{zz})(z(r), t-s)ds\right]\\
&~+\cos w{\rm Re}\Bigg[  e^{-i\theta}{\rm Re}[(\dot\xi e^{-i\theta})]\int_{-T}^tp_1(s)k(z(r), t-s)ds\\
&~\qquad\qquad\quad - \frac{re^{-i\theta}}{z^2}\left(\lambda\dot\lambda-{{\rm Re}}(re^{-i\theta}\dot\xi(t))\right)\int_{-T}^tp_1(s)zk_z(z(r), t-s)ds\Bigg]\\
&~+2\la^{-1}\rho w^2_{\rho}  {\rm Re}\left[e^{-i\theta}\int_{-T}^t p_1(s)\Big(k+\frac{r^2}{z}k_z\Big)(z,t-s)ds\right]\\
&~+2\la^{-1} w_\rho \sin w (c_1\cos\theta+c_2 \sin\theta) {\rm Re}\Bigg[  \int_{-T}^tp_0(s) e^{-i\gamma(t)} \frac{r^2}{z}k_z(z,t-s) ds\Bigg]\\
&~+2\la^{-1} w^2_\rho \sin w (c_1\cos\theta+c_2 \sin\theta) {\rm Re}\Bigg[  \int_{-T}^tp_0(s) e^{-i\gamma(t)}\Big(k+\frac{r^2}{z}k_z \Big)(z , t-s)ds\Bigg].
\end{align*}
So orthogonality condition \eqref{orthogo-W1}$_1$ implies
\begin{align*}
&~-2\pi \partial_{x_1}z^*_3(q) \la^{-1}\int_0^{+\infty} \rho w_\rho \sin^2 w d\rho-\pi\dot c_1 \int_0^{+\infty} \rho^2 w^2_\rho \sin w d\rho\\
&~+\int_0^{2\pi}\int_0^{+\infty} \left( Q_{-\gamma}\tilde L_U[\Psi-Z^*(q)] \cdot W\right) \sin w \cos\theta \rho d\rho d\theta \\
&~+\pi\dot\xi_1\int_0^{+\infty}\rho\sin^2 w  {\rm Re}\Bigg[  \int_{-T}^tp_0(s) e^{-i\gamma(t)}\Big(k +\frac{r^2}{z}k_z\Big)(z,t-s)ds\Bigg] d\rho\\
&~+\frac{\pi}{4} \la^{-1}\int_0^{+\infty}\rho^2 w_\rho^2 \sin w\cos w \int_{-T}^t {\rm Re}[p_1](s)(zk_z-z^2k_{zz})(z , t-s)ds d\rho\\
&~+\frac{\pi}{2}\dot\la \int_0^{+\infty} \rho^2 w_\rho \sin w \cos w \int_{-T}^t {\rm Re}[p_1](s) zk_z(z , t-s)ds d\rho\\
&~+2\pi\la^{-1}\int_0^{+\infty}\rho^2 w^2_{\rho} \sin w  \int_{-T}^t {\rm Re}[p_1](s)\Big(k+\frac{r^2}{z}k_z\Big)(z,t-s)ds  d\rho\\
&~+2\pi\la^{-1}c_1 \int_0^{+\infty} \rho w_\rho \sin^2 w   {\rm Re}\Bigg[  \int_{-T}^tp_0(s) e^{-i\gamma(t)} \frac{r^2}{z}k_z(z,t-s) ds\Bigg] d\rho\\
&~+2\pi \la^{-1} c_1 \int_0^{+\infty} \rho w^2_\rho \sin^2 w  {\rm Re}\Bigg[  \int_{-T}^tp_0(s) e^{-i\gamma(t)}\Big(k +\frac{r^2}{z}k_z\Big)(z,t-s)ds\Bigg] d\rho\\
&~=0,
\end{align*}
and \eqref{orthogo-W1}$_2$ implies
\begin{align*}
&~-2\pi \partial_{x_2}z^*_3(q) \la^{-1}\int_0^{+\infty} \rho w_\rho \sin^2 w d\rho-\pi\dot c_2 \int_0^{+\infty} \rho^2 w^2_\rho \sin w d\rho\\
&~+\int_0^{2\pi}\int_0^{+\infty} \left( Q_{-\gamma}\tilde L_U[\Psi-Z^*(q)] \cdot W\right) \sin w \sin\theta \rho d\rho d\theta \\
&~+\pi\dot\xi_2\int_0^{+\infty}\rho\sin^2 w  {\rm Re}\Bigg[  \int_{-T}^tp_0(s) e^{-i\gamma(t)}\Big(k(z , t-s)+\frac{r^2}{z}k_z(z,t-s)\Big)ds\Bigg] d\rho\\
&~+\frac{\pi}{4} \la^{-1}\int_0^{+\infty}\rho^2 w_\rho^2 \sin w\cos w \int_{-T}^t {\rm Im}[p_1](s)(zk_z-z^2k_{zz})(z(r), t-s)ds d\rho\\
&~+\frac{\pi}{2}\dot\la \int_0^{+\infty} \rho^2 w_\rho \sin w \cos w \int_{-T}^t {\rm Im}[p_1](s) zk_z(z , t-s)ds d\rho\\
&~+2\pi\la^{-1}\int_0^{+\infty}\rho^2 w^2_{\rho} \sin w  \int_{-T}^t {\rm Im}[p_1](s)\Big(k+\frac{r^2}{z}k_z\Big)(z,t-s)ds  d\rho\\
&~+2\pi\la^{-1}c_2 \int_0^{+\infty} \rho w_\rho \sin^2 w   {\rm Re}\Bigg[  \int_{-T}^tp_0(s) e^{-i\gamma(t)} \frac{r^2}{z}k_z(z,t-s) ds\Bigg] d\rho\\
&~+2\pi \la^{-1} c_2 \int_0^{+\infty} \rho w^2_\rho \sin^2 w  {\rm Re}\Bigg[  \int_{-T}^tp_0(s) e^{-i\gamma(t)}\Big(k +\frac{r^2}{z}k_z\Big)(z,t-s)ds\Bigg] d\rho\\
&~=0.
\end{align*}
Neglecting the terms that have faster vanishing in time, we arrive at
\begin{align*}
&~4\partial_{x_1}z^*_3(q)    -2\la \dot c_1  +\frac{\la}{\pi}\int_0^{2\pi}\int_0^{+\infty} \left( Q_{-\gamma}\tilde L_U[\Psi-Z^*(q)] \cdot W\right) \sin w \cos\theta \rho d\rho d\theta \\
&~-  \int_0^{+\infty}\rho^2 w_\rho^2 \sin w\cos w \int_{-T}^t \frac{{\rm Re}[p_1](s)}{t-s}\zeta^2 K_{\zeta\zeta}ds d\rho\\
&~+2 \int_0^{+\infty}\rho^2 w^2_{\rho} \sin w  \int_{-T}^t \frac{{\rm Re}[p_1](s)}{t-s}\Big(K(\zeta)+\frac{2\rho^2}{1+\rho^2}\zeta K_\zeta\Big) ds  d\rho\\
&~+2 c_1 \int_0^{+\infty} \rho w_\rho \sin^2 w    \int_{-T}^t \frac{ {\rm Re}[p_0(s)e^{-i\gamma(t)}] }{t-s}  \frac{2\rho^2}{1+\rho^2} \zeta K_\zeta ds  d\rho\\
&~+2  c_1 \int_0^{+\infty} \rho w^2_\rho \sin^2 w   \int_{-T}^t \frac{ {\rm Re}[p_0(s) e^{-i\gamma(t)}]}{t-s} \Big(K(\zeta)+\frac{2\rho^2}{1+\rho^2}\zeta K_\zeta\Big) ds d\rho\\
\approx &~ 0,
\end{align*}
and
\begin{align*}
&~4 \partial_{x_2}z^*_3(q)  -2\la \dot c_2  +\frac{\la}{\pi}\int_0^{2\pi}\int_0^{+\infty} \left( Q_{-\gamma}\tilde L_U[\Psi-Z^*(q)] \cdot W\right) \sin w \sin\theta \rho d\rho d\theta \\
&~-  \int_0^{+\infty}\rho^2 w_\rho^2 \sin w\cos w \int_{-T}^t \frac{{\rm Im}[p_1](s)}{t-s}\zeta^2 K_{\zeta\zeta}ds d\rho\\
&~+2 \int_0^{+\infty}\rho^2 w^2_{\rho} \sin w  \int_{-T}^t \frac{{\rm Im}[p_1](s)}{t-s}\Big(K(\zeta)+\frac{2\rho^2}{1+\rho^2}\zeta K_\zeta\Big) ds  d\rho\\
&~+2 c_2 \int_0^{+\infty} \rho w_\rho \sin^2 w    \int_{-T}^t \frac{ {\rm Re}[p_0(s)e^{-i\gamma(t)}] }{t-s}  \frac{2\rho^2}{1+\rho^2} \zeta K_\zeta ds  d\rho\\
&~+2  c_2 \int_0^{+\infty} \rho w^2_\rho \sin^2 w   \int_{-T}^t \frac{ {\rm Re}[p_0(s) e^{-i\gamma(t)}]}{t-s} \Big(K(\zeta)+\frac{2\rho^2}{1+\rho^2}\zeta K_\zeta\Big) ds d\rho\\
\approx &~ 0,
\end{align*}
where $\zeta$ and $K(\zeta)$ are defined in \eqref{def-Kzeta}. Therefore, in complex form, the orthogonality condition in the $W$-direction reads
\begin{equation*}
\begin{aligned}
&~\int_{-T}^t \frac{p_1(s)}{t-s} \Gamma_3\left(\frac{\lambda^2(t)}{t-s}\right)ds+2\la \dot \cc + \Gamma_4[p_0] \cc  \\
\approx &~  4 \big(\partial_{x_1}z^*_3(q)+i \partial_{x_2}z^*_3(q)\big) +\frac{\la}{\pi}\int_0^{2\pi}\int_0^{+\infty} \left( Q_{-\gamma}\tilde L_U[\Psi-Z^*(q)] \cdot W\right) \sin w e^{i\theta} \rho d\rho d\theta
\end{aligned}
\end{equation*}
where
\begin{equation*}
\begin{aligned}
&~ p_1(t) = -2(\la \cc)', \quad \cc(t)=c_1(t) +ic_2(t),\\
&~\Gamma_3(\tau) = \int_{0}^{+\infty}\rho^2 w_\rho^2 \sin w\left[\cos w\zeta^2 K_{\zeta\zeta}(\zeta)-2K(\zeta) - \frac{4\rho^2}{1+\rho^2}\zeta K_\zeta(\zeta)\right]_{\zeta =\tau(1+\rho^2)}d\rho,
\end{aligned}
\end{equation*}
and
\begin{equation*}
\begin{aligned}
\Gamma_4[p_0] =&~-2\int_0^{+\infty} \rho w_\rho \sin^2 w    \int_{-T}^t \frac{ {\rm Re}[p_0(s)e^{-i\gamma(t)}] }{t-s}  \frac{2\rho^2}{1+\rho^2} \zeta K_\zeta ds  d\rho\\
&~ -2\int_0^{+\infty} \rho w^2_\rho \sin^2 w   \int_{-T}^t \frac{ {\rm Re}[p_0(s) e^{-i\gamma(t)}]}{t-s} \Big(K(\zeta)+\frac{2\rho^2}{1+\rho^2}\zeta K_\zeta\Big) ds d\rho\\
=&~ \int_{-T}^t \frac{ {\rm Re}[p_0(s) e^{-i\gamma(t)}]}{t-s}\Gamma_5\left(\frac{\lambda^2(t)}{t-s}\right)ds
\end{aligned}
\end{equation*}
with
\begin{equation*}
\begin{aligned}
\Gamma_5(\tau) = -2\int_{0}^{+\infty}\rho  w_\rho  \sin^2 w\left[(1+w_\rho) \frac{2\rho^2}{1+\rho^2}\zeta K_\zeta(\zeta)+w_\rho K(\zeta)\right]_{\zeta =\tau(1+\rho^2)}d\rho.
\end{aligned}
\end{equation*}

It is direct to see that
$$
\Gamma_3(\tau)=\left\{
\begin{aligned}
&-1+O(\tau), ~&~\tau\leq 1,\\
&O\left(\frac1{\tau}\right), ~&~\tau> 1,\\
\end{aligned} \quad
\right.
\quad
\Gamma_5(\tau)=\left\{
\begin{aligned}
&-\frac23+O(\tau), ~&~\tau\leq 1,\\
&O\left(\frac1{\tau}\right), ~&~\tau> 1,\\
\end{aligned} \quad
\right.
$$
and thus above reduced problem, at main order, can be written in the complex form as
\begin{equation*}
\int_{-T}^{t-\la^2(t)}\frac{p_1(s)}{t-s}ds +\frac23\left(\int_{-T}^{t-\la^2(t)}\frac{{\rm Re}[p_0(s) e^{-i\gamma(t)}]}{t-s}ds\right)  \cc =-4 \big(\partial_{x_1}z^*_3(q)+i \partial_{x_2}z^*_3(q)\big)+f(t),
\end{equation*}
where
$$
f(t):=-\frac{\la}{\pi}\int_0^{2\pi}\int_0^{+\infty} \left( Q_{-\gamma}\tilde L_U[\Psi-Z^*(q)] \cdot W\right) \sin w e^{i\theta} \rho d\rho d\theta.
$$
In order for $p_1(t)$ to vanish as $t\to T$, we require
$$
\partial_{x_1}z^*_3(q)=\partial_{x_2}z^*_3(q)=0.
$$
Then the new balancing condition becomes
\begin{equation}\label{eqn-p_1}
\int_{-T}^{t-\la^2(t)}\frac{p_1(s)}{t-s}ds +\frac23\left(\int_{-T}^{t-\la^2(t)}\frac{{\rm Re}[p_0(s) e^{-i\gamma(t)}]}{t-s}ds\right)  \cc=f(t).
\end{equation}
Here, $f(t)$ is essentially the contribution from the outer profile $\Psi$, and we will later see from the weighted space chosen for $\Psi$ that
$$
|f(t)|\lesssim \la_*^{\Theta}(t), \quad \la_*(t)=\frac{(T-t)|\log T|}{|\log(T-t)|^2}
$$
for some $0<\Theta<1.$

\medskip

We now derive the leading asymptotic behavior of \eqref{eqn-p_1} by approximating it by an ODE, and the non-local parts remained will be solved by a linear theory in later section together with the fixed point argument.

Notice, by the leading order of $\la\sim \la_*$, that
\begin{align*}
\int_{-T}^{t-\la^2(t)}\frac{p_1(s)}{t-s}ds=&~\int_{-T}^{t-(T-t)} \frac{p_1(s)}{t-s}ds+p_1(t)\log(t-s)\Big|_{s=t-\la^2(t)}^{t-(T-t)}+\int_{t-(T-t)}^{t-\la^2(t)}\frac{p_1(s)-p_1(t)}{t-s}ds\\
=&~\int_{-T}^{t-(T-t)} \frac{p_1(s)}{t-s}ds+p_1(t)\left[\log(T-t)-2\log\la(t)\right]+\int_{t-(T-t)}^{t-\la^2(t)}\frac{p_1(s)-p_1(t)}{t-s}ds\\
\approx&~ \int_{-T}^t \frac{p_1(s)}{T-s}ds-p_1(t)\log(T-t)\\
:=&~\mathcal P_1(t),
\end{align*}
and exactly the same argument works for
\begin{align*}
\int_{-T}^{t-\la^2(t)}\frac{{\rm Re}[p_0(s) e^{-i\gamma(t)}]}{t-s}ds
\approx &~ \int_{-T}^t \frac{{\rm Re}[p_0(s) e^{-i\gamma(t)}]}{T-s}ds-{\rm Re}[p_0(t) e^{-i\gamma(t)}]\log(T-t)\\
:=&~\mathcal P_0(t).
\end{align*}
We then have an approximation for equation \eqref{eqn-p_1}:
\begin{equation*}
\mathcal P_1(t)+\frac23 \mathcal P_0(t) \cc =f(t),
\end{equation*}
and we want to solve for $p_1(t)$ from here.
Notice that
\begin{align*}
&-\log(T-t)\mathcal P_1'=\left[\log^2(T-t)p_1(t)\right]',\\
&-\log(T-t)\mathcal P_0'=\left[\log^2(T-t){\rm Re}[p_0(t) e^{-i\gamma(t)}]\right]'.
\end{align*}
Recall that
$$
p_1=-2(\la \cc)',\quad p_0=-2(\dot\la+i\la\dot\gamma) e^{i\gamma}.
$$
Then we get
\begin{equation}\label{eqn-527527527}
2\Big[\log^2(T-t) (\la \cc)'\Big]'+\frac43 \Big[\log^2(T-t) \dot\la\Big]' \cc+\frac23 \log(T-t) \mathcal P_0 \cc'=\log(T-t) f'(t).
\end{equation}
The third term above can in fact be approximated by
\begin{align*}
\mathcal P_0 \approx&~ 2\dot\la -2{\rm Re}\Big[\div(e^{-i\gamma_*}Z^*(q)) +i {\rm{curl }}(e^{-i\gamma_*}Z^*(q))\Big] \\
\approx&~ -2{\rm Re}\Big[\div(e^{-i\gamma_*}Z^*(q)) +i {\rm{curl }}(e^{-i\gamma_*}Z^*(q))\Big]
\end{align*}
due to \eqref{eqn-585858}, where $\gamma_*$ is the rotation angle. Let us now write
\begin{equation*}
\mathbf Z:=-2{\rm Re}\Big[\div(e^{-i\gamma_*}Z^*(q)) +i {\rm{curl }}(e^{-i\gamma_*}Z^*(q))\Big],
\end{equation*}
and thus equation \eqref{eqn-527527527} reads approximately
\begin{equation*}
2\Big[\log^2(T-t) (\la \cc)'\Big]'+\frac23 \log(T-t) \mathbf Z \cc'=\log(T-t) f'(t)
\end{equation*}
as $\frac43 \Big[\log^2(T-t) \dot\la\Big]' \cc$ is relatively smaller in size. Above equation should be understood in the weak sense when solving rigorously $\cc$. But for now, we only single out the main part, smooth for $t\in(0,T)$, of $\cc$. The full solvability of $\cc$ is given in Section \ref{sec-lt-c1c2}.

 Integrating above equation implies
\begin{equation*}
-2\log^2(T-t) (\la \cc)'+\frac{2\mathbf Z}{3} \int_t^T \log(T-s) \cc'(s)ds=-f(t)\log(T-t)+\int_t^T \frac{f(s)}{T-s}ds
\end{equation*}
which is again approximately
\begin{equation}\label{eqn-532532532}
-2\log^2(T-t) (\la \cc)'-\frac{2\mathbf Z}{3}\log(T-t) \cc(t)= -f(t)\log(T-t)+\int_t^T \frac{f(s)}{T-s}ds.
\end{equation}
The equation for $\la \cc$ is given by
\begin{equation*}
\left[(T-t)^{-\frac{\mathbf Z}{6}\frac{\log(T-t)}{|\log T|}}(\la\cc)\right]'=(T-t)^{-\frac{\mathbf Z}{6}\frac{\log(T-t)}{|\log T|}}\left(\frac{f(t)}{2\log(T-t)}-\frac{1}{2\log^2(T-t)}\int_t^T \frac{f(s)}{T-s}ds\right).
\end{equation*}
Thanks to \eqref{mathbfZ}, i.e., $\mathbf Z\geq 0$, one can find a solution $\cc$ which vanishes as $t\to T$. Indeed,
\begin{align*}
|\la \cc|=&~\left|(T-t)^{\frac{\mathbf Z}{6}\frac{\log(T-t)}{|\log T|}}\int_t^{T} (T-\tau)^{-\frac{\mathbf Z}{6}\frac{\log(T-\tau)}{|\log T|}}\left(\frac{f(\tau)}{2\log(T-\tau)}-\frac{1}{2\log^2(T-\tau)}\int_\tau^T \frac{f(s)}{T-s}ds\right)d\tau\right|\\
\lesssim&~\frac{|f(t)|}{|\log(T-t)|}\left|(T-t)^{\frac{\mathbf Z}{6}\frac{\log(T-t)}{|\log T|}}\int_t^{T} (T-\tau)^{-\frac{\mathbf Z}{6}\frac{\log(T-\tau)}{|\log T|}}d\tau \right|\\
\lesssim&~
\begin{cases}
\frac{3|\log T|(T-t)|f(t)|}{\mathbf Z|\log(T-t)|^2}, &~\mathbf Z >0,\\
\frac{(T-t)|f(t)|}{|\log(T-t)|}, &~\mathbf Z =0,\\
\end{cases}
\end{align*}
and
\begin{equation*}
|\cc|\lesssim
\begin{cases}
 \la_*^{\Theta}(t), &~\mathbf Z >0,\\
 \la_*^{\Theta}(t) \frac{|\log(T-t)|}{|\log T|} , &~\mathbf Z =0\\
\end{cases}
\end{equation*}
because of $|f(t)|\lesssim \la_*^{\Theta}(t)$. Throughout the rest of this paper, we consider the most representative case when $\mathbf Z >0$, so we have
\begin{equation}\label{est-cc}
|\cc|\lesssim
 \la_*^{\Theta}(t).
\end{equation}
\begin{remark}
 The case with $\mathbf Z=0$ is in fact more special as the leading dynamics of $c_1$-$c_2$ system in such case reduces to the one similar to that of $\la$-$\gamma$ system. The coupling between $\la$-$\gamma$ system and $c_1$-$c_2$ system only appears in the case $\mathbf Z>0$.
\end{remark}

\bigskip

\section{Linear theories}\label{lineartheories}

\medskip

In this section, we give the linear theories concerning the a priori estimates for the linear problem of outer problem \eqref{gluing-sys}$_3$, and the inner problems in the $W$-direction \eqref{gluing-sys}$_1$ and on $W^\perp$ \eqref{gluing-sys}$_2$.

\subsection{Linear theory for the outer problem}
For  $q\in \R^2 $ and $T>0$ sufficiently small, we consider the problem
\begin{align}
\label{heat-eq0}
\begin{cases}
\psi_t   = \Delta_x \psi + g(x,t) &\inn~ \R^2  \times (0,T), \\
\psi(x,0)   =  Z^*(x)    &\inn ~ \R^2
\end{cases}
\end{align}
for smooth initial value $Z^*$ with compact support. The RHS $g$ of \eqref{heat-eq0} is assumed to be bounded with respect to some weights that appear in the outer problem \eqref{gluing-sys}$_3$.
Thus we define the weights
\begin{align}\label{weights}
\left\{
\begin{aligned}
\varrho_1 & :=   \lambda_*^{\Theta}  (\lambda_* R)^{-1}  \1_{ \{ r \leq 3\la_*R \} },
\\
\varrho_2 & := T^{-\sigma_0}  \frac{\lambda_*^{1-\sigma_0}}{r^2}  \1_{ \{ r \geq  \la_*R \} },
\\
\varrho_3 & := T^{-\sigma_0},
\end{aligned}
\right.
\end{align}
where $r= |x-q|$, $\Theta>0$ and  $\sigma_0>0$ is  small.
For a function $g(x,t)$ we define the $L^\infty$-weighted norm
\begin{align}
\label{topo-outRHS}
\|g\|_{**} : =   \sup_{ \gamma  \times (0,T)}  \Big ( 1 + \sum_{i=1}^3 \varrho_i(x,t)\, \Big )^{-1}  {|g(x,t)|} .
\end{align}
We define the $L^\infty$-weighted norm for $\psi$
\begin{align}
\nonumber
\| \psi\|_{\sharp, \Theta,\alpha}
&:=
\lambda^{-\Theta}_*(0)
\frac{1}{|\log T|  \lambda_*(0) R(0) }\|\psi\|_{L^\infty(\R^2 \times (0,T))}
+ \lambda^{-\Theta}_*(0) \|\nabla_x \psi\|_{L^\infty(\R^2 \times (0,T))}
\\
\nonumber
&~\quad
+
\sup_{\R^2 \times (0,T)}   \lambda^{-\Theta-1}_*(t) R^{-1}(t)
\frac{1}{|\log(T-t)|} |\psi(x,t)-\psi(x,T)|
\\
\nonumber
&~\quad
+ \sup_{\R^2 \times (0,T)} \, \lambda^{-\Theta}_*(t)
|\nabla_x \psi(x,t)-\nabla_x \psi(x,T) |
\\
\label{topo-out}
& ~\quad
+ \sup_{}
\lambda^{-\Theta}_*(t)
(\lambda_*(t) R(t))^{\alpha}  \frac {|\nabla_x \psi(x,t) -\nabla_x \psi(x',t') |}{ ( |x-x'|^2 + |t-t'|)^{\alpha/2   }} ,
\end{align}
where $\Theta\in(0,1)$, $\alpha \in (0,\frac{1}{2})$, and the last supremum is taken in the region
\[
x,\, x'\in \R^2 ,\quad  t,\, t'\in (0,T), \quad |x-x'|\le 2 \la_*(t)R(t), \quad  |t-t'| < \frac 14 (T-t) .
\]

Also, we choose the initial data $Z^*(x)$ so that
\begin{equation}\label{outer-vanishing}
\psi_1(q,T)=\psi_2(q,T)=\psi_3(q,T)=\pp_{x_1}\psi_3(q,T)=\pp_{x_2}\psi_3(q,T)=0.
\end{equation}
This can be achieved by the following: choose regular maps $\mathscr Z_{k}$ with compact support, $k=1,\dots, 5$ satisfying
$$
\mathscr Z_{k}(q)=e^{(5)}_k
$$
where $\{e^{(5)}_k\}_{k=1}^5$ forms an orthonormal basis of $\R^5$. We choose $\mathscr C_k$ so that
\begin{equation}\label{vanishings}
\begin{aligned}
&~\Big(\Gamma_{\R^2} \bullet~ \mathcal G,~ [\nabla_x (\Gamma_{\R^2} \bullet~ \mathcal G)]_3\Big)(q,T)+\Big(\Gamma_{\R^2} \circ~ Z^*,~ [\nabla_x (\Gamma_{\R^2} \circ Z^*)]_3\Big)(q,T)\\
&~+\sum_{k=1}^5 \mathscr C_k(\Gamma_{\R^2}\circ \mathscr Z_{k},~[\nabla_x (\Gamma_{\R^2}\circ \mathscr Z_{k})]_3)(q,T)=0.
\end{aligned}
\end{equation}
The first three and the last two vanishing conditions  in \eqref{outer-vanishing} are needed in the gluing process and in $c_1$-$c_2$  system (i.e. $p_1(t)$), respectively.

We shall measure the solution $\psi$ to the problem \eqref{heat-eq0}
in the norm $\| \ \|_{\sharp,\Theta,\alpha}$ defined in \eqref{topo-out}. We invoke some useful estimates proved in \cite[Appendix A]{17HMF} as follows.
\begin{prop}[\cite{17HMF}]
\label{prop3}
For $T>0$ sufficiently small, there is a linear operator that maps a function $g:\R^2  \times (0,T) \to \R^3$ with  $\|g\|_{**}<\infty$ into $\psi$ which solves problem \eqref{heat-eq0}.
Moreover, the following estimate holds
\begin{align*}
\| \psi\|_{\sharp, \Theta ,\alpha}
 \leq C \|g\|_{**} ,
\end{align*}
where $\alpha\in(0,\frac{1}{2})$.
\end{prop}

\medskip

\subsection{Linear theories for the inner problems}\label{sec-linearinner}

For the inner problem, we consider the model equation for the inner problem as follows.
\begin{equation}\label{linearinnerproblem}
\left\{
\begin{aligned}
&\partial_\tau\phi = L_W[\phi] + h(y, \tau), &~|y|\leq 2 R(t(\tau)),~\tau\in(\tau_0,+\infty),\\
&\phi(y,\tau_0)  = 0, &~|y|\leq 2R(t(\tau_0)),\\
\end{aligned}
\right.
\end{equation}
where
$$
 \tau =\tau(t) = \tau_0 + \int_{0}^t\frac{ds}{\lambda(s)^2},
$$
for $\tau_0$ satisfying $t(\tau_0)=0$, and $L_W$ is the linearized operator defined in \eqref{def-linearizedoperator}. We write the solution $\phi = \phi(\rho, \theta,\tau)$, ~$y=\rho e^{i\theta}$ of (\ref{linearinnerproblem}) as
\begin{equation*}
\phi(\rho, \theta, \tau) = \phi_1(\rho, \theta, \tau)E_1 + \phi_2(\rho, \theta, \tau)E_2 + \phi_0(\rho, \theta, \tau)W.
\end{equation*}
We will deal with $\Pi_{W^\perp}[\phi]=\phi_1(\rho, \theta, \tau)E_1 + \phi_2(\rho, \theta, \tau)E_2$ and  $\Pi_{W}[\phi]=\phi_0(\rho, \theta, \tau)W$ separately, and consider
\begin{equation}\label{inner-Wperp}
\left\{
\begin{aligned}
&\partial_\tau(\Pi_{W^\perp}[\phi]) = L_W[\Pi_{W^\perp}[\phi]] + \Pi_{W^\perp}[h](y, \tau), &~|y|\leq 2 R(t(\tau)),~\tau\in(\tau_0,+\infty),\\
&\Pi_{W^\perp}[\phi](y,\tau_0)  = 0, &~|y|\leq 2R(t(\tau_0)),\\
\end{aligned}
\right.
\end{equation}
and
\begin{equation}\label{inner-W}
\left\{
\begin{aligned}
&\partial_\tau(\Pi_{W}[\phi]) = L_W[\Pi_{W}[\phi]] + \Pi_{W}[h](y, \tau), &~|y|\leq 2 R(t(\tau)),~\tau\in(\tau_0,+\infty),\\
&\Pi_{W}[\phi](y,\tau_0)  = 0, &~|y|\leq 2R(t(\tau_0)).\\
\end{aligned}
\right.
\end{equation}
We will measure the RHS by the norm:
\begin{equation}\label{topo-inRHS}
\|h\|_{\nu,a}:=\sup_{|y|\leq 2R(t),~t\in(0,T)} \la_*^{-\nu}(t)\langle y\rangle^{a}|h(y,t)|
\end{equation}
if we use $(y,t)$ variables and by
\begin{equation*}
\|h\|^{(\tau)}_{v,\ell}:=\sup_{|y|\leq 2R(t),~t\in(0,T)} v^{-1}(\tau)\langle y\rangle^{\ell}|h(y,\tau)|
\end{equation*}
if we use $(y,\tau)$ variables. Here, $v(\tau)$ is some H\"older continuous function decaying in $\tau$ as $\tau\to\infty$.

\medskip

\subsubsection{Linear theory for the inner problem on $W^\perp$}

If we use the complex notation $W^\perp$
$$
(\Pi_{W^\perp}[\phi])_{\mathbb C}=\phi_1+i\phi_2,
$$
then the equation on $W^\perp$ reads as
\begin{equation*}
\pp_{\tau} (\Pi_{W^\perp}[\phi])_{\mathbb C}=\Delta (\Pi_{W^\perp}[\phi])_{\mathbb C}-\frac1{\rho^2}(\Pi_{W^\perp}[\phi])_{\mathbb C}+\frac{8}{(1+\rho^2)^2}(\Pi_{W^\perp}[\phi])_{\mathbb C}+i\frac{2(\rho^2-1)}{\rho^2(\rho^2+1)}\pp_{\theta}(\Pi_{W^\perp}[\phi])_{\mathbb C}+(\Pi_{W^\perp}[h])_{\mathbb C}
\end{equation*}
with zero initial data.
We further expand  $(\Pi_{W^\perp}[\phi])_{\mathbb C}$ and   $(\Pi_{W^\perp}[h])_{\mathbb C}$  in Fourier modes
$$
(\Pi_{W^\perp}[\phi])_{\mathbb C}=\sum_{k\in\mathbb Z}\phi_k^{W^\perp}(\rho,\tau) e^{ik\theta}, \quad (\Pi_{W^\perp}[h])_{\mathbb C} =\sum_{k\in\mathbb Z}h_k^{W^\perp}(\rho,\tau) e^{ik\theta}.
$$
Then the complex-valued scalar $\phi_k^{W^\perp}$ at mode $k$ satisfies
\begin{equation}\label{eqn-scalar}
\begin{aligned}
\pp_{\tau} \phi_k^{W^\perp}= &~\partial_{\rho\rho}\phi_k^{W^\perp}+\frac1\rho \partial_{\rho}\phi_k^{W^\perp}-\frac{1+k^2}{\rho^2}\phi_k^{W^\perp}+\frac{8}{(1+\rho^2)^2}\phi_k^{W^\perp}-\frac{2k(\rho^2-1)}{\rho^2(\rho^2+1)}\phi_k^{W^\perp}+h_k^{W^\perp}\\
:=&~\mathcal L^{W^\perp}_k [\phi_k^{W^\perp}]+h_k^{W^\perp}
\end{aligned}
\end{equation}
which is precisely the linearization of harmonic map equation around degree $1$  harmonic map dealt with in \cite[Section 7]{17HMF}. So we have:

\begin{prop}{\rm (\cite[Proposition 7.1]{17HMF})}
\label{prop-lt-Wperp}
Assume that $a\in(2,3)$, $\nu>0$,   and $\|\Pi_{W^\perp}[h]\|_{\nu,a} <+\infty$.
Let us write $$\Pi_{W^\perp}[h] = h_0^{W^\perp} + h_1^{W^\perp} + h_{-1}^{W^\perp} + h_\perp^{W^\perp} ~\mbox{ with }~ h_{\perp}^{W^\perp} = \sum_{k\not=0,\pm 1} h_k^{W^\perp}.$$
Then there exists a solution $\Pi_{W^\perp}[\phi](h)$ of problem \eqref{inner-Wperp}, which defines
a linear operator of $\Pi_{W^\perp}[h]$, and satisfies the following estimate for
$|y|\leq 2 R(t)$ with $t\in (0,T)$:
\begin{align*}
&\quad
|\Pi_{W^\perp}[\phi](y,t)|+(1+|y|)\left |\nabla_y \Pi_{W^\perp}[\phi](y,t)\right |  \\
& \lesssim
   \lambda^\nu_*(t)   \,
 \min\left\{\frac{R^{ 5-a}(t)}{1+|y|^3},  ~\frac{R^{\frac{5-a}2}(t)}{1+|y| } \right\}
 \, \| h_0^{W^\perp} -\bar h_{0}^{W^\perp} \|_{\nu,a}
+
\frac{ \lambda^\nu_*(t)  R^2(t) } {1+ |y|}  \|\bar h_0^{W^\perp}\|_{\nu,a}
\\
& \quad
+   \frac{ \lambda^\nu_*(t) }{ 1+ |y|^{a-2} }\, \left \| h_1^{W^\perp} - \bar h_1^{W^\perp}\right  \|_{\nu,a}
+   \frac{ \lambda^\nu_*(t)  R^4(t)} {1+ |y|^2} \left \| \bar h_{1}^{W^\perp} \right  \|_{\nu,a}
\\
& \quad
+
 \lambda^\nu_*(t)   \,
 \min\left\{\frac{R^{4-a}(t)}{1+|y|^2}, ~ \log R \right\}
 \, \| h_{-1}^{W^\perp} -\bar h_{-1}^{W^\perp} \|_{\nu,a}
 + \lambda^\nu_*(t) \log R(t) \,   \| \bar h_{-1}^{W^\perp} \|_{\nu,a}
\\
& \quad
+
\frac{ \lambda^\nu_*(t)  }{ 1+ |y|^{a-2} }\,   \|h_\perp^{W^\perp} \|_{\nu,a} .
\end{align*}
Here
\begin{equation}\label{def-hbar}
\bar h^{W^\perp}_{k}(y,t)   :=       \sum_{j=1}^2   \frac {\chi Z_{k,j}(y)} { \int_{\R^2} \chi  |Z_{k,j} |^2  }  \, \int_{ \R^2 } \Pi_{W^\perp}[h](z  , t )  \cdot Z_{k,j}(z)\, dz, \quad~k=0,\pm 1,~j=1,2,
\end{equation}
with $Z_{k,j}$ defined in \eqref{kernels}, where
\begin{equation}\label{def-chichi}
\chi(y,\tau) = \begin{cases}
w_\rho^2(|y|)  & \hbox{ if }  |y|< 2R(t ),\\
0& \hbox{ if }  |y|\ge 2R(t ).
\end{cases}
\end{equation}
\end{prop}

Via a re-gluing process, better estimates can be gained at mode $0$ with a slight modification on the orthogonality condition. Consider the linear problem at mode $0$ on $W^\perp$:
\begin{equation}\label{eqn-RG-Wperp0}
\left\{
\begin{aligned}
&~\la^2 \pp_t \phi= L_W \phi+h(\rho,t)+\sum_{j=1,2} \tilde c_{0j}  Z_{0,j} w_\rho^2, \qquad |y|\leq 2 R(t ),~t\in(0,T),\\
&~\phi\perp W, \quad\phi=0 ~\mbox{ on }~\pp B_{2R}\times (0,T),\quad \phi(\cdot,0)=0 ~\mbox{ in }~B_{2R(0)}.
\end{aligned}
\right.
\end{equation}

\begin{prop}{\rm (\cite[Proposition 7.2]{17HMF})}\label{prop-RG-Wperp0}
Let $\sigma_*\in(0,1)$. Under the assumptions of Proposition \ref{prop-lt-Wperp}, there exists $(\phi,~\tilde c_{0j})$, linear in $h$, solving \eqref{eqn-RG-Wperp0} such that
$$
|\phi|+(1+\rho)|\pp_\rho \phi|\lesssim \la_*^{\nu}(t)\|h\|_{\nu,a} \min\left\{\frac{R^{\sigma_*(5-a)}(t)}{1+|y|^3},  \frac{1}{1+|y|^{a-2}} \right\}
$$
and such that
$$
\tilde c_{0j}=\tilde c_{0j}[h]=-\frac{\int_{\R^2} h\cdot Z_{0,j}}{\int_{\R^2} w_\rho^2 Z_{0,j}^2}-G[h],\quad j=1,2
$$
for some operator $G$ linear in $h$ with
$$
|G[h]|\lesssim \la_*^{\nu} R^{-\sigma_* \sigma'}\|h\|_{\nu,a},\quad \sigma'\in(0,~a-2).
$$
\end{prop}

More refined versions of above linear theory on $W^\perp$ are obtained in \cite[Section 9]{LLG} (in a special case of $a=1$, $b=0$).

\medskip

\subsubsection{Linear theory for inner problem in the $W$-direction}

For the equation \eqref{inner-W} in the $W$-direction, the use of
$$
\Pi_{W }[\phi] =\sum_{k\in\mathbb Z}\phi_k^{W }(\rho,\tau) e^{ik\theta}, \quad \Pi_{W }[h] =\sum_{k\in\mathbb Z}h_k^{W }(\rho,\tau) e^{ik\theta}
$$
gives equation in mode $k$:
\begin{equation}\label{eqn-scalarW}
\begin{cases}
\pp_{\tau} \phi_k^W(\rho,\tau)= \mathcal L_k^W[\phi_k^W] + h^W_k(\rho,\tau) ,\quad (\rho,\tau)\in (0,\infty) \times (\tau_0,\infty),\\
\phi^W_k(\rho,\tau_{0}) = 0 ,\quad \rho\in (0,\infty),
\end{cases}
\end{equation}
where
\begin{equation*}
\mathcal L_k^W:=\pp_{\rho\rho}+\frac1{\rho}\pp_{\rho}-\frac{k^2}{\rho^2}+\frac{8}{(1+\rho^2)^2}.
\end{equation*}

\medskip

\noindent $\bullet$ {\bf Mode $0$.}

\medskip

In mode $k=0$, we establish a linear theory via the distorted Fourier transform (DFT).

\begin{prop}\label{prop-lt-W0}
In \eqref{eqn-scalarW} with $k=0$, if $\| h_0^W \|^{(\tau)}_{v,\ell}  <\infty$ with $\ell >\frac 32$, then there exists a solution $\phi_{0}^W$ that satisfies the following estimate
\begin{equation*}
\begin{aligned}
|\phi^W_{0}(\rho,\tau)| \lesssim \ &
\| h_0^W \|^{(\tau)}_{v,\ell}  \1_{ \{ \rho\le \tau^{\frac 12} \} }
\begin{cases}
	v(\tau) \tau^{1-\frac{\ell}{2}}
	+
	\tau^{-\frac{\ell}{2}}
	\int_{\frac{\tau_0}{2}}^{\frac{\tau}{2}}
	v(s) ds
	&
	\mbox{ \ if  \ } \ell <2
	\\
	v(\tau) ( \log  \tau)^2
	+
	\tau^{-1}  \log  \tau
	\int_{\frac{\tau_0}{2}}^{\frac{\tau}{2}}
	v(s) ds
	&
	\mbox{ \ if  \ } \ell =2
	\\
	v(\tau)  \log  \tau
	+
	\tau^{-1}
	\int_{\frac{\tau_0}{2}}^{\frac{\tau}{2}}
	v(s) ds
	&
	\mbox{ \ if  \ } \ell >2
\end{cases}
\\
&
+
\| h_0^W \|^{(\tau)}_{v,\ell}  \1_{ \{ \rho > \tau^{\frac 12} \} }
\rho^{-1/2}
\begin{cases}
	v(\tau) \tau^{\frac 54-\frac{\ell}{2}}
	+
	\tau^{\frac 14 - \frac{\ell}{2}}
	\int_{\frac{\tau_0}{2}}^{ \frac{\tau}{2} }
	v(s)  ds
	&
	\mbox{ \ if  \ }\ell<2
	\\
	v(\tau) \tau^{\frac 14} \langle  \log  \tau \rangle
+
	\tau^{-\frac 34}
	\langle  \log  \tau \rangle
	\int_{\frac{\tau_0}{2}}^{ \frac{\tau}{2} }
	v(s) ds
	&
	\mbox{ \ if  \ }\ell=2
	\\
	v(\tau) \tau^{\frac 14 }
	+
	\tau^{-\frac 34 }
	\int_{\frac{\tau_0}{2}}^{ \frac{\tau}{2} }
	v(s)  ds
	&
	\mbox{ \ if  \ }\ell>2
\end{cases}
.
\end{aligned}
\end{equation*}
If, in addition, $2<\ell <\frac 52$ and the orthogonality condition
	\begin{equation}\label{h-1-ortho}
		\int_{0}^{\infty} h_0^W(\rho,\tau) \frac{\rho^2-1}{\rho^2+1}\rho d\rho =0
		\mbox{ \ for all \ }\tau > \tau_{0}
	\end{equation}
is satisfied, then the following estimate holds
	\begin{equation*}
		|\phi^W_{0}(\rho,\tau)|
		\lesssim
		\| h_0^W \|^{(\tau)}_{v,\ell}
		\begin{cases}
			v(\tau)
			\langle \rho\rangle^{2 - \ell }
			+
			\tau^{-\frac{\ell}{2}}
			\int_{\frac{\tau_0 }{2}}^{\frac{\tau }{2}}
			v(s) ds
			&
			\mbox{ \ if \ } \rho\le \tau^{\frac 12}
			\\
			\rho^{-1/2}
			\left(
			v(\tau) \tau^{\frac 54 -\frac{\ell}{2}}
			+
			\tau^{\frac 14-\frac{\ell}{2}} \int_{\frac{\tau_0}{2}}^{\frac{\tau}{2}}
			v(s)  ds
			\right)
			&
			\mbox{ \ if \ } \rho > \tau^{\frac 12}
		\end{cases}
	.
	\end{equation*}

\end{prop}

We postpone the proof of Proposition \ref{prop-lt-W0} in Appendix \ref{DFT}.

\medskip

\noindent $\bullet$ {\bf Modes $k=\pm1$.}

\medskip

We notice that $\mathcal L^W_{\pm 1}\equiv \mathcal L^{W^\perp}_0$, where $\mathcal L^{W^\perp}_k$ is defined in \eqref{eqn-scalar} (on $W^\perp$), meaning that the linear theory for modes $\pm1$ in the $W$-direction is exactly the same as the one for mode $0$ on $W^\perp$. So for modes $|k|=1$ in the $W$-direction, we have
\begin{prop}{\rm (\cite[Lemma 7.1, Lemma 7.3]{17HMF})}
\label{prop-lt-W1}
For modes $k=\pm 1$ in \eqref{eqn-scalarW}, there exists a solution $\phi_{k}^W$ that satisfies the following estimate
\begin{equation*}
\begin{aligned}
 |\phi_k^W|&\lesssim   \lambda^\nu_*(t)   \,
 \min\left\{\frac{R^{ 5-a}(t)}{1+|y|^3},  ~\frac{R^{\frac{5-a}2}(t)}{1+|y| } \right\}  \|h^W_k-\bar h^W_k\|_{\nu,a}+ \frac{\la_*^{\nu}(t)R^2(t)}{1+|y|}\|\bar h_k^W\|_{\nu,a} .
\end{aligned}
\end{equation*}
Here $a\in (2, 3)$, $\nu > 0$,  and
\begin{align*}
\bar h_1^W:= &~     \frac {\chi \mathcal Z_{1,1}(y)} { \int_{\R^2} \chi  |\mathcal Z_{1,1} |^2  }  \, \int_{ \R^2 } \Pi_{W}[h](z  , t )  \cdot \mathcal Z_{1,1}(z)\, dz,\\
\bar h_{-1}^W:= &~     \frac {\chi \mathcal Z_{1,2}(y)} { \int_{\R^2} \chi  |\mathcal Z_{1,2} |^2  }  \, \int_{ \R^2 } \Pi_{W}[h](z  , t )  \cdot \mathcal Z_{1,2}(z)\, dz,
\end{align*}
with $\chi$, $\mathcal Z_{1,k}$ defined respectively in \eqref{def-chichi} and \eqref{kernels}.
\end{prop}

\begin{remark}\label{rmk-651651651}
For the non-orthogonal parts in Proposition \ref{prop-lt-Wperp} and Proposition \ref{prop-lt-W1}, one can in fact relax the assumption on the spatial decay to $0<a<3$ with estimates modified accordingly; see \cite[p.~394, Lemma 7.1]{17HMF}.
Later we will use this version.
\end{remark}

An analogue of Proposition \ref{prop-RG-Wperp0} also holds true for modes $k=\pm1$. Consider the linear problem in modes $\pm1$ in the $W$-direction:
\begin{equation}\label{eqn-RG-W1}
\left\{
\begin{aligned}
&~\la^2 \pp_t \phi= L_W \phi+h(\rho,t)+\sum_{j=1,2} \hat c_{1j}  \mathcal Z_{1,j} w_\rho^2, \qquad |y|\leq 2 R(t ),~t\in(0,T),\\
&~\phi\perp W^\perp, \quad\phi=0 ~\mbox{ on }~\pp B_{2R}\times (0,T),\quad \phi(\cdot,0)=0 ~\mbox{ in }~B_{2R(0)}.
\end{aligned}
\right.
\end{equation}
We have the following refined estimates:
\begin{prop}\label{prop-RG-W1}
Let $\sigma_*\in(0,1)$. Under the assumptions of Proposition \ref{prop-lt-W1}, there exists $(\phi,~\hat c_{1j})$, linear in $h$, solving \eqref{eqn-RG-W1} such that
$$
|\phi|+(1+\rho)|\pp_\rho \phi|\lesssim \la_*^{\nu}(t)\|h\|_{\nu,a} \min\left\{\frac{R^{\sigma_*(5-a)}(t)}{1+|y|^3},  \frac{1}{1+|y|^{a-2}} \right\}
$$
and such that
$$
\hat c_{1j}=\hat c_{1j}[h]=-\frac{\int_{\R^2} h\cdot \mathcal Z_{1,j}}{\int_{\R^2} w_\rho^2 \mathcal Z_{1,j}^2}-\hat G[h],\quad j=1,2
$$
for some operator $\hat G$ linear in $h$ with
$$
|\hat G[h]|\lesssim \la_*^{\nu} R^{-\sigma_* \sigma'}\|h\|_{\nu,a},\quad \sigma'\in(0,~a-2).
$$
\end{prop}

\medskip

\medskip

\noindent $\bullet$ {\bf Higher modes $|k|\geq 2$.}

\medskip

\begin{prop}\label{prop-lt-W+}
For modes $|k|\geq 2$ in \eqref{eqn-scalarW}, there exists a solution $\phi^W_{k}$ that satisfies
$$
 |\phi_k^W| \lesssim \frac1{k^2}  \lambda^\nu_*(t) \langle \rho \rangle^{2-a} \|h_k^W \|_{\nu,a},
$$
where $a\in(2,3)$, $\nu>0$.
\end{prop}
\begin{proof}
We start with $k=2$ (the same for $k=-2$). In such case, $\mathcal L^W_2$ has a (non sign-changing) kernel function, denoted by $Z_2$, whose asymptotic behavior is $\rho^2$ both near origin and at infinity. More generally, for $|k|\geq 2$,  $\mathcal L^W_k$ has kernel functions
$$
\frac{\rho^{-k}(k\rho^2+k+\rho^2-1)}{\rho^2+1},\quad \frac{\rho^{k}(k\rho^2+k-\rho^2+1)}{\rho^2+1}.
$$
 We choose barrier function
$$
\bar \varphi_2 =2\|h_2^W \|_{\nu,a}\la_*^{\nu} \varphi_2
$$
with $\varphi_2$ solving
$$
\mathcal L^W_2[\varphi_2]+ \langle \rho\rangle^{-a}=0.
$$
With Dirichlet boundary $\varphi_2(2R)=0$, above equation admits a unique solution given by
$$
\varphi_2(\rho)=Z_2(\rho)\int_\rho^{2R} \frac{dr}{rZ^2_2(r)}\int_0^r  \langle s\rangle^{-a} Z_2(s)sds
$$
where $a\in(2,3)$. Then one has
$$
|\varphi_2|\leq C  \langle \rho\rangle^{2-a},
$$
and thus
\begin{align*}
-\la^2 \pp_t \bar \varphi_2 +\mathcal L^W_2[\bar \varphi_2]+h^W_2\leq&~ 2\nu \|h_2^W \|_{\nu,a} |\dot\la_*|\la_*^{\nu+1}|\varphi_2|-\|h_2^W \|_{\nu,a} \la_*^{\nu} \langle \rho\rangle^{-a}\\
\leq &~\|h_2^W \|_{\nu,a} \la_*^{\nu}\langle \rho\rangle^{-a}\Big( 2\nu C|\dot\la_*|\la_*\langle \rho\rangle^2 -1\Big)\\
<&~0
\end{align*}
since $|\dot\la_*|\la_* R^2\ll 1$ (we choose $\beta<\frac12$ in $R=\la_*^{-\beta}$).  Therefore, the pointwise estimate for mode $k=2$ follows. For the higher modes $|k|\geq 3$, similar argument works. Indeed, We choose barrier function
$$
\bar \varphi_2 =2\|h^W_k \|_{\nu,a}\la_*^{\nu} \varphi_k
$$
with $\varphi_k$ solving
$$
\mathcal L^W_k[\varphi_k]+ \langle \rho\rangle^{-a}=0,
$$
where in this case,
$$
|\varphi_k|\leq \frac{C}{k^2}\langle \rho\rangle^{2-a}
$$
by variation of parameters. The proof is thus completed.
\end{proof}

\bigskip

We then combine Proposition \ref{prop-lt-W0} (returning to $(y,t)$ variables), Proposition \ref{prop-lt-W1} and Proposition \ref{prop-lt-W+} and use a scaling argument together with parabolic regularity estimates to get gradient estimate for \eqref{inner-W}. We thus conclude the following linear theory for the inner problem in the $W$-direction.
\begin{prop}\label{prop-lt-W}
Assume that $a\in(2,\frac52)$, $\nu>0$, and $\|\Pi_{W}[h]\|_{\nu,a} <+\infty$.
Then there exists a solution $\Pi_{W}[\phi](h)$ of problem \eqref{inner-W}, which defines
a linear operator of $\Pi_{W}[h]$, and satisfies the following estimate for
$|y|\leq 2 R(t)$ with $t\in (0,T)$:
\begin{align*}
&\quad
|\Pi_{W }[\phi](y,t)|+(1+|y|)\left |\nabla_y \Pi_{W }[\phi](y,t)\right |  \\
& \lesssim
 \frac{\lambda^\nu_*(t)}{1+|y|^{a-2}}
 \, \| h_{0}^{W } -\bar h_{0}^{W } \|_{\nu,a}
 + \lambda^\nu_*(t) \log R(t) \,   \| \bar h_{0}^{W } \|_{\nu,a}
\\
& \quad +\sum_{k=\pm1}\left(
    \lambda^\nu_*(t)   \,
 \min\left\{\frac{R^{ 5-a}(t)}{1+|y|^3},  ~\frac{R^{\frac{5-a}2}(t)}{1+|y| } \right\}
 \, \| h_k^{W } -\bar h_{k}^{W } \|_{\nu,a}
+
\frac{ \lambda^\nu_*(t)  R^2(t)} {1+ |y|}  \|\bar h_k^{W }\|_{\nu,a}\right)
\\
& \quad
+
\frac{ \lambda^\nu_*(t)  }{ 1+ |y|^{a-2} }\,  \Big\|h-h_0^{W }-h_1^{W }- h_{-1}^{W }\Big\|_{\nu,a} .
\end{align*}
\end{prop}

\medskip

\begin{remark}
In practice, we are going to impose orthogonality conditions on the inner problems for modes $0$, $1$ on $W^\perp$  and modes $\pm 1$ in the $W$-direction, and these orthogonalities yield the dynamics of the $\la$-$\gamma$ system, translation parameter $\xi=(\xi_1,\xi_2)$ and $c_1$-$c_2$ system, respectively. Under these orthogonalities, we will solve both inner solutions $\Phi_W$ and $\Phi_{W^\perp}$ to \eqref{gluing-sys}$_1$ and \eqref{gluing-sys}$_2$ in the wighted space:
\begin{equation}\label{topo-in}
\|\phi\|_{{\rm in},\sigma_*,\nu,a}:=\sup_{|y|\leq 2R(t),~t\in(0,T)}\left[ \lambda^\nu_*(t)   \,
 \max\left\{\frac{R^{\sigma_*(5-a)}(t)}{1+|y|^3},  \frac{1}{1+|y|^{a-2}} \right\}\right]^{-1}\Big( |\phi(y,t)|+\langle y\rangle |\nabla_y\phi(y,t)|\Big)
\end{equation}
for some $\sigma_*,~\nu \in(0,1)$, $a\in (2,3)$.
\end{remark}

\bigskip

\section{Solving the gluing system}\label{gluingsystem}

\medskip

In this section, we solve the full gluing system. We are going to work in the following weighted topologies.

\begin{itemize}
\item RHS for both inner problems \eqref{gluing-sys}$_1$ and \eqref{gluing-sys}$_2$: $\|\cdot\|_{\nu,a}$-norm defined in \eqref{topo-inRHS}.

\item RHS for the outer problem \eqref{gluing-sys}$_3$: $\|\cdot\|_{**}$-norm defined in \eqref{topo-outRHS}.

\item Inner solutions $\Phi_W$ and $\Phi_{W^\perp}$: $\|\cdot\|_{{\rm in},\sigma_*,\nu,a}$-norm defined in \eqref{topo-in}.

\item Outer solution $\Psi$: $\|\cdot\|_{\sharp,\Theta,\alpha}$-norm defined in \eqref{topo-out}.
\end{itemize}
The constants measuring above weighted norms will be put into a specific range ensuring the implementation of the gluing process. We now start to estimate RHS in the gluing system \eqref{gluing-sys} and reveal these restrictions on  those constants. We are also going to use the leading asymptotics of the modulation parameters:
\begin{equation*}
|\la|\sim \la_*=\frac{(T-t)|\log T|}{|\log(T-t)|^2}, \quad |\dot\xi|\lesssim \la_*^{3\Theta-1}, \quad |c_1|\lesssim \la_*^{\Theta}, \quad |c_2|\lesssim \la_*^{\Theta},
\end{equation*}
and recall
$
R=R(t)=\la_*^{-\beta}.
$
Our first assumption on the constants is the following:
\begin{equation}\label{restri-0}
0<\nu<1, \quad 2<a<3, \quad 0<\beta<1/2,\quad 0<\Theta<1/2.
\end{equation}

\medskip

\subsection{Non-local systems}

As discussed in Section \ref{leading-dynamics}, translation parameter $\xi$ obeys essentially an ODE of first order, while the dynamics of $\la$-$\gamma$ system and $c_1$-$c_2$ system are governed by integro-differential equations. The leading non-local operator might be regarded as the Abel's integral operator in the end-point case, so the solvability is rather involved.

\medskip

For the $\la$-$\gamma$ system, we recall \eqref{def-p_0}.  To introduce the space for the parameter function $p_0(t)=-2(\dot\lambda + i\lambda\dot\gamma )e^{i\gamma }$,
we recall the  non-local operator $\mathcal B_0$  appears at mode 0 on $W^\perp$ is of the approximate form
\begin{align*}
	\mathcal B_0[p_0 ]
	=   \int_{-T} ^{t-\lambda_*^2}    \frac{p_0(s)}{t-s}ds\, + O ( p_0(t) ).
\end{align*}
For  $\Theta\in (0,1)$, $\varpi\in \R$ and a continuous function $g:[-T,T]\to \mathbb C$, we define the norm
\begin{align}\label{norm-thetavarpi}
	\|g\|_{\Theta,\varpi} = \sup_{t\in [-T,T]} \, (T-t)^{-\Theta} |\log(T-t)|^{\varpi} |g(t)| ,
\end{align}
and for  $\alpha \in (0,1)$,  $m,~\varpi \in \R$, we define the semi-norm
\begin{equation*}
	[ g]_{\frac{\alpha}2,m,\varpi} = \sup \, (T-t)^{-m}  |\log(T-t)|^{\varpi} \frac{|g(t)-g(s)|}{(t-s)^{\frac{\alpha}2}} ,
\end{equation*}
where the supremum is taken over $s \leq t$ in $[-T,T]$  such that $t-s \leq \frac{1}{4}(T-t)$.

\medskip

The following proposition gives an approximate inverse of the non-local operator $\mathcal B_0$ with a small remainder $\mathcal{R}_0$.
\begin{prop}{\rm (\cite[Proposition 6.5, Proposition 6.6]{17HMF})}
	\label{prop-laga}
	Let $\alpha_0 ,  \frac{\alpha}2 \in (0,\frac{1}{2})$, $\varpi\in \R$. There is $\flat>0$ such that if $\Theta\in(0,\flat)$ and $m \leq\Theta - \frac{\alpha}2$, then for $h(t) :[0,T]\to \mathbb C$ satisfying
	\begin{align}
		\label{hypA00}
		\left\{
		\begin{aligned}
			& | h(T) | > 0,
			\\
			& T^\Theta |\log T|^{1+\sigma-\varpi} \| h(\cdot) - h(T) \|_{\Theta,\varpi-1}
			+ [h]_{\frac{\alpha}2,m,\varpi-1}
			\leq C_1 ,
		\end{aligned}		\right.
	\end{align}
	for some $C_1>0$, $\sigma\in (0,1)$,
	then, for $T>0$ sufficiently small there exist two operators $\mathcal P $ and $\mathcal{R}_0$ so that $p_0 = \mathcal P[h]: [-T,T]\to \mathbb C$ satisfies
	\begin{align*}
		%\label{eq-modified0}
		%\nonumber
		\mathcal B_0[p_0](t)
		= h(t) + \mathcal{R}_0[h](t) , \quad t \in [0,T]
	\end{align*}
	with
	\begin{align}
		\nonumber
		& |\mathcal{R}_0[h](t) | \leq  C
		\Bigl( T^{\sigma}
		+ T^\Theta  \frac{\log |\log T|}{|\log T|}  \| h(\cdot) - h(T) \|_{\Theta,\varpi-1}
		+ [h]_{\frac{\alpha}2,m,\varpi-1} \Bigr)
		\frac{(T-t)^{m+\frac{(1+\alpha_0 ) \alpha}2}}{  |\log(T-t)|^{\varpi}}.
	\end{align}
Moreover,
\begin{align*}
	%\label{decompP}
	\mathcal P[h] = p_{0,\kappa} + \mathcal{P}_1[h] + \mathcal P_2[h],
\end{align*}
with
\begin{align}
	%\label{p0kappa}
	\nonumber
	p_{0,\kappa}(t) =
\frac{\kappa |\log T|}{|\log(T-t)|^2}
	, \quad  t\leq T  ,
\end{align}
where $\kappa =\kappa[h]$. Moreover, the following bounds hold:
\begin{equation}\label{p-est}
	\begin{aligned}
	&
\kappa = 2h(T) \left( 1+ O(|\log T|^{-1})\right),
\quad
|\pp_{t}\mathcal{P}_1[h](t) | \le C \frac{|\log T|^{1-\sigma} \left( \log (|\log T|)\right)^2}{|\log(T-t)|^{3-\sigma}},
\\
&
|\pp_{t}^2\mathcal{P}_1[h](t)| \le C\frac{|\log T|}{|\log(T-t)|^3 (T-t)},
\quad
\|\pp_{t}\mathcal{P}_2[h] \|_{\Theta,\varpi}
\le C
\left(
T^{\frac{1}{2}+\sigma-\Theta}
+
\| h(\cdot) -h(T) \|_{\Theta,\varpi-1}
\right),
\\
&
[\pp_{t}\mathcal{P}_2[h] ]_{\frac{\alpha}{2},m,\varpi}
\le C
\left(
|\log T|^{\varpi-3} T^{\flat -m-\frac{\alpha}{2}}
+
T^{\Theta} \frac{\log|\log T|}{|\log T|}
\| h(\cdot) -h(T) \|_{\Theta,\varpi-1}
+
[h]_{\frac{\alpha}{2},m,\varpi-1}
\right) .
	\end{aligned}
\end{equation}
\end{prop}

\medskip

The dealing of $c_1$-$c_2$ system is similar as the one of $\la$-$\gamma$ system as their leading non-local operator turn out to be the same. Indeed, from Section \ref{sec-lt-c1c2}, we have:
\begin{prop}\label{prop-c1c2}
For $|f(t)|\lesssim \la_*^{\Theta}(t)$, when $T > 0$ is sufficiently small, there exist two operators $\mathcal{P}_{\cc}$ and $\mathcal{R}_{\cc}$ such that $p_1 = \mathcal{P}_{\cc}[f]: [-T, T]\to \mathbb{C}$ satisfies the nonlocal equation
\begin{equation*}
\int_{-T}^t \frac{p_1(s)}{t-s} \Gamma_3\left(\frac{\lambda^2(t)}{t-s}\right)ds+2\la \dot \cc + \Gamma_4[p_0] \cc = f(t) + \mathcal{R}_{\cc}[f](t),\quad t\in [0, T]
\end{equation*}
with
\begin{equation*}
|R_{\cc}[f]|\leq C\|f\|_{*, \Theta, 0}\frac{(\lambda_*(t))^{\Theta}(T-t)^{\alpha_1}}{|\log(T-t)|^2}
\end{equation*}
and
\begin{equation*}
\left|\frac{d}{dt}R_{\cc}[f]\right|\leq C\|f\|_{*, \Theta, 0}\frac{(\lambda_*(t))^{\Theta}(T-t)^{\alpha_1-1}}{|\log(T-t)|^2}.
\end{equation*}
The function $\mathcal{P}_{\cc}[f]$ can be decomposed as $\mathcal{P}_{\cc}[f]=p_{1,0} + p_{1,1}$. Here $p_{0,1}$ and $p_{1,1}$ satisfy the following properties
\begin{equation*}
\left|p_{1,0}(t)\right|\leq C\|f\|_{*, \Theta, 0}\frac{(\lambda_*(t))^{\Theta}}{|\log(T-t)|},
\end{equation*}
\begin{equation*}
\left|p_{1,1}(t)\right|\leq C\|f\|_{*, \Theta, 0}\frac{(\lambda_*(t))^{\Theta}|\log T|^{k-1}}{|\log(T-t)|^{k+1}},
\end{equation*}
\begin{equation*}
\left|\dot p_{1,1}(t)\right|\leq C\|f\|_{*, \Theta, 0}\frac{(\lambda_*(t))^{\Theta}}{|\log(T-t)|^2(T-t)}
\end{equation*}
and
\begin{equation*}
\left|\ddot p_{1,1}(t)\right|\leq C\|f\|_{*, \Theta, 0}\frac{(\lambda_*(t))^{\Theta}}{|\log(T-t)|^2(T-t)^2}.
\end{equation*}
Here $k\in (0,1)$, $\alpha_1\in(0,1/3)$ and $$\|f\|_{*, \Theta, 0}:=\sup_{t\in [0, T]}\left|\la_*^{-\Theta}(t)f(t)\right|.$$
\end{prop}

\bigskip

\subsection{The inner-outer gluing system}\label{sec-727272}

Because of the refined linear theories for inner problems (Proposition \ref{prop-RG-Wperp0}, Proposition \ref{prop-RG-W1}), $\la$-$\gamma$ system (Proposition \ref{prop-laga}) and $c_1$-$c_2$ system (Proposition \ref{prop-c1c2}), we need to further decompose inner problems \eqref{gluing-sys}$_1$-\eqref{gluing-sys}$_2$ into three pieces. We search for
$$
\Phi^*_{W}+\Phi^*_{W^\perp}+\Phi^\dagger=\Phi_W+\Phi_{W^\perp}
$$
with
\begin{align} \notag
\la^2\pp_t \Phi^*_W=&~\Delta_y\Phi^*_{W } - 2\pp_{y_1}W\wedge \pp_{y_2}\Phi^*_{W } - 2\pp_{y_1} \Phi^*_{W }\wedge \pp_{y_2}W\\ \notag
&~+\la^2\Pi_W\Big[Q_{-\gamma}\tilde L_U[\Psi](q,t)\Big]+\la^2 \Pi_W\Big[Q_{-\gamma}(S[U_*])-\mathcal R^\dagger\Big]+\mathcal H_{{\rm in}}^{W} \\ \notag
&~+ \mathbf{R_{\cc}}[\Psi]+ \sum_{j=1,2}\hat c_{1j} \mathcal Z_{1,j}w_\rho^2 \quad\mbox{  in  }~~ \mathcal{D}_{2R},\\ \notag
\la^2\pp_t \Phi^*_{W^\perp}=&~\Delta_y\Phi^*_{W^\perp} - 2\pp_{y_1}W\wedge \pp_{y_2}\Phi^*_{W^\perp} - 2\pp_{y_1} \Phi^*_{W^\perp}\wedge \pp_{y_2}W\\ \label{gluing-inn}
&~+\la^2\Pi_{W^\perp}\Big[Q_{-\gamma}\tilde L_U[\Psi](q,t)\Big]+\la^2 \Pi_{W^\perp}\Big[Q_{-\gamma} (S[U_*])-\mathcal R^\dagger\Big]+\mathcal H_{{\rm in}}^{W^\perp}\\ \notag
&~+ \mathbf{R_{0}}[\Psi]+\la^2 \Pi_{W^\perp}\Big[Q_{-\gamma}\tilde L_U^{(1)}[\Psi]\Big]+ \sum_{j=1,2}\tilde c_{0j}  Z_{0,j}w_\rho^2+\sum_{j=1,2}c_{1j} Z_{1,j}w_\rho^2 \quad\mbox{  in  }~~ \mathcal{D}_{2R},\\ \notag
\la^2\pp_t \Phi^{\dagger}=&~\Delta_y\Phi^{\dagger} - 2\pp_{y_1}W\wedge \pp_{y_2}\Phi^{\dagger} - 2\pp_{y_1} \Phi^{\dagger}\wedge \pp_{y_2}W\\ \notag
&~+\la^2 \Big[Q_{-\gamma}\left(\tilde L_U[\Psi](\la y+\xi,t)-\tilde L_U[\Psi](q,t)\right)\Big]-\la^2 \Pi_{W^\perp}\Big[Q_{-\gamma}\tilde L_U^{(1)}[\Psi]\Big]\\ \notag
&~+\mathbf{R_{\cc}}[\Psi]+\mathbf{R_{0}}[\Psi]+\la^2   \mathcal R^\dagger \quad\mbox{  in  }~~ \mathcal{D}_{2R},
\end{align}
where
\begin{equation}\label{R_ccR_0}
\begin{aligned}
\mathbf{R_{\cc}}[\Psi]:=&~  \la\mathcal R_{\cc}\Bigg[\displaystyle\int_{\R^2} \la\Pi_W\left[Q_{-\gamma}\tilde L_U[\Psi](q,t) \right]\cdot\mathcal Z_{1,1}\Bigg]\frac{w_\rho^2 \mathcal Z_{1,1}}{{\displaystyle\int_{\R^2} w_\rho^2 \mathcal Z^2_{1,1}}}\\
&~+\la\mathcal R_{\cc}\Bigg[\displaystyle\int_{\R^2}i\la\Pi_W\left[Q_{-\gamma}\tilde L_U[\Psi](q,t) \right]\cdot\mathcal Z_{1,2}\Bigg]\frac{w_\rho^2 \mathcal Z_{1,2}}{{\displaystyle\int_{\R^2} w_\rho^2 \mathcal Z^2_{1,2}}},\\
\mathbf{R_{0}}[\Psi]:=&~\la\mathcal R_{0}\Bigg[\displaystyle\int_{\R^2} \la\Pi_{W^\perp}\left[Q_{-\gamma}\tilde L_U[\Psi](q,t) \right]\cdot Z_{0,1}\Bigg]\frac{w_\rho^2  Z_{0,1}}{{\displaystyle\int_{\R^2} w_\rho^2  Z^2_{0,1}}}\\
&~+\la\mathcal R_{0}\Bigg[\displaystyle\int_{\R^2}i\la\Pi_{W^\perp}\left[Q_{-\gamma}\tilde L_U[\Psi](q,t) \right]\cdot  Z_{0,2}\Bigg]\frac{w_\rho^2  Z_{0,2}}{{\displaystyle\int_{\R^2} w_\rho^2  Z^2_{0,2}}}
\end{aligned}
\end{equation}
with the operators $\mathcal R_0$ and $\mathcal R_{\cc}$ given in Proposition \ref{prop-laga} and Proposition \ref{prop-c1c2}, respectively; for $j=1,2$,
\begin{equation}\label{def-ccc1}
\begin{aligned}
&~\hat c_{1j}=\hat c_{1j}\Bigg[\la^2\Pi_W\Big[Q_{-\gamma}\tilde L_U[\Psi](q,t)\Big]+\la^2 \Pi_W\Big[Q_{-\gamma}(S[U_*])-\mathcal R^\dagger\Big]+\mathcal H_{{\rm in}}^{W} + \mathbf{R_{\cc}}[\Psi]\Bigg]\\
&~\tilde c_{0j}=\tilde c_{0j}\Bigg[\la^2\Pi_{W^\perp}\Big[Q_{-\gamma}\tilde L_U[\Psi](q,t)\Big]+\la^2 \Pi_{W^\perp}\Big[Q_{-\gamma} (S[U_*])-\mathcal R^\dagger\Big]+\mathcal H_{{\rm in}}^{W^\perp} + \mathbf{R_{0}}[\Psi]\Bigg]
\end{aligned}
\end{equation}
are given in Proposition \ref{prop-RG-W1} and Proposition \ref{prop-RG-Wperp0}, respectively; the term $\Pi_{W^\perp}\Big[Q_{-\gamma}\tilde L_U^{(1)}[\Psi]\Big]$ denotes the part of mode $1$ in the projection of $Q_{-\gamma}\left(\tilde L_U[\Psi](\la y+\xi,t)-\tilde L_U[\Psi](q,t)\right)$ onto $W^\perp$,
\begin{equation}\label{def-ccc2}
c_{1j}=\frac{\displaystyle\int_{\R^2}\left(\la^2\Pi_{W^\perp}\Big[Q_{-\gamma}\tilde L_U[\Psi](q,t)\Big]+\la^2 \Pi_{W^\perp}\Big[Q_{-\gamma} (S[U_*])-\mathcal R^\dagger\Big]+\mathcal H_{{\rm in}}^{W^\perp}\right)\cdot Z_{1,j}}{\displaystyle\int_{\R^2} w_\rho^2 Z_{1,j}^2},
\end{equation}
 and
\begin{equation*}
\begin{aligned}
\mathcal R^\dagger:=&~-\frac{2\eta_1 }r \left((\pp_\theta \phi_2^{(2)}+\pp_\theta \phi_2^{(-2)})+\cos w (\phi_1^{(2)}+\phi_1^{(-2)})+\sin2\theta \sin w\psi_2\right)\\
&~\quad \times\la^{-1}w_{\rho}\sin w(c_1 \cos\theta + c_2 \sin\theta) W\\
&~-\frac{2\eta_1 }r \Bigg[\la^{-1}\left(\sin2\theta w_\rho \psi_2+\pp_\rho\phi_1^{(2)}+\pp_\rho\phi_1^{(-2)}\right)\Bigg] \times \sin^2 w(c_1\cos\theta+ c_2\sin \theta) W.
\end{aligned}
\end{equation*}
We will solve \eqref{gluing-inn}$_1$ and \eqref{gluing-inn}$_2$ with orthogonality conditions imposed at corresponding modes
$$
\hat c_{1j}=0,\quad \tilde c_{0j}=0, \quad c_{1j}=0,
$$
and solve \eqref{gluing-inn}$_3$ (consisting of components both on $W^\perp$ and in the $W$-direction) without orthogonality.

\bigskip

\noindent $\bullet$ In the inner problems (supported in $B_{2R}$) with orthogonality condition imposed, by \eqref{eqn-a27a27a27}, one has
\begin{align*}
&~\left|\la^2 \Pi_W\Big[Q_{-\gamma}(S[U_*])-\mathcal R^\dagger\Big]\right|\\
\lesssim&~\la^2|Q_{-\gamma}\mathcal R_{U}|+ \la^2\Big|\Pi_W\Big[Q_{-\gamma}(\mathcal R_*)-\mathcal R^\dagger\Big]\Big|\\
&~+\la^2\Bigg|\Pi_W\Big[Q_{-\gamma}\Big(\eta_1 (\mathcal E_{U^\perp}^{(0)}+\tilde{\mathcal R}^0)+\eta_1 (\mathcal E_{U}^{(\pm1)}+  \tilde{\mathcal R}^1)+\eta_1 \tilde L_U[\Phi^{(0)}+\Phi^{(1)}]\Big)\Big]\Bigg|\\
\lesssim &~  \la_*^{4\Theta}\langle \rho \rangle^{-2}+\la_*^{1+\Theta}\langle \rho \rangle^{-3}+ \la_* \langle \rho \rangle^{-2}+\la_*^{3\Theta+1}+\la_*^2|\dot\la_*|\langle \rho \rangle^{-1} \\
&~+ \Bigg[\la_*^2+ \la_*^{2\Theta+1}  \bigg(\langle \rho \rangle^{-1-\delta}+\langle \rho \rangle^{-\delta_1}\bigg) +\la_*^{5\Theta}  \bigg(\langle \rho \rangle^{-2-\delta}+\langle \rho \rangle^{-1-\delta_1}\bigg)\\
&~\qquad+\la_*^{4\Theta} \langle \rho \rangle^{-2-2\delta}+\la_*^{\Theta+1} \langle \rho \rangle^{-3}\Bigg] \1_{\{\rho\lesssim \la_*^{-1}\}}\\
\lesssim &~  \la_*^{4\Theta}\langle \rho \rangle^{-2}+ \la_* \langle \rho \rangle^{-2}+\la_*^{3\Theta+1}+\la_*^2+  \la_*^{2\Theta+1}  \bigg(\langle \rho \rangle^{-1-\delta}+\langle \rho \rangle^{-\delta_1}\bigg),
\end{align*}
and
\begin{align*}
&~\left|\la^2 \Pi_{W^{\perp}}\Big[Q_{-\gamma}(S[U_*])-\mathcal R^\dagger\Big]\right|\\
\lesssim&~\la^2\Big|Q_{-\gamma}(\mathcal R_{U^\perp}+\mathcal E_{U^\perp}^{(1)})\Big|+ \la^2\Big|\Pi_{W^\perp}\Big[Q_{-\gamma}(\mathcal R_*)\Big]\Big|\\
&~+\la^2\Bigg|\Pi_{W^\perp}\Big[Q_{-\gamma}\Big(\eta_1 (\mathcal E_{U^\perp}^{(0)}+\tilde{\mathcal R}^0)+\eta_1 (\mathcal E_{U}^{(\pm1)}+  \tilde{\mathcal R}^1)+\eta_1 \tilde L_U[\Phi^{(0)}+\Phi^{(1)}]\Big)\Big]\Bigg|\\
\lesssim &~  \la_*^{3\Theta}\langle \rho \rangle^{-2}+\la_*  \langle \rho \rangle^{-2}+\la_*^{3\Theta+1} \\
&~+ \Bigg[\la_*^2+ \la_*^{2\Theta+1}  \langle \rho \rangle^{ -\delta} +\la_*^{4\Theta}  \bigg(\langle \rho \rangle^{-3-2\delta}+\langle \rho \rangle^{-2-\delta-\delta_1}\bigg)+\la_*^{5\Theta}  \langle \rho \rangle^{ -1-\delta} \\
&~\qquad+\la_*^{\Theta+1} \langle \rho \rangle^{-2}+\la_*^{3\Theta}  \langle \rho \rangle^{-3-\delta}\Bigg] \1_{\{\rho\lesssim \la_*^{-1}\}}\\
\lesssim &~ \la_*^{3\Theta}\langle \rho \rangle^{-2}+\la_*  \langle \rho \rangle^{-2}+\la_*^{3\Theta+1}+\la_*^2+ \la_*^{2\Theta+1}  \langle \rho \rangle^{ -\delta}+\la_*^{5\Theta}  \langle \rho \rangle^{ -1-\delta}.
\end{align*}
We thus need following restrictions on the parameters after simplification by \eqref{restri-0}
\begin{equation}\label{restri-1}
\begin{aligned}
&3\Theta-\nu-\beta(a-2)>0,\quad 1-\nu-\beta(a-2)>0,\\
&2\Theta+1-\nu-\beta(a-\delta)>0,\quad 5\Theta-\nu-\beta(a-1-\delta)>0
\end{aligned}
\end{equation}
so that above terms have finite $\|\cdot\|_{\nu,a}$-norm. Next, we estimate by Lemma \ref{linearization-lemma-2}
\begin{equation*}
\begin{aligned}
\left|\la^2\Pi_W\Big[Q_{-\gamma}\tilde L_U[\Psi](q,t)\Big]\right|,~ \left|\la^2\Pi_{W^\perp}\Big[Q_{-\gamma}\tilde L_U[\Psi](q,t)\Big]\right|\lesssim \la_* \langle y\rangle^{-2}.
\end{aligned}
\end{equation*}
So we need
\begin{equation}\label{restri-2}
1-\nu-\beta(a-2)>0.
\end{equation}
Recall \eqref{def-HWHWperp} and \eqref{def-U_*}. We estimate
\begin{equation}\label{est-710710710}
\begin{aligned}
&~\Big|\mathcal H_{{\rm in}}^{W^\perp}\Big|,\quad \Big|\mathcal H_{{\rm in}}^{W}\Big|\\
\lesssim&~ \Big| \nabla_y \Big(\eta_R Q_\gamma(\Phi_W+\Phi_{W^\perp})+\Psi\Big)\Big| \Big| \nabla_y\big(c_1 Q_\gamma \mathcal Z_{1,1}+c_2 Q_\gamma \mathcal Z_{1,2} +\eta_1 \Phi_*\big)\Big|\\
&~+\left|-\eta_R\la^2\dot\gamma J_z Q_\gamma(\Phi_W +\Phi_{W^\perp}) +\eta_R (\la \dot\xi+\la \dot\la y)\cdot \nabla_y(Q_\gamma\Phi_W +Q_\gamma\Phi_{W^\perp})\right|\\
&~+\left|\pp_{y_1}\Big(\eta_R Q_\gamma(\Phi_W+\Phi_{W^\perp})+\Psi\Big)\wedge \pp_{y_2}\Big(\eta_R Q_\gamma(\Phi_W+\Phi_{W^\perp})+\Psi\Big)\right|\\
\lesssim&~\bigg[\la_*^{\nu}R^{\sigma_*(5-a)}\langle y\rangle^{-4} \Big(\|\Phi_W\|_{{\rm in},\sigma_*,\nu,a}+\|\Phi_{W^\perp}\|_{{\rm in},\sigma_*,\nu,a}\Big)+\la_*^{\Theta}(0)\la_*(t)\|\Psi\|_{\sharp,\Theta,\alpha}\bigg]\\
&~\quad\times
\Big( \la_*(t)+\la_*^{\Theta}\langle y\rangle^{-2}+\la_*^{2\Theta}\langle y\rangle^{-\delta-1}\Big)\\
&~+(\la_*^{1+\nu}+\la_*^{3\Theta+\nu})R^{\sigma_*(5-a)}\langle y\rangle^{-3} \Big(\|\Phi_W\|_{{\rm in},\sigma_*,\nu,a}+\|\Phi_{W^\perp}\|_{{\rm in},\sigma_*,\nu,a}\Big)\\
&~+\la_*^{2\nu}R^{2\sigma_*(5-a)}\langle y\rangle^{-8}\Big(\|\Phi_W\|_{{\rm in},\sigma_*,\nu,a}+\|\Phi_{W^\perp}\|_{{\rm in},\sigma_*,\nu,a}\Big)^2 +\la_*^{2\Theta}(0)\la_*^2(t)\|\Psi\|^2_{\sharp,\Theta,\alpha},
\end{aligned}
\end{equation}
and thus need
\begin{equation}\label{restri-3}
\begin{aligned}
\Theta-\beta \sigma_*(5-a)>0, \quad \nu-2 \beta \sigma_*(5-a)>0,
\end{aligned}
\end{equation}
where parts of the restrictions are in \eqref{restri-1} already.

From Proposition \ref{prop-laga} and Proposition \ref{prop-c1c2}, we have
\begin{equation}\label{est-R_ccR_0}
\begin{aligned}
|\mathbf{R_{0}}[\Psi]|\lesssim&~(\la_*(t))^{1+m+\frac{(1+\alpha_0)\alpha}{2}} \langle \rho\rangle^{-5}\\
|\mathbf{R_{\cc}}[\Psi]|\lesssim&~(\la_*(t))^{1+\alpha_1+\Theta} \langle \rho\rangle^{-5}
\end{aligned}
\end{equation}
with sufficient decay in space-time.

\medskip

\noindent $\bullet$ In \eqref{gluing-inn}$_3$ without orthogonality imposed on the RHS, we notice that the term
$$
\la^2 \Big[Q_{-\gamma}\left(\tilde L_U[\Psi](\la y+\xi,t)-\tilde L_U[\Psi](q,t)\right)\Big]-\la^2 \Pi_{W^\perp}\Big[Q_{-\gamma}\tilde L_U^{(1)}[\Psi]\Big]
$$
has no mode $1$ on $W^\perp$. In order for it to be small such that $\Phi^\dagger$, solved from non-orthogonal versions of Proposition \ref{prop-lt-W0}, Proposition \ref{prop-lt-W} (in the $W$-direction), Proposition \ref{prop-lt-Wperp} (on $W^\perp$) and Remark \ref{rmk-651651651}, stays within the space with the $\|\cdot\|_{{\rm in},\sigma_*,\nu,a}$-topology, we need
\begin{equation}\label{restri-4}
1+\Theta+\alpha\beta-2\beta>\nu-\sigma_*\beta(5-a),\quad 0<\alpha<1.
\end{equation}
Indeed, we have
\begin{align*}
\bigg|\la^2 \Big[Q_{-\gamma}\left(\tilde L_U[\Psi](\la y+\xi,t)-\tilde L_U[\Psi](q,t)\right)\Big]-\la^2 \Pi_{W^\perp}\Big[Q_{-\gamma}\tilde L_U^{(1)}[\Psi]\Big]\bigg|
\lesssim  \la_* \langle y\rangle^{\alpha-2} \la_*^{\Theta} R^{-\alpha}\|\Psi\|_{\sharp,\Theta,\alpha}
\end{align*}
where we have used Lemma \ref{linearization-lemma-2} and
$$
|\nabla_x \Psi(x,t)-\nabla_x \Psi(q,t)|\lesssim |x-q|^{\alpha}\la_*^{\Theta}(\la_* R)^{-\alpha}\|\Psi\|_{\sharp,\Theta,\alpha}.
$$
The necessity of restriction \eqref{restri-4} is clear for all the modes except for the mode $0$ in the $W$-direction. For the mode $0$ in the $W$-direction, we can estimate
\begin{align*}
&~\bigg|\la^2 \Big[Q_{-\gamma}\left(\tilde L_U[\Psi](\la y+\xi,t)-\tilde L_U[\Psi](q,t)\right)\Big]-\la^2 \Pi_{W^\perp}\Big[Q_{-\gamma}\tilde L_U^{(1)}[\Psi]\Big]\bigg|\\
\lesssim &~ \la_* \langle y\rangle^{\alpha-2} \la_*^{\Theta} R^{-\alpha}\|\Psi\|_{\sharp,\Theta,\alpha}
\lesssim   \la_*^{1+\Theta} \langle y\rangle^{-\ell} R^{\ell-2}  \|\Psi\|_{\sharp,\Theta,\alpha}
\end{align*}
thanks to the support of the inner problem. Here $\ell>3/2$ and is close to $3/2$, and $\alpha\in(0,1)$ and is close to $1$. So we need
$$
1+\Theta-\beta(\ell-2)>\nu-\sigma_*\beta(5-a),
$$
which is true if \eqref{restri-4} holds.

For the other part $\la^2   \mathcal R^\dagger$, by definitions \eqref{eqn-correctpm2'} and \eqref{eqn-psi_2}, we first observe that it is in modes $\pm1$, $\pm3$ (in the $W$-direction), and also
$$
|\la^2   \mathcal R^\dagger|\lesssim \la_*^{3\Theta}\langle y\rangle^{-3-\delta}.
$$
 Then from Proposition \ref{prop-lt-W1} and Proposition \ref{prop-lt-W+}, we require
 \begin{equation}\label{restri-5}
 3\Theta-2\beta>\nu-\sigma_*\beta(5-a)
 \end{equation}
so that $\|\Phi^\dagger\|_{{\rm in},\sigma_*,\nu,a}<+\infty.$

For $\mathbf{R}_{\cc}[\Psi]$ and $\mathbf{R}_0[\Psi]$, since they are respectively in mode $\pm 1$ of the $W$-direction and in mode $0$ on $W^\perp$, we apply the non-orthogonal part of Proposition \ref{prop-lt-Wperp} and Proposition \ref{prop-lt-W1}. By \eqref{est-R_ccR_0}, we need the restrictions \eqref{restri-nonortho} and \eqref{restri-9} and so that $\Phi^\dagger$ survives in the desired topology.

\medskip

\medskip

\noindent $\bullet$ We now consider the outer problem
\begin{equation}\label{eqn-outerZ}
\left\{
\begin{aligned}
&\pp_t\Psi=\Delta_x \Psi+\mathcal G ~&\mbox{ in }~\R^2\times (0,T),\\
&\Psi(x,0)=Z^*(x)+\sum_{k=1}^5 \mathscr C_k \mathscr Z_k(x) ~&\mbox{ in }~\R^2,
\end{aligned}
\right.
\end{equation}
where $\mathcal G=\mathcal G[\Phi_W,\Phi_{W^\perp},\Psi,p_0,p_1,\xi,Z^*]$ is the RHS in the outer problem \eqref{gluing-sys}$_3$, and $\mathscr C_k$ will be chosen such that we have desired vanishing \eqref{outer-vanishing}. The outer problem can be formulated as a fixed-point problem
$$
\mathcal T_{{\rm out}}[\Psi]=\Gamma_{\R^2}\bullet\mathcal G+\Gamma_{\R^2}\circ\left(Z^*+\sum_{k=1}^5 \mathscr C_k \mathscr Z_k \right).
$$

For the RHS $\mathcal G$ given in \eqref{gluing-sys}$_3$, by \eqref{eqn-a27a27a27} and \eqref{eqn-a28a28a28}, we have
\begin{align*}
&~\left|(1-\eta_R)  S[U_*]\right|\\
\lesssim &~  (1-\eta_R)\Bigg|\Big[\mathcal R_{U^\perp}+\mathcal R_{U}+\mathcal E_{U^\perp}^{(1)}+\eta_1 (\mathcal E_{U^\perp}^{(0)}+\tilde{\mathcal R}^0)+\eta_1 (\mathcal E_{U}^{(\pm1)}+  \tilde{\mathcal R}^1)+\eta_1 \tilde L_U[\Phi^{(0)}+\Phi^{(1)}]\Big]\Bigg|\\
&~+(1-\eta_R)|\mathcal R_*|+|E_{\eta_1}|\\
\lesssim&~\Big(\la_*^{\Theta-1}+\la_*^{3\Theta-2}\Big)\langle \rho \rangle^{-2}  \1_{\{\la_* R\lesssim r \}} +\Big(\la_*^{-1}  \langle \rho \rangle^{-2}+\la_*^{3\Theta-1}+|\dot\la_*| \langle \rho \rangle^{-1} \Big)\1_{\{\la_* R\lesssim r \lesssim 1\}} \\
&~+\Bigg[1+ \la_*^{2\Theta-1}  \langle \rho \rangle^{ -\delta} +\la_*^{4\Theta-2}  \bigg(\langle \rho \rangle^{-3-2\delta}+\langle \rho \rangle^{-2-\delta-\delta_1}\bigg)+\la_*^{5\Theta-2}  \langle \rho \rangle^{ -1-\delta} \\
&~\qquad+\la_*^{\Theta-1} \langle \rho \rangle^{-2}+\la_*^{3\Theta-2}  \langle \rho \rangle^{-3-\delta}\Bigg] \1_{\{\la_* R\lesssim r \lesssim 1\}}+1.
\end{align*}
In order for above terms to have finite $\|\cdot\|_{**}$-norm, following restrictions should be imposed on parameters
\begin{equation}\label{restri-6}
\begin{aligned}
&\sigma_0>0, \quad 3\Theta-1\geq 0,\quad  2\Theta-1+\delta \beta>0.
\end{aligned}
\end{equation}
Also
\begin{equation*}
\begin{aligned}
\left|(1-\eta_R) \tilde L_U[\Psi]\right|\lesssim \frac{\la_*(t)}{r^2}\1_{ \{ r \geq  \la_*R \} }\lesssim T^\epsilon \varrho_2
\end{aligned}
\end{equation*}
for some $\epsilon>0$ provided $\sigma_0>0$.
For the remaining terms, we estimate
\begin{equation*}
\begin{aligned}
 &~\left|Q_\gamma(\Phi_W +\Phi_{W^\perp})\Delta_x\eta_R  + 2\nabla_x\eta_R\cdot\nabla_x(Q_\gamma\Phi_W +Q_\gamma\Phi_{W^\perp})-Q_\gamma(\Phi_W +\Phi_{W^\perp})\partial_t\eta_R\right|\\
 \lesssim &~\big[\la_*^{\nu-2}R^{-a}+\la_*^{\nu}R^{2-a}(T-t)^{-1}\big] \Big(\|\Phi_W\|_{{\rm in},\sigma_*,\nu,a}+\|\Phi_{W^\perp}\|_{{\rm in},\sigma_*,\nu,a}\Big)\1_{\{r\sim \la_*R\}},
\end{aligned}
\end{equation*}
\begin{equation*}
\begin{aligned}
&~\left|(1-\eta_R)\Big[-\eta_R\dot\gamma J_z Q_\gamma(\Phi_W +\Phi_{W^\perp}) +\eta_R (\la^{-1}\dot\xi+\la^{-1}\dot\la y)\cdot \nabla_y(Q_\gamma\Phi_W +Q_\gamma\Phi_{W^\perp})\Big]\right|\\
\lesssim&~\Big(\la_*^{\nu-1}R^{2-a}+\la_*^{3\Theta+\nu-2}R^{1-a}\Big)\Big(\|\Phi_W\|_{{\rm in},\sigma_*,\nu,a}+\|\Phi_{W^\perp}\|_{{\rm in},\sigma_*,\nu,a}\Big)\1_{\{r\geq \la_* R\}},
\end{aligned}
\end{equation*}
and by \eqref{def-U_*}, one has
\begin{align*}
&~(1-\eta_R)\bigg|-2\pp_{x_1}(U_*-U)\wedge \pp_{x_2}\Big(\eta_R Q_\gamma(\Phi_W+\Phi_{W^\perp})+\Psi\Big)\\
&~\qquad\qquad-2\pp_{x_1}\Big(\eta_R Q_\gamma(\Phi_W+\Phi_{W^\perp})+\Psi\Big)\wedge \pp_{x_2}(U_*-U)\bigg|\\
\lesssim&~(1-\eta_R) \Big| \nabla_x \Big(\eta_R Q_\gamma(\Phi_W+\Phi_{W^\perp})+\Psi\Big)\Big| \Big| \nabla_x\big(c_1 Q_\gamma \mathcal Z_{1,1}+c_2 Q_\gamma \mathcal Z_{1,2} +\eta_1 \Phi_*\big)\Big|\\
\lesssim&~\bigg[\la_*^{\nu-1}R^{1-a}\1_{\{r\sim \la_* R\}}\Big(\|\Phi_W\|_{{\rm in},\sigma_*,\nu,a}+\|\Phi_{W^\perp}\|_{{\rm in},\sigma_*,\nu,a}\Big)+\la_*^{\Theta}(0)\1_{\{r\geq \la_*R\}}\|\Psi\|_{\sharp,\Theta,\alpha}\bigg]\\
&~\times
\Big( 1+\la_*^{\Theta-1}\langle y\rangle^{-2}+\la_*^{2\Theta-1}\langle y\rangle^{-\delta-1}\Big).
\end{align*}
Also, we have
\begin{align*}
&~(1-\eta_R)\left|-2\pp_{x_1}\Big(\eta_R Q_\gamma(\Phi_W+\Phi_{W^\perp})+\Psi\Big)\wedge \pp_{x_2}\Big(\eta_R Q_\gamma(\Phi_W+\Phi_{W^\perp})+\Psi\Big)\right|\\
\lesssim &~(1-\eta_R)\Big|\nabla_x \Big(\eta_R Q_\gamma(\Phi_W+\Phi_{W^\perp})+\Psi\Big)\Big|^2\\
\lesssim&~\la_*^{2\nu-2}\langle y\rangle^{2-2a}\Big(\|\Phi_W\|_{{\rm in},\sigma_*,\nu,a}+\|\Phi_{W^\perp}\|_{{\rm in},\sigma_*,\nu,a}\Big)^2\1_{\{|y|\sim R\}}+\la_*^{2\Theta}(0)\|\Psi\|^2_{\sharp,\Theta,\alpha}.
\end{align*}
To bound above terms in the $\|\cdot\|_{**}$-topology, we need
\begin{equation}\label{restri-7}
\begin{aligned}
&~\nu-2+a\beta>\Theta-1+\beta,\quad \nu-1-\beta(2-a)>\Theta-1+\beta,\\
&~3\Theta+\nu-2-\beta(1-a)>\Theta-1+\beta,\quad \nu-1-\beta(1-a)>\Theta-1+\beta,\\
&~\nu+\Theta-2+\beta(1+a)>\Theta-1+\beta,\quad 2\nu-2-\beta(2-2a)>\Theta-1+\beta.
\end{aligned}
\end{equation}

For the Cauchy integrals $\Gamma_{\R^2}\circ\left(Z^*+\sum_{k=1}^5 \mathscr C_k \mathscr Z_k \right)$, we claim
\begin{equation*}
\left| \Gamma_{\R^2}\circ\left(Z^*+\sum_{k=1}^5 \mathscr C_k \mathscr Z_k \right)\right|\lesssim 1
\end{equation*}
since $\|Z^*+\sum_{k=1}^5 \mathscr C_k \mathscr Z_k\|_{L^\infty}\lesssim 1$. Indeed, $\mathscr C_k$ is chosen such that the vanishing \eqref{vanishings} holds, so the Cauchy integrals have the same control as the space-time convolutions estimated above. In fact, one has better estimate $\sum_{k=1}^5|\mathscr C_k|\lesssim\la_*^{-\Theta-1}(0)R^{-1}(0)|\log T|^{-1}$ that can be derived similar to \cite[Proposition A.1]{17HMF}.

\medskip

\subsection{Orthogonal equations}\label{solve-ortho}

We first employ Proposition \ref{prop-laga} and Proposition \ref{prop-c1c2} to derive restrictions on constants, measuring weighted topologies, that ensure the implementation of the gluing process, and then analyze the remainder terms neglected in Section \ref{leading-dynamics}.

\noindent $\bullet$  For the orthogonal equation of $p_0$:
$$
\tilde c_{0j}=0,
$$
we apply Proposition \ref{prop-laga} with
\begin{equation*}
h[\Psi ] :=
\left[
\pp_{x_{1}} \Psi_1 +
\pp_{x_{2}} \Psi_2  + i \left(\pp_{x_{1}} \Psi_2
-
\pp_{x_{2}}  \Psi_1 \right)
\right](q ,t)	.
\end{equation*}
 The vanishing and H\"older properties \eqref{hypA00} are exactly the ones inherited from the weighted topology \eqref{topo-out} for the outer problem, namely
 \begin{equation*}
\begin{aligned}
	&
|h[\Psi](t)-h[\Psi](T)|\lesssim \la_*^{\Theta}(t),
\\
&
\frac{|h[\Psi ](t)-h[\Psi ](s)|}{|t-s|^{\alpha/2}}\lesssim \lambda_*(t)^{\Theta-\alpha(1-\beta)}
	 \mbox{ \ for \ } |t-s|<\frac{T-t}{4}.
\end{aligned}
\end{equation*}
So it is then natural to choose in the $[\cdot]_{\frac{\alpha}2,m,\varpi-1}$-seminorm
\begin{equation*}
	m=\Theta-\alpha(1-\beta).
\end{equation*}
Then in the last line of \eqref{p-est}, $\flat -m-\frac{\alpha}{2}>0$. In order for both $\| h[\Psi ](\cdot) - h[\Psi ](T) \|_{\Theta,\varpi-1}$, $[h[\Psi ]]_{\frac{\alpha}2,m,\varpi-1}$ to be finite, we need
\begin{equation*}
		\varpi-1-2\Theta<0,
		\quad
		\varpi-1-2m<0.
\end{equation*}
Also the assumption $m\leq \Theta-\frac{\alpha}{2}$ in Proposition \ref{prop-laga} implies
\begin{equation*}
	\beta\leq 1/2,
\end{equation*}
which is in the desired self-similar regime as we require before.
Recall the estimate of $\mathcal R_0[h[\Psi ]]$. We require
$$
m+(1+\alpha_0)\frac{\alpha}{2}>\Theta,
$$
namely,
\begin{equation*}
		0<\alpha_0<1/2,\quad
		2\beta-1+\alpha_0>0,
\end{equation*}
so that the vanishing order of $\mathcal R_0[h[\Psi ]]$ as $t\to T$ is faster than the leading part $h[\Psi ]$ itself.

Under the paramter assumption above, the equation for $p_0$ will be solve by Proposition \ref{prop-laga} and the solution $p_0$ satisfies the estimate in \eqref{p-est}. We can then find a solution $\xi $ to the orthogonal equation
$$
c_{1j}=0
$$
with the estimate
\begin{equation}\label{xi-est}
	|\dot{\xi} |\lesssim \lambda_*^{3\Theta-1}
\end{equation}
provided
\begin{equation*}
2\nu -2\sigma_*\beta(5-a) -3\Theta>0 .
\end{equation*}

\medskip

 We then conclude that with above choices of $m,~\alpha_0,~\frac{\alpha}{2},~\varpi$, the remainder gains smallness
\begin{equation}\label{calR-est}
	|\mathcal R_0[h[\Psi ]](t)|\lesssim \la_*^{\Theta+\sigma_1}
\end{equation}
compare to the leading part $h[\Psi ]$ itself, where
\begin{equation*}
	0<\sigma_1<m+\frac{(1+\alpha_0)\alpha}{2}-\Theta.
\end{equation*}
We put the remainder $\mathcal R_0[h[\Psi ]]$, i.e., $\mathbf{R_0}[\Psi]$ defined in \eqref{R_ccR_0}, in the non-orthogonal inner problem, where the extra smallness measured above by $\sigma_1$ is crucial to control the non-orthogonal part. Indeed, from the version of linear theory without orthogonality condition imposed at mode $0$ on $W^\perp$ in Proposition \ref{prop-lt-Wperp}, we need
\begin{equation}\label{restri-nonortho}
\begin{aligned}
	1+\Theta-\alpha(1-\beta)+\frac{(1+\alpha_0)\alpha}{2}-2\beta>\nu-\sigma_*\beta(5-a).
\end{aligned}
\end{equation}

\medskip

In summary, the restrictions on the constants needed when dealing with the reduced problem are given by
\begin{equation}\label{restri-8}
\begin{aligned}
& 0<\beta<\frac12,
		\quad  0<\alpha_0<\frac{1}{2},
		\quad  2\beta-1+\alpha_0>0, \quad 2\nu -2\sigma_*\beta(5-a) -3\Theta>0,\\
&
	1+\Theta-\alpha(1-\beta)+\frac{(1+\alpha_0)\alpha}{2}-2\beta>\nu-\sigma_*\beta(5-a).	
	\end{aligned}
\end{equation}

\medskip

\noindent $\bullet$  By the same reasoning as in $\la$-$\gamma$ system, $c_1$-$c_2$ system will not be completely solved. Instead, the term $\mathbf R_\cc$ defined in \eqref{R_ccR_0} that produces the remainder $R_{\cc}[f]$ will be put in the piece of inner problem with no orthogonality condition imposed. Recall the non-orthogonal linear theory Proposition \ref{prop-lt-W1} for modes $\pm1$ in the $W$-direction. Therefore, from its vanishing bound, one requires
\begin{equation*}
1+\Theta+\alpha_1-2\beta> \nu-\sigma_*\beta(5-a).
\end{equation*}
A more convenient way to achieve above restriction is to assume
\begin{equation}\label{restri-9}
\alpha_1>\alpha\left(\frac{\alpha_0}{2}-\frac12+\beta\right)
\end{equation}
since it will be satisfied provided \eqref{restri-8} and \eqref{restri-9} hold.

\bigskip

 We now analyze those terms that we neglect when deriving the leading dynamics of the modulation parameters. We recall \eqref{error-decomp} and consider the remainders for the mode 0 on $W^\perp$:
\begin{align*}
&\Pi_{W^\perp}\Big[Q_{-\gamma}\tilde L_U[\Psi]\Big]+  \Pi_{W^\perp}\Big[Q_{-\gamma} (S[U_*])\Big]+\mathcal H_{{\rm in}}^{W^\perp}\\
&-
\Pi_{W^\perp}\Big[Q_{-\gamma}\tilde L_U[Z_0^*(q)]\Big]-  \Pi_{W^\perp}\Big[Q_{-\gamma} \Big((\mathcal E_{U^\perp}^{(0)}+\tilde{\mathcal R}^0)+\tilde L_U[\Phi^{(0)}]\Big)\Big].
\end{align*}
In $\mathcal D_{2R}$, one has
\begin{align*}
&~\left|\left(\Pi_{W^\perp}\Big[Q_{-\gamma} (S[U_*])\Big]-  \Pi_{W^\perp}\Big[Q_{-\gamma} \Big((\mathcal E_{U^\perp}^{(0)}+\tilde{\mathcal R}^0)+\tilde L_U[\Phi^{(0)}]\Big)\Big]\right)_{\mathbb C,0}\right|\\
=&~\left|\left(\Pi_{W^\perp}\Big[Q_{-\gamma}\left(\mathcal R_{U^\perp}+\tilde{\mathcal R}^1+\widetilde{\mathcal R}_*+\mathcal R_{*,1}\right)\Big]\right)_{\mathbb C,0}\right|\\
\lesssim &~ \la_*^{\Theta-1}\langle \rho\rangle^{-2}+\la_*^{3\Theta-1}\langle \rho\rangle^{-1}+ 1+ \la_*^{2\Theta-1}  \langle \rho \rangle^{ -\delta} +\la_*^{4\Theta-2}  \bigg(\langle \rho \rangle^{-3-2\delta}+\langle \rho \rangle^{-2-\delta-\delta_1}\bigg)+\la_*^{5\Theta-2}  \langle \rho \rangle^{ -1-\delta} \\
&~ +\la_*^{\Theta-1} \langle \rho \rangle^{-2}+\la_*^{3\Theta-2}  \langle \rho \rangle^{-3-\delta} ,
\end{align*}
where we have used \eqref{eqn-327}, \eqref{def-tildeR_*} and \eqref{eqn-a27a27a27}. Also, it follows from \eqref{topo-out} that
$$
\Big|Q_{-\gamma}\tilde L_U[\Psi-Z_0^*(q)]\Big| \lesssim \la_*^{\Theta-1},
$$
and $\mathcal H_{{\rm in}}^{W^\perp}$ is estimated in \eqref{est-710710710}.

\medskip

We next consider the remainders in mode 1 on $W^\perp$. To estimate
$
\Pi_{W^\perp}\Big[Q_{-\gamma} \widetilde{\mathcal R}_*\Big]+\mathcal H_{{\rm in}}^{W^\perp},
$
we recall \eqref{def-tildeR_*}. Similar to \eqref{eqn-a27a27a27}, one can verify that
\begin{equation}\label{est-732732732}
\begin{aligned}
\Big|\widetilde{\mathcal R}_*\cdot Q_\gamma W\Big| \lesssim &~\Bigg[1+ \la_*^{2\Theta-1}  \bigg(\langle \rho \rangle^{-1-\delta}+\langle \rho \rangle^{-\delta_1}\bigg) +\la_*^{5\Theta-2}  \bigg(\langle \rho \rangle^{-2-\delta}+\langle \rho \rangle^{-1-\delta_1}\bigg)\\
&~\qquad+\la_*^{4\Theta-2} \langle \rho \rangle^{-2-2\delta}+\la_*^{3\Theta-2}  \langle \rho \rangle^{-4-\delta}\Bigg] \1_{\{\rho\lesssim \la_*^{-1}\}},\\
\Big|\widetilde{\mathcal R}_*\cdot Q_\gamma E_1\Big|,~\Big|\widetilde{\mathcal R}_*\cdot Q_\gamma E_2\Big| \lesssim &~\Bigg[1+ \la_*^{2\Theta-1}  \langle \rho \rangle^{ -\delta} +\la_*^{4\Theta-2}  \bigg(\langle \rho \rangle^{-3-2\delta}+\langle \rho \rangle^{-2-\delta-\delta_1}\bigg)+\la_*^{5\Theta-2}  \langle \rho \rangle^{ -1-\delta} \\
&~\qquad+\la_*^{3\Theta-2}  \langle \rho \rangle^{-3-\delta}\Bigg] \1_{\{\rho\lesssim \la_*^{-1}\}}.
\end{aligned}
\end{equation}

\medskip

The estimates for remainders in modes $\pm1$ in the $W$-direction
$
\Pi_{W}\Big[Q_{-\gamma} \widetilde{\mathcal R}_*\Big]+\mathcal H_{{\rm in}}^{W}
$
are similar (cf. \eqref{est-710710710} and \eqref{est-732732732}). Recall that the outer problem will be solved within the weighted space \eqref{topo-out} with the vanishings \eqref{outer-vanishing}. Thanks to this, in the reduced problem for $c_1$ and $c_2$, we have
\begin{equation*}
\begin{aligned}
&~\left| \frac{\la}{\pi}\int_0^{2\pi}\int_0^{+\infty} \left( Q_{-\gamma}\tilde L_U[\Psi(q,t)] \cdot W\right) \sin w e^{i\theta} \rho d\rho d\theta \right|\\
=&~\frac{2}{\pi}\bigg| \int_0^{2\pi}\int_0^{+\infty}  w_\rho\sin^2 w\Big[\partial_{x_1} \Psi_3(q,t)\cos\theta +\partial_{x_2} \Psi_3(q,t)\sin\theta\Big]e^{i\theta} \rho d\rho d\theta \bigg|\\
=&~\frac{2}{\pi}\bigg| \int_0^{2\pi}\int_0^{+\infty}  w_\rho\sin^2 w\Big[\Big(\partial_{x_1} \Psi_3(q,t)-\partial_{x_1} \Psi_3(q,T)\Big)\cos\theta +\Big(\partial_{x_2} \Psi_3(q,t)-\partial_{x_2} \Psi_3(q,T)\Big)\sin\theta\Big]e^{i\theta} \rho d\rho d\theta \bigg|\\
\lesssim &~ \la_*^{\Theta}(t)\|\Psi\|_{\sharp,\Theta,\alpha},
\end{aligned}
\end{equation*}
where we have used Lemma \ref{linearization-lemma-2}. This term appears as part of $f(t)$ in Proposition \ref{prop-c1c2}.

From the estimates above, we see that, under the final choice of constants \eqref{finalchoice}, the remainders in the orthogonal equations $\tilde c_{0j}=c_{1j}=\hat c_{1j}=0$ are indeed of smaller order than the main order taken into account when deriving the leading dynamics in Section \ref{leading-dynamics}.

\bigskip

\subsection{Solving the full system: Proof of Theorem \ref{t:main}}

We now formulate the full gluing system \eqref{gluing-inn} \& \eqref{eqn-outerZ} together with the orthogonal equations
\begin{equation}\label{eqn-mod}
\tilde c_{0j}=c_{1j}=\hat c_{1j}=0
\end{equation}
as a fixed-point problem. Here the definitions for $\hat c_{1j}$, $\tilde c_{0j}$, $c_{1j}$ are given in \eqref{def-ccc1} and \eqref{def-ccc2}.

\medskip

We will solve the outer problem \eqref{eqn-outerZ} in the space
\begin{equation}\label{out-space}
	X_{\Psi} : = \Big\{ \Psi=(\Psi_1,\Psi_2,\Psi_3) :  \|  \Psi  \|_{\sharp, \Theta,\alpha}<+\infty, \  \Psi_j(q,T)=\pp_{x_1}\Psi_3(q,T)=\pp_{x_2}\Psi_3(q,T)=0,~j=1,2,3 \Big\}
\end{equation}
and the inner problems \eqref{gluing-inn}$_1$-\eqref{gluing-inn}$_3$ for $\Phi^*_{W}$, $\Phi^*_{W^\perp}$ and $\Phi^\dagger$ all in the space
\begin{equation}\label{inner-space}
\begin{aligned}
	&X_{\Phi^*_{W}} : = \Big\{ \Phi^*_{W},~\nabla_y\Phi^*_{W}\in L^{\infty}(\mathcal D_{2R}) :  \|\Phi^*_{W}\|_{{\rm in},\sigma_*,\nu,a}<+\infty \Big\}, \\
	&X_{\Phi^*_{W^\perp}} : = \Big\{ \Phi^*_{W^\perp},~\nabla_y\Phi^*_{W^\perp}\in L^{\infty}(\mathcal D_{2R}) :  \|\Phi^*_{W^\perp}\|_{{\rm in},\sigma_*,\nu,a}<+\infty \Big\},\\
	&X_{\Phi^\dagger} : = \Big\{ \Phi^\dagger,~\nabla_y\Phi^\dagger\in L^{\infty}(\mathcal D_{2R}) :  \|\Phi^\dagger\|_{{\rm in},\sigma_*,\nu,a}<+\infty \Big\}.
\end{aligned}
\end{equation}

\medskip

We consider the modulation problem \eqref{eqn-mod}. For $\tilde c_{0j}=0$, Proposition~\ref{prop-laga} gives an approximate inverse $\mathcal P$  of the operator $\mathcal B_0$, so that for given $h(t)$ satisfying \eqref{hypA00},  $ p_0 := \mathcal P  \left[ h \right] $
satisfies
$$
\mathcal B_0[ p_0 ]   = h +\mathcal R_0[ h] \quad \text{ in }~[0,T]
$$
for a small remainder $\mathcal R_0[h]$. Moreover, estimates \eqref{p-est} for
$p_{0,1}:= \mathcal P_1[h]+ \mathcal P_2[h]$
 in Proposition \ref{prop-laga} lead us to define the space
\begin{equation*}
 X_{p_0} := \{  p_{0,1} \in C([-T,T;\mathbb C]) \cap C^1([-T,T;\mathbb C]) \ : \ p_{0,1}(T) = 0 , \ \|p_{0,1}\|_{*,3-\sigma_0}<+\infty \}
\end{equation*}
for some $\sigma_0\in(0,1)$, where we represent $p_0$ by the pair $(\kappa,p_{0,1})$ in the form $p_0 = p_{0,\kappa} + p_{0,1}$, and the $\|\cdot\|_{*,3-\sigma_0}$-seminorm is defined by
$$
\|f\|_{*,3-\sigma_0} := \sup_{t\in [-T,T]}  |\log(T-t)|^{3-\sigma_0} |\dot f(t)|.
$$

For $c_{1j}=0$, we define the space for $\xi(t)$ as
\begin{equation*}
X_{\xi}=\left\{\xi\in C^1((0,T);\R^2)~:~\dot\xi(T)=0,~\|\xi\|_{X_{\xi}}<+\infty\right\}
\end{equation*}
where
\begin{equation*}
\|\xi\|_{X_{\xi}}=\|\xi\|_{L^{\infty}(0,T)}+\sup_{t\in[-T,T]} \la_*^{-\sigma_1}(t)|\dot\xi(t)|
\end{equation*}
for some $\sigma_1\in(0,3\Theta-1)$.

Similarly, for $\hat c_{1j}=0$, Proposition \ref{prop-c1c2} concerns the solvability of $c_1$-$c_2$ system, i.e., $p_1=-2\big[\la(c_1+ic_2)\big]'$, up to a small remainder $\mathcal{R}_{\cc}[f](t)$. The control of the resolution
$
p_1=\mathcal P_{\cc}[f]=p_{1,0}+p_{1,1}
$
for the modified non-local problem motivates us to solve $p_1$ in the space
\begin{equation*}
 X_{p_1} := \{  p_{1} \in C([-T,T;\mathbb C]) \cap C^1([-T,T;\mathbb C]) \ : \ p_{1}(T) = 0 , \ \|p_{1}\|_{\Theta,1}<+\infty \},
\end{equation*}
where the norm above is defined in \eqref{norm-thetavarpi}.

\medskip

We define $\mathfrak X:=X_\Psi\times X_{\Phi^*_{W}} \times X_{\Phi^*_{W^\perp}}\times X_{\Phi^\dagger}\times \mathbb C \times X_{p_0}\times X_{\xi}\times X_{p_1}$ and take a closed subset $\mathfrak B\subset \mathfrak X$ for which
$
(\Psi,\Phi^*_{W},~\Phi^*_{W^\perp},~\Phi^\dagger,~\kappa,~p_{0,1},~\xi,~p_1)\in \mathfrak B
$
 satisfies
\begin{equation*}
\begin{aligned}
&\|  \Psi  \|_{\sharp, \Theta,\alpha}+ \|\Phi^*_{W}\|_{{\rm in},\sigma_*,\nu,a}+ \|\Phi^*_{W^\perp}\|_{{\rm in},\sigma_*,\nu,a}+ \|\Phi^\dagger\|_{{\rm in},\sigma_*,\nu,a}\leq 1,\\
&|\kappa-\kappa_0|\leq |\log T|^{-1},\quad \kappa_0=\div Z^*(q) +i {\rm{curl }} Z^*(q),\\
&\|p_{0,1}\|_{*,3-\sigma_0}\leq C_0 |\log T|^{1-\sigma_0}(\log(|\log T|))^2, \quad \|\xi\|_{X_{\xi}}+\|p_{1}\|_{\Theta,\sigma_2}\leq 1
\end{aligned}
\end{equation*}
for a sufficiently large constant $C_0$. Then we define an operator $\mathfrak F$ which returns the solution from $\mathfrak B$ to $\mathfrak X$
\begin{equation*}
\begin{aligned}
\mathfrak F: \mathfrak B\subset \mathfrak X~\rightarrow& ~\mathfrak X\\
v~\mapsto&~ \mathfrak F(v):=\Big(\mathfrak F_{\Psi}(v),~\mathfrak F_{\Phi^*_W}(v),~\mathfrak F_{\Phi^*_{W^\perp}}(v),~\mathfrak F_{\Phi^\dagger}(v),~\mathfrak F_{\kappa}(v),~\mathfrak F_{p_{0,1}}(v),~\mathfrak F_{\xi}(v),~\mathfrak F_{p_1}(v)\Big).
\end{aligned}
\end{equation*}
Here the operator $\mathfrak F_{\Psi}$ corresponds to the outer problem \eqref{eqn-outerZ} with linear theory given by Proposition \ref{prop3}. The operators $\mathfrak F_{\Phi^*_W}$, $\mathfrak F_{\Phi^*_{W^\perp}}$, $\mathfrak F_{\Phi^\dagger}$ handle respectively three pieces of inner problems \eqref{gluing-inn}$_1$, \eqref{gluing-inn}$_2$, \eqref{gluing-inn}$_3$, and their linear theories are given in Proposition \ref{prop-lt-W} and Proposition \ref{prop-lt-Wperp}. The operators $\mathfrak F_{\kappa}$, $\mathfrak F_{p_{0,1}}$ deal with the $\la$-$\gamma$ system from $\tilde c_{0j}=0$ whose linear theory is in Proposition \ref{prop-laga}. The operator $\mathfrak F_{\xi}$ concerns $c_{1j}=0$ yielding a first order ODE for $\xi$. The operator $\mathfrak F_{p_1}$ is related to the $c_1$-$c_2$ system from $\hat c_{1j}=0$ with linear theory given by Proposition \ref{prop-c1c2}.

The property that $\mathfrak F: \mathfrak B\rightarrow \mathfrak B$ follows from  those estimates in Section \ref{sec-727272} and Section \ref{solve-ortho} under all the restrictions \eqref{restri-0}, \eqref{restri-1}, \eqref{restri-2}, \eqref{restri-3}, \eqref{restri-4}, \eqref{restri-5}, \eqref{restri-6}, \eqref{restri-7}, \eqref{restri-8} and \eqref{restri-9} for the constants. These can be simplified as
\begin{equation}\label{restri}
\begin{aligned}
&0<\nu<1, \quad 2<a<3, \quad 1/3<\Theta<\beta<1/2, \\
&1-\nu-\beta(a-2)>0, \quad 0<\sigma_*<1,\\
&\Theta-\beta \sigma_*(5-a)>0, \quad \nu-2 \beta \sigma_*(5-a)>0,\\
&3\Theta-2\beta>\nu-\sigma_*\beta(5-a),\quad \nu-2+a\beta>\Theta-1+\beta,\\
&0<\alpha<1,\quad \alpha_0\approx1/2, \quad 0<\alpha_1<1/3, \quad \alpha_1>\alpha\left(\frac{\alpha_0}{2}-\frac12+\beta\right),\\
&1+\Theta-\alpha(1-\beta)+\frac{(1+\alpha_0)\alpha}{2}-2\beta>\nu-\sigma_*\beta(5-a),
\end{aligned}
\end{equation}
where we take $\delta=\frac{999}{1000}\approx1$, $\alpha_0=\frac{49}{100}\approx1/2$. With the aid of Mathematica, the system \eqref{restri} admits a valid choice of constants:
\begin{equation}\label{finalchoice}
\begin{aligned}
&\frac{37}{60}<\nu<\frac58, \quad \frac{209-120\nu}{45}<a<3, \quad \beta=\frac38, \quad \Theta=\frac{11}{30},\\
&\frac{8\nu-\frac{14}{5}}{15-3a}<\sigma_*< \frac{4\nu}{15-3a},\quad \alpha=\frac{99}{100},\quad \frac{297}{2500}<\alpha_1<\frac13.
\end{aligned}
\end{equation}
The compactness of the operator $\mathfrak F$ can be proved by suitable variants of \eqref{finalchoice}. Indeed, one can vary slightly the constants such that all the restrictions in \eqref{restri} still hold, and get \eqref{finalchoice} with the weighted norms measured by the new constants with the closed ball $\mathfrak B$ remains the same. For instance, for fixed $\Theta',~\alpha'$ close to $\Theta,\alpha$, one can show that if $v\in\mathfrak B$, then
$$\|\mathfrak F_{\Psi}(v)\|_{\sharp,\Theta',\alpha'}\leq CT^{\epsilon'}$$
for some constants $C,~\epsilon'>0$.
Moreover, one can show that for $\alpha'>\alpha$ and $\Theta'-\Theta>2(\alpha'-\alpha)$, one has a compact embedding in the sense that if a sequence $\{\Psi_k\}_k$ is bounded in the $\|\cdot\|_{\sharp,\Theta',\alpha'}$-norm, then there exists a subsequence that converges in the $\|\cdot\|_{\sharp,\Theta,\alpha}$-norm. The compactness thus follows directly from Arzel\`a--Ascoli's theorem by a standard diagonal argument. The compactness of the rest operators can be proved in a similar manner. Therefore,  the Schauder fixed-point theorem implies the existence of a desired solution. The proof of Theorem \ref{t:main} is complete.

\bigskip

\section{Non-local $c_1$-$c_2$ system}\label{sec-lt-c1c2}
In this section, we solve the equation for $c_1$ and $c_2$ defined in Section \ref{leading-dynamics}
\begin{equation}\label{lambda-c-system}
\int_{-T}^t \frac{p_1(s)}{t-s} \Gamma_3\left(\frac{\lambda^2(t)}{t-s}\right)ds+2\la \dot \cc + \Gamma_4[p_0] \cc = f(t).
\end{equation}
Recall from Section \ref{leading-dynamics} that $p_1=-2(\la \cc)',\quad p_0=-2(\dot\la+i\la\dot\gamma) e^{i\gamma}$, $\cc(t) := c_1(t)+i c_2(t)$ and $f(t)$ is a smooth function satisfying the condition $|f(t)|\lesssim \la_*^{\Theta}(t)$. For simplicity, we will use the notations in this following: $\cc^{(0)}$, $\cc^{(1)}$ and $p_{1,0}(t) = \left(-2\lambda\cc^{(0)}\right)'$, $p_{1,1}(t) = \left(-2\lambda\cc^{(1)}\right)'$.
\subsection{The construction of a solution}
According to the computations in Section \ref{leading-dynamics}, (\ref{lambda-c-system}) can be approximated by the following equation
\begin{equation*}\label{lambda-c-system-100}
\mathcal P_1(t)+\frac23 \mathcal P_0(t)\cc  = f(t) \text{ in } [0, T),
\end{equation*}
so we need to construct an operator $\mathcal{P}_{\cc}$ which assigns $p_1 = \mathcal{P}_{\cc}[f]$ to a function $f$ in a suitable class such that                  \begin{equation}\label{lambda-c-system-1}
\mathcal P_1(t)+\frac23 \mathcal P_0(t)\cc  = f(t) + \mathcal{R}_{\cc}[f](t)\text{ in } [0, T),
\end{equation}
so that $\mathcal{R}_{\cc}[f](t)$ is a suitable small remainder term. The first approximation is the following function
\begin{equation*}
\begin{aligned}
\la \cc^{(0)}=&~(T-t)^{\frac{\mathbf Z}{6}\log(T-t)}\int_t^{T} (T-\tau)^{-\frac{\mathbf Z}{6}\log(T-\tau)}\left(\frac{f(\tau)}{2\log(T-\tau)}-\frac{1}{2\log^2(T-\tau)}\int_\tau^T \frac{f(s)}{T-s}ds\right)d\tau,
\end{aligned}
\end{equation*}
which is a solution of the equation
\begin{equation*}
\mathcal{P}(t):=\int_{-T}^t\frac{p_{1, 0}(s)}{T-s}ds - p_{1, 0}(t)\log(T-t) - \frac{2}{3}\mathbf{Z} \cc^{(0)}  =  f(t),\quad p_{1, 0}(t) = (-2\lambda\mathbf c^{(0)})'
\end{equation*}
and it is an approximation to the equation (\ref{lambda-c-system-1}). Observe that we have $|\lambda\mathbf c^{(0)}|\leq \frac{\lambda_*^\Theta(t)|\log T|(T-t)}{|\log(T-t)|^2}$ and $|(\lambda\mathbf c^{(0)})'|\leq \frac{\lambda_*^\Theta(t)}{|\log(T-t)|}$ since $f(t)$ satisfies the condition $|f(t)|\lesssim \la_*^{\Theta}(t)$.

We look for $p_1$ with the form $p_1 = p_{1, 0} + p_{1, 1}$, where $p_{1, 0}(t) = (-2\lambda\cc^{(0)})'$ and $p_1$ satisfies the following equation
\begin{equation*}
\mathcal{I}[p_{1, 0}] + \mathcal{I}[p_{1, 1}] + \tilde{B}[p_{1, 0} + p_{1, 1}]- \frac{2}{3}\mathbf{Z} \left(\cc^{(0)} + \cc^{(1)}\right)  = f(t) + \mathcal{R}_{\cc}[f](t) \text{ for }t\in [0, T].
\end{equation*}
Here we define $\mathcal{I}[p] = \int_{-T}^{t-\lambda_*^2(t)}\frac{p(s)}{t-s}ds$, the term $\tilde{B}[p_{1, 0} + p_{1, 1}]$ are small terms defined as follows,
\begin{equation*}
\tilde{\mathcal B}[p_1] = \tilde{\mathcal B}_1[p_1] + \tilde{\mathcal B}_2[p_1]\cc + \tilde{\mathcal B}_3[p_1]\cc - 2\lambda\dot\cc,
\end{equation*}
\begin{equation*}
\tilde{\mathcal B}_1[p_1] = -\int_{-T}^{t-\lambda_*^2(t)}\frac{p_1(s)}{t-s}\left(\Gamma_3\left(\frac{\lambda^2(t)}{t-s}\right)+1\right)ds-\int_{t-\lambda_*^2(t)}^t\frac{p_1(s)}{t-s}\Gamma_3\left(\frac{\lambda^2(t)}{t-s}\right)ds,
\end{equation*}
\begin{equation*}
\tilde{\mathcal B}_2[p_1] = \int_{-T}^{t-\lambda_*^2(t)} \frac{ {\rm Re}[p_0(s) e^{-i\gamma(t)}]}{t-s}\left(\Gamma_5\left(\frac{\lambda^2(t)}{t-s}\right)+\frac{2}{3}\right)ds + \int_{t-\lambda_*^2(t)}^t \frac{ {\rm Re}[p_0(s) e^{-i\gamma(t)}]}{t-s}\Gamma_5\left(\frac{\lambda^2(t)}{t-s}\right)ds,
\end{equation*}
\begin{equation*}
\tilde{\mathcal B}_3[p_1] = -\frac{2}{3}\int_{-T}^{t-\lambda_*^2(t)} \frac{ {\rm Re}[p_0(s) e^{-i\gamma(t)}]}{t-s}ds + \frac{2}{3}\mathbf{Z}.
\end{equation*}
The idea is to decompose $\mathcal{I}[p_{1,1}]$ into $\mathcal{I}[p_{1,1}] = S_{\alpha_1}[p_{1, 1}] + \mathcal{R}_{\alpha_1}[p_{1, 1}]$, then replace the operator $\mathcal{I}[p_{1,1}]$ by $S_{\alpha_1}[p_{1,1}]$ and try to solve the corresponding equation. If $\alpha_1 > 0$ is small, then we can find $p_{1,1}$ such that
\begin{equation*}
\mathcal{I}[p_{1, 0}] + S_{\alpha_1}[p_{1, 1}] + \tilde{B}[p_{1, 0} + p_{1, 1}] - \frac{2}{3}\mathbf{Z} \left(\cc^{(0)} + \cc^{(1)}\right) = f(t) \text{ for }t\in [0, T].
\end{equation*}
This means that we have
\begin{equation*}
\mathcal{B}[p_{1, 0} + p_{1, 1}] - \frac{2}{3}\mathbf{Z} \left(\cc^{(0)} + \cc^{(1)}\right) = f(t)  + \mathcal{R}_{\alpha_1}[p_{1, 1}] \text{ for }t\in [0, T].
\end{equation*}
We will prove that
\begin{equation*}
\left|\mathcal{R}_{\alpha_1}[p_{1, 1}]\right |\leq C(T-t)^{\Theta+\alpha_1} \text{ for }t\in [0, T].
\end{equation*}
Now we decompose
\begin{equation*}
S_{\alpha_1}[g] = \tilde L_0[g] + \tilde L_1[g]
\end{equation*}
where
\begin{equation*}
\tilde L_0[g](t) = (1-\alpha_1)|\log(T-t)|g(t)
\end{equation*}
and $\tilde L_1[g]$ contains all other terms, we have
\begin{equation*}
\begin{aligned}
\tilde L_1[g] &= \int_{-T}^t\frac{g(s)}{T-s}ds + \int_{t-(T-t)}^{t-(T-t)^{1+\alpha_1}}\frac{g(s)}{t-s}ds-\int_{t-(T-t)}^t\frac{g(s)}{T-s}ds\\
&\quad + \int_{-T}^{t-(T-t)}g(s)\left(\frac{1}{t-s}-\frac{1}{T-s}\right)ds + (4\log(|\log(T-t)|)-2\log(|\log T|))g(t).
\end{aligned}
\end{equation*}
Now we look for a function $g$ solving the following problem
\begin{equation*}
S_{\alpha_1}[g] - \frac{1}{3} \frac{\mathbf Z}{\lambda}g = \tilde f(t)\text{ in }[-T, T].
\end{equation*}
We solve a modified version of this equation. Let $\eta$ be a smooth cut-off function such that
\begin{equation*}
\eta(s)= 1\text{ for }s\geq 0, \,\, \eta(s) = 0\text{ for }s\leq -\frac{1}{4}.
\end{equation*}
We will find a function $g$ such that
\begin{equation}\label{modified-problem}
\tilde L_0[g] + \eta(\frac{t}{T})\tilde L_1[g] - \frac{1}{3} \frac{\mathbf Z}{\lambda}g = \tilde f(t)\text{ in }[-T, T].
\end{equation}
\begin{lemma}\label{lemma7.2}
When $\alpha_1\in (0, \frac{1}{3})$ and $T > 0$ are sufficiently small, there is a linear operator $T_1$ such that $g=T_1[\tilde f]$ satisfies (\ref{modified-problem}) and the following estimate holds
\begin{equation*}
\|g\|_{*,\Theta,k+1}\leq C\|\tilde f\|_{*,\Theta,k}.
\end{equation*}
Here the norms are defined by
\begin{equation*}
\|h\|_{*, \Theta, k} = \sup_{t\in [-T, T]}(\lambda_*(t))^{-\Theta}|\log(T-t)|^k|h(t)|.
\end{equation*}
Here $k\in (0,1)$.
\end{lemma}

Let us start with the construction of the linear operator $T_1$. We will find an inverse for $\tilde L_0$, namely given a function $\tilde f$, find $g$ such that $\tilde L_0[g] - \frac{1}{3} \frac{\mathbf Z}{\lambda}g =\tilde  f$.
Set $g = (-2\lambda \cc)'$, differentiating this equation we get
\begin{equation*}
\begin{aligned}
-2(\lambda \cc)''-2\frac{1}{(T-t)|\log(T-t)|}(\lambda\cc)' + \frac{2\mathbf Z}{3(1-\alpha_1)}\frac{\cc'}{|\log(T-t)|} = \frac{1}{1-\alpha_1}\frac{\tilde f'(t)}{|\log(T-t)|}.
\end{aligned}
\end{equation*}
This equation can be rewritten as
\begin{equation*}
\begin{aligned}
-2\left(|\log(T-t)|(\lambda \cc)'\right)'  + \frac{2\mathbf Z}{3(1-\alpha_1)}\cc'  = \frac{1}{1-\alpha_1}\tilde f'(t).
\end{aligned}
\end{equation*}
Integrating this equation gives us
\begin{equation*}
\begin{aligned}
2|\log(T-t)|(\lambda \cc)'  + \frac{2\mathbf Z}{3(1-\alpha_1)}\int_{t}^T\cc'(s)ds= -\frac{1}{1-\alpha_1}\tilde f(t),
\end{aligned}
\end{equation*}
which is equivalent to
\begin{equation}\label{equation-for-lambda-c}
\begin{aligned}
&(\lambda \cc)' - \frac{\mathbf Z}{3(1-\alpha_1)}|\log(T-t)|^{-1}\cc(t)= \frac{1}{2(1-\alpha_1)}|\log(T-t)|^{-1}\tilde f(t).
\end{aligned}
\end{equation}
Rewrite it as
\begin{equation*}
\begin{aligned}
&\left[(T-t)^{-\frac{\mathbf Z}{6(1-\alpha_1)}\frac{\log(T-t)}{|\log T|}}(\lambda\cc)\right]'= (T-t)^{-\frac{\mathbf Z}{6(1-\alpha_1)}\frac{\log(T-t)}{|\log T|}}\frac{1}{2(1-\alpha_1)}|\log(T-t)|^{-1}\tilde f(t).
\end{aligned}
\end{equation*}
We finally get
\begin{equation}\label{formulaforg}
\begin{aligned}
\lambda\cc = (T-t)^{\frac{\mathbf Z}{6(1-\alpha_1)}\frac{\log(T-t)}{|\log T|}}\int_{t}^T(T-\tau)^{-\frac{\mathbf Z}{6(1-\alpha_1)}\frac{\log(T-\tau)}{|\log T|}}\frac{1}{2(1-\alpha_1)}|\log(T-\tau)|^{-1}\tilde f(\tau)d\tau.
\end{aligned}
\end{equation}
We denote the the operator which assigns $\tilde f(t)$ to $(-2\lambda\cc)'$ by the formulas (\ref{formulaforg}) and (\ref{equation-for-lambda-c}) as $T_0$.
\begin{lemma}\label{construction-of-T0}
Set $\alpha_1\in (0, \frac{1}{3})$, $k\in (0,1)$ and $(-2\lambda\cc)' := T_0(\tilde f)$ given by (\ref{formulaforg}). Then the following estimate holds
\begin{equation*}
\|T_0[\tilde f]\|_{*,\Theta, k+1}\leq \frac{2}{1-\alpha_1}\|\tilde f\|_{*, \Theta, k}
\end{equation*}
for a constant $C > 0$ which is independent of $k$ and $T$.
\end{lemma}
\begin{proof}
We recall that
\begin{equation*}
\|\tilde f\|_{*, \Theta, k} = \sup_{t\in [-T, T]}(\lambda_*(t))^{-\Theta}|\log(T-t)|^k|\tilde f(t)|.
\end{equation*}
Then from the formula (\ref{formulaforg}), we have the following estimate
\begin{equation*}
\begin{aligned}
\left|\lambda\cc\right| & = (T-t)^{\frac{\mathbf Z}{6(1-\alpha_1)}\frac{\log(T-t)}{|\log T|}}\int_{t}^T(T-\tau)^{-\frac{\mathbf Z}{6(1-\alpha_1)}\frac{\log(T-\tau)}{|\log T|}}\frac{1}{2(1-\alpha_1)}|\log(T-\tau)|^{-1}\tilde f(\tau)d\tau\\
&\leq \frac{1}{2(1-\alpha_1)}\frac{(\lambda_*(t))^{\Theta}\|\tilde f\|_{**,\Theta,k}}{|\log(T-t)|^{k+1}}(T-t)^{\frac{\mathbf Z}{6(1-\alpha_1)}\frac{\log(T-t)}{|\log T|}}\int_{t}^T(T-\tau)^{-\frac{\mathbf Z}{6(1-\alpha_1)}\frac{\log(T-\tau)}{|\log T|}}d\tau\\
& \leq  \frac{3(\lambda_*(t))^{\Theta}|\log T|(T-t)}{2\mathbf Z|\log(T-t)|^{k+2}}\|\tilde f\|_{**,\Theta,k}.
\end{aligned}
\end{equation*}
Using equation (\ref{equation-for-lambda-c}), we know that
\begin{equation*}
\begin{aligned}
\left|(-2\lambda\cc)'\right| \leq  \frac{2(\lambda_*(t))^{\Theta}}{(1-\alpha_1)|\log(T-t)|^{k+1}}\|\tilde f\|_{**,\Theta,k}.
\end{aligned}
\end{equation*}
\end{proof}
\noindent {\bf Proof of Lemma \ref{lemma7.2}:} We construct $g$ as a solution of the fixed point problem
\begin{equation*}
g = T_0\left[\tilde f-\eta(\frac{t}{T})\tilde L_1[g]\right]
\end{equation*}
where $T_0$ is the operator defined in (\ref{formulaforg}) and $\eta$ is the cut-off function.  By Lemma \ref{construction-of-T0}, we have
\begin{equation*}
\|T_0\left[\eta(\frac{t}{T})\tilde L_1[g]\right]\|_{*,\Theta,k+1}\leq \frac{2}{1-\alpha_1}\|\tilde L_1[g]\|_{*,\Theta,k}.
\end{equation*}
Let us estimate the different terms in $\tilde L_1$, which are defined by
$$
\tilde L_1[g] = \sum_{j=0}^4\tilde L_{1j}[g]
$$
where
$$
\tilde L_{10}[g] = \int_{-T}^t\frac{g(s)}{T-s}ds,\quad
\tilde L_{11}[g] = \int_{t-(T-t)}^{t-(T-t)^{1+\alpha_1}}\frac{g(s)}{t-s}ds,\quad
\tilde L_{12}[g] = \int_{t-(T-t)}^{t}\frac{g(s)}{T-s}ds
$$
$$
\tilde L_{13}[g] = \int_{-T}^{t-(T-t)}g(s)\left(\frac{1}{t-s}-\frac{1}{T-s}\right)ds,\quad
\tilde L_{14}[g] = (4\log(|\log(T-t)|)-2\log(|\log T|))g(t).
$$
Then we have the following estimates. First,
\begin{equation*}
\begin{aligned}
|\tilde L_{10}[g]| & \leq \|g\|_{*,\Theta,k+1} \int_{-T}^{t}\frac{(\lambda_*(s))^{\Theta}}{(T-s)|\log(T-s)|^{k+1}}ds\leq \|g\|_{*,\Theta,k+1}\frac{(\lambda_*(t))^{\Theta}}{|\log(T-t)|^{k+1}},
\end{aligned}
\end{equation*}
and
\begin{equation*}
\begin{aligned}
|\tilde L_{11}[g]| & \leq \|g\|_{*,\Theta,k+1} \int_{t-(T-t)}^{t-(T-t)^{1+\alpha_1}}\frac{(\lambda_*(s))^{\Theta}}{(t-s)|\log(T-s)|^{k+1}}ds\\
&\leq \|g\|_{*,\Theta,k+1}\frac{(\lambda_*(t))^{\Theta}}{|\log(T-t)|^{k+1}}\int_{t-(T-t)}^{t-(T-t)^{1+\alpha_1}}\frac{1}{t-s}ds\\
&\leq \|g\|_{*,\Theta,k+1}\frac{\alpha_1(\lambda_*(t))^{\Theta}}{|\log(T-t)|^{k}},
\end{aligned}
\end{equation*}
therefore
\begin{equation*}
\|\tilde L_{10}[g]\|_{*,\Theta,k}\leq \frac{1}{|\log T|}\|g\|_{*,\Theta,k+1}, \quad \|\tilde L_{11}[g]\|_{*,\Theta,k}\leq \alpha_0\|g\|_{*,\Theta,k+1}.
\end{equation*}
Second,
\begin{equation*}
\begin{aligned}
|\tilde L_{12}[g]| & \leq \|g\|_{*,\Theta,k+1} \int_{t-(T-t)}^{t}\frac{(\lambda_*(s))^{\Theta}}{(T-s)|\log(T-s)|^{k+1}}ds\\
&\leq \|g\|_{*,\Theta,k+1}\frac{(\lambda_*(t))^{\Theta}}{|\log(T-t)|^{k+1}}\int_{t-(T-t)}^{t}\frac{1}{T-s                             }ds\\
&\leq \|g\|_{*,\Theta,k+1}\frac{(\lambda_*(t))^{\Theta}}{|\log(T-t)|^{k+1}}\log 2
\end{aligned}
\end{equation*}
which implies that
\begin{equation*}
\|\tilde L_{12}[g]\|_{*,\Theta,k}\leq \frac{\log 2}{|\log T|}\|g\|_{*,\Theta,k+1}.
\end{equation*}
For $\tilde L_{13}$ we have
\begin{equation*}
\begin{aligned}
|\tilde L_{13}[g]| & \leq \|g\|_{*,\Theta,k+1}\int_{-T}^{t-(T-t)}\frac{(\lambda_*(s))^{\Theta}}{|\log(T-s)|^{k+1}}\left(\frac{1}{t-s}-\frac{1}{T-s}\right)ds\\
& \leq C\|g\|_{*,\Theta,k+1}(T-t)\int_{-T}^{t-(T-t)}\frac{(\lambda_*(s))^{\Theta}}{(T-s)^2|\log(T-s)|^{k+1}}ds\\
& \leq C\|g\|_{*,\Theta,k+1}\frac{(\lambda_*(t))^{\Theta}}{|\log(T-t)|^{k+1}}
\end{aligned}
\end{equation*}
and this gives
\begin{equation*}
\|\tilde L_{13}[g]\|_{*,\Theta,k}\leq \frac{C}{|\log T|}\|g\|_{*,\Theta,k+1}.
\end{equation*}
Finally, we have
\begin{equation*}
\begin{aligned}
|\tilde L_{14}[g]| & \leq \|g\|_{*,\Theta,k+1}(\lambda_*(t))^{\Theta}\frac{\log(|\log(T-t)|) + \log(|\log T|))}{|\log(T-t)|^{k+1}}
\end{aligned}
\end{equation*}
and this gives us the estimate
\begin{equation*}
\|\tilde L_{14}[g]\|_{*,\Theta,k}\leq \frac{C\log(|\log T|)}{|\log T|}\|g\|_{*,\Theta,k+1}.
\end{equation*}
Combine all the estimates above we get
\begin{equation*}
\begin{aligned}
\|T_0[\eta(\frac{t}{T})\tilde L_1[g]]\|_{*,\Theta,k+1}\leq \frac{2}{1-\alpha_1} \left(\alpha_1 + \frac{1}{|\log T|} + \frac{\log|\log T|}{|\log T|}\right)\|g\|_{*,\Theta,k+1}.
\end{aligned}
\end{equation*}
If $0 < \alpha_1 < \frac{1}{3}$ is fixed and $T>0$ is sufficiently small, then we get
\begin{equation*}
\|g\|_{*,\Theta,k+1}\leq C\|\tilde f\|_{*,\Theta,k}.
\end{equation*} \qed

\noindent Let
\begin{equation}\label{definitionofet}
E(t) = \mathcal{I}[p_{1,0}](t) + \frac{2}{3}\mathbf{Z}\cc^{(0)} - f(t)
\end{equation}
and we consider the fixed point problem
\begin{equation}\label{fixedproblemforp1}
p_{1,1} = \mathcal{A}[p_{1,1}]
\end{equation}
where
\begin{equation*}
\mathcal{A}[p_{1,1}] = T_1[-\eta(\frac{t}{T}) E - \tilde{\mathcal{B}}[p_{1,0}+p_{1,1}]].
\end{equation*}
Then we have
\begin{prop}\label{nonlinear-system-for-linear-problem}
When $\alpha\in (0, \frac{1}{3})$ and $k\in (0,1)$, there is a function $p_{1,1}$ satisfying (\ref{fixedproblemforp1}) satisfying $\|p_{1,1}\|_{*,k+1}\leq C_0|\log T|^{k-1}$. Here $C_0$ is a sufficiently large but fixed constant.
\end{prop}
\noindent We divide the proof into the following lemmas.
\begin{lemma}\label{estimateoftheerror}
For the term $E(t)$ defined in (\ref{definitionofet}), we have
\begin{equation}\label{estimate-for-E(t)}
|E(t)|\leq C\|f\|_{*,\Theta,0}\frac{(\lambda_*(t))^{\Theta}}{\left|\log(T-t)\right |},\quad -\frac{T}{4}\leq t \leq T.
\end{equation}
\end{lemma}
\begin{proof}
By definition (\ref{definitionofet}) we have
\begin{equation*}
E(t) = \int_{-T}^{t-\lambda_*(t)^2}\frac{p_{1,0}(s)}{t-s}ds + \frac{2}{3}\mathbf{Z}\cc^{(0)} - f(t).
\end{equation*}
Let $t\in [-\frac{T}{4}, T]$ and we write
\begin{equation*}
\begin{aligned}
E(t) & = \int_{-T}^{t-(T-t)/5}\frac{p_{1,0}(s)}{t-s}ds + \int_{t-(T-t)/5}^{t-\lambda_*(t)^2}\frac{p_{1,0}(s)}{t-s}ds+ \frac{2}{3}\mathbf{Z}\cc^{(0)} - f(t)\\
& = \int_{-T}^{t}\frac{p_{1,0}(s)}{T-s}ds - \int_{t-(T-t)/5}^{t}\frac{p_{1,0}(s)}{T-s}ds\\
&\quad +\int_{-T}^{t-(T-t)/5}p_{1,0}(s)\left(\frac{1}{t-s}-\frac{1}{T-s}\right)ds + \int_{t-(T-t)/5}^{t-\lambda_*(t)^2}\frac{p_{1,0}(s)}{t-s}ds+ \frac{2}{3}\mathbf{Z}\cc^{(0)} - f(t).
\end{aligned}
\end{equation*}
Then we estimate as follows,
\begin{equation*}
\begin{aligned}
\left|\int_{t-(T-t)/5}^{t}\frac{p_{1,0}(s)}{T-s}ds\right| &\lesssim \|f\|_{*,\Theta,0}\int_{t-(T-t)/5}^{t}\frac{(\lambda_*(s))^{\Theta}}{\left|\log(T-s)\right |}\frac{1}{T-s}ds\\
&\quad \lesssim \|f\|_{*,\Theta,0}\frac{(\lambda_*(t))^{\Theta}}{\left |\log(T-t)\right |}\frac{1}{T-t}\int_{t-(T-t)/5}^{t}ds\lesssim \|f\|_{*,\Theta,0}\frac{(\lambda_*(t))^{\Theta}}{\left|\log(T-t)\right |}
\end{aligned}
\end{equation*}
and
\begin{equation*}
\begin{aligned}
\left|\int_{-T}^{t-(T-t)/5}p_{1,0}(s)\left(\frac{1}{t-s}-\frac{1}{T-s}\right)ds\right| &\lesssim \|f\|_{*,\Theta,0}\int_{-T}^{t-(T-t)/5}\frac{(\lambda_*(s))^{\Theta}}{\left|\log(T-s)\right|}\frac{T-t}{(t-s)(T-s)}ds\\
& \lesssim \|f\|_{*,\Theta,0}\int_{-T}^{t-(T-t)/5}\frac{(\lambda_*(s))^{\Theta}}{\left|\log(T-s)\right|}\frac{T-t}{(T-s)^2}ds\\
& \lesssim \|f\|_{*,\Theta,0}\frac{(\lambda_*(t))^{\Theta}}{\left|\log(T-t)\right|}.
\end{aligned}
\end{equation*}
For the fourth term in $E$ we have
\begin{equation*}
\begin{aligned}
\int_{t-(T-t)/5}^{t-\lambda_*(t)^2}\frac{p_{1,0}(s)}{t-s}ds  & = p_{1,0}(t)\int_{t-(T-t)/5}^{t-\lambda_*(t)^2}\frac{1}{t-s}ds - \int_{t-(T-t)/5}^{t-\lambda_*(t)^2}\frac{p_{1,0}(t)-p_{1,0}(s)}{t-s}ds\\
& = p_{1,0}(t)\left(\log(T-t)/5-2\log(\lambda_*(t))\right) - \int_{t-(T-t)/5}^{t-\lambda_*(t)^2}\frac{p_{1,0}(t)-p_{1,0}(s)}{t-s}ds.
\end{aligned}
\end{equation*}
Furthermore,
\begin{equation*}
\begin{aligned}
\left|\int_{t-(T-t)/5}^{t-\lambda_*(t)^2}\frac{p_{1,0}(t)-p_{1,0}(s)}{t-s}ds\right |\leq \sup_{-T\leq s\leq t}|\dot p_{1,0}(s)|(T-t)\leq \|f\|_{*,\Theta,0}\frac{(\lambda_*(t))^{\Theta}}{\left|\log(T-t)\right |}.
\end{aligned}
\end{equation*}
Since $p_{1,0}$ satisfies the equation $\int_{-T}^t\frac{p_{1, 0}(s)}{T-s}ds - p_{1, 0}(t)\log(T-t) - \frac{2}{3}\mathbf{Z} \cc^{(0)}  - f(t) = 0$, we have
\begin{equation*}
\begin{aligned}
E(t) \leq C \|f\|_{*,\Theta,0}\left(\frac{(\lambda_*(t))^{\Theta}}{\left|\log(T-t)\right |}\right ).
\end{aligned}
\end{equation*}
This is the desired estimate (\ref{estimate-for-E(t)}).
\end{proof}
\begin{lemma}\label{basicestimate}
We have the estimate
\begin{equation*}
\int_{-T}^{t-(\lambda_*(t))^2}\frac{(\lambda_*(s))^{\Theta}}{(t-s)^a|\log(T-s)|^b}ds \leq C\frac{(\lambda_*(t))^{2(1-a)+\Theta}}{|\log(T-t)|^b},\quad t\in [0, T]
\end{equation*}
for $a > 1$ and $b > 0$.
For $\mu\in (0,1)$, $l\in \mathbb{R}$, we have
\begin{equation*}
\int_{-T}^{t-(\lambda_*(t))^2}\frac{(T-s)^{\mu}}{(t-s)^2|\log(T-s)|^l}ds \leq C\frac{(T-t)^\mu}{(\lambda_*(t))^2|\log(T-t)|^l}.
\end{equation*}
\end{lemma}
\begin{proof}
Let us write
\begin{equation*}
\begin{aligned}
&\int_{-T}^{t-(\lambda_*(t))^2}\frac{(\lambda_*(s))^{\Theta}}{(t-s)^a|\log(T-s)|^b}ds\\
&= \int_{-T}^{t-(T-t)}\frac{(\lambda_*(s))^{\Theta}}{(t-s)^a|\log(T-s)|^b}ds + \int_{t-(T-t)}^{t-(\lambda_*(t))^2}\frac{(\lambda_*(s))^{\Theta}}{(t-s)^a|\log(T-s)|^b}ds.
\end{aligned}
\end{equation*}
Then, for $t\in [0, T]$, we estimate,
\begin{equation*}
\begin{aligned}
&\int_{-T}^{t-(T-t)}\frac{(\lambda_*(s))^{\Theta}}{(t-s)^a|\log(T-s)|^b}ds \leq C\int_{-T}^{t-(T-t)}\frac{1}{(T-s)^{a-\Theta}|\log(T-s)|^{b+2\Theta}}ds\\
&\leq C\frac{1}{(T-t)^{a-\Theta-1}|\log(T-t)|^{b+2\Theta}}\leq C\frac{(\lambda_*(t))^{\Theta}}{(T-t)^{a-1}|\log(T-t)|^{b}}\leq C\frac{(\lambda_*(t))^{2(1-a)+\Theta}}{|\log(T-t)|^b}
\end{aligned}
\end{equation*}
and
\begin{equation*}
\begin{aligned}
&\int_{t-(T-t)}^{t-(\lambda_*(t))^2}\frac{(\lambda_*(s))^{\Theta}}{(t-s)^a|\log(T-s)|^b}ds\\
&\leq C\frac{(\lambda_*(t))^{\Theta}}{|\log(T-t)|^b}\int_{t-(T-t)}^{t-(\lambda_*(t))^2}\frac{1}{(t-s)^a}ds\leq C\frac{(\lambda_*(t))^{\Theta}}{(\lambda_*(t))^{2a-2}|\log(T-t)|^{b}}.
\end{aligned}
\end{equation*}
The case for $t\in [-T, 0]$ is similar. Indeed, we have
\begin{equation*}
\begin{aligned}
\left|\int_{-T}^{t-(\lambda_*(t))^2}\frac{(\lambda_*(s))^{\Theta}}{(t-s)^a|\log(T-s)|^b}ds\right |\leq \frac{(\lambda_*(t))^{\Theta}}{|\log(T-t)|^b}\int_{-T}^{t-(\lambda^*(t))^2}\frac{1}{(t-s)^a}ds\leq C\frac{(\lambda_*(t))^{\Theta+2(1-a)}}{|\log(T-t)|^b}.
\end{aligned}
\end{equation*}
Similarly, we have
\begin{equation*}
\int_{-T}^{t-(\lambda_*(t))^2}\frac{(T-s)^{\mu}}{(t-s)^2|\log(T-s)|^l}ds \leq C\frac{(T-t)^\mu}{(\lambda_*(t))^2|\log(T-t)|^l}.
\end{equation*}
\end{proof}

\begin{lemma}\label{estimateforB}
Let $M = C|\log T|^{k-1}$, $k\in (0,1)$, then for $\|p_{1,1}\|_{*,\Theta,k+1}\leq M$, we have
\begin{equation*}
\|\tilde B[p_{1,0}+p_{1,1}](\cdot)\|_{*,\Theta,k}\leq C\|f\|_{*,\Theta,0}|\log T|^{k-1}.
\end{equation*}
\end{lemma}
\begin{proof}
Observe that with the choice of $M$, if $\|p_{1,1}\|_{*,\Theta,k+1}\leq M$, then we have
\begin{equation*}
\left|\frac{p_{1,1}}{p_{1,0}}\right|\leq \frac{C_0}{|\log T|^k}
\end{equation*}
since $T > 0$ is sufficiently small.
Let us recall that
\begin{equation*}
\tilde{\mathcal B}[p_1] = \tilde{\mathcal B}_1[p_1] + \tilde{\mathcal B}_2[p_1]\cc + \tilde{\mathcal B}_3[p_1]\cc - 2\lambda\dot\cc,
\end{equation*}
\begin{equation*}
\tilde{\mathcal B}_1[p_1] = -\int_{-T}^{t-\lambda_*^2(t)}\frac{p_1(s)}{t-s}\left(\Gamma_3\left(\frac{\lambda^2(t)}{t-s}\right)+1\right)ds-\int_{t-\lambda_*^2(t)}^t\frac{p_1(s)}{t-s}\Gamma_3\left(\frac{\lambda^2(t)}{t-s}\right)ds,
\end{equation*}
\begin{equation*}
\tilde{\mathcal B}_2[p_1] = \int_{-T}^{t-\lambda_*^2(t)} \frac{ {\rm Re}[p_0(s) e^{-i\gamma(t)}]}{t-s}\left(\Gamma_5\left(\frac{\lambda^2(t)}{t-s}\right)+\frac{2}{3}\right)ds + \int_{t-\lambda_*^2(t)}^t \frac{ {\rm Re}[p_0(s) e^{-i\gamma(t)}]}{t-s}\Gamma_5\left(\frac{\lambda^2(t)}{t-s}\right)ds,
\end{equation*}
\begin{equation*}
\tilde{\mathcal B}_3[p_1] = -\frac{2}{3}\int_{-T}^{t-\lambda_*^2(t)} \frac{ {\rm Re}[p_0(s) e^{-i\gamma(t)}]}{t-s}ds + \frac{2}{3}\mathbf{Z}.
\end{equation*}
From the definition, we have
\begin{equation*}
\|2\lambda\dot\cc\|_{*,\Theta,k}\leq C\|f\|_{*,\Theta,0}|\log T|^{k-1}.
\end{equation*}
For the other terms, we write
\begin{equation*}
\tilde B_{1,a}[p_{1,0}+p_{1,1}](t) = \int_{-T}^{t-\lambda_*^2(t)}\frac{p_1(s)}{t-s}\left(\Gamma_3\left(\frac{\lambda^2(t)}{t-s}\right)+1\right)ds,
\end{equation*}
\begin{equation*}
\tilde B_{2,a}[p_{1,0}+p_{1,1}](t) = \int_{-T}^{t-\lambda_*^2(t)} \frac{ {\rm Re}[p_0(s) e^{-i\gamma(t)}]}{t-s}\left(\Gamma_5\left(\frac{\lambda^2(t)}{t-s}\right)+\frac{2}{3}\right)ds
\end{equation*}
and
\begin{equation*}
\tilde B_{1,b}[p_{1,0}+p_{1,1}](t) = \int_{t-\lambda_*^2(t)}^t\frac{p_1(s)}{t-s}\Gamma_3\left(\frac{\lambda^2(t)}{t-s}\right)ds,
\end{equation*}\begin{equation*}
\tilde B_{2,b}[p_{1,0}+p_{1,1}](t) = \int_{t-\lambda_*^2(t)}^t \frac{ {\rm Re}[p_0(s) e^{-i\gamma(t)}]}{t-s}\Gamma_5\left(\frac{\lambda^2(t)}{t-s}\right)ds.
\end{equation*}
From Lemma \ref{basicestimate} and the asymptotic estimates for $\Gamma_3$, $\Gamma_5$, we have the following estimates,
\begin{equation*}
\begin{aligned}
|\tilde B_{1,a}[p_{1,0}+p_{1,1}](t)|&\leq C(\lambda_*(t))^{2\sigma}\left|\int_{-T}^{t-(\lambda_*(t))^2}\frac{(p_{1,0}+p_{1,1})(s)}{(t-s)^{1+\sigma}}ds\right|\\
&\leq C\|f\|_{*,\Theta,0}(\lambda_*(t))^{2\sigma}\int_{-T}^{t-(\lambda_*(t))^2}\frac{(\lambda_*(s))^{\Theta}}{(t-s)^{1+\sigma}|\log(T-s)|}ds\\
&\leq C\|f\|_{*,\Theta,0}\frac{(\lambda_*(t))^{\Theta}}{|\log(T-t)|},
\end{aligned}
\end{equation*}
\begin{equation*}
\begin{aligned}
|\tilde B_{2,a}[p_{1,0}+p_{1,1}](t)\cc(t)|&\leq C(\lambda_*(t))^{2\sigma}\left|\int_{-T}^{t-(\lambda_*(t))^2}\frac{p_{0}(s)}{(t-s)^{1+\sigma}}ds\cc(t)\right|\\
&\leq C\|f\|_{*,\Theta,0}(\lambda_*(t))^{2\sigma+\Theta}\int_{-T}^{t-(\lambda_*(t))^2}\frac{1}{(t-s)^{1+\sigma}|\log(T-s)|^2}ds\\
&\leq C\|f\|_{*,\Theta,0}\frac{(\lambda_*(t))^{\Theta}}{|\log(T-t)|^2}
\end{aligned}
\end{equation*}
and
\begin{equation*}
\begin{aligned}
|\tilde B_{1,b}[p_{1,0}+p_{1,1}](t)|&\leq C\frac{1}{(\lambda_*(t))^{2}}\int_{t-(\lambda_*(t))^2}^{t}\left|(p_{1,0}+p_{1,1})(s)\right |ds\leq C\|f\|_{*,\Theta,0}\frac{(\lambda_*(t))^{\Theta}}{|\log(T-t)|},
\end{aligned}
\end{equation*}
\begin{equation*}
\begin{aligned}
|\tilde B_{2,b}[p_{1,0}+p_{1,1}](t)\cc(t)|&\leq C\frac{\|f\|_{*,\Theta,0}}{(\lambda_*(t))^{2}}\int_{t-(\lambda_*(t))^2}^{t}\left|p_{0}(s)\right |ds(\lambda_*(t))^{\Theta}\leq C\|f\|_{*,\Theta,0}\frac{(\lambda_*(t))^{\Theta}}{|\log(T-t)|^2},
\end{aligned}
\end{equation*}
\begin{equation*}
\begin{aligned}
&|\tilde B_{3}[p_{1,0}+p_{1,1}](t)\cc(t)|\\
&\leq C\|f\|_{*,\Theta,0}(\lambda_*(t))^{\Theta}|\dot \lambda_*| + C\|f\|_{*,\Theta,0}(\lambda_*(t))^{\Theta}\left|\int_{t-(T-t)}^{t-\la^2(t)}\frac{p_0(s)-p_0(t)}{t-s}ds\right|\\
&\leq C\|f\|_{*,\Theta,0}\frac{(\lambda_*(t))^{\Theta}}{|\log(T-t)|^2}.
\end{aligned}
\end{equation*}
Combine all the above estimates, we obtain
\begin{equation*}
\|\tilde B[p_{1,0}+p_{1,1}](\cdot)\|_{*,\Theta,k}\leq C\|f\|_{*,\Theta,0}|\log T|^{k-1}.
\end{equation*}
\end{proof}

\begin{lemma}
Let $M = C|\log T|^{k-1}$, then for $\|p_{1,i}\|_{*,\Theta,k}\leq M$, $i = 1, 2$, we have the following estimate,
\begin{equation*}
\|\tilde B[p_{1,0}+p_{1,1}](\cdot)-\tilde B[p_{1,0}+p_{1,2}](\cdot)\|_{*,\Theta,k}\leq \frac{C}{|\log T|}\|p_{1,1}-p_{1,2}\|_{*,\Theta,k+1}.
\end{equation*}
\end{lemma}
\begin{proof}
Let us use the notations in Lemma \ref{estimateforB}. The estimate for $2\lambda\dot\cc$ is from definition directly. For the other terms, we write
\begin{equation*}
D_{1,a}:= \tilde B_{1,a}[p_{1,0}+p_{1,1}](t)-\tilde B_{1,a}[p_{1,0}+p_{1,2}](t),\quad D_{2,a}:= \tilde B_{2,a}[p_{1}](\cc^{(0)} + \cc^{(1,1)})(t)-\tilde B_{2,a}[p_{1}](\cc^{(0)} + \cc^{(1,2)})(t)
\end{equation*}
\begin{equation*}
D_{1,b}:= \tilde B_{1,b}[p_{1,0}+p_{1,1}](t)-\tilde B_{1,b}[p_{1,0}+p_{1,2}](t),\quad D_{2,b}:= \tilde B_{2,b}[p_{1}](\cc^{(0)} + \cc^{(1,1)})(t)-\tilde B_{2,b}[p_{1}](\cc^{(0)} + \cc^{(1,2)})(t),
\end{equation*}
\begin{equation*}
D_{3}:= \tilde B_{3}[p_{1}](\cc^{(0)} + \cc^{(1,1)})(t)-\tilde B_{3}[p_{1}](\cc^{(0)} + \cc^{(1,2)})(t).
\end{equation*}
Let us consider the term
\begin{equation*}
\begin{aligned}
\frac{d}{d\zeta}\tilde B_{1,a}[p_{1,0}+p_{1,\zeta}](t) & = \int_{-T}^{t-(\lambda_*(t))^2}\frac{p_{1,1}(s)-p_{1,2}(s)}{t-s}\left(\Gamma_3\left(\frac{|(p_0+p_1)(t)|^2}{t-s}\right)+1\right)ds
\end{aligned}
\end{equation*}
with $p_{1,\zeta}(t) = \zeta p_{1,1}(t)+(1-\zeta)p_{1,2}(t)$. Using the decaying estimates for $\Gamma_3$ and Lemma \ref{basicestimate}, we have
\begin{equation*}
\begin{aligned}
&\left|\int_{-T}^{t-(\lambda_*(t))^2}\frac{p_{1,1}(s)-p_{1,2}(s)}{t-s}\left(\Gamma_3\left(\frac{|(p_0+p_1)(t)|^2}{t-s}\right)+1\right)ds\right |\\
&\leq C \int_{-T}^{t-(\lambda_*(t))^2}\frac{|p_{1,1}(s)-p_{1,2}(s)|}{t-s}\left(\frac{|(p_0+p_\zeta)(t)|^2}{t-s}\right)^\sigma ds\\
&\leq C\|p_{1,1}-p_{1,2}\|_{*,\Theta,k+1}(\lambda_*(t))^{2\sigma} \int_{-T}^{t-(\lambda_*(t))^2}\frac{(\lambda_*(s))^{\Theta}}{(t-s)^{1+\sigma}|\log(T-s)|^{k+1}} ds\\
&\leq C\|p_{1,1}-p_{1,2}\|_{*,\Theta,k+1}(\lambda_*(t))^{2\sigma} \frac{(\lambda_*(t))^{\Theta}}{(\lambda_*(t))^{2\sigma}|\log(T-t)|^{k+1}}\\
&= C\|p_{1,1}-p_{1,2}\|_{*,\Theta,k+1}\frac{(\lambda_*(t))^{\Theta}}{|\log(T-t)|^{k+1}}.
\end{aligned}
\end{equation*}
Therefore we have
\begin{equation*}
\begin{aligned}
&\left|D_{1,a}(t)\right| = \left|\int_{0}^1\frac{d}{d\zeta}\tilde B_{1,a}[p_{1,0}+p_{1,\zeta}](t)\right|\\
&\leq \left|\int_{-T}^{t-(\lambda_*(t))^2}\frac{p_{1,1}(s)-p_{1,2}(s)}{t-s}\left(\Gamma_3\left(\frac{|(p_0+p_1)(t)|^2}{t-s}\right)-1\right)ds\right |\\
& \leq C\|p_{1,1}-p_{1,2}\|_{*,\Theta,k+1}\frac{(\lambda_*(t))^{\Theta}}{|\log(T-t)|^{k+1}}.
\end{aligned}
\end{equation*}
Similarly, we estimate $D_{1,b}$ as follows.
\begin{equation*}
\begin{aligned}
\frac{d}{d\zeta}\tilde B_{1,b}[p_{1,0}+p_{1,\zeta}](t) & = \int_{t-(\lambda_*(t))^2}^t\frac{p_{1,1}(s)-p_{1,2}(s)}{t-s}\Gamma_3\left(\frac{|(p_0+p_1)(t)|^2}{t-s}\right)ds.
\end{aligned}
\end{equation*}
For this term we have
\begin{equation*}
\begin{aligned}
&\left|D_{1,b}(t)\right| = \left|\int_{0}^1\frac{d}{d\zeta}\tilde B_{1,a}[p_{1,0}+p_{1,\zeta}](t)\right|\\
&\leq \left|\int_{t-(\lambda_*(t))^2}^t\frac{p_{1,1}(s)-p_{1,2}(s)}{t-s}\Gamma_3\left(\frac{|(p_0+p_1)(t)|^2}{t-s}\right)ds\right |\\
&\leq \frac{C}{|(p_0+p_\zeta)(t)|^2}\int_{t-(\lambda_*(t))^2}^t|p_{1,1}(s)-p_{1,2}(s)| ds\\
&\leq C\|p_{1,1}-p_{1,2}\|_{*,\Theta,k+1}\frac{(\lambda_*(t))^{\Theta}}{|\log(T-t)|^{k+1}}.
\end{aligned}
\end{equation*}
Now we estimate $D_{2,a}$. We have
\begin{align*}
&|D_{2,a}| = |\tilde B_{2,a}[p_{1}](\cc^{(0)} + \cc^{(1,1)})(t)-\tilde B_{2,a}[p_{1}](\cc^{(0)} + \cc^{(1,2)})(t)|\\
&\leq C(\lambda_*(t))^{2\sigma}\left|\int_{-T}^{t-(\lambda_*(t))^2}\frac{p_{0}(s)}{(t-s)^{1+\sigma}}ds(\cc^{(1,1)} - \cc^{(1,2)})(t)\right|\\
&\leq C(\lambda_*(t))^{2\sigma-1}\left|\int_{-T}^{t-(\lambda_*(t))^2}\frac{p_{0}(s)}{(t-s)^{1+\sigma}}ds\int_{t}^T(p_{1,1}(s)-p_{1,2}(s))ds\right|\\
&\leq C\|p_{1,1}-p_{1,2}\|_{*,\Theta,k+1}(\lambda_*(t))^{2\sigma-1}\left|\int_{-T}^{t-(\lambda_*(t))^2}\frac{p_{0}(s)}{(t-s)^{1+\sigma}}ds\int_{t}^T\frac{(\lambda_*(s))^\Theta}{|\log(T-s)|^{k+1}}ds\right|\\
&\leq C\|p_{1,1}-p_{1,2}\|_{*,\Theta,k+1}(\lambda_*(t))^{2\sigma-1}\left|\int_{-T}^{t-(\lambda_*(t))^2}\frac{|\log T|}{(t-s)^{1+\sigma}|\log(T-s)|^2}ds\int_{t}^T\frac{(\lambda_*(s))^\Theta}{|\log(T-s)|^{k+1}}ds\right|\\
&\leq C\|p_{1,1}-p_{1,2}\|_{*,\Theta,k+1}(\lambda_*(t))^{2\sigma-1}\frac{|\log T|(\lambda_*(t))^{-2\sigma}}{|\log(T-t)|^2}\frac{(\lambda_*(t))^\Theta}{|\log(T-t)|^{k+1}}(T-t)\\
&\leq C\|p_{1,1}-p_{1,2}\|_{*,\Theta,k+1}\frac{(\lambda_*(t))^{\Theta}}{|\log(T-t)|^{k+1}}.
\end{align*}
For $D_{2,b}$, we have
\begin{align*}
&|D_{2,b}| = |\tilde B_{2,b}[p_{1}](\cc^{(0)} + \cc^{(1,1)})(t)-\tilde B_{2,b}[p_{1}](\cc^{(0)} + \cc^{(1,2)})(t)|\\
&\leq C(\lambda_*(t))^{-2}\left|\int_{t-(\lambda_*(t))^2}^{t}|p_0(s)|ds(\cc^{(1,1)} - \cc^{(1,2)})(t)\right|\\
&\leq C(\lambda_*(t))^{-2}\left|\int_{t-(\lambda_*(t))^2}^{t}|p_0(s)|ds\int_{t}^T(p_{1,1}(s)-p_{1,2}(s))ds\right|\\
&\leq C\|p_{1,1}-p_{1,2}\|_{*,\Theta,k+1}(\lambda_*(t))^{-2}\left|\int_{t-(\lambda_*(t))^2}^{t}|p_{0}(s)|ds\int_{t}^T\frac{(\lambda_*(s))^\Theta}{|\log(T-s)|^{k+1}}ds\right|\\
&\leq C\|p_{1,1}-p_{1,2}\|_{*,\Theta,k+1}(\lambda_*(t))^{-2}\left|\frac{|\log T|(\lambda_*(t))^2}{|\log(T-t)|^2}\int_{t}^T\frac{(\lambda_*(s))^\Theta}{|\log(T-s)|^{k+1}}ds\right|\\
&\leq C\|p_{1,1}-p_{1,2}\|_{*,\Theta,k+1}(\lambda_*(t))^{-2}\frac{|\log T|(\lambda_*(t))^{2}}{|\log(T-t)|^2}\frac{(\lambda_*(t))^\Theta}{|\log(T-t)|^{k+1}}(T-t)\\
&\leq C\|p_{1,1}-p_{1,2}\|_{*,\Theta,k+1}\frac{(\lambda_*(t))^{\Theta+1}}{|\log(T-t)|^{k+1}}.
\end{align*}
For $D_{3}$, we have
\begin{align*}
&|D_{3}| = |\tilde B_{3}[p_{1}](\cc^{(0)} + \cc^{(1,1)})(t)-\tilde B_{3}[p_{1}](\cc^{(0)} + \cc^{(1,2)})(t)|\\
&\leq C\left|\left(|\dot \lambda_*| + \left|\int_{t-(T-t)}^{t-\la^2(t)}\frac{p_0(s)-p_0(t)}{t-s}ds\right|\right)(\cc^{(1,1)} - \cc^{(1,2)})(t)\right|\\
&\leq C\frac{|\log T|}{|\log(T-t)|^2}\left|\int_{t}^T(p_{1,1}(s)-p_{1,2}(s))ds\right|\\
&\leq C\|p_{1,1}-p_{1,2}\|_{*,\Theta,k+1}\left|\frac{|\log T|}{|\log(T-t)|^2}\int_{t}^T\frac{(\lambda_*(s))^\Theta}{|\log(T-s)|^{k+1}}ds\right|\\
&\leq C\|p_{1,1}-p_{1,2}\|_{*,\Theta,k+1}\frac{|\log T|}{|\log(T-t)|^2}\frac{(\lambda_*(t))^\Theta}{|\log(T-t)|^{k+1}}(T-t)\\
&\leq C\|p_{1,1}-p_{1,2}\|_{*,\Theta,k+1}\frac{(\lambda_*(t))^{\Theta+1}}{|\log(T-t)|^{k+1}}.
\end{align*}
Combine all the above estimates, we obtain
\begin{equation*}
\|\tilde B[p_{1,0}+p_{1,1}](\cdot)-\tilde B[p_{1,0}+p_{1,2}](\cdot)\|_{*,\Theta,k}\leq \frac{C}{|\log T|}\|p_{1,1}-p_{1,2}\|_{*,\Theta,k+1}.
\end{equation*}
\end{proof}
\noindent {\bf Proof of Proposition \ref{nonlinear-system-for-linear-problem}}: Now we prove Proposition \ref{nonlinear-system-for-linear-problem} based on the above estimates and the fixed point arguments. Indeed, we choose $k\in (0,1)$. Then from Lemma \ref{lemma7.2} we have
\begin{equation*}
\|\mathcal{A}[p_{1,1}]\|_{*,\Theta,k+1}\leq C\left(\|\eta \tilde E\|_{*,\Theta,k} + \|\tilde B[p_{1,0}+p_{1,1}]\|_{*,\Theta,k}\right)
\end{equation*}
and
\begin{equation*}
\|\eta \tilde E\|_{*,\Theta,k}\leq C_E|\log T|^{k-1}
\end{equation*}
\begin{equation*}
\|\tilde B[p_{1,0}+p_{1,1}]\|_{*,\Theta,k}\leq C|\log T|^{k-1}.
\end{equation*}
Therefore we have
\begin{equation*}
\|\mathcal{A}[p_{1,1}]\|_{*,\Theta,k+1}\leq C\cdot C_E|\log T|^{k-1} + C|\log T|^{k-1}\leq C_0|\log T|^{k-1}
\end{equation*}
for fixing $C_0$ large. Hence the operator maps $\overline{B}_M(0)$ into itself for $M = C_0|\log T|^{k-1}$.
Moreover, we have
\begin{equation*}
\|\mathcal{A}[p_{1,1}]-\mathcal{A}[p_{1,2}]\|_{*,\Theta,k+1}\leq C \|\tilde B[p_{1,0}+p_{1,1}] - \tilde B[p_{1,0}+p_{1,2}]\|_{*,\Theta,k}\leq \frac{C}{|\log T|}\|p_{1,1}-p_{1,2}\|_{*,\Theta,k+1}.
\end{equation*}
Then we obtain a fixed point of problem (\ref{fixedproblemforp1}) by the contraction mapping theorem. \qed

\subsection{Estimates of the derivatives}
\begin{lemma}\label{estimateforsecondderivative}
We have the following estimates for $p_{1,1}$,
\begin{equation}\label{estimateforsecondderivative1}
\left|\dot p_{1,1}(t)\right|\leq C\|f\|_{*, \Theta, 0}\frac{(\lambda_*(t))^{\Theta}}{|\log(T-t)|^2(T-t)},
\end{equation}
\begin{equation}\label{estimateforsecondderivative2}
\left|\ddot p_{1,1}(t)\right|\leq C\|f\|_{*, \Theta, 0}\frac{(\lambda_*(t))^{\Theta}}{|\log(T-t)|^2(T-t)^2}.
\end{equation}
\end{lemma}
\begin{proof}
To prove (\ref{estimateforsecondderivative1}), let us recall that $p_{1,1}$ satisfies the following equation
\begin{equation*}
\tilde L_0[p_{1,1}] + \eta(\frac{t}{T})\tilde L_1[p_{1,1}] + \eta(\frac{t}{T}) E + \tilde{B}[p_{1,0}+p_{1,1}](t) - \frac{2\mathbf Z}{3} \cc^{(1)} = 0\text{ in }[-T, T].
\end{equation*}
Differentiate this equation with respect to $t$, we have
\begin{equation*}
\begin{aligned}
(1-\alpha_1)|\log(T-t)|\dot p_{1,1} + (1-\alpha_1)\frac{p_{1,1}}{T-t} & + \eta(\frac{t}{T})\frac{d}{dt}\tilde L_1[p_{1,1}] + \frac{1}{T}\eta'\left(\frac{t}{T}\right)\tilde L_1[p_{1,1}] + \eta\left(\frac{t}{T}\right) \frac{d}{dt} E \\
&+ \frac{1}{T}\eta'\left(\frac{t}{T}\right) E + \frac{d}{dt}\tilde{B}[p_{1,0}+p_{1,1}](t) - \frac{d}{dt}\left(\frac{2\mathbf Z}{3} \cc^{(1)}\right) = 0.
\end{aligned}
\end{equation*}
Rewrite the above equation as
\begin{equation*}
(1-\alpha_1)|\log(T-t)|\dot p_{1,1} + \eta(\frac{t}{T})\tilde L_1[\dot p_{1,1}] + \mathcal{U}[\dot p_{1,1}](t)= h(t)
\end{equation*}
where $h$ is a function satisfying
\begin{equation*}
|h(t)|\leq C\|f\|_{*, \Theta, 0}\frac{(\lambda_*(t))^{\Theta}}{|\log(T-t)|(T-t)}
\end{equation*}
and $\mathcal{U}$ is the operator defined by
\begin{equation*}
\mathcal{U}[\dot p_{1,1}] = -\int_{-T}^{t-(\lambda_*(t))^2}\frac{\dot p_{1,1}(s)}{t-s}\left(\Gamma_3\left(\frac{(\lambda(t))^2}{t-s}\right)+1\right)ds - \int_{t-(\lambda_*(t))^2}^t\frac{\dot p_{1,1}(s)}{t-s}\Gamma_3\left(\frac{(\lambda(t))^2}{t-s}\right)ds.
\end{equation*}
Observe that One of the terms in $h$ is $\eta(\frac{t}{T})\frac{d}{dt} E$, from the proof of Lemma \ref{estimateoftheerror}, we have
\begin{equation*}
\left|\eta(\frac{t}{T})\frac{d}{dt} E\right|\leq C\|f\|_{*, \Theta, 0}\frac{(\lambda_*(t))^{\Theta}}{|\log(T-t)|(T-t)}.
\end{equation*}
Let us recall the proof of Lemma \ref{lemma7.2}, we estimate the term $\frac{d}{dt}\tilde L_{10}[p_{1,1}]$ as follows
\begin{equation*}
\begin{aligned}
\left|\frac{d}{dt}\tilde L_{10}[p_{1,0}]\right| = \left|\frac{d}{dt}\int_{-T}^t\frac{p_{1,0}(s)}{T-s}ds\right|=\left|\frac{p_{1,0}(t)}{T-t}\right|\leq C\|f\|_{*, \Theta, 0}\frac{(\lambda_*(t))^{\Theta}}{|\log(T-t)|(T-t)}.
\end{aligned}
\end{equation*}
Next we compute $\frac{d}{dt}\tilde L_{11}[p_{1,1}]$,
\begin{align*}
\frac{d}{dt}\tilde L_{11}[p_{1,1}] &= \frac{d}{dt}\int_{(T-t)^{1+\alpha_1}}^{T-t}\frac{p_{1,1}(t-r)}{r}dr\\
& = \int_{(T-t)^{1+\alpha_1}}^{T-t}\frac{\dot p_{1,1}(t-r)}{r}dr -\frac{p_{1,1}(t-(T-t))}{T-t}\\
&\quad + (1+\alpha_1)\frac{p_{1,1}(t-(T-t)^{1+\alpha_0})}{T-t}\\
& = \tilde L_{11}[\dot p_{1,1}] -\frac{p_{1,1}(t-(T-t))}{T-t} + (1+\alpha_1)\frac{p_{1,1}(t-(T-t)^{1+\alpha_0})}{T-t}.
\end{align*}
The last two terms can be estimated as follows
\begin{equation*}
\left|\frac{p_{1,1}(t-(T-t))}{T-t}\right|\leq C\|f\|_{*, \Theta, 0}\frac{|\log T|^{k-1}(\lambda_*(t))^{\Theta}}{|\log(T-t)|^{k+1}(T-t)}\leq C\|f\|_{*, \Theta, 0}\frac{(\lambda_*(t))^{\Theta}}{|\log(T-t)|(T-t)},
\end{equation*}
\begin{equation*}
\left|(1+\alpha_1)\frac{p_{1,1}(t-(T-t)^{1+\alpha_0})}{T-t}\right|\leq C\|f\|_{*, \Theta, 0}\frac{|\log T|^{k-1}(\lambda_*(t))^{\Theta}}{|\log(T-t)|^{k+1}(T-t)}\leq C\|f\|_{*, \Theta, 0}\frac{(\lambda_*(t))^{\Theta}}{|\log(T-t)|(T-t)}.
\end{equation*}
Similarly, the term $\frac{d}{dt}\tilde L_{12}[p_{1,1}]$ can be estimated as
\begin{equation*}
\begin{aligned}
\frac{d}{dt}\tilde L_{12}[p_{1,1}] &= \frac{d}{dt}\int_{0}^{T-t}\frac{p_{1,1}(t-r)}{T-t+r}dr\\
& = \int_{0}^{T-t}\frac{\dot p_{1,1}(t-r)}{T-t+r}dr -\frac{p_{1,1}(t-(T-t))}{2(T-t)}\\
&\quad + \int_{0}^{T-t}\frac{p_{1,1}(t-r)}{(T-t+r)^2}dr\\
& = \tilde L_{12}[\dot p_{1,1}] -\frac{p_{1,1}(t-(T-t))}{2(T-t)} + \int_{t-(T-t)}^{t}\frac{p_{1,1}(s)}{(T-s)^2}ds
\end{aligned}
\end{equation*}
and
\begin{equation*}
\left|\frac{p_{1,1}(t-(T-t))}{2(T-t)}\right|\leq C\|f\|_{*, \Theta, 0}\frac{|\log T|^{k-1}(\lambda_*(t))^{\Theta}}{|\log(T-t)|^{k+1}(T-t)}\leq C\|f\|_{*, \Theta, 0}\frac{(\lambda_*(t))^{\Theta}}{|\log(T-t)|(T-t)},
\end{equation*}
\begin{equation*}
\left|\int_{t-(T-t)}^{t}\frac{p_{1,1}(s)}{(T-s)^2}ds\right|\leq C\|f\|_{*, \Theta, 0}\frac{|\log T|^{k-1}(\lambda_*(t))^{\Theta}}{|\log(T-t)|^{k+1}(T-t)}\leq C\|f\|_{*, \Theta, 0}\frac{(\lambda_*(t))^{\Theta}}{|\log(T-t)|(T-t)}.
\end{equation*}
The term $\frac{d}{dt}\tilde L_{13}[p_{1,1}]$ can be estimated as
\begin{align*}
\frac{d}{dt}\tilde L_{13}[p_{1,1}] &= \frac{d}{dt}\int_{T-t}^{T+t}\frac{p_{1,1}(t-r)(T-t)}{r(T-t+r)}dr\\
& = \int_{T-t}^{T+t}\frac{\dot p_{1,1}(t-r)(T-t)}{r(T-t+r)}dr + \frac{p_{1,1}(t-(T-t))}{2(T-t)} + \frac{p_{1,1}(t-(T+t))(T-t)}{2T(T+t)}\\
&\quad - \int_{T-t}^{T+t}\frac{p_{1,1}(t-r)}{r(T-t+r)}dr + \int_{T-t}^{T+t}\frac{p_{1,1}(t-r)(T-t)}{r(T-t+r)^2}dr\\
& = \tilde L_{13}[\dot p_{1,1}] + \frac{p_{1,1}(t-(T-t))}{2(T-t)} + \frac{p_{1,1}(-T)(T-t)}{2T(T+t)}\\
&\quad - \int_{-T}^{t-(T-t)}\frac{p_{1,1}(s)}{(t-s)(T-s)}ds + \int_{-T}^{t-(T-t)}\frac{p_{1,1}(s)(T-t)}{(t-s)(T-s)^2}ds
\end{align*}
and
\begin{equation*}
\left|\frac{p_{1,1}(t-(T-t))}{2(T-t)}\right|\leq C\|f\|_{*, \Theta, 0}\frac{|\log T|^{k-1}(\lambda_*(t))^{\Theta}}{|\log(T-t)|^{k+1}(T-t)}\leq C\|f\|_{*, \Theta, 0}\frac{(\lambda_*(t))^{\Theta}}{|\log(T-t)|(T-t)},
\end{equation*}
\begin{equation*}
\left|\frac{p_{1,1}(-T)(T-t)}{2T(T+t)}\right|\leq C\left|\frac{p_{1,1}(-T)}{2T}\right|\leq C\|f\|_{*, \Theta, 0}\frac{|\log T|^{k-1}(\lambda_*(t))^{\Theta}}{|\log(T-t)|^{k+1}(T-t)}\leq C\|f\|_{*, \Theta, 0}\frac{(\lambda_*(t))^{\Theta}}{|\log(T-t)|(T-t)},
\end{equation*}
\begin{equation*}
\begin{aligned}
\left|\int_{-T}^{t-(T-t)}\frac{p_{1,1}(s)}{(t-s)(T-s)}ds\right|&\leq C\|f\|_{*, \Theta, 0}\int_{-T}^{t-(T-t)}\frac{|\log T|^{k-1}(\lambda_*(s))^{\Theta}}{(T-s)^2|\log(T-s)|^{k+1}}ds\\
&\leq C\|f\|_{*, \Theta, 0}\frac{|\log T|^{k-1}(\lambda_*(t))^{\Theta}}{|\log(T-t)|^{k+1}(T-t)}\leq C\|f\|_{*, \Theta, 0}\frac{(\lambda_*(t))^{\Theta}}{|\log(T-t)|(T-t)},
\end{aligned}
\end{equation*}
\begin{equation*}
\begin{aligned}
\left|\int_{-T}^{t-(T-t)}\frac{p_{1,1}(s)(T-t)}{(t-s)(T-s)^2}ds\right|&\leq C\|f\|_{*, \Theta, 0}\int_{-T}^{t-(T-t)}\frac{|\log T|^{k-1}(\lambda_*(s))^{\Theta}(T-t)}{(T-s)^3|\log(T-s)|^{k+1}}ds\\
&\leq C\|f\|_{*, \Theta, 0}\frac{|\log T|^{k-1}(\lambda_*(t))^{\Theta}}{|\log(T-t)|^{k+1}(T-t)}\leq C\|f\|_{*, \Theta, 0}\frac{(\lambda_*(t))^{\Theta}}{|\log(T-t)|(T-t)}.
\end{aligned}
\end{equation*}
The term $\frac{d}{dt}\tilde L_{14}[p_{1,1}]$ can be estimated as
\begin{equation*}
\begin{aligned}
\frac{d}{dt}\tilde L_{14}[p_{1,1}] &= \frac{d}{dt}\left((4\log(|\log(T-t)|)-2\log(|\log T|))p_{1,1}(t)\right)\\
& = \tilde L_{14}[\dot p_{1,1}] + \frac{4}{(T-t)|\log(T-t)|}p_{1,1}
\end{aligned}
\end{equation*}
and
\begin{equation*}
\left|\frac{4}{(T-t)|\log(T-t)|}p_{1,1}\right|\leq C\|f\|_{*, \Theta, 0}\frac{|\log T|^{k-1}(\lambda_*(t))^{\Theta}}{|\log(T-t)|^{k+2}(T-t)}\leq C\|f\|_{*, \Theta, 0}\frac{(\lambda_*(t))^{\Theta}}{|\log(T-t)|^{2}(T-t)}.
\end{equation*}
For the term $\frac{1}{T}\eta'\left(\frac{t}{T}\right)\tilde L_1[p_{1,1}]$, we have\begin{equation*}
\begin{aligned}
\left|\frac{1}{T}\eta'\left(\frac{t}{T}\right)\tilde L_1[p_{1,1}]\right| & \leq C\|f\|_{*, \Theta, 0}\frac{|\log T|^{k-1}(\lambda_*(t))^{\Theta}}{T|\log(T-t)|^{k}}\leq C\|f\|_{*, \Theta, 0}\frac{|\log T|^{k-1}(\lambda_*(t))^{\Theta}}{|\log(T-t)|^{k}(T-t)}\\
& \leq C\|f\|_{*, \Theta, 0}\frac{(\lambda_*(t))^{\Theta}}{|\log(T-t)|(T-t)}
\end{aligned}
\end{equation*}
since $|\log(T-t)|\sim |\log T|$ in the interval $t\in [-\frac{T}{4},0]$ where $\eta'\left(\frac{t}{T}\right)$ is not zero.
Now we consider the term $\frac{d}{dt}\tilde{B}[p_{1,0}+p_{1,1}](t)$. Observe that
one of these terms are $$\frac{d}{dt}\int_{-T}^{t-\lambda_*^2(t)}\frac{p_1(s)}{t-s}\left(\Gamma_3\left(\frac{\lambda^2(t)}{t-s}\right)+1\right)ds,$$ we have
\begin{align*}
&\frac{d}{dt}\int_{-T}^{t-\lambda_*^2(t)}\frac{p_1(s)}{t-s}\left(\Gamma_3\left(\frac{\lambda^2(t)}{t-s}\right)+1\right)ds = \frac{d}{dt}\int_{\lambda_*^2(t)}^{t+T}\frac{p_1(t-r)}{r}\left(\Gamma_3\left(\frac{\lambda^2(t)}{r}\right)+1\right)dr\\
& = \int_{\lambda_*^2(t)}^{t+T}\frac{\dot p_1(t-r)}{r}\left(\Gamma_3\left(\frac{\lambda^2(t)}{r}\right)+1\right)dr + \int_{\lambda_*^2(t)}^{t+T}\frac{ p_1(t-r)}{r}\left(\Gamma_3'\left(\frac{\lambda^2(t)}{r}\right)\right)\frac{2\lambda(t)\lambda'(t)}{r}dr\\
&\quad + \frac{p_1(-T)}{t+T}\left(\Gamma_3\left(\frac{\lambda^2(t)}{t+T}\right)+1\right) + \frac{p_1(t-\lambda_*^2(t))}{\lambda_*^2(t)}\left(\Gamma_3\left(\frac{\lambda^2(t)}{\lambda_*^2(t)}\right)+1\right)2\lambda_*(t)\lambda_*'(t)
\end{align*}
and from Lemma \ref{basicestimate}, it holds that
\begin{align*}
&\left|\int_{\lambda_*^2(t)}^{t+T}\frac{ p_1(t-r)}{r}\left(\Gamma_3'\left(\frac{\lambda^2(t)}{r}\right)\right)\frac{2\lambda(t)\lambda'(t)}{r}dr\right| = \left|\int_{-T}^{t-\lambda_*^2(t)}\frac{ p_1(s)}{t-s}\left(\Gamma_3'\left(\frac{\lambda^2(t)}{t-s}\right)\right)\frac{2\lambda(t)\lambda'(t)}{t-s}ds\right|\\
&\leq C|\lambda(t)\lambda'(t)|\int_{-T}^{t-\lambda_*^2(t)}\frac{|p_1(s)|}{(t-s)^2}\left(\frac{t-s}{\lambda^2(t)}\right)^\sigma ds\\
&\leq C\lambda(t)^{1-2\sigma}|\lambda'(t)|\|f\|_{*,\Theta,0}\int_{-T}^{t-\lambda_*^2(t)}\frac{1}{(t-s)^{2-\sigma}}\frac{(\lambda_*(s))^{\Theta}}{|\log(T-s)|}ds\\
&\leq C\lambda(t)^{1-2\sigma}|\lambda'(t)|\|f\|_{*,\Theta,0}\frac{(\lambda_*(t))^{\Theta + 2(\sigma-1)}}{|\log(T-t)|}\leq C\|f\|_{*,\Theta,0}\frac{(\lambda_*(t))^{\Theta}}{|\log(T-t)|(T-t)},
\end{align*}
\begin{equation*}
\begin{aligned}
&\left|\frac{p_1(-T)}{t+T}\left(\Gamma_3\left(\frac{\lambda^2(t)}{t+T}\right)+1\right)\right| \leq \left|\frac{|p_1(-T)}{T}\right|\leq C\|f\|_{*,\Theta,0}\frac{(\lambda_*(2T))^{\Theta}}{|\log(2T)|T}\leq C\|f\|_{*,\Theta,0}\frac{(\lambda_*(t))^{\Theta}}{|\log(T-t)|(T-t)},
\end{aligned}
\end{equation*}
\begin{equation*}
\begin{aligned}
&\left|\frac{p_1(t-\lambda_*^2(t))}{\lambda_*^2(t)}\left(\Gamma_3\left(\frac{\lambda^2(t)}{\lambda_*^2(t)}\right)+1\right)2\lambda_*(t)\lambda_*'(t)\right|\\
&\leq C\|f\|_{*,\Theta,0}\frac{(\lambda_*(t))^{\Theta}}{|\log(T-t)|}\frac{|\lambda_*'(t)|}{\lambda_*(t)}\leq C\|f\|_{*,\Theta,0}\frac{(\lambda_*(t))^{\Theta}}{|\log(T-t)|(T-t)}.
\end{aligned}
\end{equation*}
The other terms in $h$ as well as the proof of (\ref{estimateforsecondderivative2}) can be estimated similarly, we omit the details here.
\end{proof}

\subsection{Estimate of the remainder $R_{\alpha_1}[p_{1,1}]$}
\begin{lemma}\label{estimateoftheremainder}
Let $p_{1,1}$ be the solution constructed in Proposition \ref{nonlinear-system-for-linear-problem}, then we have
\begin{equation*}
|R_{\alpha_1}[p_{1,1}]|\leq C\|f\|_{*, \Theta, 0}\frac{(\lambda_*(t))^{\Theta}(T-t)^{\alpha_1}}{|\log(T-t)|^2}.
\end{equation*}
Furthermore, it holds that
\begin{equation*}
\left|\frac{d}{dt}R_{\alpha_1}[p_{1,1}]\right|\leq C\|f\|_{*, \Theta, 0}\frac{(\lambda_*(t))^{\Theta}(T-t)^{\alpha_1-1}}{|\log(T-t)|^2}.
\end{equation*}
\end{lemma}
\begin{proof}
From the definition $\mathcal{R}_{\alpha_1}[p_{1, 1}] = \mathcal{I}[p_{1,1}] - S_{\alpha_1}[p_{1, 1}]$ and the estimates in Lemma \ref{estimateforsecondderivative}, we have
\begin{equation*}
\begin{aligned}
|R_{\alpha_1}[p_{1,1}]|& \leq \int_{t-(T-t)^{1+\alpha_1}}^{t-(\lambda_*(t))^2}\frac{|p_{1,1}(t) - p_{1,1}(s)|}{t-s}ds\\
& \leq \sup_{r\in (t-(T-t)^{1+\alpha_1}, t-(\lambda_*(t))^2) } |\dot p_{1,1}(r)|(T-t)^{1+\alpha_1}\\
& \leq C\|f\|_{*, \Theta, 0}\frac{(\lambda_*(t))^{\Theta}(T-t)^{\alpha_1}}{|\log(T-t)|^2}.
\end{aligned}
\end{equation*}
On the other hand, we have
\begin{equation*}
\begin{aligned}
\frac{d}{dt}R_{\alpha_1}[p_{1,1}]& = -\frac{d}{dt}\int_{t-(T-t)^{1+\alpha_1}}^{t-(\lambda_*(t))^2}\frac{p_{1,1}(t) - p_{1,1}(s)}{t-s}ds\\
& = -\frac{d}{dt}\int_{(\lambda_*(t))^2}^{(T-t)^{1+\alpha_1}}\frac{p_{1,1}(t) - p_{1,1}(t-r)}{r}dr\\
& = -\int_{(\lambda_*(t))^2}^{(T-t)^{1+\alpha_1}}\frac{\dot p_{1,1}(t) - \dot p_{1,1}(t-r)}{r}dr -2(1+\alpha_1)\frac{p_{1,1}(t-(\lambda_*(t))^2)}{\lambda_*(t)}\dot \lambda_*(t)\\
&\quad + (1+\alpha_1)\frac{p_{1,1}(t-(T-t)^{1+\alpha_0})}{T-t}.
\end{aligned}
\end{equation*}
From the estimates in Lemma \ref{estimateforsecondderivative}, we obtain
\begin{equation*}
\left|\frac{d}{dt}R_{\alpha_1}[p_{1,1}]\right|\leq C\|f\|_{*, \Theta, 0}\frac{(\lambda_*(t))^{\Theta}(T-t)^{\alpha_1-1}}{|\log(T-t)|^2}.
\end{equation*}
\end{proof}
Combine the results in Proposition \ref{nonlinear-system-for-linear-problem}, Lemma \ref{estimateforsecondderivative} and Lemma \ref{estimateoftheremainder}, define $\mathcal{R}_{\cc}[f](t):=R_{\alpha_1}[p_{1,1}](t)$, then we obtain the conclusion of Proposition \ref{prop-c1c2}.

\bigskip

\appendix

\section{Proof of technical lemmas}\label{appendixA}

Several technical lemmas will be proved in this section.

\subsection{Decomposition of the linearization}\label{appendix-linearization}

\medskip

\begin{proof}[Proof of Lemma \ref{linearization-lemma-0}]
We compute
\begin{equation*}
\begin{aligned}
\Delta(\phi_3 W )=&~\left[\partial_{\rho\rho}\phi_3+\frac1\rho \partial_{\rho}\phi_3+\frac{1}{\rho^2}\partial_{\theta\theta}\phi_3-\left(w_{\rho}^2+\frac{\sin^2 w}{\rho^2}\right)\phi_3\right]W \\
&~+\left[\left(w_{\rho\rho}+\frac{w_{\rho}}{\rho}-\frac{\sin w\cos w}{\rho^2}\right)\phi_3+2w_{\rho}\partial_{\rho}\phi_3\right]E_1+2\frac{\sin w \partial_{\theta}\phi_3}{\rho^2} E_2,
\end{aligned}
\end{equation*}
\begin{equation*}
\begin{aligned}
\Delta(\phi_1 E_1)=&~-\left[\left(w_{\rho\rho}+\frac{w_{\rho}}{\rho}+\frac{\sin w\cos w}{\rho^2}\right)\phi_1+2w_{\rho}\partial_{\rho}\phi_1\right]W \\
&~+\left[\partial_{\rho\rho}\phi_1+\frac1\rho \partial_{\rho}\phi_1+\frac{1}{\rho^2}\partial_{\theta\theta}\phi_1-\left(w_{\rho}^2+\frac{\cos^2 w}{\rho^2}\right)\phi_1\right]E_1+2\frac{\cos w \partial_{\theta}\phi_1}{\rho^2} E_2,
\end{aligned}
\end{equation*}
and
\begin{equation*}
\begin{aligned}
\Delta(\phi_2 E_2)=&~-\frac{2\sin w \partial_{\theta}\phi_2}{\rho^2}W -\frac{2\cos w\partial_{\theta}\phi_2}{\rho^2}E_1\\
&~+\left[\partial_{\rho\rho}\phi_2+\frac1\rho \partial_{\rho}\phi_2+\frac{1}{\rho^2}\partial_{\theta\theta}\phi_2-\frac{\phi_2}{\rho^2}\right]E_2.\\
\end{aligned}
\end{equation*}
So we get
\begin{align*}
&~~\Delta\phi\\
=&~\left[\partial_{\rho\rho}\phi_3+\frac1\rho \partial_{\rho}\phi_3+\frac{1}{\rho^2}\partial_{\theta\theta}\phi_3-\left(w_{\rho}^2+\frac{\sin^2 w}{\rho^2}\right)\phi_3-\left(w_{\rho\rho}+\frac{w_{\rho}}{\rho}+\frac{\sin w\cos w}{\rho^2}\right)\phi_1-2w_{\rho}\partial_{\rho}\phi_1-\frac{2\sin w \partial_{\theta}\phi_2}{\rho^2}\right]W \\
&~+\left[\left(w_{\rho\rho}+\frac{w_{\rho}}{\rho}-\frac{\sin w\cos w}{\rho^2}\right)\phi_3+2w_{\rho}\partial_{\rho}\phi_3+\partial_{\rho\rho}\phi_1+\frac1\rho \partial_{\rho}\phi_1+\frac{1}{\rho^2}\partial_{\theta\theta}\phi_1-\left(w_{\rho}^2+\frac{\cos^2 w}{\rho^2}\right)\phi_1-\frac{2\cos w\partial_{\theta}\phi_2}{\rho^2}\right]E_1\\
&~+\left[\partial_{\rho\rho}\phi_2+\frac1\rho \partial_{\rho}\phi_2+\frac{1}{\rho^2}\partial_{\theta\theta}\phi_2-\frac{\phi_2}{\rho^2}+2\frac{\sin w \partial_{\theta}\phi_3}{\rho^2}+2\frac{\cos w \partial_{\theta}\phi_1}{\rho^2} \right]E_2.
\end{align*}
Also, one has
\begin{align*}
\phi_{y_1}=&~\left(\cos\theta \pp_{\rho}-\frac{\sin\theta}{\rho}\pp_{\theta}\right)(\phi_3 W +\phi_1 E_1+\phi_2 E_2)\\
=&~\left[\cos\theta \pp_{\rho}\phi_3-\frac{\sin\theta}{\rho}\pp_{\theta}\phi_3-\cos\theta w_{\rho}\phi_1+\frac{\sin\theta\sin w}{\rho}\phi_2\right]W \\
&~+\left[\cos\theta \pp_{\rho}\phi_1-\frac{\sin\theta}{\rho}\pp_{\theta}\phi_1+\cos\theta w_{\rho}\phi_3+\frac{\sin\theta\cos w}{\rho}\phi_2\right]E_1\\
&~+\left[\cos\theta \pp_{\rho}\phi_2-\frac{\sin\theta}{\rho}\pp_{\theta}\phi_2-\frac{\sin\theta\sin w}{\rho}\phi_3-\frac{\sin\theta \cos w}{\rho}\phi_1\right]E_2\\
\end{align*}
\begin{align*}
\phi_{y_2}=&~\left(\sin\theta \pp_{\rho}+\frac{\cos\theta}{\rho}\pp_{\theta}\right)(\phi_3 W +\phi_1 E_1+\phi_2 E_2)\\
=&~\left[\sin\theta \pp_{\rho}\phi_3+\frac{\cos\theta}{\rho}\pp_{\theta}\phi_3-\sin\theta w_{\rho}\phi_1-\frac{\cos\theta \sin w}{\rho}\phi_2\right]W \\
&~+\left[\sin\theta \pp_{\rho}\phi_1+\frac{\cos\theta}{\rho}\pp_{\theta}\phi_1+\sin\theta w_{\rho} \phi_3-\frac{\cos \theta \cos w}{\rho}\phi_2\right]E_1\\
&~+\left[\sin\theta \pp_{\rho}\phi_2+\frac{\cos\theta}{\rho}\pp_{\theta}\phi_2+\frac{\cos\theta\sin w}{\rho}\phi_3+\frac{\cos\theta\cos w}{\rho}\phi_1\right]E_2\\
\end{align*}
and thus
\begin{align*}
U_{y_1}\wedge \phi_{y_2}=&~\cos\theta w_{\rho}\left[\sin\theta \pp_{\rho}\phi_2+\frac{\cos\theta}{\rho}\pp_{\theta}\phi_2+\frac{\cos\theta\cos w}{\rho}\phi_1\right]W \\
&~+\frac{\sin\theta\sin w}{\rho}\left[\sin\theta \pp_{\rho}\phi_1+\frac{\cos\theta}{\rho}\pp_{\theta}\phi_1-\frac{\cos \theta \cos w}{\rho}\phi_2\right]W +\frac{w_{\rho}\sin w}{\rho}\phi_3 W \\
&~-\frac{\sin\theta\sin w}{\rho}\left[\sin\theta \pp_{\rho}\phi_3+\frac{\cos\theta}{\rho}\pp_{\theta}\phi_3-\sin\theta w_{\rho}\phi_1-\frac{\cos\theta \sin w}{\rho}\phi_2\right]E_1\\
&~-\cos\theta w_{\rho}\left[\sin\theta \pp_{\rho}\phi_3+\frac{\cos\theta}{\rho}\pp_{\theta}\phi_3-\sin\theta w_{\rho}\phi_1-\frac{\cos\theta \sin w}{\rho}\phi_2\right]E_2,\\
\end{align*}
\begin{align*}
\phi_{y_1}\wedge U_{y_2}
=&~\frac{\cos\theta\sin w}{\rho}\left[\cos\theta \pp_{\rho}\phi_1-\frac{\sin\theta}{\rho}\pp_{\theta}\phi_1+\frac{\sin\theta\cos w}{\rho}\phi_2\right] W \\
&~-\sin\theta w_{\rho}\left[\cos\theta \pp_{\rho}\phi_2-\frac{\sin\theta}{\rho}\pp_{\theta}\phi_2-\frac{\sin\theta \cos w}{\rho}\phi_1\right] W +\frac{w_{\rho}\sin w}{\rho}\phi_3 W \\
&~-\frac{\cos\theta\sin w}{\rho}\left[\cos\theta \pp_{\rho}\phi_3-\frac{\sin\theta}{\rho}\pp_{\theta}\phi_3-\cos\theta w_{\rho}\phi_1+\frac{\sin\theta\sin w}{\rho}\phi_2\right] E_1\\
&~+\sin\theta w_{\rho}\left[\cos\theta \pp_{\rho}\phi_3-\frac{\sin\theta}{\rho}\pp_{\theta}\phi_3-\cos\theta w_{\rho}\phi_1+\frac{\sin\theta\sin w}{\rho}\phi_2\right]E_2.\\
\end{align*}
Therefore, linearization
\begin{equation*}
\Delta_y \phi-2U_{y_1} \wedge \phi_{y_2}-2\phi_{y_1} \wedge U_{y_2}=0
\end{equation*}
implies
\begin{align*}
&~\left[\partial_{\rho\rho}\phi_3+\frac1\rho \partial_{\rho}\phi_3+\frac{1}{\rho^2}\partial_{\theta\theta}\phi_3-\left(w_{\rho}^2+\frac{\sin^2 w}{\rho^2}\right)\phi_3-\left(w_{\rho\rho}+\frac{w_{\rho}}{\rho}+\frac{\sin w\cos w}{\rho^2}\right)\phi_1-2w_{\rho}\partial_{\rho}\phi_1-\frac{2\sin w \partial_{\theta}\phi_2}{\rho^2}\right]W \\
&~+\left[\left(w_{\rho\rho}+\frac{w_{\rho}}{\rho}-\frac{\sin w\cos w}{\rho^2}\right)\phi_3+2w_{\rho}\partial_{\rho}\phi_3+\partial_{\rho\rho}\phi_1+\frac1\rho \partial_{\rho}\phi_1+\frac{1}{\rho^2}\partial_{\theta\theta}\phi_1-\left(w_{\rho}^2+\frac{\cos^2 w}{\rho^2}\right)\phi_1-\frac{2\cos w\partial_{\theta}\phi_2}{\rho^2}\right]E_1\\
&~+\left[\partial_{\rho\rho}\phi_2+\frac1\rho \partial_{\rho}\phi_2+\frac{1}{\rho^2}\partial_{\theta\theta}\phi_2-\frac{\phi_2}{\rho^2}+2\frac{\sin w \partial_{\theta}\phi_3}{\rho^2}+2\frac{\cos w \partial_{\theta}\phi_1}{\rho^2} \right]E_2\\
=&~2\left[\frac{w_{\rho}}{\rho}\pp_{\theta}\phi_2+\frac{w_{\rho}\cos w}{\rho}\phi_1+\frac{\sin w}{\rho}\pp_{\rho}\phi_1+2\frac{w_{\rho}\sin w}{\rho}\phi_3\right]W \\
&~+2\left[-\frac{\sin w}{\rho}\pp_{\rho}\phi_3+\frac{\sin w}{\rho} w_{\rho}\phi_1\right] E_1+2\left[-\frac{w_{\rho}}{\rho}\pp_{\theta}\phi_3+\frac{w_{\rho}\sin w}{\rho} \phi_2\right] E_2,
\end{align*}
and the proof is complete by collecting terms in the same direction.

\medskip
\end{proof}

\medskip

\begin{proof}[Proof of Lemma \ref{linearization-lemma-1}]
First, we notice the fact
\begin{equation*}
U_r = \frac{1}{\lambda}w_{\rho}(\rho)Q_\gamma  E_1, \quad \frac{1}{r}\partial_\theta U = -\frac{1}{\lambda}w_{\rho}(\rho) Q_\gamma  E_2.
\end{equation*}
By direct computation, we have
\begin{align*}
\frac1r U_r\wedge \Phi_\theta &= \frac1r \left(\frac{1}{\lambda}w_{\rho}(\rho)Q_\gamma  E_1\right)\wedge \Phi_\theta\\
& = \frac{1}{\lambda r}w_{\rho}(\rho)(Q_\gamma  E_1\wedge \Phi_\theta)\\
& = \frac{1}{\lambda r}w_{\rho}(\rho)\left[(-\Phi_\theta\cdot (Q_\gamma  W )) Q_\gamma  E_2 + (\Phi_\theta\cdot (Q_\gamma  E_2))Q_\gamma  W \right]
\end{align*}
and
\begin{align*}
\frac1r \Phi_r\wedge U_\theta &= \Phi_r\wedge \left(-\frac{1}{\lambda}w_{\rho}(\rho)Q_\gamma  E_2\right)\\
& = \left(-\frac{1}{\lambda}w_{\rho}(\rho)\right)\Phi_r\wedge(Q_\gamma  E_2)\\
& = \left(-\frac{1}{\lambda}w_{\rho}(\rho)\right)\left[\left(-\Phi_r\cdot (Q_\gamma  W )\right)Q_\gamma  E_1 + (\Phi_r\cdot(Q_\gamma  E_1))Q_\gamma  W \right].
\end{align*}
Therefore, we have
\begin{align*}
\tilde{L}_U[\Phi] =&~ -\frac{2}{\lambda r}w_{\rho}(\rho)\left[(-\Phi_\theta\cdot (Q_\gamma   W ))Q_\gamma   E_2 + (\Phi_\theta\cdot (Q_\gamma  E_2))Q_\gamma  W \right]\\
&~ +\frac{2}{\lambda}w_{\rho}(\rho)\left[\left(-\Phi_r\cdot (Q_\gamma   W )\right) Q_\gamma  E_1 + (\Phi_r\cdot (Q_\gamma   E_1)) Q_\gamma   W \right]\\
 =&~ -\frac{2}{\lambda}w_{\rho}(\rho)\left(\Phi_r\cdot(Q_\gamma   W )\right)Q_\gamma   E_1 + \frac{2}{\lambda r}w_{\rho}(\rho)(\Phi_\theta\cdot (Q_\gamma  W ))Q_\gamma   E_2\\
 &~ + \frac{2}{\lambda r}w_{\rho}(\rho)\left[ r(\Phi_r\cdot(Q_\gamma   E_1))-(\Phi_\theta\cdot (Q_\gamma  E_2))  \right] Q_\gamma  W .
\end{align*}
\end{proof}

\medskip

\begin{proof}[Proof of Lemma \ref{linearization-lemma-2}]
Clearly,
\begin{equation*}
\Phi_r = \cos\theta\partial_{x_1}\Phi + \sin\theta\partial_{x_2}\Phi, \quad \frac{1}{r}\Phi_{\theta} = -\sin\theta\partial_{x_1}\Phi+\cos\theta\partial_{x_2}\Phi.
\end{equation*}
We have $(\Phi\cdot Q_\gamma  E_j)=(Q_{-\gamma } \Phi\cdot E_j)$ and write $\phi=Q_{-\gamma } \Phi=(\phi_1,~\phi_2,~\phi_3)^{T}$ so that
\begin{equation}\label{eqn-252525}
\Phi=\begin{bmatrix}
\varphi_1+i\varphi_2\\
\varphi_3\\
\end{bmatrix}=\begin{bmatrix}
e^{i\gamma }(\phi_1+i\phi_2)\\
\phi_3\\
\end{bmatrix}.
\end{equation}
Then, one has
\begin{equation*}
\begin{aligned}
\phi_r\cdot W  &= \frac{1}{2}\sin w\Big[[\partial_{x_1}\phi_1+\partial_{x_2}\phi_2]+\cos(2\theta)[\partial_{x_1}\phi_1-\partial_{x_2}\phi_2]+\sin(2\theta)[\partial_{x_2}\phi_1+\partial_{x_1}\phi_2]\Big]\\
& \quad + \cos w[\partial_{x_1}\phi_3\cos\theta +\partial_{x_2}\phi_3\sin\theta]
\end{aligned}
\end{equation*}
\begin{equation*}
\begin{aligned}
\frac1r\phi_\theta\cdot W  &= \frac{1}{2}\sin w\Big[[\partial_{x_2}\phi_1-\partial_{x_1}\phi_2]+\cos(2\theta)[\partial_{x_1}\phi_2+\partial_{x_2}\phi_1]+\sin(2\theta)[\partial_{x_2}\phi_2-\partial_{x_1}\phi_1]\Big]\\
& \quad + \cos w[-\partial_{x_1}\phi_3\sin\theta +\partial_{x_2}\phi_3\cos\theta]
\end{aligned}
\end{equation*}
and
\begin{equation*}
\begin{aligned}
\phi_r\cdot E_1 &= \frac{1}{2}\cos w\Big[[\partial_{x_1}\phi_1+\partial_{x_2}\phi_2]+\cos(2\theta)[\partial_{x_1}\phi_1-\partial_{x_2}\phi_2]+\sin(2\theta)[\partial_{x_2}\phi_1+\partial_{x_1}\phi_2]\Big]\\
& \quad -\sin w[\partial_{x_1}\phi_3\cos\theta +\partial_{x_2}\phi_3\sin\theta]
\end{aligned}
\end{equation*}
\begin{equation*}
\begin{aligned}
\frac1r\phi_\theta\cdot E_2 &= [-\sin\theta\partial_{x_1}\phi_1+\cos\theta\partial_{x_2}\phi_1](-\sin\theta)+[-\sin\theta\partial_{x_1}\phi_2+\cos\theta\partial_{x_2}\phi_2](\cos\theta)\\
& = [\partial_{x_1}\phi_1\sin^2\theta+\partial_{x_2}\phi_2\cos^2\theta]-\frac{1}{2}\sin(2\theta)[\partial_{x_2}\phi_1+\partial_{x_1}\phi_2]\\
& = [\partial_{x_1}\phi_1\frac{1-\cos(2\theta)}{2}+\partial_{x_2}\phi_2\frac{\cos(2\theta)+1}{2}]-\frac{1}{2}\sin(2\theta)[\partial_{x_2}\phi_1+\partial_{x_1}\phi_2]\\
& = \frac{1}{2}[\partial_{x_1}\phi_1+\partial_{x_2}\phi_2]-\frac{1}{2}\sin(2\theta)[\partial_{x_2}\phi_1+\partial_{x_1}\phi_2]-\frac{1}{2}\cos(2\theta)[\partial_{x_1}\phi_1-\partial_{x_2}\phi_2]\\
\end{aligned}
\end{equation*}
\begin{equation*}
\begin{aligned}
\phi_r\cdot E_1 - \frac1r\phi_\theta\cdot E_2 &= \frac{1}{2}\cos w\Big[[\partial_{x_1}\phi_1+\partial_{x_2}\phi_2]+\cos(2\theta)[\partial_{x_1}\phi_1-\partial_{x_2}\phi_2]+\sin(2\theta)[\partial_{x_2}\phi_1+\partial_{x_1}\phi_2]\Big]\\
& \quad - \sin w[\partial_{x_1}\phi_3\cos\theta +\partial_{x_2}\phi_3\sin\theta]\\
& \quad - \frac{1}{2}[\partial_{x_1}\phi_1+\partial_{x_2}\phi_2]+\frac{1}{2}\sin(2\theta)[\partial_{x_2}\phi_1+\partial_{x_1}\phi_2]+\frac{1}{2}\cos(2\theta)[\partial_{x_1}\phi_1-\partial_{x_2}\phi_2].
\end{aligned}
\end{equation*}
Combining above terms with above Lemma and \eqref{eqn-252525}, we obtain the desired decomposition.
\end{proof}

\medskip

\medskip

\subsection{Analysis of the new error}\label{app-remainder}

\medskip

\begin{proof}[Proof of Lemma \ref{lemma-newerror}]
We have
\begin{align*}
&~S[U_*]\\
=&~\mathcal R_{U^\perp}+\mathcal R_{U}+\eta_1 (\mathcal E_{U^\perp}^{(0)}-\pp_t \Phi^{(0)}+\Delta \Phi^{(0)})+\eta_1 (\mathcal E_{U}^{(\pm1)}-\pp_t \Phi^{(1)}+\Delta \Phi^{(1)})+\mathcal E_{U^\perp}^{(1)}\\
&~+(1-\eta_1 )\Big(\mathcal E_{U^\perp}^{(0)}+ \mathcal E_{U}^{(\pm1)}+\mathcal E_{U^\perp}^{(\pm2)}+\mathcal E_{U}^{(\pm2)}\Big)+(\Delta \eta_1 -\pp_t \eta_1)\Phi_*+2\nabla \Phi_*\cdot \nabla \eta_1\\
&~-\eta_1\left[\phi_1^{(2)}\pp_t(Q_\gamma E_1)+\phi_2^{(2)}\pp_t(Q_\gamma E_2)+2\theta_t (\phi_1^{(2)}Q_\gamma E_2-\phi_2^{(2)}Q_\gamma E_1)+\pp_\rho\Phi_{U^\perp}^{(2)} \rho_t\right]\\
&~-\eta_1\left[\phi_1^{(-2)}\pp_t(Q_\gamma E_1)+\phi_2^{(-2)}\pp_t(Q_\gamma E_2)-2\theta_t (\phi_1^{(-2)}Q_\gamma E_2-\phi_2^{(-2)}Q_\gamma E_1)+\pp_\rho\Phi_{U^\perp}^{(-2)} \rho_t\right]\\
&~-\eta_1\left[\sin2\theta \psi_2 \pp_t(Q_\gamma W)+2\cos2\theta \psi_2 \theta_t Q_\gamma W+\sin2\theta \pp_\rho \psi_2 \rho_t Q_\gamma W\right]\\
&~+\eta_1\tilde L_U[\Phi^{(0)}+ \Phi^{(1)}]-2\pp_{x_1}(c_1Q_\gamma\mathcal Z_{1,1}+c_2Q_\gamma\mathcal Z_{1,2})\wedge \pp_{x_2}(\eta_1\Phi_*) -2\pp_{x_1}(\eta_1\Phi_*)\wedge \pp_{x_2}(c_1Q_\gamma\mathcal Z_{1,1}+c_2Q_\gamma\mathcal Z_{1,2})\\
&~-2\pp_{x_2}\eta_1 \pp_{x_1}(Q_\gamma W)\wedge \Phi_* -2\pp_{x_1}\eta_1   \Phi_*\wedge \pp_{x_2}(Q_\gamma W)  -2\pp_{x_1}(\eta_1\Phi_*)\wedge \pp_{x_2}(\eta_1\Phi_*)\\
=&~\mathcal R_{U^\perp}+\mathcal R_{U}+\eta_1 (\mathcal E_{U^\perp}^{(0)}+\tilde{\mathcal R}^0)+\eta_1(\mathcal E_{U}^{(\pm1)}+\tilde{\mathcal R}^1)+\mathcal E_{U^\perp}^{(1)}\\
&~+(1-\eta_1 )\Big(\mathcal E_{U^\perp}^{(0)}+ \mathcal E_{U}^{(\pm1)}+\mathcal E_{U^\perp}^{(\pm2)}+\mathcal E_{U}^{(\pm2)}\Big)+(\Delta \eta_1 -\pp_t \eta_1)\Phi_*+2\nabla \Phi_*\cdot \nabla \eta_1\\
&~-\eta_1 \la^{-1}w_\rho(\dot\xi_1\cos\theta+\dot\xi_2 \sin\theta+\dot\la \rho)(\phi_1^{(2)}+\phi_1^{(-2)})Q_\gamma W\\
&~+\eta_1 \sin w[\dot\gamma+\la^{-1}\rho^{-1}(\dot\xi_1 \sin\theta-\dot\xi_2 \cos\theta)] (\phi_2^{(2)}+\phi_2^{(-2)})Q_\gamma W\\
&~-\eta_1 \la^{-1}w_\rho(\dot\xi_1\cos\theta+\dot\xi_2 \sin\theta+\dot\la \rho)(\phi_1^{(2)}+\phi_1^{(-2)})Q_\gamma W\\
&~+\eta_1 \la^{-1}(\dot\xi_1\cos\theta+\dot\xi_2 \sin\theta+\dot\la \rho)\sin2\theta \pp_\rho \psi_2 Q_\gamma W\\
&~-2\eta_1 \la^{-1}\rho^{-1}(\dot\xi_1 \sin\theta-\dot\xi_2 \cos\theta)\cos2\theta \psi_2 Q_\gamma W\\
&~+\eta_1 \cos w[\dot\gamma+\la^{-1}\rho^{-1}(\dot\xi_1 \sin\theta-\dot\xi_2 \cos\theta)](\phi_2^{(2)}+\phi_2^{(-2)})Q_\gamma E_1\\
&~+2\eta_1 \la^{-1}\rho^{-1}(\dot\xi_1 \sin\theta-\dot\xi_2 \cos\theta)[\phi_2^{(2)}-\phi_2^{(-2)}]Q_\gamma E_1\\
&~+\eta_1 \la^{-1}(\dot\xi_1\cos\theta+\dot\xi_2 \sin\theta+\dot\la \rho) (\pp_\rho \phi_1^{(2)}+\pp_\rho \phi_1^{(-2)})Q_\gamma E_1\\
&~+\eta_1 \la^{-1}w_\rho(\dot\xi_1\cos\theta+\dot\xi_2 \sin\theta+\dot\la \rho) \sin2\theta \psi_2 Q_\gamma E_1\\
&~-\eta_1 \cos w[\dot\gamma+\la^{-1}\rho^{-1}(\dot\xi_1 \sin\theta-\dot\xi_2 \cos\theta)](\phi_1^{(2)}+\phi_1^{(-2)}) Q_\gamma E_2\\
&~+2\eta_1 \la^{-1}\rho^{-1}(\dot\xi_1 \sin\theta-\dot\xi_2 \cos\theta)[\phi_1^{(-2)}-\phi_1^{(2)}]Q_\gamma E_2\\
&~+\eta_1 \la^{-1}(\dot\xi_1\cos\theta+\dot\xi_2 \sin\theta+\dot\la \rho) (\pp_\rho \phi_2^{(2)}+\pp_\rho \phi_2^{(-2)})Q_\gamma E_2\\
&~-\eta_1 \sin w[\dot\gamma+\la^{-1}\rho^{-1}(\dot\xi_1 \sin\theta-\dot\xi_2 \cos\theta)] \sin2\theta \psi_2 Q_\gamma E_2\\
&~+\eta_1 \tilde L_U[\Phi^{(0)}+\Phi^{(1)}]-\frac2r\pp_{r}(\eta_1 \Phi_*)\wedge \pp_{\theta}(\eta_1 \Phi_*)\\
&~-\frac2r\Big[\pp_r (c_1 Q_\gamma\mathcal Z_{1,1}+c_2Q_\gamma\mathcal Z_{1,2})\wedge \pp_\theta(\eta_1 \Phi_*)+\pp_r(\eta_1 \Phi_*)\wedge \pp_\theta (c_1Q_\gamma\mathcal Z_{1,1}+c_2Q_\gamma\mathcal Z_{1,2})\Big]\\
&~-2\pp_{x_2}\eta_1 \pp_{x_1}(Q_\gamma W)\wedge   \Phi_* -2\pp_{x_1}\eta_1   \Phi_*\wedge \pp_{x_2}(Q_\gamma W) ,
\end{align*}
where we have used
\begin{align*}
&~\rho_t=-\la^{-1}(\dot\xi_1\cos\theta+\dot\xi_2 \sin\theta+\dot\la \rho),\\
&~\theta_t=\la^{-1}\rho^{-1}(\dot\xi_1 \sin\theta-\dot\xi_2 \cos\theta),\\
&~\pp_t (Q_\gamma W)=-\la^{-1}(\dot\xi_1\cos\theta+\dot\xi_2 \sin\theta+\dot\la \rho)w_\rho Q_\gamma E_1+ [\dot\gamma+\la^{-1}\rho^{-1}(\dot\xi_1 \sin\theta-\dot\xi_2 \cos\theta)]\sin w Q_\gamma E_2,\\
&~\pp_t (Q_\gamma E_1)=\la^{-1}(\dot\xi_1\cos\theta+\dot\xi_2 \sin\theta+\dot\la \rho)w_\rho Q_\gamma W+ [\dot\gamma+\la^{-1}\rho^{-1}(\dot\xi_1 \sin\theta-\dot\xi_2 \cos\theta)]\cos w Q_\gamma E_2,\\
&~\pp_t (Q_\gamma E_2)=-[\dot\gamma+\la^{-1}\rho^{-1}(\dot\xi_1 \sin\theta-\dot\xi_2 \cos\theta)](\cos w Q_\gamma E_1+\sin w Q_\gamma W).
\end{align*}
We collect some extra terms produced by the cut-off $\eta_1$ and define
\begin{equation}\label{def-E_eta1}
\begin{aligned}
E_{\eta_1}:=&~ (1-\eta_1 )\Big(\mathcal E_{U^\perp}^{(0)}+ \mathcal E_{U}^{(\pm1)}+\mathcal E_{U^\perp}^{(\pm2)}+\mathcal E_{U}^{(\pm2)}\Big)+(\Delta \eta_1 -\pp_t \eta_1)\Phi_*+2\nabla \Phi_*\cdot \nabla \eta_1\\
&~-2\pp_{x_2}\eta_1 \pp_{x_1}(Q_\gamma W)\wedge   \Phi_* -2\pp_{x_1}\eta_1   \Phi_*\wedge \pp_{x_2}(Q_\gamma W) -\frac{2\eta_1\pp_r \eta_1}{r}\Phi_*\wedge \pp_\theta \Phi_*\\
&~-\frac{2\pp_r \eta_1}{r} \Phi_* \wedge \pp_\theta (c_1Q_\gamma\mathcal Z_{1,1}+c_2Q_\gamma\mathcal Z_{1,2}). \\
\end{aligned}
\end{equation}

To further expand the  error, we use \eqref{eqn-325325325} and first compute
\begin{equation*}
\begin{aligned}
\pp_r \Phi^{(0)}=\begin{bmatrix}
e^{i\theta}(\psi^0+\frac{r^2}{z}\pp_z \psi^0)\\
0\\
\end{bmatrix},\quad \pp_\theta \Phi^{(0)}=\begin{bmatrix}
ire^{i\theta}\psi^0\\
0\\
\end{bmatrix},
\end{aligned}
\end{equation*}

\begin{equation*}
\begin{aligned}
\pp_r \Phi^{(1)}=\begin{bmatrix}
0\\
0\\
{\rm Re}[e^{-i\theta}(\psi^1+\frac{r^2}{z}\pp_z\psi^1)]\\
\end{bmatrix}, \quad \pp_\theta \Phi^{(1)}=\begin{bmatrix}
0\\
0\\
{\rm Im}(r e^{-i\theta}\psi^1)\\
\end{bmatrix},
\end{aligned}
\end{equation*}

\begin{align*}
&~\pp_r (\Phi^{(2)}_{U^\perp}+\Phi^{(-2)}_{U^\perp}+\Phi^{(2)}_U)\\
=&~\la^{-1}\left[\pp_\rho\phi_1^{(2)} Q_\gamma E_1+\pp_\rho\phi_2^{(2)} Q_\gamma E_2-w_\rho\phi_1^{(2)}Q_\gamma W\right]\\
&~+\la^{-1}\left[\pp_\rho\phi_1^{(-2)} Q_\gamma E_1+\pp_\rho\phi_2^{(-2)} Q_\gamma E_2-w_\rho\phi_1^{(-2)}Q_\gamma W\right]\\
&~+\la^{-1}\sin2\theta(\pp_\rho \psi_2 Q_\gamma W+w_\rho \psi_2 Q_\gamma E_1),
\end{align*}
\begin{align*}
&~\pp_\theta (\Phi^{(2)}_{U^\perp}+\Phi^{(-2)}_{U^\perp}+\Phi^{(2)}_U)\\
=&~-\phi_2^{(2)}\sin w Q_\gamma W +(\pp_\theta \phi_1^{(2)}-\cos w \phi_2^{(2)})Q_\gamma E_1+(\pp_\theta \phi_2^{(2)}+\cos w \phi_1^{(2)})Q_\gamma E_2\\
&~-\phi_2^{(-2)}\sin w Q_\gamma W +(\pp_\theta \phi_1^{(-2)}-\cos w \phi_2^{(-2)})Q_\gamma E_1+(\pp_\theta \phi_2^{(-2)}+\cos w \phi_1^{(-2)})Q_\gamma E_2\\
&~+\psi_2(2\cos2\theta Q_\gamma W+\sin2\theta \sin w Q_\gamma E_2).
\end{align*}
We then have
\begin{align*}
\pp_r \Phi_*=&~\begin{bmatrix}
e^{i\theta}(\psi^0+\frac{r^2}{z}\pp_z \psi^0)\\
0\\
\end{bmatrix}+\begin{bmatrix}
0\\
0\\
{\rm Re}[e^{-i\theta}(\psi^1+\frac{r^2}{z}\pp_z\psi^1)]\\
\end{bmatrix}\\
&~+\la^{-1}\left[\pp_\rho\phi_1^{(2)} Q_\gamma E_1+\pp_\rho\phi_2^{(2)} Q_\gamma E_2-w_\rho\phi_1^{(2)}Q_\gamma W\right]\\
&~+\la^{-1}\left[\pp_\rho\phi_1^{(-2)} Q_\gamma E_1+\pp_\rho\phi_2^{(-2)} Q_\gamma E_2-w_\rho\phi_1^{(-2)}Q_\gamma W\right]\\
&~+\la^{-1}\sin2\theta(\pp_\rho \psi_2 Q_\gamma W+w_\rho \psi_2 Q_\gamma E_1)\\
=&~\cos w {\rm Re}\left[e^{-i\gamma }(\psi^0+\frac{r^2}{z}\pp_z \psi^0)\right]Q_\gamma  E_1 + {\rm Im}\left[e^{-i\gamma }(\psi^0+\frac{r^2}{z}\pp_z \psi^0)\right]Q_\gamma  E_2\\
&~ + \sin w {\rm Re}\left[e^{-i\gamma }(\psi^0+\frac{r^2}{z}\pp_z \psi^0)\right]Q_\gamma  W \\
&~+{\rm Re}[e^{-i\theta}(\psi^1+\frac{r^2}{z}\pp_z\psi^1)]\cos w Q_\gamma  W-{\rm Re}[e^{-i\theta}(\psi^1+\frac{r^2}{z}\pp_z\psi^1)]\sin w Q_\gamma  E_1   \\
&~+\la^{-1}\left[\pp_\rho\phi_1^{(2)} Q_\gamma E_1+\pp_\rho\phi_2^{(2)} Q_\gamma E_2-w_\rho\phi_1^{(2)}Q_\gamma W\right]\\
&~+\la^{-1}\left[\pp_\rho\phi_1^{(-2)} Q_\gamma E_1+\pp_\rho\phi_2^{(-2)} Q_\gamma E_2-w_\rho\phi_1^{(-2)}Q_\gamma W\right]\\
&~+\la^{-1}\sin2\theta(\pp_\rho \psi_2 Q_\gamma W+w_\rho \psi_2 Q_\gamma E_1)\\
=&~\left(\sin w {\rm Re}\left[e^{-i\gamma }(\psi^0+\frac{r^2}{z}\pp_z \psi^0)\right]+{\rm Re}[e^{-i\theta}(\psi^1+\frac{r^2}{z}\pp_z\psi^1)]\cos w\right)Q_\gamma W\\
&~+\la^{-1}\left[\sin2\theta\pp_\rho \psi_2-w_\rho(\phi_1^{(2)}+\phi_1^{(-2)})\right]Q_\gamma W\\
&~+\left(\cos w {\rm Re}\left[e^{-i\gamma }(\psi^0+\frac{r^2}{z}\pp_z \psi^0)\right]-{\rm Re}[e^{-i\theta}(\psi^1+\frac{r^2}{z}\pp_z\psi^1)]\sin w\right)Q_\gamma E_1\\
&~+\la^{-1}\left(\sin2\theta w_\rho \psi_2+\pp_\rho\phi_1^{(2)}+\pp_\rho\phi_1^{(-2)}\right)Q_\gamma E_1\\
&~+\left({\rm Im}\left[e^{-i\gamma }(\psi^0+\frac{r^2}{z}\pp_z \psi^0)\right]+\la^{-1}(\pp_\rho\phi_2^{(2)}+\pp_\rho\phi_2^{(-2)})\right)Q_\gamma  E_2,
\end{align*}
\begin{align*}
\pp_\theta \Phi_*=&~\begin{bmatrix}
ire^{i\theta}\psi^0\\
0\\
\end{bmatrix}+\begin{bmatrix}
0\\
0\\
{\rm Im}(r e^{-i\theta}\psi^1)\\
\end{bmatrix}\\
&~-\phi_2^{(2)}\sin w Q_\gamma W +(\pp_\theta \phi_1^{(2)}-\cos w \phi_2^{(2)})Q_\gamma E_1+(\pp_\theta \phi_2^{(2)}+\cos w \phi_1^{(2)})Q_\gamma E_2\\
&~-\phi_2^{(-2)}\sin w Q_\gamma W +(\pp_\theta \phi_1^{(-2)}-\cos w \phi_2^{(-2)})Q_\gamma E_1+(\pp_\theta \phi_2^{(-2)}+\cos w \phi_1^{(-2)})Q_\gamma E_2\\
&~+\psi_2(2\cos2\theta Q_\gamma W+\sin2\theta \sin w Q_\gamma E_2)\\
=&~\cos w {\rm Re}\left[e^{-i\gamma }ir\psi^0\right]Q_\gamma  E_1 + {\rm Im}\left[e^{-i\gamma }ir\psi^0\right]Q_\gamma  E_2+ \sin w {\rm Re}\left[e^{-i\gamma }ir\psi^0\right]Q_\gamma  W \\
&~+{\rm Im}(r e^{-i\theta}\psi^1)\cos w Q_\gamma  W-{\rm Im}(r e^{-i\theta}\psi^1)\sin w Q_\gamma  E_1   \\
&~-\phi_2^{(2)}\sin w Q_\gamma W +(\pp_\theta \phi_1^{(2)}-\cos w \phi_2^{(2)})Q_\gamma E_1+(\pp_\theta \phi_2^{(2)}+\cos w \phi_1^{(2)})Q_\gamma E_2\\
&~-\phi_2^{(-2)}\sin w Q_\gamma W +(\pp_\theta \phi_1^{(-2)}-\cos w \phi_2^{(-2)})Q_\gamma E_1+(\pp_\theta \phi_2^{(-2)}+\cos w \phi_1^{(-2)})Q_\gamma E_2\\
&~+\psi_2(2\cos2\theta Q_\gamma W+\sin2\theta \sin w Q_\gamma E_2)\\
=&~\left(\sin w {\rm Re}\left[e^{-i\gamma }ir\psi^0\right]+{\rm Im}(r e^{-i\theta}\psi^1)\cos w-\sin w(\phi_2^{(2)}+\phi_2^{(-2)})+2\psi_2\cos2\theta\right)Q_\gamma W\\
&~+\left(\cos w {\rm Re}\left[e^{-i\gamma }ir\psi^0\right]-{\rm Im}(r e^{-i\theta}\psi^1)\sin w +(\pp_\theta \phi_1^{(2)}+\pp_\theta \phi_1^{(-2)})-\cos w (\phi_2^{(2)}+\phi_2^{(-2)})\right)Q_\gamma E_1\\
&~+\left({\rm Im}\left[e^{-i\gamma }ir\psi^0\right]+(\pp_\theta \phi_2^{(2)}+\pp_\theta \phi_2^{(-2)})+\cos w (\phi_1^{(2)}+\phi_1^{(-2)})+\sin2\theta \sin w\psi_2\right)Q_\gamma E_2.
\end{align*}

So we obtain
\begin{align*}
&~\pp_{r}\Phi_*\wedge \pp_{\theta}\Phi_*\\
=&~\Bigg[\left(\cos w {\rm Re}\left[e^{-i\gamma }(\psi^0+\frac{r^2}{z}\pp_z \psi^0)\right]-{\rm Re}[e^{-i\theta}(\psi^1+\frac{r^2}{z}\pp_z\psi^1)]\sin w\right)\\
&~\qquad+\la^{-1}\left(\sin2\theta w_\rho \psi_2+\pp_\rho\phi_1^{(2)}+\pp_\rho\phi_1^{(-2)}\right)\Bigg]\\
&~\quad\times \left({\rm Im}\left[e^{-i\gamma }ir\psi^0\right]+(\pp_\theta \phi_2^{(2)}+\pp_\theta \phi_2^{(-2)})+\cos w (\phi_1^{(2)}+\phi_1^{(-2)})+\sin2\theta \sin w\psi_2\right)Q_\gamma W\\
&~-\left({\rm Im}\left[e^{-i\gamma }(\psi^0+\frac{r^2}{z}\pp_z \psi^0)\right]+\la^{-1}(\pp_\rho\phi_2^{(2)}+\pp_\rho\phi_2^{(-2)})\right)\\
&~\quad \times \left(\cos w {\rm Re}\left[e^{-i\gamma }ir\psi^0\right]-{\rm Im}(r e^{-i\theta}\psi^1)\sin w +(\pp_\theta \phi_1^{(2)}+\pp_\theta \phi_1^{(-2)})-\cos w (\phi_2^{(2)}+\phi_2^{(-2)})\right)Q_\gamma W\\
&~+\left({\rm Im}\left[e^{-i\gamma }(\psi^0+\frac{r^2}{z}\pp_z \psi^0)\right]+\la^{-1}(\pp_\rho\phi_2^{(2)}+\pp_\rho\phi_2^{(-2)})\right)\\
&~\quad \times \left(\sin w {\rm Re}\left[e^{-i\gamma }ir\psi^0\right]+{\rm Im}(r e^{-i\theta}\psi^1)\cos w-\sin w(\phi_2^{(2)}+\phi_2^{(-2)})+2\psi_2\cos2\theta\right)Q_\gamma E_1\\
&~-\left(\sin w {\rm Re}\left[e^{-i\gamma }(\psi^0+\frac{r^2}{z}\pp_z \psi^0)\right]+{\rm Re}[e^{-i\theta}(\psi^1+\frac{r^2}{z}\pp_z\psi^1)]\cos w\right)\\
&~\quad\times \left({\rm Im}\left[e^{-i\gamma }ir\psi^0\right]+(\pp_\theta \phi_2^{(2)}+\pp_\theta \phi_2^{(-2)})+\cos w (\phi_1^{(2)}+\phi_1^{(-2)})+\sin2\theta \sin w\psi_2\right)Q_\gamma E_1\\
&~+\Bigg[\left(\sin w {\rm Re}\left[e^{-i\gamma }(\psi^0+\frac{r^2}{z}\pp_z \psi^0)\right]+{\rm Re}[e^{-i\theta}(\psi^1+\frac{r^2}{z}\pp_z\psi^1)]\cos w\right)\\
&~\qquad+\la^{-1}\left[\sin2\theta\pp_\rho \psi_2-w_\rho(\phi_1^{(2)}+\phi_1^{(-2)})\right]\Bigg]\\
&~\quad \times \left(\cos w {\rm Re}\left[e^{-i\gamma }ir\psi^0\right]-{\rm Im}(r e^{-i\theta}\psi^1)\sin w +(\pp_\theta \phi_1^{(2)}+\pp_\theta \phi_1^{(-2)})-\cos w (\phi_2^{(2)}+\phi_2^{(-2)})\right)Q_\gamma E_2\\
&~-\Bigg[\left(\cos w {\rm Re}\left[e^{-i\gamma }(\psi^0+\frac{r^2}{z}\pp_z \psi^0)\right]-{\rm Re}[e^{-i\theta}(\psi^1+\frac{r^2}{z}\pp_z\psi^1)]\sin w\right)\\
&~\qquad+\la^{-1}\left(\sin2\theta w_\rho \psi_2+\pp_\rho\phi_1^{(2)}+\pp_\rho\phi_1^{(-2)}\right)\Bigg]\\
&~\quad \times \left(\sin w {\rm Re}\left[e^{-i\gamma }ir\psi^0\right]+{\rm Im}(r e^{-i\theta}\psi^1)\cos w-\sin w(\phi_2^{(2)}+\phi_2^{(-2)})+2\psi_2\cos2\theta\right)Q_\gamma E_2,
\end{align*}
and
\begin{align*}
&~\pp_r (c_1 Q_\gamma\mathcal Z_{1,1}+c_2Q_\gamma\mathcal Z_{1,2})\wedge \pp_\theta\Phi_*+\pp_r\Phi_*\wedge \pp_\theta (c_1Q_\gamma\mathcal Z_{1,1}+c_2Q_\gamma\mathcal Z_{1,2})\\
=&~\left({\rm Im}\left[e^{-i\gamma }ir\psi^0\right]+(\pp_\theta \phi_2^{(2)}+\pp_\theta \phi_2^{(-2)})+\cos w (\phi_1^{(2)}+\phi_1^{(-2)})+\sin2\theta \sin w\psi_2\right)\\
&~\quad \times\la^{-1}w_{\rho}\sin w(c_1 \cos\theta + c_2 \sin\theta) Q_\gamma W\\
&~-\left({\rm Im}\left[e^{-i\gamma }ir\psi^0\right]+(\pp_\theta \phi_2^{(2)}+\pp_\theta \phi_2^{(-2)})+\cos w (\phi_1^{(2)}+\phi_1^{(-2)})+\sin2\theta \sin w\psi_2\right)\\
&~\quad\times \la^{-1}w_{\rho} \cos w (c_1 \cos\theta + c_2 \sin\theta) Q_\gamma E_1\\
&~+\left(\cos w {\rm Re}\left[e^{-i\gamma }ir\psi^0\right]-{\rm Im}(r e^{-i\theta}\psi^1)\sin w +(\pp_\theta \phi_1^{(2)}+\pp_\theta \phi_1^{(-2)})-\cos w (\phi_2^{(2)}+\phi_2^{(-2)})\right)\\
&~\quad\times \la^{-1}w_{\rho} \cos w (c_1 \cos\theta + c_2 \sin\theta) Q_\gamma E_2\\
&~-\left(\sin w {\rm Re}\left[e^{-i\gamma }ir\psi^0\right]+{\rm Im}(r e^{-i\theta}\psi^1)\cos w-\sin w(\phi_2^{(2)}+\phi_2^{(-2)})+2\psi_2\cos2\theta\right)\\
&~\quad \times \la^{-1}w_{\rho}\sin w(c_1 \cos\theta + c_2 \sin\theta) Q_\gamma E_2\\
&~+\Bigg[\left(\cos w {\rm Re}\left[e^{-i\gamma }(\psi^0+\frac{r^2}{z}\pp_z \psi^0)\right]-{\rm Re}[e^{-i\theta}(\psi^1+\frac{r^2}{z}\pp_z\psi^1)]\sin w\right)\\
&~\qquad+\la^{-1}\left(\sin2\theta w_\rho \psi_2+\pp_\rho\phi_1^{(2)}+\pp_\rho\phi_1^{(-2)}\right)\Bigg] \times \sin^2 w(c_1\cos\theta+ c_2\sin \theta) Q_\gamma W\\
&~-\Bigg[\left(\sin w {\rm Re}\left[e^{-i\gamma }(\psi^0+\frac{r^2}{z}\pp_z \psi^0)\right]+{\rm Re}[e^{-i\theta}(\psi^1+\frac{r^2}{z}\pp_z\psi^1)]\cos w\right)\\
&~\qquad+\la^{-1}\left[\sin2\theta\pp_\rho \psi_2-w_\rho(\phi_1^{(2)}+\phi_1^{(-2)})\right]\Bigg]\times\sin^2 w(c_1\cos\theta+ c_2\sin \theta) Q_\gamma E_1\\
&~+\left({\rm Im}\left[e^{-i\gamma }(\psi^0+\frac{r^2}{z}\pp_z \psi^0)\right]+\la^{-1}(\pp_\rho\phi_2^{(2)}+\pp_\rho\phi_2^{(-2)})\right)\times \sin w (c_2 \cos\theta-c_1\sin\theta) Q_\gamma E_1\\
&~-\Bigg[\left(\cos w {\rm Re}\left[e^{-i\gamma }(\psi^0+\frac{r^2}{z}\pp_z \psi^0)\right]-{\rm Re}[e^{-i\theta}(\psi^1+\frac{r^2}{z}\pp_z\psi^1)]\sin w\right)\\
&~\qquad+\la^{-1}\left(\sin2\theta w_\rho \psi_2+\pp_\rho\phi_1^{(2)}+\pp_\rho\phi_1^{(-2)}\right)\Bigg]\times \sin w (c_2 \cos\theta-c_1\sin\theta) Q_\gamma E_2
\end{align*}
since
\begin{equation*}
\begin{aligned}
&~\pp_r (c_1\mathcal Z_{1,1}+c_2\mathcal Z_{1,2})
=\la^{-1}w_{\rho}(c_1 \cos\theta + c_2 \sin\theta)(\cos w W+\sin w E_1),\\
&~\pp_\theta (c_1\mathcal Z_{1,1}+c_2\mathcal Z_{1,2})
=\sin w (c_2 \cos\theta-c_1\sin\theta) W +\sin^2 w(c_1\cos\theta+ c_2\sin \theta) E_2.
\end{aligned}
\end{equation*}
Collecting all the identities above, we conclude the desired results.
\end{proof}

\bigskip

To estimate terms appearing in the remainder $\mathcal R_*$, we give now several estimates for the corrections $\psi^0$, $\psi^1$,  $\phi_j^{(k)}$, $j=1,2$, ~$k=\pm 2$, and $\psi_2$, where their definitions can be found in \eqref{eqn-psi^0}, \eqref{eqn-psi^1}, \eqref{eqn-correctpm2'},  \eqref{eqn-correctpm2}, \eqref{eqn-psi_2} and \eqref{eqn-correctL2}.

 By the definition of $k(z,t)$ in \eqref{eqn-psi^0}, we have
\begin{equation*}
|\psi^j|\lesssim \int_{-T}^t \frac{ |p_j(s)| }{t-s}
		\left( \1_{\{ \zeta \le 1\}}
		+
		\zeta^{-1}  \1_{\{ \zeta  > 1\}} \right)\Big|_{\zeta = z^2 (t-s)^{-1}}
		ds, \quad j=0,~1,
\end{equation*}
where $\1_A$ denotes the characteristic function of the set $A$. For instance, we estimate
\begin{align*}
|\psi^1|\lesssim z^{-2}\int_{-T}^t |p_1(s)|ds\lesssim z^{-2}(t+T) T^{\Theta}|\log T|^{-1-2\Theta} ~\mbox{ for }~z^2\geq t,
\end{align*}
where we have used $$|p_1|\sim \frac{\la_*^{\Theta}}{|\log(T-s)|},\quad \la_*=\frac{(T-t)|\log T|}{|\log(T-t)|^2}.$$
For $z^2<t$, we have
\begin{align*}
|\psi^1|\lesssim \int_{-T}^{t-z^2} \frac{|p_1(s)|}{t-s}ds + z^{-2}\int_{t-z^2}^t |p_1(s)|ds,
\end{align*}
and if $t\leq \frac{T}2$,
\begin{align*}
|\psi^1|\lesssim T^{\Theta}|\log T|^{-1-2\Theta}\left\langle \log\Big(\frac{t+T}{z^2}\Big)\right\rangle.
\end{align*}
If $z^2<t$, $z^2<T-t$ and $t>\frac{T}{2}$, one has
\begin{align*}
|\psi^1|\lesssim&~ \int_{-T}^{t-(T-t)} \frac{|p_1(s)|}{T-s}ds+\int_{t-(T-t)}^{t-z^2} \frac{|p_1(s)|}{t-s}ds + z^{-2}\int_{t-z^2}^t |p_1(s)|ds\\
\lesssim &~ |\log T|^{\Theta}\int_{-T}^{t-(T-t)} \frac{(T-s)^{\Theta-1}}{|\log(T-s)|^{1+2\Theta}} ds+|p_1(t)|\log\frac{T-t}{z^2} +\la_*^{\Theta}(t) |\log(T-t)|^{-1}\\
\lesssim &~ \frac{T^{\Theta}}{|\log T|^{1+\Theta}}.
\end{align*}
If $z^2<t$, $z^2\geq T-t$ and $t>\frac{T}{2}$, we estimate
\begin{align*}
|\psi^1|\lesssim&~ \frac{T^{\Theta}}{|\log T|^{1+\Theta}}+z^{-2}\int_{t-z^2}^t \frac{(T-s)^{\Theta}|\log T|^\Theta}{|\log(T-s)|^{1+2\Theta}} ds\\
\lesssim&~ \frac{T^{\Theta}}{|\log T|^{1+\Theta}}+ \frac{(T-t+z^2)^\Theta}{|\log(T-t+z^2)|^{1+2\Theta}}.
\end{align*}
Above estimates then give
\begin{equation*}
|\psi^1|\lesssim  \1_{\{ z^2< t \}}  + z^{-2}(t+T) T^{\Theta}|\log T|^{-1-2\Theta}  \1_{\{ z^2\ge t \}},
\end{equation*}
and similarly,
\begin{equation}\label{est-psi^1}
|\psi^1|+z|\pp_z \psi^1|\lesssim  \1_{\{ z^2< t \}}  + z^{-2}(t+T) T^{\Theta}|\log T|^{-1-2\Theta}  \1_{\{ z^2\ge t \}}.
\end{equation}
$\psi^0$ is estimated in the same way. One directly checks that
\begin{equation}\label{est-psi^0}
|\psi^0|+z|\pp_z \psi^0|\lesssim  \1_{\{ z^2< t \}}  + z^{-2}(t+T) |\log T|^{-1}  \1_{\{ z^2\ge t \}}.
\end{equation}

\medskip

Terms $\phi_j^{(k)}$, $j=1,2$, ~$k=\pm 2$ and $\psi_2$ are estimated via the linear theory in both $W$-direction and on $W^\perp$ in mode $\pm 2$, namely Proposition \ref{prop-lt-Wperp} and Proposition \ref{prop-lt-W+} with $k=\pm 2$.

In the scalar equation \eqref{eqn-correctpm2}, the right hand sides for modes $j=2$ and $j=-2$ satisfy
\begin{align*}
&~\left|\la^2\left(Q_{-\gamma}\mathcal E_{U^\perp}^{(2)}\right)_{\mathbb C}\right|=\Big|-i  c_1c_2  w^2_\rho \sin w e^{2i\theta} (1-\cos w)\Big| \lesssim \la_*^{2\Theta}(t)  \rho \langle \rho\rangle^{-8},\\
&~\left|\la^2\left(Q_{-\gamma}\mathcal E_{U^\perp}^{(-2)}\right)_{\mathbb C}\right|=\Big|-i  c_1c_2  w^2_\rho \sin w e^{-2i\theta}(1+\cos w) \Big| \lesssim \la_*^{2\Theta}(t) \rho^3 \langle \rho\rangle^{-8},
\end{align*}
where we have used \eqref{est-cc}. Then Proposition \ref{prop-lt-Wperp}  implies
\begin{equation}\label{est-varphi_pm2}
\begin{aligned}
|\varphi_2|\lesssim &~\la_*^{2\Theta}(t) \langle \rho\rangle^{-\delta}, \\
|\varphi_{-2}|\lesssim &~\la_*^{2\Theta}(t) \langle \rho\rangle^{-\delta}
\end{aligned}
\end{equation}
for some $\delta\in(0,1)$ and close to $1$. In fact, a more careful inspection on the proof of Proposition \ref{prop-lt-Wperp} enables one to obtain faster spatial decay for corrections in higher modes as well as vanishing properties near the origin. For instance, $\delta$ can be taken close to $3$ for mode $2$ and close to $1$ for mode $-2$.  Similarly, in \eqref{eqn-correctL2}, since the RHS
$$
|2c_1c_2 \rho w_\rho^2 \sin^2 w|\lesssim \la_*^{2\Theta}(t)  \rho^3 \langle \rho\rangle^{-8},
$$
one has
\begin{equation}\label{est-psi_2}
|\psi_2|\lesssim \la_*^{2\Theta}(t)  \langle \rho\rangle^{-\delta_1}
\end{equation}
by Proposition \ref{prop-lt-W+}, where $\delta_1$ can be taken close to $2$.

One then uses the bounds \eqref{est-psi^1}, \eqref{est-psi^0}, \eqref{est-varphi_pm2} and \eqref{est-psi_2} and directly checks that
\begin{equation}\label{eqn-a27a27a27}
\begin{aligned}
\Big|\mathcal R_*\cdot Q_\gamma W\Big| \lesssim &~\Bigg[1+ \la_*^{2\Theta-1}  \bigg(\langle \rho \rangle^{-1-\delta}+\langle \rho \rangle^{-\delta_1}\bigg) +\la_*^{5\Theta-2}  \bigg(\langle \rho \rangle^{-2-\delta}+\langle \rho \rangle^{-1-\delta_1}\bigg)\\
&~\qquad+\la_*^{4\Theta-2} \langle \rho \rangle^{-2-2\delta}+\la_*^{\Theta-1} \langle \rho \rangle^{-3}+\la_*^{3\Theta-2}  \langle \rho \rangle^{-4-\delta}\Bigg] \1_{\{\rho\lesssim \la_*^{-1}\}},\\
\Big|\mathcal R_*\cdot Q_\gamma E_1\Big|,~\Big|\mathcal R_*\cdot Q_\gamma E_2\Big| \lesssim &~\Bigg[1+ \la_*^{2\Theta-1}  \langle \rho \rangle^{ -\delta} +\la_*^{4\Theta-2}  \bigg(\langle \rho \rangle^{-3-2\delta}+\langle \rho \rangle^{-2-\delta-\delta_1}\bigg)+\la_*^{5\Theta-2}  \langle \rho \rangle^{ -1-\delta} \\
&~\qquad+\la_*^{\Theta-1} \langle \rho \rangle^{-2}+\la_*^{3\Theta-2}  \langle \rho \rangle^{-3-\delta}\Bigg] \1_{\{\rho\lesssim \la_*^{-1}\}},
\end{aligned}
\end{equation}
where the expression of the remainder $\mathcal R_*$ is given in Lemma \ref{lemma-newerror}, and we have used
$$
|\cc|\lesssim \la_*^{\Theta},\quad |\dot\xi|\lesssim \la_*^{3\Theta-1}.
$$
 For the remaining terms $E_{\eta_1}$ (defined in \eqref{def-E_eta1}) supported outside of the inner region, we have
\begin{equation}\label{eqn-a28a28a28}
|E_{\eta_1}|\lesssim 1
\end{equation}
by \eqref{est-psi^1}, \eqref{est-psi^0}, \eqref{est-varphi_pm2}, \eqref{est-psi_2}, and \eqref{def-error1}.

\bigskip

\section{Spectral analysis of the linearized Liouville equation}\label{sec-DFT}

\medskip

\subsection{Generalized eigenfunctions and density of spectral measure}

We consider the operator
\begin{equation*}
L_0:=\partial_{rr}+\frac1r \partial_{r}+\frac{8}{(1+r^2)^2},
\end{equation*}
which has kernels
$$
K_1=\frac{r^2-1}{r^2+1}, \quad K_2=\frac{(r^2-1)\log r-2}{r^2+1}.
$$
The operator $L_0$ corresponds to the linearization at mode $0$ of the Liouville equation. Let $u(\rho)=r^{-1/2}v(r)$. Then
$$
L_0(u)=r^{-1/2}\mathcal L_0 v,
$$
with
$$
\mathcal L_0 v:= \pp_{rr} v+\frac{v}{4r^2}+\frac{8v}{(1+r^2)^2}.
$$
The new operator $\mathcal L_0$ has kernel
$$
\Phi_0^0(r)=\frac{r^{1/2}(r^2-1)}{r^2+1},$$
and the other one is given by
\begin{equation*}
\begin{aligned}
\Theta_0^0(r)=&~-\Phi_0^0(r)\int \frac1{(\Phi(r))^2} dr\\
=&~-\frac{r^{1/2}(r^2-1)}{r^2+1}\left(\log r-\frac2{r^2-1}+C\right)
\end{aligned}
\end{equation*}
for which
$$
W[\Theta_0^0,\Phi_0^0]=1.
$$
Here for any $C$, $\Theta_0^0(1)=0$, and we take $C=0$.
\begin{lemma} (\cite[Lemma 4.2]{KST09Duke}, \cite{GZ06Mana})
The spectrum of $\mathcal L_{0}$ equals
\begin{equation*}
spec(-\mathcal L_{0}) = \{ \xi_{d}\} \cup [ 0,\infty),
\end{equation*}
where the only negative eigenvalue
$\xi_{d}$ is negative and simple, and its corresponding eigenfunction $\phi_{d}$ has exponential decay.
\end{lemma}

\begin{remark}
The operator $\mathcal L_0$ has a resonance at zero since $\mathcal L_0[\Phi_0^0]=0$ and $\Theta_0^0\not\in L^2(dr)$.
\end{remark}

We now consider the fundamental system of solutions $\phi(r,z)$ and $\theta(r,z)$ to
\begin{equation*}
-\mathcal L_0 y=zy
\end{equation*}
for all $z\in \mathbb C$ so that
$$
W[\theta(\cdot,z),\phi(\cdot,z)]=1.
$$
Notice that these functions are entire in $z$, $\phi(r,0)=c\Phi_0^0(r)$ for some normalization constant $c$. Let $\psi(r,z)$ be a Weyl-Titchmarsh solution. The generalized Weyl-Titchmarsh $m$ function is defined as
$$
C\psi(\cdot,z)=\theta(\cdot,z)+m(z)\phi(\cdot,z)
$$
for some constant $C\neq0$. Then
$$
m(z)=\frac{W[\theta(\cdot,z),\psi(\cdot,z)]}{W[\psi(\cdot,z),\phi(\cdot,z)]}.
$$
A spectral measure of $\mathcal L_0$ is obtained as
$$
\rho((\la_1,\la_2])=\frac1{\pi}\lim_{\delta\to0^+}\lim_{\epsilon\to0^+}\int_{\la_1+\delta}^{\la_2+\delta}{\rm Im}~ m(\la+i\epsilon) d\la.
$$
For the detailed definitions and properties, see \cite{GZ06Mana}.

\begin{prop}
The  $m$ function is
\begin{equation*}
d \rho  =\delta_{\xi_{d}} + \rho(\xi )d\xi,
\quad \rho(\xi) = \pi^{-1} \mathrm{Im} \  m(\xi + i 0^{+}).
\end{equation*}
The distorted Fourier transform (DFT) defined as
\begin{equation*}
\mathcal F : f \rightarrow \hat{f},
\end{equation*}
\begin{equation*}
\hat{f}(\xi_{d}) = \int_{0}^{\infty} \phi_{d}(r) f(r) d r,
\quad
\hat{f}(\xi) =
\lim\limits_{b\rightarrow \infty}
\int_{0}^{b} \phi(r,\xi) f(r) d r, \quad \xi \ge 0,
\end{equation*}
is a unitary operator from $L^2(\RR^{+})$ to $L^2(\{\xi_{d} \}\cup \RR^{+},\rho)$, and its inverse is given by
\begin{equation*}
\mathcal F^{-1}: \hat{f} \rightarrow
f(r) = \hat{f}(\xi_{d}) \phi_{d}(r) +
\lim\limits_{\mu\rightarrow \infty}
\int_{0}^{\mu} \phi(r,\xi )\hat{f}(\xi ) \rho(\xi) d \xi.
\end{equation*}
\end{prop}

Next, we want to give the asymptotic expansion of distorted basis $\Phi^{0}(r,\xi)$ satisfying $-\mathcal L_{0} \Phi^{0}(r,\xi) = \xi \Phi^{0}(r,\xi)$.
\begin{prop}
For any $\xi \in \mathbb{C}$, $\Phi^{0}(r,\xi)$ admits
the asymptotic expansion
\begin{equation*}
\Phi^{0}(r,\xi)
= \Phi_{0}^{0}(r)
+
r^{1/2} \sum\limits_{j=1}^{\infty}
(-  r^2 \xi)^{j} \Phi_{j}^{0}(r^2),
\end{equation*}
which converges absolutely for any fixed $r$, $\xi$ and converges uniformly if $r^2 \xi$ remains bounded with a fixed $r$. Here for $j\geq 1$, $\Phi_{j}^{0}(u)$ are smooth functions in $u\ge 0$ satisfying
\begin{equation*}
\begin{aligned}
\Phi_{1}^{0}(u) =&~\frac{u-1}{12u(u+1)^2}\Bigg[(3u-2\pi^2)(1+u)+6(1+u)\log(1+u)[2+\log(1+u)-2\log u]\\
&~\qquad\qquad\qquad+12(1+u){\rm Polylog}\left(2,\frac{1}{1+u}\right)\Bigg]+\frac{(u+3)\log(u+1)-3u}{u(u+1)},
\end{aligned}
\end{equation*}
\begin{equation*}
|\Phi_{j}^{0}(u)| \leq C,\quad j\geq 2
\end{equation*}
for a large constant $C>0$.
\end{prop}

\begin{proof}
We look for
$$
\Phi^{0}(r, \xi ) = r^{1/2}
\sum\limits_{j=0}^{\infty} (-\xi)^{j} f_{j}(r), \quad f_0(r)=\frac{r^2-1}{r^2+1}
$$
with
$$
\mathcal L_0(r^{1/2}f_j)=r^{1/2} f_{j-1},\quad f_{-1}\equiv 0.
$$
Then above ansatz yields the recurrence relation
\begin{equation*}
\begin{aligned}
f_j(r)=&~\int_0^r r^{-1/2}\left[\Phi_0^0(r)\Theta_0^0(s)-\Theta_0^0(r)\Phi_0^0(s)\right] s^{1/2} f_{j-1}(s) ds\\
=&~\int_0^r \left[\frac{(r^2-1)}{r^2+1}\frac{s\left[2-(s^2-1)\log s\right]}{s^2+1}+\frac{\left[(r^2-1)\log r-2\right]}{r^2+1}\frac{s(s^2-1)}{s^2+1}\right]  f_{j-1}(s) ds\\
=&~\int_0^r \frac{2s(r^2-s^2)+s(s^2-1)(r^2-1)\log(r/s)}{(r^2+1)(s^2+1)} f_{j-1}(s) ds
\end{aligned}
\end{equation*}
 Now we change variable $v=s^2$, $u=r^2$ and define
 $$
 \tilde f_j(v)=f_j(s).
 $$
 In particular,
 $$
 \tilde f_0(u)=\frac{u-1}{u+1}.
 $$
Then we get
\begin{equation*}
\begin{aligned}
\tilde f_j(u)=\int_0^{u} \frac{4(u-v)+(v-1)(u-1)\log (u/v)}{4(u+1)(v+1)} \tilde f_{j-1}(v) dv.
\end{aligned}
\end{equation*}
%Then by writing $v=tu$, we get
%\begin{equation}
%\tilde f_j(u)=u\int_0^{1} \frac{4(u-tu)-(tu-1)(u-1)\log t}{4(u+1)(tu+1)} \tilde f_{j-1}(tu) dt.
%\end{equation}
We first consider $\tilde f_1(u)$
\begin{equation*}
\begin{aligned}
\tilde f_1(u)%=&~u\int_0^{1} \frac{4(u-tu)-(tu-1)(u-1)\log t}{4(u+1)(tu+1)}\frac{tu-1}{tu+1}dt\\
=&~\int_0^u \frac{4(u-z)(z-1)-(z-1)^2(u-1)\log (z/u)}{4(u+1)(z+1)^2}dz\\
=&~\frac{1}{1+u}\int_0^u \frac{(u-z)(z-1)}{(z+1)^2}dz-\frac{u-1}{4(u+1)}\int_0^u \frac{(z-1)^2\log(z/u)}{(z+1)^2}dz,
%=&~\frac1{1+u}\left[(u+3)\log(u+1)-3u\right]+\frac{(u-1)\log u}{4(u+1)}\left[\frac{u(5+u)}{1+u}-4\log(1+u)\right]\\
%&~-\frac{(u-1)u\log u}{4(u+1)}+O\left(\frac{u(u-1)}{u+1}\right)\\
%=&~\frac{(u+3)\log(u+1)-3u}{u+1}+\frac{(u-1)\log u}{u+1}\left[\frac{u}{u+1}-\log(1+u)\right]+O\left(\frac{u(u-1)}{u+1}\right)
\end{aligned}
\end{equation*}
where we have
$$
\int_0^u \frac{(u-z)(z-1)}{(z+1)^2}dz=(u+3)\log(u+1)-3u,
$$
%$$
%\int_0^u \frac{(z-1)^2}{(z+1)^2}dz=\frac{u(5+u)}{1+u}-4\log(1+u),
%$$
%and
%$$
%\int_0^u \frac{(z-1)^2\log z}{(z+1)^2}dz\sim u\log u+O(u).
%$$
and
\begin{equation*}
\begin{aligned}
&~\int_0^u \frac{(z-1)^2\log(z/u)}{(z+1)^2} dz\\
=&~-\frac{1}{3(1+u)}\Bigg[3u(1+u)-2\pi^2(1+u)-15 u\log u-3u^2\log u+12(1+u)\log(1+u)\\
&~\qquad\qquad\qquad+6(1+u)\log^2 (1+u)+3 [u(5+u)-4(1+u)\log(1+u)](\log u)\\
&~\qquad\qquad\qquad+12(1+u){\rm Polylog}\left(2,\frac1{1+u}\right)\Bigg]
\end{aligned}
\end{equation*}
with
\begin{equation*}
{\rm Polylog}\left(2,\frac1{1+u}\right)\sim
\left\{
\begin{aligned}
&\frac{\pi^2}{6}+(\log u -1)u-\frac14(2\log u +1)u^2+\frac1{18}(6\log u + 7)u^3+O(u^4\log u )~\mbox{ for }~u\ll1,\\
&\frac1u+\frac1{4u^2}+\frac1{9u^3}+O\left(\frac1{u^4}\right)~\mbox{ for }~u\gg1,\\
\end{aligned}
\right.
\end{equation*}
and for $u\ll 1$
\begin{align*}
\frac{u-1}{u+1}{\rm Polylog}\left(2,\frac1{1+u}\right)=&~-\frac{\pi^2}6+\frac13(3+\pi^2-3\log u)u+\frac1{12}(-21-4\pi^2+30\log u)u^2\\
&~+\frac1{9}(10+3\pi^2-30\log u)u^3+O(u^4\log u).
\end{align*}
So we obtain
\begin{equation*}
\begin{aligned}
\tilde f_1(u)=&~\frac{(u+3)\log(u+1)-3u}{u+1}\\
&~+\frac{u-1}{12(u+1)^2}\Bigg[(3u-2\pi^2)(1+u)+6(1+u)\log(1+u)[2+\log(1+u)-2\log u]\\
&~\qquad\qquad\qquad\quad+12(1+u){\rm Polylog}\left(2,\frac{1}{1+u}\right)\Bigg],
\end{aligned}
\end{equation*}
and
\begin{equation}\label{eqn-310310310}
\tilde f_1(u)\sim
\left\{
\begin{aligned}
&-\frac{u}4+\frac{u^2}{4}-\frac{2u^3}{9}+O(u^4)~&\mbox{ for }~u\ll1,\\
&\frac{u}{4}-\frac{3\log^2 u-12\log u+21+\pi^2}{6}+\frac{6\log^2 u+39+2\pi^2}{6u}+O\left(\frac{\log^2 u}{u^2}\right)
~&\mbox{ for }~u\gg1.\\
\end{aligned}
\right.
\end{equation}

\medskip

Next
%$$
%\tilde f_j(u)=u\int_0^{1} \frac{4(u-tu)-(tu-1)(u-1)\log t}{4(u+1)(tu+1)} \tilde f_{j-1}(tu) dt,
%$$
we estimate $\tilde f_j(u)$ inductively.  %We write
%$$
%h_j(u):=\frac{\tilde f_j(u)}{u^{j+1}}.
%$$
%Then
%$$
%h_{j}(u)=\int_0^{1} \frac{4(u-tu)-(tu-1)(u-1)\log t}{4(u+1)(tu+1)} t^j h_{j-1}(tu) dt,\quad h_0(u)=\frac{u-1}{u(u+1)}.
%$$
%We notice that $h_0$ is sign-changing, so we might only be able to get an upper bound. To do so, we use
We consider
\begin{equation*}
\tilde f_j(u)=\int_0^{u} \frac{4(u-v)+(v-1)(u-1)\log (u/v)}{4(u+1)(v+1)} \tilde f_{j-1}(v) dv
\end{equation*}
and bound the kernel function as
\begin{equation*}
\left|\frac{4(u-v)+(v-1)(u-1)\log (u/v)}{4(u+1)(v+1)}\right|\leq C[\log (u/v)+1]
\end{equation*}
for some constant $C>0$ since $v\leq u$. Then from \eqref{eqn-310310310} we have
$$
|\tilde f_1(u)|\leq C u,
$$
and thus
\begin{equation*}
\begin{aligned}
|\tilde f_2(u)|\leq &~ C\int_0^{u} [\log (u/v)+1]|\tilde f_{1}(v)| dv\\
\leq &~ C\int_0^{u} v[\log (u/v)+1] dv\\
\leq &~ C u^2.
\end{aligned}
\end{equation*}
By induction, clearly one has
$$
|\tilde f_j(u)|\leq C u^j.
$$
In the original variable, we have
$
r^{2j}\Phi^0_j(r^2)=f_j(r)=\tilde f_j(r^2)=\tilde f_j(u),
$
and thus
$$
|\Phi^0_j(u)|\leq C
$$
as desired.
\end{proof}

Next we estimate the Weyl-Titchmarsh function $\Psi_0^+(r,\xi)$ with
$$
-\mathcal L_0 \Psi_0^+(r,\xi)=\xi \Psi_0^+(r,\xi),\quad \xi>0,
$$
namely
$$
\pp_{rr}\Psi_0^+(r,\xi) +\frac{\Psi_0^+(r,\xi)}{4r^2}+\xi \Psi_0^+(r,\xi)=-\frac{8\Psi_0^+(r,\xi)}{(1+r^2)^2}.
$$
We write
$$
\Psi_0^+(r,\xi)=g(r\xi^{1/2}).
$$
Then one has
$$
g''(r\xi^{1/2})+\frac{1}{4r^2\xi}g(r\xi^{1/2})+g(r\xi^{1/2})=-\frac8{\xi(1+r^2)^2}g(r\xi^{1/2}),
$$
and
$$
g''(q)+\frac{1}{4q^2}g(q)+g(q)=-\frac8{\xi(1+q^2/\xi)^2}g(q)
$$
for
$$
q:=r\xi^{1/2}.
$$
Homogeneous solutions can be written in terms of special functions. Here we follow the argument in \cite[Proposition 5.6]{Krieger08} and \cite[Section 4]{KMS20WM}. We now consider the case $r\xi^{1/2}\gtrsim 1$.

\begin{lemma}\label{lem-323232}
A Weyl-Titchmarsh function is of the form
\begin{equation*}
\Psi_{0}^{+}(r,\xi) = \xi^{-1/4} e^{ir\xi^{1/2}}
\sigma (r\xi^{1/2},r) \mbox{ \ for \ }  r\xi^{1/2} \gtrsim 1
\end{equation*}
where $\sigma$ admits the asymptotic series approximation about $q$ with a fixed $r$,
\begin{equation*}
\sigma(q,r) \sim  \sum\limits_{j=0}^{\infty} q^{-j} \psi_{j}^{+}(r), \quad
\psi_{0}^{+}(r)=1,
\end{equation*}
\begin{equation*}
\psi_{1}^{+}(r)=-\frac{i}{8} + i\left(\frac{2r^2}{r^2+1}+2r\arctan r -r\pi\right)
\end{equation*}
with
\begin{equation*}
\sup_{r>0}|(r\pp_r)^k \psi_j^+|\le d_{jk},
\end{equation*}
and
\begin{equation*}
	|(r\pp_{r})^{\alpha} (q\pp_{q})^{\beta} (\sigma(q,r) - \sum\limits_{j=0}^{j_0} q^{-j} \psi_{j}^{+}(r) )|
	\le
	e_{\alpha,\beta,j_0} q^{-j_0-1} .
\end{equation*}
\end{lemma}
\begin{proof}
From the form of $\Psi_{0}^{+}(r,\xi)$, we only need to consider
$$
\left(\pp_{rr}+2i\xi^{1/2}\pp_r+\frac1{4r^2}+\frac{8}{(1+r^2)^2}\right)\sigma(r\xi^{1/2},r)=0.
$$
We look for a formal power series that solves above equation
$$
\sigma(r\xi^{1/2},r)=\sum_{j=0}^{\infty}\xi^{-j/2}f_j(r)
$$
which implies the following recurrence relation
$$
2i\pp_r f_j(r)=\left(-\pp_{rr}-\frac1{4r^2}-\frac8{(1+r^2)^2}\right)f_{j-1}(r),\quad j\geq 1, ~f_0(r)=1.
$$
A solution is given by
$$
f_j(r)=\frac{i}{2}\pp_r f_{j-1}(r)+\frac{i}2\int_r^{\infty}\left(-\frac1{4s^2}-\frac8{(1+s^2)^2}\right)f_{j-1}(s)ds.
$$
Then we get
$$
|f_1(r)|\leq c_{10} r^{-1},\quad |(r\pp_r)^k f_1|\leq c_{1k} r^{-1},~k\in\mathbb N,
$$
and by induction
$$
|(r\pp_r)^k f_j|\leq c_{jk} r^{-j},\quad j\geq 1,~r>0.
$$
Now write
$$
\psi_j^+(r):=r^j f_j(r),
$$
and it follows from the control of $f_j$ that
$$
\sup_{r>0}|(r\pp_r)^k\psi_j^+|\leq d_{jk}
$$
for some constant $d_{jk}$. Write $q=r\xi^{1/2}$. Then
$$q^{-j} \psi_{j}^{+}(r) = \xi^{-j/2} f_j(r).$$
Define
$$\sigma_{ap}(q,r):= \sum\limits_{j=0}^{\infty} q^{-j} \psi_{j}^{+}(r) (1-\eta(q\delta_j)),
$$
where $\delta_j \rightarrow 0^{+}$ sufficiently fast to ensure the convergence of the summation. Indeed, one has
\begin{equation*}
\begin{aligned}
&
\bigg|\sigma_{ap}(q,r) - \sum\limits_{j=0}^{j_0} q^{-j} \psi_{j}^{+}(r) \bigg|
\\
= \ &
\bigg|\sum\limits_{j=j_0 +1}^{\infty} q^{-j} \psi_{j}^{+}(r) (1-\eta(q\delta_j))
-
\sum\limits_{j=0}^{j_0} q^{-j} \psi_{j}^{+}(r) \eta(q\delta_j)
\bigg|
\\
\le \ &
e_{j_0} q^{-j_0-1} .
\end{aligned}
\end{equation*}
In the last step, we have used the following estimates. For $j\ge j_0+2$,
$$\big|q^{-j} \psi_{j}^{+}(r) (1-\eta(q\delta_j)) \big| \le q^{-j} \1_{\{q\delta_j >1 \} } d_{j0} \le q^{-j_0-1} 2^{-j} ~\mbox{ if }~\delta_j d_{j0} \le 2^{-j}.
$$
For $0\le j \le j_0$,
$$|q^{-j} \psi_{j}^{+} \eta(q\delta_j)| = |q^{-j_0-1}q^{j_0+1-j} \psi_{j}^{+} \eta(q\delta_j)| \le q^{-j_0-1} (2\delta_{j}^{-1})^{j_0+1} d_{j0}
\le q^{-j_0-1} \mathop{\max}\limits_{0\le j\le j_0}(2\delta_{j}^{-1})^{j_0+1} d_{j0}.
$$
Similarly, we have
\begin{equation*}
		\bigg|(r\pp_{r})^{\alpha} (q\pp_{q})^{\beta} (\sigma_{ap}(q,r) - \sum\limits_{j=0}^{j_0} q^{-j} \psi_{j}^{+}(r) )\bigg|
		\le
		e_{\alpha,\beta,j_0} q^{-j_0-1} .
\end{equation*}
This implies that $\sigma_{ap}(r\xi^{1/2},r)$ is a good approximation at infinity, namely the error
\begin{equation*}
\begin{aligned}
e(r\xi^{1/2},r) = \ & \left(\pp_{rr} + 2i\xi^{1/2} \pp_{r} +\frac{1}{4r^2} + \frac{8}{(1+r^2)^2}\right)  \sigma_{ap}
\\
= \ &
\left(\pp_{rr} + 2i\xi^{1/2} \pp_{r} +\frac{1}{4r^2} + \frac{8}{(1+r^2)^2}\right)
\left(\sigma_{ap} - \sum\limits_{j=0}^{j_0}
q^{-j} \psi_{j}^+(r) + \sum\limits_{j=0}^{j_0}
q^{-j} \psi_{j}^+(r) \right)
\\
\end{aligned}
\end{equation*}
is small. Indeed, for the first term, we estimate
\begin{equation*}
	\left|\left(\pp_{rr} +\frac{1}{4r^2} + \frac{8}{(1+r^2)^2} \right)
	\left(\sigma_{ap} - \sum\limits_{j=0}^{j_0}
	q^{-j} \psi_{j}^+(r)  \right) \right|
\lesssim
r^{-2} q^{-j_0},
\end{equation*}
and
\begin{equation*}
\left|2i\xi^{1/2} \pp_{r}
	\left(\sigma_{ap} - \sum\limits_{j=0}^{j_0}
	q^{-j} \psi_{j}^+(r)  \right)\right|
	\lesssim
	\xi^{1/2}  r^{-1} q^{-j_0 -1}
= r^{-2} q^{-j_0}.
\end{equation*}
For the second term, we have
\begin{align*}
&
\left(\pp_{rr} + 2i\xi^{1/2} \pp_{r} +\frac{1}{4r^2} + \frac{8}{(1+r^2)^2}\right)
\sum\limits_{j=0}^{j_0}
q^{-j} \psi_{j}^+(r)
\\
= \ &
\left(\pp_{rr} + 2i\xi^{1/2} \pp_{r} +\frac{1}{4r^2} + \frac{8}{(1+r^2)^2}\right)
\sum\limits_{j=0}^{j_0}
\xi^{-j/2} f_{j}(r)
\\
= \ &
\left(\pp_{rr}  +\frac{1}{4r^2} + \frac{8}{(1+r^2)^2}\right)
\sum\limits_{j=0}^{j_0}
\xi^{-j/2} f_{j}(r)
+
2i
\left( \sum\limits_{j=0}^{j_0 -1}
\xi^{-j/2}  \pp_{r}   f_{j+1}(r)
\right)
\\
= \ &
\left(\pp_{rr}  +\frac{1}{4r^2} + \frac{8}{(1+r^2)^2}\right)
\xi^{-{j_0}/2} f_{j_0}(r),
\end{align*}
so
\begin{equation*}
\left|\left(\pp_{rr} + 2i\xi^{1/2} \pp_{r} +\frac{1}{4r^2} + \frac{8}{(1+r^2)^2}\right)
\sum\limits_{j=0}^{j_0}
q^{-j} \psi_{j}^+(r) \right|\lesssim \xi^{-{j_0}/2} r^{-j_0 -2} =
r^{-2} q^{-j_0} .
\end{equation*}
Therefore, the error has the following bound
\begin{equation*}
\big|e(r\xi^{1/2},r)\big| \le c_{j_0}  r^{-2} q^{-j_0} \mbox{ \ for all \ } j_0\ge 0 .
\end{equation*}
We then look for perturbation $\sigma_1 = -\sigma +\sigma_{ap}$ solved by
\begin{equation*}
\left(\pp_{rr} + 2i\xi^{1/2} \pp_{r} +\frac{1}{4r^2} + \frac{8}{(1+r^2)^2}\right) \sigma_{1} = e(r\xi^{1/2} ,r),
\end{equation*}
which can be formulated, by writing $(v_1,v_2)=(\sigma, r\pp_r \sigma)$, as a system of first order
\begin{equation}\label{eqn-v}
\left\{
\begin{aligned}
&\pp_r v_1=r^{-1}v_2\\
&\pp_r v_2=\left(-\frac1{4r}-\frac{8r}{(1+r^2)^2}\right) v_1+(r^{-1}-2i\xi^{1/2})v_2+r e(r\xi^{1/2},r)\\
\end{aligned}
\right.
\end{equation}
We then estimate
\begin{equation*}
\begin{aligned}
&~\pp_r(|v_1|^2+|v_2|^2)\\
=&~ 2{\rm Re}(\bar v_1\pp_r v_1)+2{\rm Re}(\bar v_2\pp_r v_2)\\
=&~2r^{-1}{\rm Re}(\bar v_1 v_2)+2{\rm Re}\Big[\bar v_2\Big(\Big(-\frac1{4r}-\frac{8r}{(1+r^2)^2}\Big) v_1+(r^{-1}-2i\xi^{1/2})v_2+r e(r\xi^{1/2},r)\Big)\Big]\\
=&~2r^{-1}{\rm Re}(\bar v_1 v_2)-\left(\frac1{2r}+\frac{16r}{(1+r^2)^2}\right){\rm Re}(\bar v_2 v_1)+2r^{-1}|v_2|^2+2r{\rm Re}\big(\bar v_2 e(r\xi^{1/2},r)\big)\\
\geq&~ -C (r^{-1}|v|^2+r|v||e(r\xi^{1/2},r)|)
\end{aligned}
\end{equation*}
for some $C>0$. So we have
$$
\pp_r |v|\geq -C (r^{-1}|v|+r|e(r\xi^{1/2},r)|),
$$
and
$$
|v|\leq Cr^{-C}\int_r^{\infty}s|e(s\xi^{1/2},s)|ds.
$$
Recalling the estimate
$$
|e|\lesssim r^{-2}q^{-j_0}=r^{-2-j_0}\xi^{-j_0/2},
$$
we get
$$
|v|\lesssim r^{-j_0}\xi^{-j_0/2}.
$$
For the estimate of higher order derivatives $(r\pp_r)^{\alpha}(\xi\pp_{\xi})^{\beta}v$, one needs to differentiate \eqref{eqn-v} and repeat above process. The argument is a verbatim repetition of the proof in \cite{Krieger08}.
\end{proof}

\medskip

By the asymptotic expansion of $\Psi_{0}^{+}$, we have $W(\Psi_{0}^{+}, \overline{\Psi_{0}^{+}})= -2i$. Concerning the density of spectral measure, we have the following
\begin{prop}
We have
\begin{equation}\label{41}
\Phi^{0}(r,\xi) = a_{0}(\xi ) \Psi_{0}^{+}(r,\xi) +
\overline{a_{0}(\xi ) \Psi_{0}^{+}(r,\xi)},
\end{equation}
where
\begin{equation*}
|a_{0}(\xi)| \sim
1. %\begin{cases}
%1
%&
%~\mbox{ for }~ 0<\xi \le 1
%\\
%\xi^{1/2}
%&
%~\mbox{ for }~ \xi \geq 1
%\end{cases}\sim \langle \xi\rangle^{1/2}
\end{equation*}
The spectral measure $\rho_0$ has density estimate
\begin{equation*}
	\frac{d \rho_0}{d \xi}(\xi)  \sim |a_0(\xi)|^{-2}
	\sim 1.
	%\begin{cases}
	%	1
	%	&
	%	~\mbox{ for }~  0<\xi \le 1
	%	\\
	%	\xi^{-1}(\log\xi)^{-2}
	%	&
	%	~\mbox{ for }~  \xi \geq 1
	%\end{cases}\sim \langle \xi\rangle^{-1}.
\end{equation*}
\end{prop}
\begin{proof}
We follow the argument in \cite[Proposition 5.7]{Krieger08}.
\begin{equation*}
a_{0}(\xi) = \frac{i}{2} W(\Phi^{0}, \overline{\Psi_{0}^{+}} ) = \frac{i}{2}
(\Phi^{0}(r,\xi) \pp_{r} \overline{\Psi_{0}^{+} (r,\xi) } - \overline{\Psi_{0}^{+} (r,\xi)  } \pp_{r} \Phi^{0} (r,\xi) ) .
\end{equation*}
We evaluate the Wronskian in the region where both the $\Psi_{0}^{+} (r,\xi)$ and $\Phi^{0} (r,\xi) $ asymptotics are useful, i.e., where $r^2 \xi \sim 1$.  Recall that for $r^2 \xi \sim 1$, one has
$$
\Phi^0(r,\xi)\sim \frac{r^{1/2}(r^2-1)}{r^2+1}+r^{1/2}\Phi_1^0(r^2),
$$
and
\begin{equation*}
\Phi_1^0(u)\sim
\left\{
\begin{aligned}
&-\frac14+\frac{u}4+O(u^2)~&\mbox{ for }~ 0<u\leq 1,\\
&\frac14-\frac{\log^2 u}{2u}+O\left(\frac{\log u}{u}\right)~&\mbox{ for }~ u\geq 1.\\
\end{aligned}
\right.
\end{equation*}
For all $\xi >0$,  it follows from Lemma \ref{lem-323232} that
$$
|\Psi_{0}^{+}(r,\xi)| \lesssim  \xi^{-1/4},\quad |\pp_{r}\Psi_{0}^{+}(r,\xi)| \lesssim  \xi^{1/4}.
$$
For $r\sim \xi^{-1/2}$, we have
$$
|\Phi^{0}(r,\xi)| \lesssim \xi^{-1/4} ,\quad |\pp_{r} \Phi^{0}(r,\xi)| \lesssim \xi^{1/4},
$$
and thus
$$
|a_{0}(\xi)| \lesssim 1.
$$
Therefore, we have
%\begin{equation*}
%|a_{0}(\xi)| \lesssim
%\begin{cases}
%1
%&
%~\mbox{ for }~  0<\xi \le 1
%\\
%\xi^{1/2}
%&
%~\mbox{ for }~  \xi \geq 1
%\end{cases}
%\sim \langle \xi\rangle^{1/2},
%\end{equation*}
%and
\begin{equation*}
	\frac{d \rho_0}{d \xi}(\xi)  \sim |a_0(\xi)|^{-2}
	\gtrsim 1.%\begin{cases}
%1
%&
%~\mbox{ for }~  0<\xi \le 1
%\\
%\xi^{-1}
%&
%~\mbox{ for }~  \xi \geq 1
%\end{cases}\sim \langle \xi\rangle^{-1}.
	\end{equation*}
In order to get the lower bound of $|a_{0}(\xi)| $, by \eqref{41}, we have
\begin{equation*}
|a_{0}(\xi)| \ge \frac{|\Phi^{0}(r,\xi)|}{2 |\Psi_{0}^{+} (r,\xi) |}.
\end{equation*}
For $0<\xi \le 1$, we take $r=M\xi^{-1/2}$ with a sufficiently large $M$. Then above estimates imply
$$
|\Phi^0(r,\xi)|\gtrsim \xi^{-1/4}.
$$
For $\xi\geq1,$ we similarly take $r=m\xi^{-1/2}$ for a sufficiently small $m$. Then
$$
|\Phi^0(r,\xi)|\gtrsim \xi^{-1/4}.
$$
Collecting above estimates, we conclude the estimate of density
\begin{equation*}
	\frac{d \rho_0}{d \xi}(\xi)  \sim  1 % \begin{cases}
%1
%&
%~\mbox{ for }~  0<\xi \le 1
%\\
%\xi^{-1}
%&
%~\mbox{ for }~  \xi \geq 1
%\end{cases} \sim \langle \xi\rangle^{-1}
\end{equation*}
as desired.
\end{proof}

\medskip

\subsection{Duhamel's representation by DFT}\label{DFT}

\medskip

Now we formulate the Duhamel's formula for the linearization at mode 0
$$
\partial_{\tau} \phi=\partial_{\rho\rho}\phi+\frac1{\rho}\partial_{\rho} \phi+V_0\phi+h,
$$
where
$$
V_0=\frac{8}{(1+\rho^2)^2},
$$
and $h$ has fast decay in space-time. If we take $\phi=\rho^{-1/2} A$, we get
$$
\partial_{\tau} A=A''+\frac{A}{4\rho^2}+V_0 A+h\rho^{1/2}=\mathcal L_0 A+h\rho^{1/2}.
$$
Consider the Cauchy problem
\begin{equation}\label{eqn-Lmode0}
\begin{cases}
\partial_{\tau} A=\mathcal L_0 A+h\rho^{1/2},\\
A(\cdot,t_0)=0.\\
\end{cases}
\end{equation}
For
$$
\mathcal L_0\Phi^0=-\xi\Phi^0(\rho,\xi),
$$
the generalized eigenfunction satisfies the following pointwise upper bound
\begin{equation}\label{est-geigen}
|\Phi^0(\rho,\xi)|\lesssim
\left\{
\begin{aligned}
&\rho^{1/2},\quad &\rho^2\xi\ll 1,\\
&\xi^{-1/4},\quad &\rho^2\xi\gg1,\\
\end{aligned}
\right.
\end{equation}
and its associated spectral measure has density estimate
\begin{equation}\label{est-sm}
	\frac{d \rho_0}{d \xi}(\xi)  \sim  1.
	\end{equation}
Taking distorted Fourier transform on both sides of \eqref{eqn-Lmode0}, we obtain
\begin{equation*}
\left\{
\begin{aligned}
&\partial_{\tau} \hat A(\xi,{\tau})=-\xi\hat A(\xi,{\tau})+\int_0^{\infty}h\rho^{1/2}\Phi^0(\rho,\xi)d\rho,\\
&\hat A(\cdot,{\tau}_0)=0.\\
\end{aligned}
\right.
\end{equation*}
Therefore, one has
\begin{equation*}
\phi(\rho,{\tau})=\rho^{-1/2}A(\rho,{\tau})=\int_{{\tau}_0}^{\tau}\int_0^{\infty}\int_{0}^{\infty} e^{-\xi({\tau}-s)}\rho^{-1/2}\Phi^0(\rho,\xi)\Phi^0(x,\xi)x^{1/2}h(x,s)\frac{d\rho_0}{d\xi}(\xi)d\xi dx ds.
\end{equation*}
Now assume the RHS $h$ is smooth enough with the pointwise space-time control
$$
|h(\rho,{\tau})|\lesssim v(\tau)\langle \rho\rangle^{-\ell}\| h \|^{(\tau)}_{v,\ell} ,\quad ~\ell>\frac32,
$$
where $v(\tau)$ is a regular function decay in $\tau$. Then
$$
|\phi|\lesssim \| h \|^{(\tau)}_{v,\ell}\int_{{\tau}_0}^{\tau} ds \int_{0}^{\infty} v(s)e^{-\xi({\tau}-s)}\rho^{-1/2}\Phi^0(\rho,\xi)\frac{d\rho_0}{d\xi}(\xi)  d\xi\int_0^{\infty}\Phi^0(x,\xi)x^{1/2} \langle x\rangle^{-\ell}dx.
$$

Before estimating $|\phi|$, we first invoke a lemma that will frequently used below. The following lemma is proved in \cite[Lemma A.3]{LLG}.
\begin{lemma}[\cite{LLG}]\label{asy-lem-1}
	Assume constants $a$, $b$ satisfy either $a>-1$ or $a=-1$ and $b<-1$.  For $0\le x_{0} \le x_{1} \le \frac 12$, we have	
	\begin{equation}\label{general case}
		\begin{aligned}
			\int_{x_{0}}^{x_{1}}  e^{-\lambda x} x^{a}
			(- \log  x)^{b} dx
			\lesssim
			\begin{cases}
				\begin{cases}
					x_{1}^{a+1}  (- \log  x_{1})^{b} & \mbox{ \ if \ } a>-1
					\\
					( - \log  x_{1})^{b+1}
					-
					( - \log  x_{0})^{b+1}
					& \mbox{ \ if \ } a=-1, b<-1
				\end{cases}
				&
				\mbox{ \ for \ }
				0\le \lambda \le x_{1}^{-1}
				\\
				\frac{ ( \log  \lambda)^b }{\lambda^{a+1}}
				+
				\begin{cases}
					0
					&
					\mbox{ \ if \ } a>-1
					\\
					( \log  \lambda)^{b+1} - (- \log  x_{0})^{b+1}
					&
					\mbox{ \ if \ } a=-1, b<-1
				\end{cases}
				&
				\mbox{ \ for \ }
				x_{1}^{-1} \le \lambda  \le x_{0}^{-1}
				\\
				\frac{ ( \log  \lambda)^b }{\lambda^{a+1}} e^{-\frac{x_{0 } \lambda}{2}}
				&
				\mbox{ \ for \ }  \lambda \ge x_{0}^{-1}
			\end{cases}
			.
		\end{aligned}
	\end{equation}		
\end{lemma}

The linear theory for the mode $0$ is stated as follows, and its proof is in a similar spirit as \cite[Proposition 9.8]{LLG}.

\begin{proof}[Proof of Proposition \ref{prop-lt-W0}]

We first estimate the Duhamel's representation without imposing orthogonality on the RHS $h(\rho,\tau)$ and compute
\begin{equation*}
\begin{aligned}
|F(\xi)|:=&~\left|\int_0^{\infty}\Phi^0(x,\xi)x^{1/2}\langle x\rangle^{-\ell} dx\right|
\lesssim \int_0^{\xi^{-1/2}}x\langle x\rangle^{-\ell}  dx+\xi^{-1/4}\int_{\xi^{-1/2}}^{\infty}x^{1/2}\langle x\rangle^{-\ell} dx\\
\lesssim &~\begin{cases}
\xi^{-1/4} ~&\mbox{ for }~\xi\geq 1,\\
\begin{cases}
\xi^{\frac{\ell}{2}-1},~&\ell<2\\
\log\xi,~&\ell=2\\
1,~&\ell>2\\
\end{cases}
 ~&\mbox{ for }~\xi\leq 1.\\
\end{cases}
\end{aligned}
\end{equation*}

Now we consider the integration in $\xi$:
\begin{equation*}
\begin{aligned}
\left|\int_{0}^{\infty} e^{-\xi({\tau}-s)}\rho^{-1/2}\Phi^0(\rho,\xi)F(\xi)\frac{d\rho_0}{d\xi}(\xi) d\xi\right| = &~\left(\int_0^{\rho^{-2}} + \int_{\rho^{-2}}^{\infty}\right)\cdots:=P_1 + P_2.
\end{aligned}
\end{equation*}

\noindent $\bullet$ For $P_1$, if $\rho>1$, then we have
\begin{equation*}
	\begin{aligned}
		|P_{1}| \lesssim \ &
		\int_{0}^{\rho^{-2}}
		e^{-  \xi (\tau -s)}
\begin{cases}
	\xi^{\frac{\ell}{2} -1}
	&
	\mbox{ \ if  \ } \ell<2
	\\
	 \log \xi
	&
	\mbox{ \ if  \ } \ell=2
	\\
	1
	&
	\mbox{ \ if  \ } \ell>2
\end{cases}
		~~d\xi
\lesssim
\begin{cases}
\begin{cases}
	\rho^{-\ell }
&
\mbox{ \ if \ } \tau -s \le \rho^2
\\
(\tau-s)^{-\frac{\ell}{2}}
&
\mbox{ \ if \ } \tau -s > \rho^2
\end{cases}
	&
	\mbox{ \ if  \ } \ell<2
	\\
\begin{cases}
\rho^{-2} \langle  \log  \rho \rangle
	&
	\mbox{ \ if \ } \tau -s \le \rho^2
	\\
 (\tau-s)^{-1} \langle \log(\tau-s) \rangle
	&
	\mbox{ \ if \ } \tau -s > \rho^2
\end{cases}
	&
	\mbox{ \ if  \ } \ell=2
	\\
\begin{cases}
	\rho^{-2}
	&
	\mbox{ \ if \ } \tau -s \le \rho^2
	\\
 (\tau-s)^{-1}
	&
	\mbox{ \ if \ } \tau -s > \rho^2
\end{cases}
	&
	\mbox{ \ if  \ } \ell>2
\end{cases}
	\end{aligned}
\end{equation*}
If $\rho\leq 1$, then one has
\begin{equation*}
\begin{aligned}
P_1=&~\int_0^{\rho^{-2}}e^{-\xi({\tau}-s)}\rho^{-1/2}\Phi^0(\rho,\xi)\frac{d\rho_0}{d\xi}(\xi)F(\xi) d\xi\\
=&~\left(\int_0^1+\int_1^{\rho^{-2}}\right)e^{-\xi({\tau}-s)}\rho^{-1/2}\Phi^0(\rho,\xi)\frac{d\rho_0}{d\xi}(\xi)F(\xi) d\xi\\
:=&~P_{11}+P_{12},
\end{aligned}
\end{equation*}
where we estimate
\begin{equation*}
\begin{aligned}
|P_{11}|\lesssim&~\int_0^1 e^{-  \xi (\tau -s)}
\begin{cases}
	\xi^{\frac{\ell}{2} -1}
	&
	\mbox{ \ if  \ } \ell<2
	\\
	 \log \xi
	&
	\mbox{ \ if  \ } \ell=2
	\\
	1
	&
	\mbox{ \ if  \ } \ell>2
\end{cases}
		~~d\xi
\lesssim \begin{cases}
	\begin{cases}
	1
		&
		\mbox{ \ if \ } \tau -s \le 1
		\\
		(\tau-s)^{-\frac{\ell}{2}}
		&
		\mbox{ \ if \ } \tau -s > 1
	\end{cases}
	&
	\mbox{ \ for  \ } \ell<2
	\\
	\begin{cases}
		1
		&
		\mbox{ \ if \ } \tau -s \le 1
		\\
		(\tau-s)^{-1} \langle \log(\tau-s) \rangle
		&
		\mbox{ \ if \ } \tau -s > 1
	\end{cases}
	&
	\mbox{ \ for  \ } \ell=2
	\\
	\begin{cases}
		1
		&
		\mbox{ \ if \ } \tau -s \le 1
		\\
		(\tau-s)^{-1}
		&
		\mbox{ \ if \ } \tau -s > 1
	\end{cases}
	&
	\mbox{ \ for  \ } \ell>2
\end{cases}
\end{aligned}
\end{equation*}
and
\begin{equation*}
\begin{aligned}
|P_{22}|\lesssim&~\int_1^{\rho^{-2}} e^{-  \xi (\tau -s)} \xi^{-1/4} d\xi
\lesssim
\begin{cases}
	\rho^{-\frac 32}
&
\mbox{ \ if \ } \tau-s \le \rho^2
\\
(\tau-s)^{-\frac 34}
&
\mbox{ \ if \ } \rho^2 < \tau-s \le 1
\\
(\tau-s)^{-\frac 34}
e^{-\frac{\tau-s}{2}}
&
\mbox{ \ if \ }  \tau-s > 1
\end{cases}.
\end{aligned}
\end{equation*}
Hence one obtains for $\rho\leq 1$
\begin{equation*}
	|P_{1}| \lesssim
	\begin{cases}
		\begin{cases}
\rho^{-\frac32}
&
\mbox{ \ if \ } \tau-s \le \rho^2
\\
	 (\tau-s)^{-\frac 34}
&
\mbox{ \ if \ } \rho^2 < \tau-s \le 1
\\
	(\tau-s)^{-\frac{\ell}{2}}
			&
			\mbox{ \ if \ } \tau -s > 1
		\end{cases}
		&
		\mbox{ \ for  \ } \ell<2
		\\
		\begin{cases}
		\rho^{-\frac32}
		&
		\mbox{ \ if \ } \tau-s \le \rho^2
		\\
		(\tau-s)^{-\frac 34}
		&
		\mbox{ \ if \ } \rho^2 < \tau-s \le 1
		\\
				(\tau-s)^{-1} \langle \log(\tau-s ) \rangle
			&
			\mbox{ \ if \ } \tau -s > 1
		\end{cases}
		&
		\mbox{ \ for  \ } \ell=2
		\\
		\begin{cases}
			\rho^{-\frac32}
			&
			\mbox{ \ if \ } \tau-s \le \rho^2
			\\
			(\tau-s)^{-\frac 34}
			&
			\mbox{ \ if \ } \rho^2 < \tau-s \le 1
			\\
				(\tau-s)^{-1}
			&
			\mbox{ \ if \ } \tau -s > 1
		\end{cases}
		&
		\mbox{ \ for  \ } \ell>2
	\end{cases}
.
\end{equation*}

\noindent $\bullet$ For $P_2$, if $\rho\leq 1$, we estimate
\begin{equation*}
\begin{aligned}
	|P_{2}| \lesssim \rho^{-1/2}\ &
	\int_{\rho^{-2}}^{\infty}
	e^{- \xi (\tau -s)}
	\xi^{-\frac 12} 	d\xi
\lesssim
\begin{cases}
\rho^{-1/2}(\tau-s)^{-\frac 12}
&
\mbox{ \ if \ } \tau-s \le \rho^2
\\
\rho^{-1/2}(\tau-s)^{-\frac 12}
e^{- \frac{\tau-s}{2\rho^2}}
&
\mbox{ \ if \ } \tau-s > \rho^2
\end{cases}
.
\end{aligned}
\end{equation*}
If $\rho>1$, then we have
\begin{equation*}
\begin{aligned}
P_2=&~\int^{\infty}_{\rho^{-2}}e^{-\xi({\tau}-s)}\rho^{-1/2}\Phi^0(\rho,\xi)\frac{d\rho_0}{d\xi}(\xi)F(\xi) d\xi\\
=&~\left(\int_{\rho^{-2}}^1+\int_1^{\infty}\right)e^{-\xi({\tau}-s)}\rho^{-1/2}\Phi^0(\rho,\xi)\frac{d\rho_0}{d\xi}(\xi)F(\xi) d\xi\\
:=&~P_{21}+P_{22},
\end{aligned}
\end{equation*}
where
\begin{equation*}
\begin{aligned}
|P_{21}|\lesssim &~\rho^{-1/2}\int_{\rho^{-2}}^1 e^{-  \xi (\tau -s)} \xi^{-1/4}
\begin{cases}
	\xi^{\frac{\ell}{2} -1}
	&
	\mbox{ \ if  \ } \ell<2
	\\
	 \log \xi
	&
	\mbox{ \ if  \ } \ell=2
	\\
	1
	&
	\mbox{ \ if  \ } \ell>2
\end{cases}
		~~d\xi\\
		\lesssim &~\rho^{-1/2}
\begin{cases}
\begin{cases}
1
&
\mbox{ \ if \ } \tau-s\le 1
\\
(\tau-s)^{\frac 14 - \frac{\ell}{2}}
&
\mbox{ \ if \ } 1<\tau-s\le \rho^2
\\
(\tau-s)^{\frac 14 - \frac{\ell}{2}}
e^{-\frac{\tau-s}{2\rho^2}}
&
\mbox{ \ if \ } \tau-s >\rho^2
\end{cases}
	&
	\mbox{ \ for  \ } \ell<2
	\\
\begin{cases}
	1
	&
	\mbox{ \ if \ } \tau-s\le 1
	\\
	(\tau-s)^{-\frac 34}
	\langle \log (\tau-s) \rangle
	&
	\mbox{ \ if \ } 1<\tau-s\le \rho^2
	\\
	(\tau-s)^{-\frac 34}
	\langle \log (\tau-s) \rangle
	e^{-\frac{\tau-s}{2\rho^2}}
	&
	\mbox{ \ if \ } \tau-s >\rho^2
\end{cases}
	&
	\mbox{ \ for  \ } \ell=2
	\\
\begin{cases}
	1
	&
	\mbox{ \ if \ } \tau-s\le 1
	\\
	(\tau-s)^{-\frac 34}
	&
	\mbox{ \ if \ } 1<\tau-s\le \rho^2
	\\
	(\tau-s)^{-\frac 34}
	e^{-\frac{\tau-s}{2\rho^2}}
	&
	\mbox{ \ if \ } \tau-s >\rho^2
\end{cases}
	&
	\mbox{ \ for  \ } \ell>2
\end{cases}
\end{aligned}
\end{equation*}
and
\begin{equation*}
\begin{aligned}
|P_{22}|\lesssim&~\rho^{-1/2}\int_{1}^{\infty}
	e^{- \xi (\tau -s)}
	\xi^{-\frac 12}
	d\xi
	\lesssim \rho^{-1/2}
	\begin{cases}
		(\tau-s)^{-\frac 12}
		&
		\mbox{ \ if \ } \tau-s \le 1
		\\
		(\tau-s)^{-\frac 12}
		e^{- \frac{\tau-s}{2}}
		&
		\mbox{ \ if \ } \tau-s > 1
	\end{cases}
\end{aligned}
\end{equation*}
Therefore, we obtain for $\rho>1$ that
\begin{equation*}
	|P_2| \lesssim \rho^{-1/2}
	\begin{cases}
		\begin{cases}
			(\tau-s)^{-\frac 12}
			&
			\mbox{ \ if \ } \tau-s\le 1
			\\
			(\tau-s)^{\frac 14 - \frac{\ell}{2}}
			&
			\mbox{ \ if \ } 1<\tau-s\le \rho^2
			\\
			(\tau-s)^{\frac 14 - \frac{\ell}{2}}
			e^{-\frac{\tau-s}{4\rho^2}}
			&
			\mbox{ \ if \ } \tau-s >\rho^2
		\end{cases}
		&
		\mbox{ \ for  \ } \ell<2
		\\
		\begin{cases}
			(\tau-s)^{-\frac 12}
			&
			\mbox{ \ if \ } \tau-s\le 1
			\\
			(\tau-s)^{-\frac 34}
			\langle \log (\tau-s) \rangle
			&
			\mbox{ \ if \ } 1<\tau-s\le \rho^2
			\\
			(\tau-s)^{-\frac 34}
			\langle \log (\tau-s) \rangle
			e^{-\frac{\tau-s}{4 \rho^2}}
			&
			\mbox{ \ if \ } \tau-s >\rho^2
		\end{cases}
		&
		\mbox{ \ for  \ } \ell=2
		\\
		\begin{cases}
			(\tau-s)^{-\frac 12}
			&
			\mbox{ \ if \ } \tau-s\le 1
			\\
			(\tau-s)^{-\frac 34}
			&
			\mbox{ \ if \ } 1<\tau-s\le \rho^2
			\\
			(\tau-s)^{-\frac 34}
			e^{-\frac{\tau-s}{4 \rho^2}}
			&
			\mbox{ \ if \ } \tau-s >\rho^2
		\end{cases}
		&
		\mbox{ \ for  \ } \ell>2
	\end{cases}
	.
\end{equation*}

\medskip

Collecting above estimates, we conclude that

\noindent $\bullet$ for $\rho\le 1$,
\begin{equation*}
	|P_1+P_2| \lesssim
	\begin{cases}
		\begin{cases}
\rho^{-1/2}(\tau-s)^{-\frac 12}
&
\mbox{ \ if \ } \tau-s \le \rho^2
\\
	 (\tau-s)^{-\frac 34}
&
\mbox{ \ if \ } \rho^2 < \tau-s \le 1
\\
	(\tau-s)^{-\frac{\ell}{2}}
			&
			\mbox{ \ if \ } \tau -s > 1
		\end{cases}
		&
		\mbox{ \ for  \ } \ell<2
		\\
		\begin{cases}
		\rho^{-1/2}(\tau-s)^{-\frac 12}
		&
		\mbox{ \ if \ } \tau-s \le \rho^2
		\\
		(\tau-s)^{-\frac 34}
		&
		\mbox{ \ if \ } \rho^2 < \tau-s \le 1
		\\
				(\tau-s)^{-1} \langle \log(\tau-s ) \rangle
			&
			\mbox{ \ if \ } \tau -s > 1
		\end{cases}
		&
		\mbox{ \ for  \ } \ell=2
		\\
		\begin{cases}
			\rho^{-1/2}(\tau-s)^{-\frac 12}
			&
			\mbox{ \ if \ } \tau-s \le \rho^2
			\\
			(\tau-s)^{-\frac 34}
			&
			\mbox{ \ if \ } \rho^2 < \tau-s \le 1
			\\
				(\tau-s)^{-1}
			&
			\mbox{ \ if \ } \tau -s > 1
		\end{cases}
		&
		\mbox{ \ for  \ } \ell>2
	\end{cases}
,
\end{equation*}

\noindent $\bullet$  for $\rho>1$,
\begin{equation*}
	|P_1+P_2| \lesssim
	\begin{cases}
		\begin{cases}
			\rho^{-1/2}(\tau-s)^{-\frac 12}
			&
			\mbox{ \ if \ } \tau-s\le 1
			\\
			\rho^{-1/2}(\tau-s)^{\frac 14 - \frac{\ell}{2}}
			&
			\mbox{ \ if \ } 1<\tau-s\le \rho^2
			\\
			(\tau-s)^{-\frac{l}{2}}
			&
			\mbox{ \ if \ } \tau-s >\rho^2
		\end{cases}
		&
		\mbox{ \ for  \ } \ell<2
		\\
		\begin{cases}
			\rho^{-1/2}(\tau-s)^{-\frac 12}
			&
			\mbox{ \ if \ } \tau-s\le 1
			\\
			\rho^{-1/2}(\tau-s)^{-\frac 34}
			\langle \log (\tau-s) \rangle
			&
			\mbox{ \ if \ } 1<\tau-s\le \rho^2
			\\
			(\tau-s)^{-1}\langle\log(\tau-s)\rangle
			&
			\mbox{ \ if \ } \tau-s >\rho^2
		\end{cases}
		&
		\mbox{ \ for  \ } \ell=2
		\\
		\begin{cases}
			\rho^{-1/2}(\tau-s)^{-\frac 12}
			&
			\mbox{ \ if \ } \tau-s\le 1
			\\
			\rho^{-1/2}(\tau-s)^{-\frac 34}
			&
			\mbox{ \ if \ } 1<\tau-s\le \rho^2
			\\
			(\tau-s)^{-1}
			&
			\mbox{ \ if \ } \tau-s >\rho^2
		\end{cases}
		&
		\mbox{ \ for  \ } \ell>2
	\end{cases}
	.
\end{equation*}

\medskip

We now estimate the convolution in time.

\medskip

\noindent $\bullet$ For $\rho\leq 1$, one has

\begin{align*}
\left|	\phi(\rho,\tau)
\right|
\lesssim &~
\left(
\int_{\tau-\rho^2}^{\tau}
+
\int_{\tau-1}^{\tau-\rho^2}
+
\int_{\frac{\tau_{0}}{2}}^{\tau-1}
\right)
v(s) \Big|P_1(\rho,\tau,s)+P_2(\rho,\tau,s)\Big| ds
\\
\lesssim &~
\rho^{-1/2} v(\tau)\int_{\tau-\rho^2}^{\tau}
(\tau-s)^{-\frac 12} ds
+
v(\tau)\int_{\tau-1}^{\tau-\rho^2}
(\tau-s)^{-\frac34} ds
\\
&~
+
\begin{cases}
\int_{\frac{\tau_{0}}{2}}^{\tau-1}
v(s) (\tau-s)^{-\frac{\ell}{2}} ds
 &
 \mbox{ \ if \ } \ell <2
 \\
 \int_{\frac{\tau_{0}}{2}}^{\tau-1}
v(s) (\tau-s)^{-1} \langle \log(\tau-s ) \rangle ds
&
\mbox{ \ if \ } \ell =2
\\
\int_{\frac{\tau_{0}}{2}}^{\tau-1}
v(s) (\tau-s)^{-1} ds
&
\mbox{ \ if \ } \ell >2
\end{cases}
\\\lesssim & ~
v(\tau)
+
\begin{cases}
	v(\tau) \tau^{1-\frac{\ell}{2}}
+
\tau^{-\frac{\ell}{2}}
\int_{\frac{\tau_{0}}{2}}^{ \frac{\tau}{2} }
v(s)  ds
	&
	\mbox{ \ if \ } \ell <2
	\\
v(\tau) ( \log  \tau)^2
+
\tau^{-1}  \log  \tau
	\int_{\frac{\tau_{0}}{2}}^{\frac{\tau}{2}}
	v(s)   ds
	&
	\mbox{ \ if \ } \ell =2
	\\
		v(\tau)  \log  \tau
		+
	\tau^{-1}
	\int_{\frac{\tau_{0}}{2}}^{ \frac{\tau}{2} } v(s) ds
	&
	\mbox{ \ if \ } \ell >2
\end{cases}
.
\end{align*}

\medskip

\noindent $\bullet$ For $1<\rho\leq (\frac{\tau}{2})^{1/2}$, we estimate
\begin{align*}
		\left|	\phi(\rho,\tau)
		\right|
		\lesssim &~
		\left(
		\int_{\tau -1}^{\tau}
		+	
\int_{\tau -\rho^2}^{\tau-1}
+
\int_{\frac{\tau_0}{2}}^{\tau -\rho^2}
		\right)
		v(s) \Big|P_1(\rho,\tau,s)+P_2(\rho,\tau,s)\Big| ds
		\\
\lesssim  & ~
v(\tau)\rho^{-1/2}
\int_{\tau -1}^{\tau} (\tau-s)^{-\frac 12} ds
+	
v(\tau)\rho^{-1/2}
\int_{\tau -\rho^2}^{\tau-1}
\begin{cases}
(\tau-s)^{\frac 14 - \frac{\ell}{2}}
&
\mbox{ \ if \ }\ell <2
\\
(\tau-s)^{-\frac 34}
\langle \log (\tau-s) \rangle
&
\mbox{ \ if \ }\ell =2
\\
(\tau-s)^{-\frac 34}
&
\mbox{ \ if \ }\ell >2
\end{cases}
~~ ds
\\
&
+
\int_{\frac{\tau_0}{2}}^{\tau -\rho^2}
v(s)
\begin{cases}
(\tau-s)^{-\frac{\ell}{2}}
&
\mbox{ \ if  \ } \ell <2
\\
(\tau-s)^{-1} \langle \log(\tau-s) \rangle
&
\mbox{ \ if  \ } \ell =2
\\
(\tau-s)^{-1}
&
\mbox{ \ if  \ } \ell >2
\end{cases}
~~ds
\\
\lesssim  & ~
\rho^{-1/2}
v(\tau)
+	
v(\tau)
\begin{cases}
\rho^{2 -\ell}
	&
	\mbox{ \ if \ }\ell <2
	\\
 \langle \log \rho \rangle
	&
	\mbox{ \ if \ }\ell =2
	\\
1
	&
	\mbox{ \ if \ }\ell >2
\end{cases}
+
\begin{cases}
	v(\tau) \tau^{1-\frac{\ell}{2}}
	+
	\tau^{-\frac{\ell}{2}}
	\int_{\frac{\tau_0}{2}}^{\frac{\tau}{2}}
	v(s) ds
	&
	\mbox{ \ if  \ } \ell <2
	\\
v(\tau)
\int_{\rho^2}^{\frac{\tau}{2}} \langle \log z \rangle z^{-1} dz
+
\tau^{-1}  \log  \tau
	\int_{\frac{\tau_0}{2}}^{\frac{\tau}{2}}
	v(s) ds
	&
	\mbox{ \ if  \ } \ell =2
	\\
	v(\tau) \log (\frac{\tau}{2\rho^2} )
	+
	\tau^{-1}
	\int_{\frac{\tau_0}{2}}^{\frac{\tau}{2}}
	v(s) ds
	&
	\mbox{ \ if  \ } \ell >2
\end{cases}
\\
\lesssim & ~
\begin{cases}
	v(\tau) \tau^{1-\frac{\ell}{2}}
	+
	\tau^{-\frac{\ell}{2}}
	\int_{\frac{\tau_0}{2}}^{\frac{\tau}{2}}
	v(s) ds
	&
	\mbox{ \ if  \ } \ell <2
	\\
	v(\tau) (\langle \log \rho \rangle +
	\int_{\rho^2}^{\frac{\tau}{2}} \langle \log z \rangle z^{-1} dz)
	+
	\tau^{-1} \log \tau
	\int_{\frac{\tau_0}{2}}^{\frac{\tau}{2}}
	v(s) ds
	&
	\mbox{ \ if  \ } \ell =2
	\\
	v(\tau) \langle \log (\frac{\tau}{2\rho^2} ) \rangle
	+
	\tau^{-1}
	\int_{\frac{\tau_0}{2}}^{\frac{\tau}{2}}
	v(s) ds
	&
	\mbox{ \ if  \ } \ell >2
\end{cases}
.
	\end{align*}

\medskip

\noindent $\bullet$ For $(\frac{\tau}{2})^{1/2}\leq \rho \leq \tau^{1/2}$, we estimate
\begin{align*}
		\left|	\phi(\rho,\tau)
		\right|
		\lesssim &~
		\left(
		\int_{\tau -1}^{\tau}
		+	
		\int_{\tau -\rho^2}^{\tau-1}
		+
		\int_{\frac{\tau_0}{2}}^{\tau -\rho^2}
		\right)
		v(s) \Big|P_1(\rho,\tau,s)+P_2(\rho,\tau,s)\Big| ds
		\\
		\lesssim &~
		\rho^{-1/2}v(\tau)
		\int_{\tau -1}^{\tau} (\tau-s)^{-\frac 12} ds
		+	
		\rho^{-1/2}\int_{\tau -\rho^2}^{\tau-1}
		v(s)
		\begin{cases}
			(\tau-s)^{\frac 14 - \frac{\ell}{2}}
			&
			\mbox{ \ if \ }\ell <2
			\\
			(\tau-s)^{-\frac 34}
			\langle \log (\tau-s) \rangle
			&
			\mbox{ \ if \ }\ell =2
			\\
			(\tau-s)^{-\frac 34}
			&
			\mbox{ \ if \ }\ell >2
		\end{cases}
		~~ds
		\\
		&
		+
		\int_{\frac{\tau_0}{2}}^{\tau -\rho^2}
		v(s)
		\begin{cases}
			(\tau-s)^{-\frac{\ell}{2}}
			&
			\mbox{ \ if  \ } \ell <2
			\\
			(\tau-s)^{-1} \langle \log(\tau-s)  \rangle
			&
			\mbox{ \ if  \ } \ell =2
			\\
			(\tau-s)^{-1}
			&
			\mbox{ \ if  \ } \ell >2
		\end{cases}
		~~ds
		\\
		\lesssim & ~
\rho^{-1/2}v(\tau)
+	
\rho^{-1/2}v(\tau)
\begin{cases}
\tau^{\frac 54 - \frac{\ell}{2}}
	&
	\mbox{ \ if \ }\ell <2
	\\
	\tau^{\frac 14}
	\langle \log\tau \rangle
	&
	\mbox{ \ if \ }\ell =2
	\\
	\tau^{\frac 14}
	&
	\mbox{ \ if \ }\ell >2
\end{cases}
+
\rho^{-1/2}\int_{\tau -\rho^2}^{\frac{\tau}{2}}
v(s) ds
\begin{cases}
	\tau^{\frac 14 - \frac{\ell}{2}}
	&
	\mbox{ \ if \ }\ell <2
	\\
	\tau^{-\frac 34}
	\langle \log \tau \rangle
	&
	\mbox{ \ if \ }\ell =2
	\\
	\tau^{-\frac 34}
	&
	\mbox{ \ if \ }\ell >2
\end{cases}
\\
&
+
\int_{\frac{\tau_0}{2}}^{\tau -\rho^2}
v(s) ds
\begin{cases}
	\tau^{-\frac{\ell}{2}}
	&
	\mbox{ \ if  \ } \ell <2
	\\
	\tau^{-1} \langle \log\tau  \rangle
	&
	\mbox{ \ if  \ } \ell =2
	\\
\tau^{-1}
	&
	\mbox{ \ if  \ } \ell >2
\end{cases}
\\
	\lesssim  &~
\rho^{-1/2}v(\tau)
\begin{cases}
	\tau^{\frac 54 - \frac{\ell}{2}}
	&
	\mbox{ \ if \ }\ell <2
	\\
	\tau^{\frac 14}
	\langle \log\tau \rangle
	&
	\mbox{ \ if \ }\ell =2
	\\
	\tau^{\frac 14}
	&
	\mbox{ \ if \ }\ell >2
\end{cases}
+
\rho^{-1/2}\int_{ \frac{\tau_0}{2} }^{\frac{\tau}{2}}
v(s) ds
\begin{cases}
	\tau^{\frac 14 - \frac{\ell}{2}}
	&
	\mbox{ \ if \ }\ell <2
	\\
	\tau^{-\frac 34}
	\langle \log \tau \rangle
	&
	\mbox{ \ if \ }\ell =2
	\\
	\tau^{-\frac 34}
	&
	\mbox{ \ if \ }\ell >2
\end{cases}
\\
\sim & ~
\begin{cases}
		\tau^{1 - \frac{\ell}{2}} v(\tau) +
	\tau^{ - \frac{\ell}{2}} \int_{ \frac{\tau_0}{2} }^{\frac{\tau}{2}}
	v(s) ds
	&
	\mbox{ \ if \ }\ell <2
	\\
		\langle \log\tau \rangle v(\tau)+
	\tau^{-1}
	\langle \log \tau \rangle \int_{ \frac{\tau_0}{2} }^{\frac{\tau}{2}}
	v(s) ds
	&
	\mbox{ \ if \ }\ell =2
	\\
	v(\tau)
	+
	\tau^{-1} \int_{ \frac{\tau_0}{2} }^{\frac{\tau}{2}}
	v(s) ds
	&
	\mbox{ \ if \ }\ell >2
\end{cases}
.
	\end{align*}

\medskip

\noindent $\bullet$ For $\rho\geq \tau^{1/2}$, we estimate
\begin{align*}
\left|	\phi(\rho,\tau)
\right|
\lesssim &~
\left(
\int_{\tau -1}^{\tau}
+	
\int_{\frac{\tau_0}{2}}^{\tau-1}
\right)
v(s) \Big|P_1(\rho,\tau,s)+P_2(\rho,\tau,s)\Big| ds
\\
\lesssim &~
\rho^{-1/2}v(\tau)
\int_{\tau -1}^{\tau}
(\tau-s)^{-\frac 12} ds
+	
\rho^{-1/2}
\int_{\frac{\tau_0}{2}}^{\tau-1}
v(s)
\begin{cases}
(\tau-s)^{\frac 14 - \frac{\ell}{2}}
&
\mbox{ \ if  \ }\ell<2
\\
(\tau-s)^{-\frac 34}
\langle \log (\tau-s) \rangle
&
\mbox{ \ if  \ }\ell=2
\\
(\tau-s)^{-\frac 34}
&
\mbox{ \ if  \ }\ell>2
\end{cases}
~~ds
\\
\lesssim  &~
\rho^{-1/2}v(\tau)
+	
\rho^{-1/2}
\begin{cases}
	v(\tau) \tau^{\frac 54-\frac{\ell}{2}}
	+
	\tau^{\frac 14 - \frac{\ell}{2}}
	\int_{\frac{\tau_0}{2}}^{ \frac{\tau}{2} }
	v(s)  ds
	&
	\mbox{ \ if  \ }\ell<2
	\\
	v(\tau) \tau^{\frac 14} \langle \log \tau \rangle
	+
	\tau^{-\frac 34}
	\langle \log \tau \rangle
	\int_{\frac{\tau_0}{2}}^{ \frac{\tau}{2} }
	v(s) ds
	&
	\mbox{ \ if  \ }\ell=2
	\\
	v(\tau) \tau^{\frac 14 }
	+
	\tau^{-\frac 34 }
	\int_{\frac{\tau_0}{2}}^{ \frac{\tau}{2} }
	v(s)  ds
	&
	\mbox{ \ if  \ }\ell>2
\end{cases}
\\
\lesssim  &~
\rho^{-1/2}
\begin{cases}
	v(\tau) \tau^{\frac 54-\frac{\ell}{2}}
	+
	\tau^{\frac 14 - \frac{\ell}{2}}
	\int_{\frac{\tau_0}{2}}^{ \frac{\tau}{2} }
	v(s)  ds
	&
	\mbox{ \ if  \ }\ell<2
	\\
	v(\tau) \tau^{\frac 14} \langle \log \tau \rangle
	+
	\tau^{-\frac 34}
	\langle \log \tau \rangle
	\int_{\frac{\tau_0}{2}}^{ \frac{\tau}{2} }
	v(s) ds
	&
	\mbox{ \ if  \ }\ell=2
	\\
	v(\tau) \tau^{\frac 14 }
	+
	\tau^{-\frac 34 }
	\int_{\frac{\tau_0}{2}}^{ \frac{\tau}{2} }
	v(s)  ds
	&
	\mbox{ \ if  \ }\ell>2
\end{cases}
.
\end{align*}
Collecting above estimates, we obtain the a priori estimates without orthogonality. Estimates with orthogonality can be derived in a similar manner, where the orthogonality condition
\begin{equation*}
\int_0^{\infty} h(\rho,\tau)\frac{\rho^2-1}{\rho^2+1}\rho d\rho=0\quad \forall \tau\in(\tau_0,+\infty)
\end{equation*}
is used when estimating
\begin{equation*}
\left|\int_0^{\infty}\Phi^0(x,\xi)x^{1/2}h(x,s) dx\right|=\Bigg|\int_0^{\infty} \left(\Phi^0(x,\xi)-\frac{x^{1/2}(x^2-1)}{x^2+1}\right)x^{1/2}h(x,s) dx\Bigg|.
\end{equation*}
We refer the interested readers to \cite[Proposition 9.8]{LLG}.
\end{proof}

\bigskip

\section*{Acknowledgements}
Y. Sire is partially supported by the NSF DMS Grant 2154219, ``Regularity vs singularity formation in elliptic and parabolic equations''. J. Wei is partially supported by NSERC of Canada. Y. Zheng is partially supported by NSF of China (No. 12171355). Part of this work was accomplished when Y. Zheng visited the Department of Mathematical Sciences at Tsinghua University in spring 2023; he is very grateful to the institution and Professor Wenming Zou for the kind hospitality. Y. Zhou is supported in part by a start-up grant at Wuhan University.
\bigskip

%\bibliography{RefDatabase}{}
%\bibliographystyle{plain}

\end{document}